\documentclass[1 leqno,11pt,psfig]{amsart}
\usepackage{amssymb, amsmath,amsmath,latexsym,amssymb,amsfonts,amsbsy, amsthm,mathtools,graphicx,CJKutf8,CJKnumb,CJKulem,color}
\usepackage{graphicx,epsfig}
\usepackage{graphicx}
\usepackage{amssymb}
\usepackage{graphicx}
\usepackage{graphicx,epsfig}
\usepackage{amsmath}
\usepackage{mathrsfs}
\usepackage{amssymb}
\usepackage{float}
\usepackage{amsmath}

\usepackage{amssymb,  amsmath, amsmath, latexsym, amssymb, amsfonts, amsbsy, amsthm,mathtools, graphicx, CJKutf8, CJKnumb, CJKulem, extarrows, color, dsfont,verbatim}
\usepackage{amssymb, amssymb,  amsmath, amsmath, latexsym, amssymb, amsfonts, amsbsy, amsthm,mathtools, graphicx, CJKutf8, CJKnumb, CJKulem, extarrows, color, dsfont,verbatim}
\usepackage{amssymb, amsmath,amsmath,latexsym,amssymb,amsfonts,amsbsy, amsthm,mathtools,graphicx,CJKutf8,CJKnumb,CJKulem,color}
\usepackage{graphicx,epsfig}
\usepackage{graphicx}
\usepackage{amssymb,makecell}
\usepackage{graphicx}
\usepackage[numbers,compress]{natbib}
\usepackage{graphicx,epsfig}
\usepackage{amssymb}
\usepackage[numbers,compress]{natbib}
\usepackage{graphicx,epsfig}
\usepackage{mathrsfs}
\usepackage{amssymb}

\usepackage{amssymb,mathrsfs,amsfonts,amsmath,amsbsy,tikz}\usepackage{fontenc}\usepackage{textcomp}

\numberwithin{equation}{section}

\allowdisplaybreaks

\let\ep=\epsilon

\def\bbT{\mathbb{T}}

\newcommand{\beq}{\begin{equation}}
\newcommand{\eeq}{\end{equation}}
\newcommand{\ben}{\begin{eqnarray}}
\newcommand{\een}{\end{eqnarray}}
\newcommand{\beno}{\begin{eqnarray*}}
\newcommand{\eeno}{\end{eqnarray*}}

\newtheorem{theorem}{Theorem}[section]

\newtheorem{lemma}[theorem]{Lemma}

\newtheorem{remark}[theorem]{Remark}

\newtheorem{Theorem}{Theorem}[section]

\newtheorem{Proposition}[Theorem]{Proposition}
\newtheorem{Lemma}[Theorem]{Lemma}
\newtheorem{Corollary}[Theorem]{Corollary}
\newtheorem{Remark}[Theorem]{Remark}

\begin{document}
\begin{CJK*}{GBK}{song}
\title[The onset of instability for zonal stratospheric flows]{\textbf {The onset of instability for zonal stratospheric flows}}

\author{Adrian Constantin}
\address{Faculty of Mathematics, University of Vienna, Oskar-Morgenstern-Platz 1, 1090 Vienna, Austria}
\email{adrian.constantin@univie.ac.at}

\author{Pierre Germain}
\address{Department of Mathematics, Imperial College London, South Kensington Campus, London SW7 2AZ, United Kingdom}
\email{pgermain@ic.ac.uk}

\author{Zhiwu Lin}
\address{School of Mathematical Sciences, Fudan University,  200433, Shanghai, P. R. China}
\email{zwlin@fudan.edu.cn}

\author{Hao Zhu}
\address{School of Mathematics, Nanjing University,  210093, Nanjing, Jiangsu, P. R. China \&
 Faculty of Mathematics, University of Vienna, Oskar-Morgenstern-Platz 1, 1090 Vienna, Austria}
\email{haozhu@nju.edu.cn \& hao.zhu@univie.ac.at}

\date{\today}

\begin{abstract}
We investigate some qualitative aspects of the dynamics of the Euler equation on a rotating sphere that are relevant for stratospheric flows. Zonal flow dominates
the dynamics of the stratosphere and for most known planetary stratospheres the observed flow pattern is a small perturbation of an $n$-jet, which corresponds to choosing the Legendre polynomial of degree $n$ as the stream function. Since the $1$-jet and the $2$-jet are stable, the main interest is the onset of instability for the $3$-jet.
We confirm long-standing conjectures based on numerical simulations by proving that there exist positive and negative critical rotation rates $\omega_{cr}^+$ and $\omega_{cr}^-$ such that
the $3$-jet is linearly unstable if and only if $\omega \in (\omega_{cr}^-,\omega_{cr}^+)$. Turning to the nonlinear problem, we prove that linear instability implies nonlinear instability and that, as $\omega$ goes to infinity, nearby traveling waves gradually change from a cat's eyes streamline pattern to a profile with no stagnation points.
\end{abstract}

\maketitle


\noindent

\tableofcontents

\section{Introduction}

\subsection{Planetary stratospheric flows} 
In our galaxy, more than 3200 stars with planets orbiting them were detected, and our solar system alone has 8 planets, 5 dwarf planets
and over 160 moons. While some small astronomical bodies (comets, meteoroids, asteroids) are irregularly shaped, planets and dwarf planets are nearly spherical because
they are massive enough for their self-gravity to overcome the inherent strength of the materials they are made of. Large moons are
also nearly spherical but some less massive moons (for example, the eighth-largest moon of Saturn, Hyperion) are not rounded.
Astronomical bodies retain an atmosphere when their escape velocity is significantly larger than the average molecular velocity of the gases present
by gravitational accretion onto the celestial body or released from the celestial body itself. In our solar system all planets except Mercury and some moons have an atmosphere,
while some dwarf planets (e.g. Pluto) have an atmosphere and others (e.g. Makemake) lack one. A bewildering range of flows occur in the planetary atmospheres, mainly driven by the inhomogeneity of the energy input from the sun, with
the rotation acquired when the celestial body was formed (taking angular momentum from the impacts that shaped it) also playing an important role.
 Winds -- flows in a direction orthogonal to the normal direction -- are the most significant large-scale motions in the stratosphere of a spherical celestial body, with vertical movements
of a much smaller order of magnitude; in contrast to this, in the troposphere one can at least occasionally observe strong updrafts and downdrafts. In particular, the banded zonal cloud patterns
are among the most striking features of the visible atmospheres of Jupiter and Saturn  (see Fig. \ref{f1}).
The temperature inversion -- a stratosphere characteristic -- inhibits vertical flow, the dynamics in this atmospheric layer being inviscid and, due
to the stable stratification, layered. Consequently, stratospheric flow is governed at leading order at any fixed height by the Euler equation on the rotating unit sphere, written in non-dimensional form as
the evolution equation for the vorticity $(\mathcal{E}_\omega)$ which will be introduced below.
 \begin{table}[h!]
\centering

\begin{tabular}{| c | c|c|c|c|}
 \hline
 Celestial body & $R'$  & $\omega'$ & $U'$ &  $\omega$  \\
 \hline\hline
 Earth   & 6371 km  & $7.27\times 10^{-5}\,\text{rad/s}$ & 50 m/s & 9.26 \\
Jupiter &  69911 km  & $1.76\times 10^{-4}\,\text{rad/s}$ & 100 m/s  &  123  \\
Saturn &  58232 km  & $1.62\times 10^{-4}\,\text{rad/s}$ & 100 m/s   & 94.3  \\
Neptune &  24622 km &  $1.08\times 10^{-4}\,\text{rad/s}$ & 200 m/s  & 13.2  \\
Uranus &  25362 km   & $-1.04\times 10^{-4}\,\text{rad/s}$ & 200 m/s  & $-13.1$  \\
Pluto &  1188 km & $-1.1\times 10^{-5}\,\text{rad/s}$& 10 m/s & $-1.31$\\
Titan & 2576 km & $4.55\times 10^{-6}\,\text{rad/s}$ & 100 m/s & 0.11 \\
HD 209458b & 94380 km & $2.06\times 10^{-5}\,\text{rad/s}$ & 1940 m/s  & 1.01 \\
WASP-39b & 91000 km & $4.05\times 10^{-7}\,\text{rad/s}$ & 2000 m/s & 0.01 \\
 \hline
\end{tabular}
\vspace{0.2cm}
\caption{Approximate values of $\omega$ for astronomic bodies with a stratosphere.}
\label{tbl:Tb-omega}
\end{table}

 An important nondimensional parameter $\omega$ is defined in terms of the physical scales of the problem as
\begin{equation*}
\omega=\frac{\omega'R'}{U'}\,,
\end{equation*}
where $R'$ is the size of the planet which rotates with the speed $\omega'$ about its polar axis (measured counterclockwise with respect to the polar axis oriented towards 
the North Pole), having  zonal velocity scale $U'$ (see the discussion in \cite{cg22}). Using the data provided in \cite{cat, i, pe, rl},
we collect in Table \ref{tbl:Tb-omega} approximate values of $\omega$ for various astronomic bodies with a stratosphere: the planets Earth, Jupiter, Saturn,
Neptune, Uranus, HD 209458b, WASP-39b, the dwarf planet Pluto and the moon Titan. Note that the retrograde rotations about the polar
axis (spinning from east to west), exhibited by Uranus and Pluto, are atypical.

While the vorticity of atmospheric flows is typically calculated from velocity measurements, rather than being measured directly, the study of the evolution of the vorticity is at the core of
theoretical considerations in geophysics. This is even more so for quasi-two-dimensional flows since these showcase the emergence of long-lived vortices with the ability to self-propagate, such
structures being important in determining the weather and the climate.

\begin{figure}[H]
    \centering
	\includegraphics[scale = 0.4]{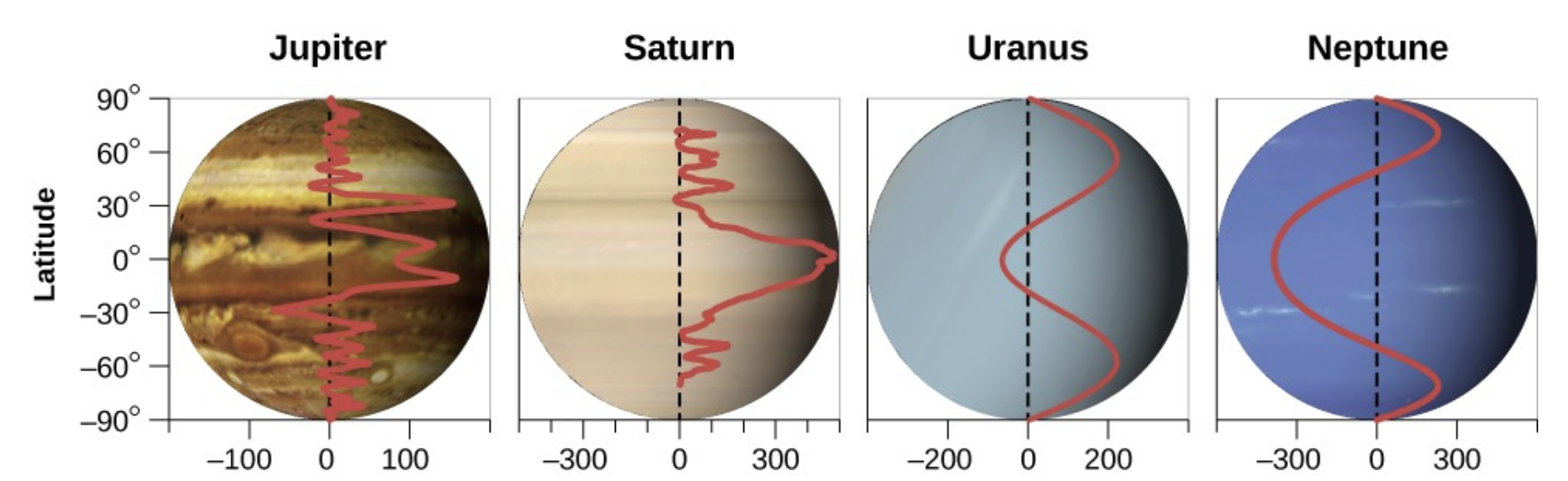}
	\caption{{\footnotesize The zonal wind profile of the outer planets in our solar system, measured in m/s relative to each planet's rotation speed (Credit: Open-Stax CNX). Zonal bandings
	are the most prominent visual features on Jupiter and Saturn, called ``zones" if they have an eastward jet along their poleward boundary and a westward one on the boundary
	nearest the equator and ``belts" if the direction of the jets along their boundary is reversed (on Jupiter, they have a strong color contrast as bright, respectively dark regions).
	Zonal flow also dominates the dynamics of the stratosphere of Uranus and Neptune, with a broad
	westward equatorial flow and an eastward flow at higher latitudes in each hemisphere. These pictures show the high altitude clouds just beneath the stratosphere (at the top of the
troposphere).}}
	\label{f1}
\end{figure}

Stratospheric winds can be decomposed into a steady zonal flow  -- the time and zonal average -- and
a wave-perturbation, the departures from the mean flow being generically called eddies \cite{vallis}. These steady zonal flows create transport barriers that have a crucial influence on mixing and confinement (e.g. the wintertime
stratospheric polar jet over Antarctica, a particularly robust flow, is effectively isolated from the rest of the atmosphere, enabling the chemistry of ozone depletion to take place \cite{mci}) and also give rise to
spectacular long-lasting patterns (e.g. Jupiter's Great Red Spot \cite{cj-jas} and Saturn's polar hexagon \cite{cj-jas2}) by shear instability.
Studies aimed at differentiating between stable and unstable mean flows in order to identify states that are expected to change rapidly due to growth of small amplitude perturbations
were initiated over half a century ago (see the discussion in \cite{pal}). Of interest is not only the stability of steady zonal flows under small perturbations, making them easily detectable to observations, but
also the type of patterns triggered by the onset of instability.

In our solar system the gas giants (Jupiter, Saturn) and the ice giants (Uranus and Neptune) have a stratosphere, as does the dwarf planet Pluto, but Earth is the only
terrestrial planet with a stratosphere, since Mercury lacks an atmosphere and on Venus and Mars the mesosphere is adjacent to the troposphere. Few moons in our solar system have
an atmosphere, and if so, the atmosphere is typically very thin. Titan, the second largest moon -- both by radius and mass -- in the solar system,
orbiting Saturn, is heralded as being the only natural satellite that has a fully developed atmosphere. However, Titan's atmosphere is denser than Earth's and
major aspects of its stratospheric motion remain to be explored. We note that, while the study of extrasolar planets in our galaxy is one of the fastest-growing subdisciplines in astronomy and planetary science,
understanding the dynamics of their atmospheres remains a daunting task where theoretical considerations can guide spectroscopic investigations. In this context, the exploration of the dynamics of the simplest dynamically
relevant zonal flows ($n$-jet combinations of retrograde and prograde zonal flows, similar to those that dominate the visible atmosphere of the outer planets in our solar system) can provide invaluable insight.
For example, the extrasolar planets HD 209458b and WASP-39b, about 160 and 700 light-years from Earth, respectively, are both known to have a stratosphere (see \cite{pe}).

The dynamics of the stratosphere is dominated by steady zonal flows with a coherent, banded structure showcasing characteristic jet-like velocity profiles. Zonal flows like the latitudinally-aligned
belts and zones observed on Jupiter and Saturn are not solely extra-terrestrial phenomena -- near the Earth's tropopause one finds the polar and subtropical jet streams which eastward-bound aircraft takes advantage of.

\subsection{Euler's equation on the sphere}
On the sphere $\mathbb{S}^2$, we choose the following coordinates:
the angle of longitude $\varphi \in [-\pi,\pi]$ and $s=\sin\theta$ with $\theta \in \big[-\frac{\pi}{2},\frac{\pi}{2}\big]$ the angle of latitude. Then the Euler equation written on the vorticity $\Upsilon(t,\varphi,s)$ is
\begin{equation*}
{\rm(\mathcal{E}_\omega)} \qquad \qquad\qquad\qquad\qquad
\partial_t \Upsilon + \Big( \partial_\varphi\Psi\,\partial_s - \partial_s\Psi \,\partial_\varphi \Big) (\Upsilon + 2\omega s)=0\,,\qquad\qquad\qquad\qquad\qquad\qquad
\end{equation*}
where the stream function is given by the Poisson equation ($\Delta$ denoting the Laplace-Beltrami operator on ${\mathbb S}^2$)
\begin{equation*}
\Upsilon=\Delta\Psi= \partial_s \Big( (1-s^2)\,\partial_s\Psi\Big) + \frac{1}{1-s^2}\,\partial_\varphi^2\Psi 
\end{equation*}
with zero mean.
The azimuthal and meridional velocity components can be expressed by means of the stream function $\Psi(t,\varphi,s)$ through
\begin{equation*}
u=-\sqrt{1-s^2}\,\partial_s\Psi \quad\text{and}\quad  v={\frac{1}{\sqrt{1-s^2}}}\,\partial_\varphi\Psi.
\end{equation*}
Denote $D_T=\mathbb{T}_{2\pi}\times [-1,1]$. Note that integrating the vorticity
over $D_T$ yields the Gauss constraint
$$\int_{D_T}{\Upsilon(t,\varphi,s)}\,{\rm d}s\,{\rm d}\varphi =0\,,$$
which ensures that one can recover the velocity field $(u,v)$ from the vorticity $\Upsilon$ by means of the Biot-Savart formula (see the discussion in \cite{cg22}).

Finally, the Euler equation on the sphere enjoys the following symmetry, which corresponds to changing the rotation of the frame of observation: $\psi_0$ solves $(\mathcal{E}_0)$ if and only if
\begin{equation}
\label{symmetry}
\psi_\omega(t,\varphi,s) = \psi_0(t,\varphi+\omega t,s) + \omega s
\end{equation}
solves $(\mathcal{E}_\omega)$.

Since the vorticity is transported by the Euler flow in dimension 2, an analogue of the Beale-Kato-Majda
theorem \cite{bkm} ensures that smooth solutions can be  continued indefinitely (see \cite{Taylor2016}). With the question of global existence settled, 
our aim will be to obtain more qualitative information on the dynamics by studying the stability of steady zonal flows. 
Note that the governing equation $(\mathcal{E}_\omega)$ for stratospheric flow can be written in the Hamiltonian form 
$$\partial_t (\Delta \Psi + 2\omega s) = \{ \Delta \Psi + 2\omega s,\,\Psi \}$$
with respect to the symplectic structure on ${\mathbb S}^2$ whose Poisson bracket is given by
$$\{f,h\}= \partial_s h \partial_\varphi f - \partial_s f \partial_\varphi h =\frac{\partial_s h \partial_\varsigma f - \partial_s f \partial_\varsigma h }{\sqrt{1-s^2}}\,,$$
corresponding to the Mathieu transformation $(s,\varsigma)=\big(\sin\theta,\,\frac{\varphi}{\cos\theta}\big)$ of the latitude-longitude coordinates $(\theta,\varphi)$ on ${\mathbb S}^2$, which 
define the standard symplectic structure in spherical coordinates (see \cite{cg22}) with Poisson bracket 
$$\{f,h\}=\frac{\partial_\theta h \partial_\varphi f - \partial_\theta f \partial_\varphi h }{\cos\theta}$$
(we refer to \cite{mh} for a discussion of symplectic transformations). The stability of specific solutions of nonlinear Hamiltonian systems is very challenging 
and often studied by linearization techniques. The spectrum of the resulting linear operator does provide clues as to the
stability of the solution of the nonlinear equation, an element of the spectrum having a strictly positive real part being indicative of instability. Since the spectrum 
of a linear Hamiltonian operator is symmetric with respect to the imaginary axis of the complex plane, for Hamiltonian systems spectral stability is equivalent to a purely 
imaginary spectrum. Linear stability is more demanding, requiring that small perturbations remain bounded for all times. Linear stability implies spectral stability but there are even 
systems with finitely many degrees of freedom for which the converse is not true; moreover, linear stability does not imply nonlinear stability and linear instability does not preclude nonlinear 
stability (we refer to \cite{hmrw} for examples). In view of these considerations, only a nonlinear stability analysis can be expected to provide genuine insight into the dynamics of physically relevant flows.

\subsection{Zonal flows, $n$-jets and their stability}
Steady zonal flows, which correspond to stream functions $\Psi=\Psi_\ast(s)$ only dependent on the $s$-variable, solve $(\mathcal{E}_\omega)$.

The study of the linear instability of zonal flows on the rotating unit sphere ${\mathbb S}^2$ is of long-standing interest in geophysics. One of the most important general results is
the classical Rayleigh's criterion, providing a necessary condition: if a zonal flow with stream function $\Psi_\ast(s)$ is linearly unstable, then
$\Upsilon_\ast'+2\omega$ must change sign on $(-1,1)$, where $\Upsilon_\ast=\Delta\Psi_\ast$ is the vorticity of the flow. An improvement is Fjortoft's necessary condition for linear instability:
for any $\gamma\in \mathbb{R}$, $(\Upsilon_\ast'(s_0)+2\omega)(\Psi_\ast'(s_0)+\gamma)<0$ at some $s_0 \in (-1,1)$ (see the discussions in \cite{cg22,Skiba2017}). Of relevance is also the semicircle theorem, stating
(see \cite{Ishioka-Yoden92,Sasaki-Takehiro-Yamada2012}) that an unstable eigenvalue  $\lambda$  must lie in the upper semicircle with center at ${\max(- \Psi_\ast')+\min(- \Psi_\ast')\over2}$ and radius larger than or equal to  ${\max(- \Psi_\ast')-\min(- \Psi_\ast')\over2}$; we refer to  \cite{Thuburn-Haynes1996} for comparisons of the semicircle theorem between the cases of spherical and flat-space geometry.
The most far-reaching general nonlinear stability result for a zonal flow with stream function $\Psi_\ast(s)$, proved in \cite{Caprino-Marchioro1988,Taylor2016}, requires a monotone total vorticity
$\Upsilon_\ast(s)+2\omega s$ on $[-1,1]$.

\medskip

The $n$-jets are distinguished zonal flows; by definition, the $n$-jet has its stream function given by the rescaled Legendre polynomial
$$P_n(s) = \sum_{k=0}^{[n/2]} \frac{(-1)^k (2n-2k)! \,s^{n-2k}}{2^n k! (n-k)!(n-2k)!} = \frac{1}{2^n n!} \,\frac{{\rm d}^n}{{\rm d}s^n}\,(s^2-1)^n\,, \qquad s \in [-1,1]\,,$$
where $n \in \mathbb{N}$ and $[n/2]$ stands for the integer part of $n/2$. The Legendre polynomial $P_n$ is the zonal eigenfunction
corresponding to the eigenvalue $-n(n+1)$ of the Laplace-Beltrami operator on ${\mathbb S}^2$ with $\int_{-1}^1 P_n^2(s)\,{\rm d}s=\frac{2}{2n+1}$. The polynomial $P_n$ has
$n$ simple roots and $n-1$ local extrema in $(-1,1)$, so that for $n \ge 2$ its graph features alternating prograde and retrograde jets (see \cite{sz} and Fig. \ref{f2}); moreover, $P_n(1)=1$ for all $n \ge 0$ and 
$$P_n(-s)=(-1)^n P_n(s)\,,\quad P_{2n}(0)=\frac{(-1)^n }{2^{2n}}\,{2n \choose n} \quad\text{for all}\quad n \ge 1\,,\ s \in [-1,1]\,.$$
The fact that any zonal flow in $L^2({\mathbb S}^2)$ can be decomposed into a combination of the $n$-jets is a consequence of the fact that
the spherical harmonics (the normalized eigenfunctions of the Laplace-Beltrami operator) form an orthonormal basis of $L^2({\mathbb S}^2)$.

\begin{figure}[ht]
    \centering
	\includegraphics[scale = 0.65]{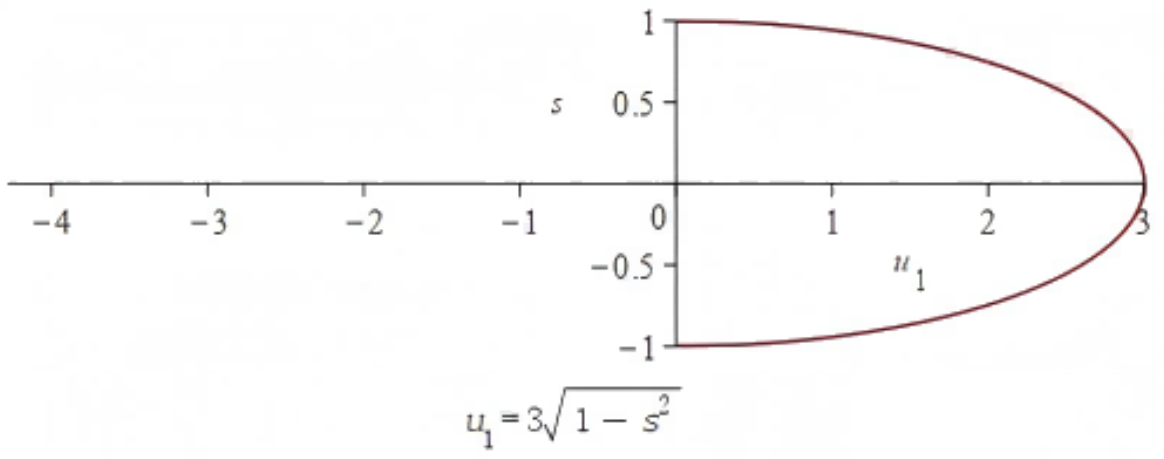}\;
        \includegraphics[scale = 0.67]{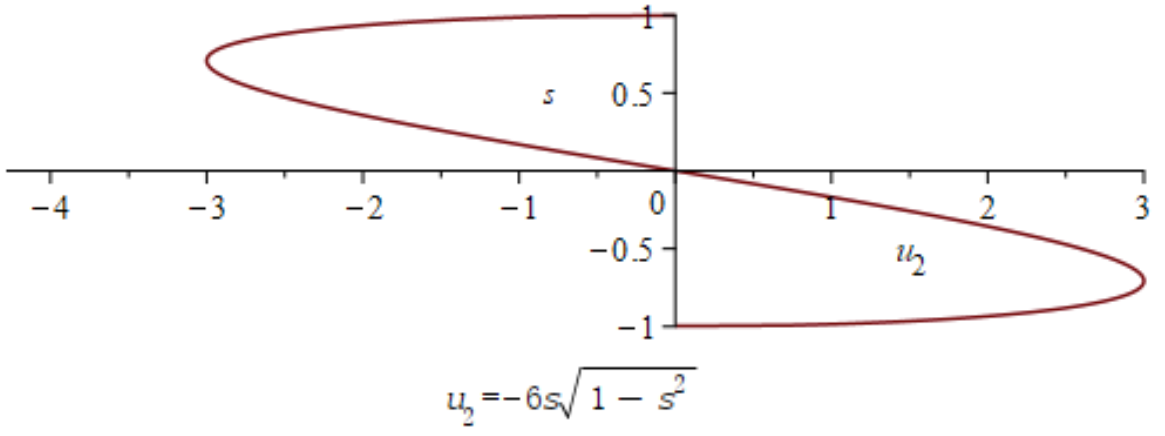}\;
        \includegraphics[scale = 0.67]{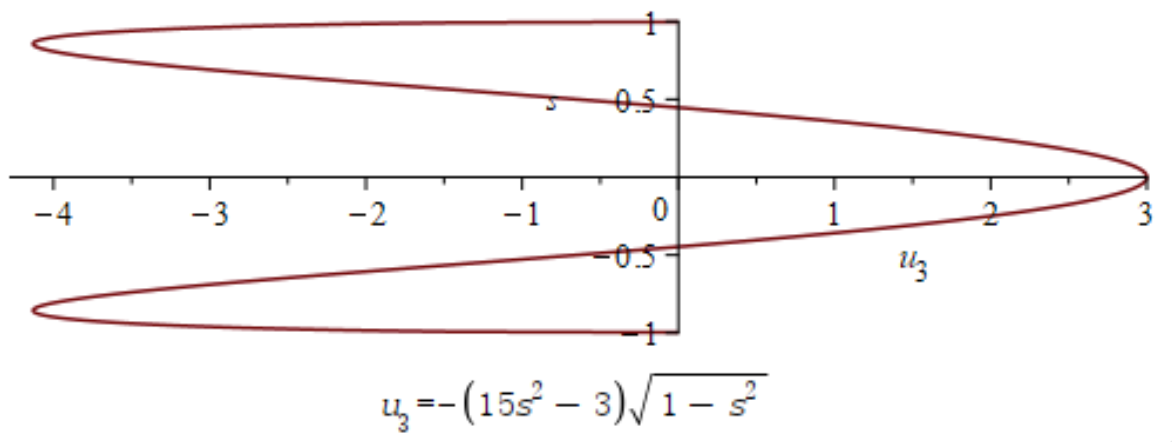}
        \caption{Zonal velocity profiles $u_n$ of the first three  jets, rescaled such that $\max\limits_{s \in [-1,1]}\{u_n(s)\}=3$ for $1 \le n \le 3$, the corresponding stream functions being $\Psi_1=-3P_1$, $\Psi_2=2P_2$ and
        $\Psi_3=2P_3$ in terms of the Legendre polynomials.}
	\label{f2}
\end{figure}

The $n$-jets are also of interest because of their connection to Rossby-Haurwitz waves \cite{Haurwitz1940, Rossby1939}. By definition, the stream function of a Rossby-Haurwitz wave is the sum of two eigenfunctions of the Laplace-Betrami operator, belonging to the first and the $n$-th eigenspaces, respectively. In the case of zonal flows, Rossby-Haurwitz waves reduce to $n$-jets. Because of their relevance in climate dynamics, much effort has been devoted to the linear and nonlinear stability/instability of general Rossby-Haurwitz waves, analytically \cite{Hoskins1973,Baines1976,Skiba1992,Thuburn-Li2000,Skiba2000,Skiba2004,Skiba-PerezGarcia2005,Skiba2009,Skiba2017,Jeong-Cheong2017,Skiba2018,Benard2020} as well as numerically \cite{Hoskins1973,Baines1976,Thuburn-Li2000,Skiba-PerezGarcia2005,Jeong-Cheong2017,Benard2020} but it remains a very challenging problem.

Are the $n$-jets stable?
\begin{itemize}
\item The stability of the $1$-jet is a consequence of the conservation of the angular momentum, see for instance \cite{cg22}.
\item The stability of the $2$-jet was proved in \cite{cg22} by relying on the Hamiltonian structure, in particular on the conservation of energy and on the
interplay between high-order Casimirs (cubic, quartic and quintic invariants).
\item The $3$-jet is the first that appears to succumb to instability mechanisms. It is numerically found in \cite{Baines1976} that
there exist positive and negative critical rotation rates $\omega_{cr}^\pm$ such that the 3-jet
is linearly unstable for $ \omega\in (\omega_{cr}^-,\omega_{cr}^+$) and spectrally stable otherwise. The accuracy of the numerical computations in \cite{Baines1976} was improved in \cite{Sasaki-Takehiro-Yamada2012,Taylor2016}: 
the formal analysis in \cite{Sasaki-Takehiro-Yamada2012} predicts that $\omega_{cr}^+={99\over2}$ and $\omega_{cr}^- \approx -16.0732$, while the numerical approximations in 
\cite{Baines1976} and \cite{Taylor2016} are $\omega_{cr}^- \approx -15.9652$ and $\omega_{cr}^- \approx -16.875$, respectively\footnote{Note that in the notation of 
\cite{Sasaki-Takehiro-Yamada2012} this critical rotation rate  is $-{33\sqrt{7}\over16}$, being
transformed in our notation to $-{33\sqrt{7}\over16}\times \left(-{24\over \sqrt{7}}\right)={99\over2}$, while the critical rotation rate $1.7719$ in the notation of \cite{Sasaki-Takehiro-Yamada2012} corresponds
to $1.7719\times \left(-{24\over \sqrt{7}}\right)\approx-16.0732$ in our notation. Similar scaling transformations yield the corresponding values for $\omega_{cr}^-$ obtained in \cite{Baines1976,Taylor2016}.}. From the analytical perspective, the $3$-jet is  spectrally stable for $\omega\in(-\infty,-18]\cup[72,\infty)$ by the Rayleigh's criterion, and is actually nonlinearly stable for $\omega\in(-\infty,-18]\cup[72,\infty)$, as proved  in \cite{Caprino-Marchioro1988,Taylor2016,Cao-Wang-Zuo2023}.
\end{itemize}

This shows that the $3$-jet is the current frontier of our understanding of stability of jets; they also provide the first example of instability of such flows. Note also that the $3$-jet plays a key role in the description of the
wind profile on Uranus: according to \cite{snh} it is the linear combination
\begin{equation}\label{ur}
u(s)=\tfrac{68}{3}\,\big(u_1(s) - 2u_3(s)\big)
\end{equation}
of the $1$-jet and the $3$-jet, where $u_1, u_3$ are rescaled as in Fig. \ref{f2}. Using the symmetry \eqref{symmetry} of the Euler equation on the sphere, this corresponds to the 
$3$-jet (with stream function $2P_3$) for $\omega = -{1647\over1360}$. These considerations prompt the following question, which will be our focus in the present paper.

\medskip

\noindent
{\bf Main question.}\label{Secondquestion} \textit{For which values of the planetary rotation speed $\omega$ is the $3$-jet (with stream function $2P_3$) stable, respectively unstable?}

\subsection{Main results: linear aspects}
Our aim is to understand the local dynamics near the $3$-jet zonal flow, the stream function of which has the form
\begin{align}\label{stream function of 3-jet zonal flow}
\Psi_0(s)=5s^3-3s=2P_3(s),
\end{align}
As a first step, we consider the linearized problem.

\subsubsection{Ranges of linear stability or instability}

Our first main result is a rigorous proof of   the critical rotation rate $\omega_{cr}^+={99\over2}$ for linear stability/instability of the $3$-jet in the positive half-line.  Our analytical result is consistent with the numerical findings in \cite{Sasaki-Takehiro-Yamada2012}.

\begin{Theorem}\label{positive half-line critical rotation rate}
 The $3$-jet is  linearly unstable  for $\omega\in\left[0,{99\over2}\right)$ and spectrally  stable for $\omega\in\left[{99\over2},\infty\right)$.
\end{Theorem}

As we shall see, the linear instability of the $3$-jet can only occur for the first and second Fourier modes (in the azimuthal variable). The positive critical rotation rate  for the $k$-th Fourier mode is denoted by $\omega_{cr,k}^+$ for $k=1, 2$.
We shall prove that   $\omega_{cr,1}^+={99\over2}$ in Theorems \ref{linear instability} and \ref{k=1 positive half-line critical rotation rate}, and  $ \omega_{cr,2}^+={69\over2}$  in Theorems \ref{linear instability} and \ref{k=2 positive half-line critical rotation rate}.

\medskip

 Our second main result is an exact determination and a rigorous proof  of   the critical rotation rate $\omega_{cr}^-$ on the negative half-line.
In particular, the exact value of $\omega_{cr}^-$
is not found in the numerical literature \cite{Baines1976,Sasaki-Takehiro-Yamada2012,Taylor2016} and is
 based on the principal eigenvalues of a modified Rayleigh equation, which has the form
 \begin{align}\label{Rayleigh-type equation thm 1.2}
((1-s^2)\Phi')'-{4\over 1-s^2}\Phi-{2\omega+12\mu\over 15s^2-3+\mu}\Phi=\tilde\lambda\Phi
 \end{align}
 with $\Phi\in X_{\omega,\mu,e}$ (defined in \eqref{def-X-omega-mu-e k=2}).
We denote its principal (i.e. maximal) eigenvalue by $\tilde\lambda_1(\mu,\omega)$, where $\mu\in[3,\infty)$ and $\omega\in[-18,-3]$. Next, we define the function
\begin{align}\label{def-g}
g(\omega)=\max_{\mu\in[3,183]}\tilde\lambda_1(\mu,\omega),\quad \omega\in[-18,-3].
\end{align}
In Lemma \ref{properties of g}, we will prove that $g$ is decreasing and continuous on $\omega\in[-18,-3]$, $g(-18)=-6$ and $g(-3)<-12$.
Then $g^{-1}(-12)\in(-18,-3)$. Now we can state our result for the negative half-line.

 \begin{Theorem}\label{negative half-line critical rotation rate}
 The $3$-jet is  linearly unstable  for $\omega\in\left(g^{-1}(-12),0\right]$ and spectrally  stable for $\omega\in\left(-\infty, g^{-1}(-12)\right]$.
\end{Theorem}


As for the positive half-line, linear instability only appears  for the first and second Fourier modes for the negative half-line. The negative critical rotation rate  for the $k$-th Fourier mode is denoted by $\omega_{cr,k}^-$ for $k=1, 2$.
 We prove that  $\omega_{cr,1}^-= -3$  in Theorems \ref{linear instability} and \ref{k=1 negative half-line}, and   $\omega_{cr,2}^-=g^{-1}(-12)$ in Theorems \ref{linear instability} and \ref{k=2 negative half-line}. While 
 Theorems \ref{positive half-line critical rotation rate}-\ref{negative half-line critical rotation rate} apply to vorticity eigenfunctions in $L^2$, Lemma 6.3 shows that the  unstable vorticity eigenfunctions are more regular.

The analytical value $\omega_{cr}^-=g^{-1}(-12)$ has not been found in the literature. To check whether our analytical  critical rotation rate $g^{-1}(-12)$ is consistent with the numerical computation in \cite{Sasaki-Takehiro-Yamada2012}, we
  compute $g^{-1}(-12)$ by Matlab. Our computation reveals that $g^{-1}(-12)\approx-16.0735$, which is very close to the numerical critical rotation rate $-16.0732$ in \cite{Sasaki-Takehiro-Yamada2012}.
   The $\omega$-range of linear stability or instability of the $3$-jet is illustrated in Fig. \ref{f3}.
 \begin{figure}[ht]
    \centering
	\includegraphics[scale = 0.52]{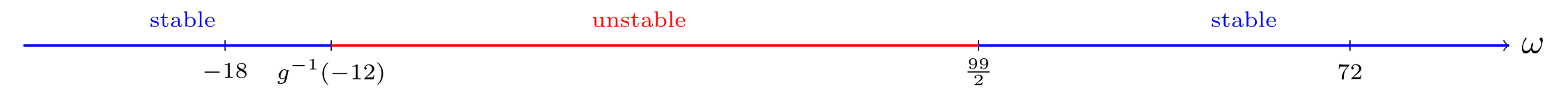}
        \caption{Rayleigh's criterion ensures that the $3$-jet is 
 spectrally stable   for $\omega\in(-\infty,-18]\cup[72,\infty)$,
  but is not helpful for $\omega\in(-18,72)$.
Theorems \ref{positive half-line critical rotation rate}-\ref{negative half-line critical rotation rate} give the sharp  $\omega$-range of linear stability or instability.}
	\label{f3}
\end{figure}

\subsubsection{Ideas of the proof}
An important difference of the Euler equation on $\mathbb{S}^2$  from that on a flat geometry is the conservation of angular momentum.
The linearized equation of ${\rm(\mathcal{E}_\omega)}$ around   $\Upsilon_0=\Delta \Psi_0$  is $\partial_t\Upsilon=\mathcal{L}_\omega\Upsilon$, where the linearized operator is
\begin{align}\label{original linearized operator near the 3 jet flow}
\mathcal{L}_\omega=\Psi_0'\partial_{\varphi}-(\Upsilon_0'+2\omega)\partial_{\varphi}\Delta^{-1}.
\end{align}
Note that the study of
linear instability of the $3$-jet with stream function $\Psi_0(s)$  in the rotating case $\omega\neq0$ is equivalent  to that  of another zonal flow with stream function  $\tilde \Psi_{\omega}(s)=\Psi_0(s)-\omega s$  in the non-rotating case. Instead of  looking at \eqref{original linearized operator near the 3 jet flow} directly,
in  the frame $(\varphi-{5\over 6}\omega t,s)$ we study the linearized vorticity  equation of ${\rm(\mathcal{E}_0)}$ around the zonal flow $\tilde \Upsilon_{\omega}=\Delta\tilde \Psi_{\omega}$, which has a Hamiltonian structure
 \begin{align*}
\partial_t\Upsilon=J_\omega L\Upsilon.
\end{align*}
Here
\begin{align}\label{def-ham-JL}
J_\omega=-\tilde \Upsilon_{\omega}'\partial_\varphi=-(\Upsilon_0'+2\omega)\partial_\varphi:X^*\supset D(J)\to X,\quad L={1\over 12}+ \Delta^{-1}:X\to X^*,
\end{align}
and
\begin{align}\label{def-space-X}
X=\left\{\Upsilon\in L^2(D_T):\iint_{D_T}\Upsilon d\varphi ds=0,\iint_{D_T}\Upsilon Y_1^m d\varphi ds=0, m=0,\pm1\right\}.
\end{align}
The conservation of angular momentum allows us to look for the unstable eigenvalues of the linearized operator $J_{\omega}L$ in the invariant subspace $X=\oplus_{n=2}^{\infty}E_n$  orthogonal to $E_1$. The relation between $\mathcal{L}_\omega$ and $J_\omega L$ is $\mathcal{L}_\omega=
J_\omega L+{1\over 6}\omega\partial_\varphi$.

A basic ingredient is now  provided by the index formulae \eqref{index formula 1o1}-\eqref{index formula 1o2}, which we do not reproduce here for the sake of conciseness.
Through these index formulae, the question of spectral stability reduces to finding neutral modes and then
calculating the signature of the corresponding energy quadratic form. A neutral mode is
a quadruple $(c,k,\omega,\Phi)$ solving the Rayleigh equation
 \begin{align*}
  \Delta_k\Phi-{\tilde \Upsilon_{\omega}'\over \tilde \Psi_{\omega}'+c}\Phi=((1-s^2)\Phi')'-{k^2\over 1-s^2}\Phi-{\tilde \Upsilon_{\omega}'\over \tilde \Psi_{\omega}'+c}\Phi=0
 \end{align*}
with $\Delta_k\Phi\in L^2(-1,1)$, $c\in \mathbb{R}$. If $c\notin Ran(-\Psi_{\omega}')$, then the neutral mode $(c,k,\omega,\Phi)$ is non-resonant.  A neutral mode $(c,k,\omega,\Phi)$ corresponds to a purely imaginary eigenvalue $-ik(c-c_\omega)$ of $J_{\omega}L$.
The analysis of the above Rayleigh equation is very delicate, we refer to  Remarks \ref{ideas thm k=1 positive half-line critical rotation rate and k=2 positive half-line critical rotation rate}, \ref{ideas in the proof of k=1 negative half-line} and \ref{ideas in the proof of k=2 negative half-line} for a description of the ideas involved. Here, we only point out that  the discovery of the analytical value of the negative critical  rotation rate can be traced back to the lift-up jump of the principal eigenvalues of the Rayleigh equations 
(see Lemma 4.10). Leaving this delicate proof aside, we will present the mechanisms which account for the onset of instability.

It is worth mentioning that geometric curvature effects play an important role in  the stability of zonal flows on the sphere, leading to significant differences to flat geometry, in particular with regard to 
the presence of non-resonant neutral modes and to their role as the stability boundary (see
Remark \ref{geometric curvature effects and differences with the flat geometry}).

\subsubsection{Mechanism  causing  instability in the positive half-line:} The positive critical rotation rate $\omega_{cr}^+={99\over2}$ in Theorem \ref{positive half-line critical rotation rate} is from the first Fourier mode.
The transition from stability to instability is
caused by purely imaginary  isolated eigenvalues  hitting an embedded eigenvalue in the continuous spectrum of the linearized  operator $\mathcal{L}_{\omega,1}$ for the first Fourier mode, see Fig. \ref{Mechanism causing  instability on the positive half-line}.
We refer to the spectral analysis in the proof of Theorem \ref{k=1 positive half-line critical rotation rate} (illustrated in Fig. \ref{fig-eigenvalue curves1}) and Lemma \ref{ometa=12} for more details.
  Similar mechanism appears in  structural instability for some equilibria, and we refer to \cite{Grillakis1990,Cuccagna-Pelinovsky-Vougalter2005} for the nonlinear
Schr\"{o}dinger equation, \cite{Grillakis1990} for the Klein-Gordon equation and \cite{lin2022instability} for general Hamiltonian systems.
Such structural instability is proved by constructing specific perturbations, and cannot be used to study our problem.
\begin{center}
 \begin{tikzpicture}[scale=0.8]
  \draw [->](-7, -2)--(7, -2)node[right]{Re};
 \draw [->](0,-5.6)--(0,6.5) node[above]{Im};
        \node[blue] (a) at (0,4.5) {$\bullet$};
        \node[blue] (a) at (0,3.5) {$\bullet$};
        \node[blue] (a) at (0,2.9) {$\bullet$};
        \node[blue] (a) at (0,2.45) {$\bullet$};
        \node[blue] (a) at (0,2.15) {$\bullet$};
        \node[blue] (a) at (0,1.9) {$\bullet$};
        \node[blue] (a) at (0,1.68) {$\bullet$};
         \node[red] (a) at (0,1.5) {$\bullet$};
       \draw  (-0.1, -3).. controls (0.1, -3) and (0.1, -3)..(0.1, -3);
        \draw  (-0.1, 1.5).. controls (0.1, 1.5) and (0.1, 1.5)..(0.1, 1.5);
     \draw [red][->](0, 1.5)--(2, 1.5)node[right]{};
       \draw [red][->](0, 1.5)--(-2, 1.5)node[left]{};

        \draw  [green][thick] (0, -3).. controls (0, -3) and (0, 1.45)..(0, 1.45);
        \node (a) at (-1.4,2.1) {\tiny$\omega\to {99\over2}^+$};
        \node (a) at (0.5,-3) {\tiny$-3i$};
        \node (a) at (0.5,1.3) {\tiny$12i$};
        \node (a) at (0.8,2.8) {\tiny$-i\mu_{1,\omega}$};
        \node (a) at (1.2,-0.8) {\tiny$\sigma_e(\mathcal{L}_{\omega,1})$};
       \node (a) at (4,3.5) {\tiny$\sigma(\mathcal{L}_{\omega,1})$};
       \node (a) at (4,1.5) {\tiny$\text{unstable eigenvalues}$};
 \end{tikzpicture}
\end{center}\vspace{-0.5cm}
\begin{figure}[ht]
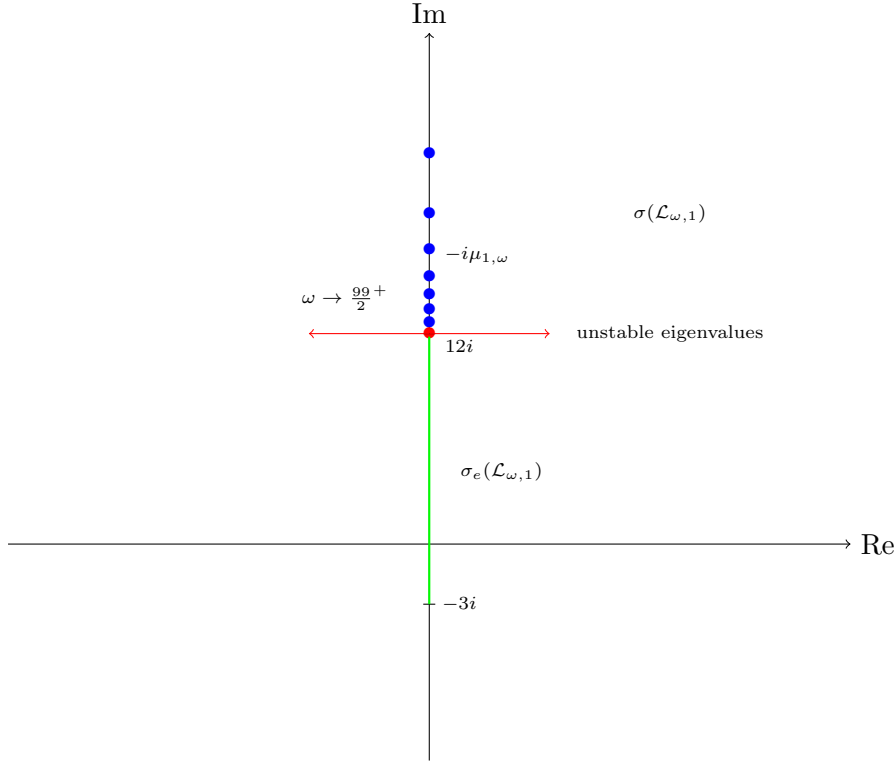

\centering
\caption{
The green interval is the essential spectrum of $\mathcal{L}_{\omega,1}$, which is the projection of $\mathcal{L}_{\omega}$ on the first Fourier mode. For $\omega$ larger than and close to ${99\over2}$, the blue bold points are the isolated eigenvalues $-i\mu_{1,\omega}$ of $\mathcal{L}_{\omega,1}$. For $\omega={99\over2}$, the red bold point is the embedded (edge) eigenvalue $12i$ of $\mathcal{L}_{{99\over2},1}$. As $\omega\to{99\over2}^+$,   the isolated eigenvalue $-i\mu_{1,\omega}$ of $\mathcal{L}_{\omega,1}$ hits the embedded eigenvalue $12i$ of $\mathcal{L}_{{99\over2},1}$, where an unstable eigenvalue emerges.
The embedded  eigenvalue $12i$  has negative sign of the energy quadratic form.
}
\label{Mechanism causing  instability on the positive half-line}
\end{figure}
\vspace{-0.2cm}

\subsubsection{Mechanism causing  instability in the negative half-line:} The negative critical rotation rate $\omega_{cr}^-=g^{-1}(-12)$ in Theorem \ref{negative half-line critical rotation rate} is from the second Fourier mode.
The transition from stability to instability is
induced by the collision of purely imaginary isolated eigenvalues  of the linearized  operator $\mathcal{L}_{\omega,2}$ with opposite Krein signatures, see Fig. \ref{Mechanism causing  instability on the negative half-line}. For more details, see the spectral analysis in the proof of Theorem \ref{k=2 negative half-line} (illustrated  in Fig. \ref{fig-eigenvalue curves2}) and Remark \ref{krein signature-rem}. Other instances of loss of stability of equilibria through the collision of purely imaginary eigenvalues with opposite Krein signatures appeared in \cite{MacKay1987,Kapitula-Kevrekidis-Sandstede2004,lin2022instability}.

\subsection{Main results: nonlinear aspects}
With Theorems \ref{positive half-line critical rotation rate}-\ref{negative half-line critical rotation rate} at hand, it is natural to turn to the nonlinear problem. Our results on the nonlinear problem will be of two kinds: first, we will see that linear instability implies nonlinear instability; and second, we will be able to describe streamline patterns of traveling waves in a neighborhood of the $3$-jet. Both results rely heavily on the precise description of the linearized problems obtained in Theorems  \ref{positive half-line critical rotation rate}-\ref{negative half-line critical rotation rate}.

\begin{center}
 \begin{tikzpicture}[scale=0.8]
 \draw [->](-7, -2)--(7, -2)node[right]{Re};
 \draw [->](0,-10.5)--(0,2) node[above]{Im};
       \node (a) at (5,1) {\tiny$\sigma(\mathcal{L}_{\omega,2})$};
       \node[blue] (a) at (0,-3) {$\bullet$};
       \node[blue] (a) at (0,-3.7) {$\bullet$};
       \node[blue] (a) at (0,-4.3) {$\bullet$};
       \node[blue] (a) at (0,-4.8) {$\bullet$};
       \node[blue] (a) at (0,-5.2) {$\bullet$};
       \node[blue] (a) at (0,-5.5) {$\bullet$};
       \node[blue] (a) at (0,-5.73) {$\bullet$};
       \node[red] (a) at (0,-5.92) {$\bullet$};
       \node[brown] (a) at (0,-6.1) {$\bullet$};
       \node[brown] (a) at (0,-6.4) {$\bullet$};
       \node[brown] (a) at (0,-6.75) {$\bullet$};
       \node[brown] (a) at (0,-7.2) {$\bullet$};
       \node[brown] (a) at (0,-7.7) {$\bullet$};
       \node[brown] (a) at (0,-8.2) {$\bullet$};
       \node[brown] (a) at (0,-9) {$\bullet$};
       \draw [red][->](0, -5.92)--(2, -5.92)node[right]{};
       \draw [red][->](0, -5.92)--(-2, -5.92)node[left]{};
      \node (a) at (4,-5.92) {\tiny$\text{unstable eigenvalues}$};
       \node (a) at (-2,-5.3) {\tiny$\omega\to g^{-1}(-12)^-$};
        \node (a) at (0.8,-4.8) {\tiny$-2i\mu_{3,\omega}$};
        \node (a) at (0.8,-7.2) {\tiny$-2i\mu_{2,\omega}$};
        \draw  [green][thick] (0, -2.5).. controls (0, -2.5) and (0, 0.5)..(0, 0.5);
         \draw  (-0.1, -2.5).. controls (0.1, -2.5) and (0.1, -2.5)..(0.1, -2.5);
         \draw  (-0.1, 0.5).. controls (0.1, 0.5) and (0.1, 0.5)..(0.1, 0.5);
          \node (a) at (0.5,-2.5) {\tiny$-6i$};
        \node (a) at (0.5,0.5) {\tiny$24i$};
        \node (a) at (1.2,-0.8) {\tiny$\sigma_e(\mathcal{L}_{\omega,2})$};
 \end{tikzpicture}
\end{center}
\vspace{-0.5cm}
\begin{figure}[ht]
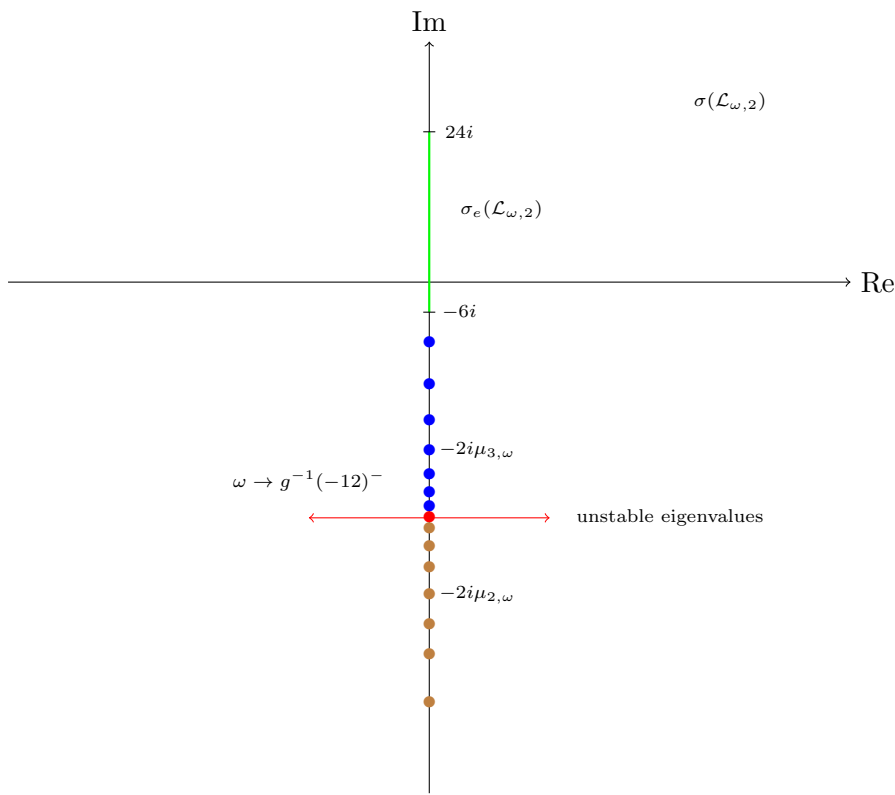

    \centering
    \caption{The green interval is the essential spectrum of $\mathcal{L}_{\omega,2}$, which is the projection of $\mathcal{L}_{\omega}$ on the second Fourier mode. For $\omega$ smaller than and close to $g^{-1}(-12)$, the blue bold points are the isolated eigenvalues $-2i\mu_{3,\omega}$ of $\mathcal{L}_{\omega,2}$ and the brown bold points are the isolated eigenvalues $-2i\mu_{2,\omega}$ of $\mathcal{L}_{\omega,2}$. As $\omega\to g^{-1}(-12)^-$, the two eigenvalues with opposite Krein signatures collide at the isolated eigenvalue $-2i\mu_{2,g^{-1}(-12)}$ of  $\mathcal{L}_{g^{-1}(-12),2}$. After the collision, an unstable eigenvalue emerges.
}
\label{Mechanism causing  instability on the negative half-line}
\end{figure}
\vspace{-0.2cm}

\subsubsection{Nonlinear instability}
This question is classical for the Euler equation set on flat domains
\cite{Friedlander-Strauss-Vishik1997,Grenier2000,Bardos-Guo-Strauss2002,Vishik-Friedlander2003,lin2004}. These results need to overcome two difficulties: the loss of derivative in the nonlinearity, and the positivity of the Lyapunov exponent of the flow generated by the steady velocity field. The latter difficulty was addressed by introducing the  averaging Lyapunov exponent of the flow and proving that it is zero in \cite{lin2004}.
In order to generalize these methods to the case of flows set on the sphere, we  resort to the tools of nonlinear analysis, including various  Sobolev embeddings and inequalities, in Riemannian manifolds and study the nonlinear  problem by an intrinsic geometric method. The quantities like the averaging Lyapunov exponent are defined globally in nature and it is preferable to  make flexible use of intrinsic quantities and coordinate charts.
For example, to prove that the globally-defined averaging  Lyapunov exponent  is zero, we can choose the poles  such that they do not touch the particle trajectories emitted from the ``dangerous" region, which is near non-degenerate saddle points and the trajectories connecting them.

Our main result in this subsection applies to a general steady flow. It is natural to consider the nonlinear orbital instability since it is embedded in the zonally translational orbit. Let $T$ be a one-parameter group of unitary operators on $L^p(T\mathbb{S}^2)$   defined by
$
(T(\tau)\mathbf{u})(\mathbf{x})=\mathbf{u}(\zeta^{-1}(\varphi+\tau,s))
$
 for $\mathbf{x}\notin \{N, S\}$ and
$(T(\tau)\mathbf{u})(\mathbf{x})=\mathbf{u}(\mathbf{x})$  for $\mathbf{x}\in\{N, S\}$, where $\mathbf{u}\in L^p(T\mathbb{S}^2)$, 
 $p\geq1$, $\tau\in\mathbb{R}$, the coordinate chart  $\zeta$ is defined in \eqref{chart1}, and $N, S$ are the North and South poles.

\begin{Theorem}[Linear to nonlinear orbital instability of general steady flows]\label{thm-nonlinear instability}  
For a $C^1$ steady flow $\mathbf{u}_{G}$ with  finite
 stagnation points,
if it is linearly unstable in $L^2(T\mathbb{S}^2)$, then it is nonlinearly orbitally unstable in the sense that
there exists $\epsilon_1>0$ such that for any $\delta>0$, there exists a solution $ \mathbf{u}_{\delta,G}$ to the nonlinear Euler equation and $t_{1}=O(|\ln\delta|)$ satisfying
\begin{align*}
\|\Omega_{\delta,G}(0)-\Omega_{G}\|_{L^{p_2}(\mathbb{S}^2)}+\|\nabla( \Omega_{\delta,G}(0)-\Omega_{G})\|_{L^{p_1}(T\mathbb{S}^2)}\leq \delta
\end{align*}
and
\begin{align*}
\inf_{\tau\in\mathbb{R}}\|{\mathbf{u}}_{\delta,G}(t_{1})-T(\tau){\mathbf{u}}_{G}\|_{L^{p_0}(T\mathbb{S}^2)}\geq\epsilon_1,
 \end{align*}
where  $\Omega_{G}=curl(\mathbf{u}_{G})$, $\Omega_{\delta,G}=curl(\mathbf{u}_{\delta,G})$, $p_0\in(1,\infty)$, $p_1\in[1,b_0)$, $p_2\in[1,\infty)$,
$b_0=\infty$ if $\mu\leq {\rm Re}(\lambda_1)$ while $b_0={\mu\over \mu-{\rm Re \lambda_1}}$ if $\mu>{\rm Re(\lambda_1)}$,  $\mu$ is the
Lyapunov exponent of the flow generated by $\mathbf{u}_{G}$, and $\lambda_1$ is an unstable  eigenvalue  with the largest real part.
\end{Theorem}

Here, the linear instability condition  is in the weak sense  that the regularity of unstable (velocity) eigenfunction is $L^2$.
For  a zonal flow, the translational orbit is itself and the Lyapunov exponent is zero, which implies the following.

\begin{Corollary}
Let $p_0\in(1,\infty)$, $p_1\in[1,\infty)$ and $p_2\in[1,\infty)$.
If a $C^1$ zonal flow $\mathbf{u}_Z$ is linearly unstable in $L^2(T\mathbb{S}^2)$, then it is nonlinearly unstable in the following sense: there exists $\epsilon_1>0$ such that for any $\delta>0$, there exists a solution $ {\mathbf{u}}_{\delta,Z}$ to the nonlinear Euler equation and $t_1=O(|\ln\delta|)$ satisfying
\begin{align*}
\|\Omega_{\delta,Z}(0)-\Omega_Z\|_{L^{p_2}(\mathbb{S}^2)}+\|\nabla( \Omega_{\delta,Z}(0)-\Omega_Z)\|_{L^{p_1}(T\mathbb{S}^2)}\leq \delta\quad\text{and}\quad
\|{\mathbf{u}}_{\delta,Z}(t_1)-{\mathbf{u}}_Z\|_{L^{p_0}(T\mathbb{S}^2)}\geq\epsilon_1,
 \end{align*}
where  $\Omega_Z=curl({\mathbf{u}}_Z)$ and $\Omega_{\delta,Z}=curl({\mathbf{u}}_{\delta,Z})$.
\end{Corollary}

In particular, these results apply to the wind profile \eqref{ur}.

\begin{Corollary}
The wind profile \eqref{ur} on Uranus is nonlinearly unstable.
\end{Corollary}

\subsubsection{Nearby traveling waves}
Traveling waves near a given stationary solution play an important role in
 understanding the global dynamics. Indeed, they are potential limits, or even attractors, of the flow as $t \to \infty$. In flat geometry, it was shown that cat's eyes (or non-shear) structures exist in an $H^s$ neighborhood of the Couette flow if $s<\frac 32$ but not if $s>{3\over2}$ \cite{lin-zeng11}, see also \cite{Zelati-Elgindi-Widmayer2023,Castro-Lear2024} for other shear flows. If rotation is added via the Coriolis force, new traveling waves (unidirectional) are constructed in \cite{lin-yang-zhu20,Lin-Wei-Zhang-Zhu22,Wang-Zhang-Zhu2023}. Finally, we refer to \cite{Nualart2023} where non-zonal steady flows in an analytic neighborhood of the $2$-jet are constructed.

Turning to the case of the $3$-jet, the following theorem asserts that (1)
when the rotation rate is small, the streamlines of the travelling waves near the $3$-jet have a cat's eyes structure;
(2) when the rotation rate is slightly larger, there are both  cat's eyes travelling waves and unidirectional travelling waves near the $3$-jet; and (3) when the rotation rate becomes larger, the streamlines of the travelling waves near the $3$-jet are unidirectional.

\begin{Theorem}[Existence of nearby traveling waves] \label{travelling wave solutions cat eye unidirectional thm}
$(1)$ Let $\omega\in(-3,{69\over2})$. For any $\varepsilon>0$, there exists a cat's eyes travelling wave  $\Psi_{1,\varepsilon}(\varphi-c_{1,\varepsilon}t,s)$
satisfying
\begin{align*}
\|\Psi_{1,\varepsilon}-\Psi_0\|_{\mathcal{G}_\lambda}<\varepsilon,
\end{align*}
where $\|\cdot\|_{\mathcal{G}_\lambda}$ is the Gevrey-$1$ norm.

\smallskip

$(2)$ Let $\omega\in(-18,-3)\cup({69\over2},72)$. Then
 for any $\varepsilon>0$, there exist both a cat's eyes travelling wave  $\Psi_{2,\varepsilon}(\varphi-c_{1,\varepsilon}t,s)$ and a non-zonal  unidirectional travelling wave $\Psi_{3,\varepsilon}(\varphi-c_{1,\varepsilon}t,s)$
 satisfying
\begin{align*}
\|\Psi_{2,\varepsilon}-\Psi_0\|_{\mathcal{G}_\lambda}<\varepsilon\quad\text{and}\quad\|\Psi_{3,\varepsilon}-\Psi_0\|_{H_2^4(\mathbb{S}^2)}<\varepsilon.
\end{align*}

\smallskip

$(3)$ Let $\omega\in(-\infty,-18)\cup(72,\infty)$. For any $\varepsilon>0$, there exists a non-zonal unidirectional travelling wave $\Psi_{4,\varepsilon}(\varphi-c_{4,\varepsilon}t,s)$
 satisfying
\begin{align*}
\|\Psi_{1,\varepsilon}-\Psi_0\|_{\mathcal{G}_\lambda}<\varepsilon.
\end{align*}
\end{Theorem}

Finally, we conclude with a rigidity statement.

\begin{Theorem}[Rigidity of nearby traveling waves]
\label{Rigidity of nearby traveling waves}
$(1)$ 
Let $\omega\in(-3,{69\over2})$, $p\geq3$ and $\delta\in(0,1)$. Then there exists $\varepsilon_\delta>0$ such that any unidirectional travelling wave  $\Psi(\varphi-ct,s)$   satisfying that
    $\text{dist} (c, Ran(-\Psi_0'))\geq\delta$, $c\neq-\omega$
    and
\begin{align}\label{travelling wave norm1}
\left\lVert\Delta\Psi-\Delta\Psi_0\right\rVert_{H_{p}^{2}(\mathbb{S}^2)}
+\left\lVert\partial_{s}\Psi-\Psi_{0}'\right\rVert_{C^0(\mathbb{S}^2)}\leq\varepsilon_\delta
\end{align}
must be a zonal flow.

\smallskip

$(2)$ Let $\omega\in(-\infty,-18)\cup(72,\infty)$, $p\geq3$ and $\delta\in(0,1)$. There exists $\varepsilon_0>0$ such that any travelling wave $\Psi(\varphi-ct,s)$ satisfying that $\text{dist} (c,\mathbb{R}\setminus Ran(-\Psi_0'))\geq\delta$ and
\begin{align}\label{travelling wave norm2}
\left\lVert\Delta\Psi-\Delta\Psi_0\right\rVert_{H_{p}^{2}(\mathbb{S}^2)}
+\sum_{j=1}^2\left\lVert\partial_{s}^j(\Psi-\Psi_{0})\right\rVert_{C^0(\mathbb{S}^2)}\leq\varepsilon_\delta
\end{align}
must be a zonal flow.
\end{Theorem}

 \begin{figure}[ht]
    \centering
	\includegraphics[scale = 0.8]{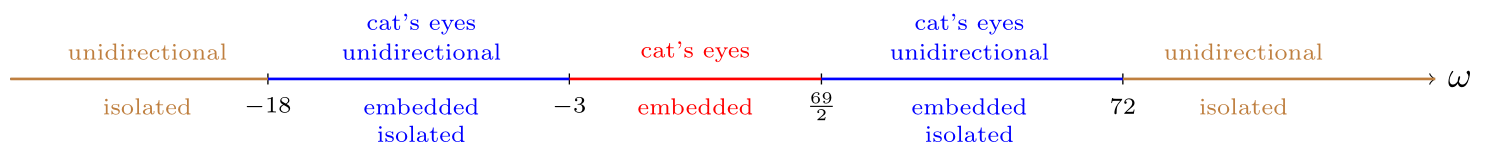}
        \caption{Nearby travelling waves and their causes: above the line are streamline patterns of nearby travelling waves, below the line are  types  of imaginary eigenvalues  of the linearized operators.}
	\label{Gradual-changes-in-streamline-patterns-of-travelling-waves}
\end{figure}

Basically, existence and rigidity  of nearby traveling waves are
 reflections of imaginary eigenvalues of the linearized operators  on  nonlinear dynamics near the $3$-jet, as illustrated in Fig. \ref{Gradual-changes-in-streamline-patterns-of-travelling-waves}. 
 Isolated imaginary eigenvalues produce unidirectional travelling waves, and embedded eigenvalues produce cat's eyes travelling waves.
See Lemma \ref{spectra of the linearized operatorLrigidity} for more details about imaginary eigenvalues of the linearized operators.

To construct curves of unidirectional  travelling waves   for $\omega\in\left(-18,-3\right)\cup\left({69\over2},72\right)$, one may study the bifurcation at modified zonal flows as in \cite{Lin-Wei-Zhang-Zhu22}. But this can not ensure that the travelling waves  form a curve.
To overcome the difficulty, we adopt  Kielh\"{o}fer's degenerate bifurcation theorem to the $3$-jet itself and carry out a more delicate spectral analysis  of the linearization of the nonlinear functional, see Remark $\ref{remark on the proof of travelling wave solutions thm}$.

\subsection{Organization of the article}
The rest of this paper is organized as follows. In Section 2, we give a detailed analysis of the Hamiltonian structure of the linearized equations, index formulae and properties of neutral modes. In Sections 3-4,
 we prove that the positive and negative critical rotation rates are ${99\over2}$ and $g^{-1}(-12)$, respectively.  In Section 5, we show the invariant subspace decomposition and  exponential trichotomy of the semigroup $e^{tJ_{\omega}L}$.  In Section 6, we prove nonlinear orbital instability from linear instability for general steady flows. In the last section, we study how the streamline patterns of traveling waves  near the 3-jet gradually change  as the rotation rate increases.

\section{Hamiltonian structure and neutral modes}

\subsection{Hamiltonian structure of the linearized equations and index formulae}\label{subsection-Hamiltonian structure of the linearized equations and index formulae}
Let $\Psi_0$ be the stream function  of the $3$-jet  in \eqref{stream function of 3-jet zonal flow}.
By (2.9) in \cite{cg22}, instead of studying linear instability of $\Psi_0(s)$ for the equation ${\rm(\mathcal{E}_\omega)}$ directly, we equivalently study linear instability of $\tilde \Psi_{\omega}(s)=\Psi_0(s)-\omega s$ for the equation ${\rm(\mathcal{E}_0)}$.
The transformed $3$-jet is explicitly expressed as
\begin{align}\label{def-tilde-Upsilon0-derivative}
\begin{array}{llll}
\tilde \Psi_{\omega}(s)=&\Psi_0(s)- \omega s=  5s^3-3s-\omega s,\\
\tilde \Psi_{\omega}'(s)=&\Psi_0'(s)- \omega =  15s^2-3-\omega \Longrightarrow \text{Ran}(-\tilde \Psi_{\omega}')=[-12+\omega,3+\omega],\\
\tilde \Upsilon_{\omega}(s)=&\Delta(\Psi_0-\omega s)=\Upsilon_0+2\omega s=-12\Psi_0+2\omega s=-12(5s^3-3s)+2\omega s,\\
\tilde \Upsilon_{\omega}'(s)=&\Upsilon_0'+2\omega=-12\Psi_0'+2\omega  =-12(15s^2-3)+2\omega.
\end{array}
\end{align}
Linearizing ${\rm(\mathcal{E}_0)}$ around $\tilde \Upsilon_{\omega}(s)$ in a traveling frame  $(\varphi-c_\omega t,s)$, we have
\begin{align*}
\partial_t\Upsilon-c_\omega\partial_\varphi\Upsilon=\left(\tilde \Psi_{\omega}'\partial_\varphi-\tilde \Upsilon_{\omega}'\partial_\varphi\Delta^{-1}\right)\Upsilon,
\end{align*}
where $c_\omega$ is to be determined. Thus,
\begin{align}\label{linearized equation}
\partial_t\Upsilon=-\tilde \Upsilon_{\omega}'\partial_\varphi\left(-{\tilde \Psi_{\omega}'+c_\omega\over \tilde \Upsilon_{\omega}'}+ \Delta^{-1}\right)\Upsilon.
\end{align}
 Choosing $c_\omega={5\over6}\omega$, we have
\begin{align}\label{c omega k}
-{\tilde \Psi_{\omega}'+c_\omega\over \tilde \Upsilon_{\omega}'}=-{\Psi_0'-{1\over 6}\omega\over -12\Psi_0'+2\omega}={1\over 12}.
\end{align}
Then the linearized equation \eqref{linearized equation} can be written in a Hamiltonian form
\begin{align}\label{Hamiltonian form}
\partial_t\Upsilon=-\tilde \Upsilon_{\omega}'\partial_\varphi\left({1\over 12}+ \Delta^{-1}\right)\Upsilon=J_{\omega}L\Upsilon,
\end{align}
where $J_\omega, L$ are defined in \eqref{def-ham-JL} and $X$ is defined in \eqref{def-space-X}.
Note that for the nonlinear Euler equation $e^{im\omega t}\iint_{D_T}\Upsilon Y_1^md\varphi ds$  is invariant, and this is also true for the linearized Euler equation, where $m=0,\pm1$. Thus, for a growing mode $e^{\lambda t}\Upsilon(\varphi,s)$ solving the linearized Euler equation,
$\iint_{D_T}\Upsilon Y_1^md\varphi ds=0$.
So it suffices to consider the perturbation of the vorticity in $X$ when studying the
existence of growing modes.
Since
\begin{align}\label{invariant-1}
{d\over dt}\iint_{D_T}e^{tJ_{\omega}L}\Upsilon d\varphi ds|_{t=0}=\iint_{D_T}J_{\omega}L\Upsilon d\varphi ds=0
\end{align}
and
\begin{align}\label{invariant-2}
{d\over dt}\iint_{D_T}e^{tJ_{\omega}L}\Upsilon Y_1^m d\varphi ds|_{t=0}=\iint_{D_T}J_{\omega}L\Upsilon Y_1^m d\varphi ds=0
 \end{align}
 for $\Upsilon\in X$ and $m=0,\pm1$,  $X$ is an invariant subspace for the linearized operator $J_{\omega}L$.
 For fixed $k\in\mathbb{Z}$,  recall that $\sigma(\Delta_k)=\{-l(l+1)\}_{l\geq |k|}$. For the eigenvalue $-l(l+1)$, $l\geq |k|$, the corresponding eigenfunction is the associated Legendre polynomial $P_l^k$.
Using such polynomials, we decompose the space $X$ into Fourier modes as follows. We define $X^k=\oplus_{l=|k|}^\infty\text{span}\{ P_l^k(s)\}$ for $|k|\geq2$ and $X^k=\oplus_{l=2}^\infty\text{span}\{ P_l^k(s)\}$ for $k=0,\pm1$. Note that $X^k=X^{-k}$ for $k\in \mathbb{Z}$. Then $X$ has the decomposition
\begin{align*}
X=\oplus_{k\in\mathbb{Z}}e^{ik\varphi} X^k,
\end{align*}
where $e^{ik\varphi} X^k=\{e^{ik\varphi}\Phi|\Phi\in X^k\}$.
We write $\Upsilon$ in the Fourier series $\Upsilon=\sum_{k\in\mathbb{Z}\setminus\{0\}}\Upsilon_k(s)e^{ik\varphi}$.
On $X^k$, the Hamiltonian form  \eqref{Hamiltonian form} can be reduced to
\begin{align*}
\partial_t\Upsilon_k=J_{\omega,k}L_k\Upsilon_k,
\end{align*}
where
\begin{align*}
J_{\omega,k}=-ik\tilde \Upsilon_{\omega}':(X^k)^*\supset D(J_{\omega,k})\to X^k,\quad L_k={1\over 12}+\Delta_k^{-1}:X^k\to (X^k)^*.
\end{align*}
Note that $J_{\omega,k}$ is not a real operator on $X^k$. We now reformulate a real  Hamiltonian system.
Define
the space
 \begin{align*}
 Y^k=e^{ik\varphi}X^k\oplus e^{-ik\varphi} X^{-k}=\{\cos(k\varphi)\Upsilon_{k,1}(s)+\sin(k\varphi)\Upsilon_{k,2}(s),\Upsilon_{k,1},\Upsilon_{k,2}\in X^k \}.
 \end{align*}
For any
$
\Upsilon=\cos(k\varphi)\Upsilon_{k,1}(s)+\sin(k\varphi)\Upsilon_{k,2}(s)\in Y^k$,
we have
\begin{align*}
J_{\omega}L\Upsilon=(\cos(k\varphi),\sin(k\varphi))J_{\omega}^kL^k\begin{pmatrix}\Upsilon_{k,1}\\ \Upsilon_{k,2}\end{pmatrix},
\end{align*}
where
\begin{align*}
J_{\omega}^k=\begin{pmatrix}0&-k\tilde \Upsilon_{\omega}'\\ k\tilde \Upsilon_{\omega}' &0\end{pmatrix},\quad L^k=\begin{pmatrix}L_k&0\\ 0 &L_k\end{pmatrix}.
\end{align*}
Note that $\sigma(J_{\omega,k}L_k|_{X^k})=$ $\overline{\sigma(J_{\omega,-k}L_{-k}|_{X^{-k}})}$  and $\sigma(J_{\omega}^kL^k|_{X^k\times X^k})=$ $\sigma(J_{\omega,k}L_k|_{X^k})\cup\sigma(J_{\omega,-k}$ $L_{-k}|_{X^{-k}})$.
By Theorem 2.3 in \cite{lin2022instability}, we have
\begin{align}\label{index-for real}
2\tilde k_{i,J_{\omega}^kL^k}^{\leq0}+\tilde k_{0,J_{\omega}^kL^k}^{\leq0}+2\tilde k_{c,J_{\omega}^kL^k}+\tilde k_{r,J_{\omega}^kL^k}=n^-(L^k)=2n^-(L_k),
\end{align}
where $n^-(L^k)$ is the negative dimension of the quadratic form $\langle L^k\cdot,\cdot\rangle$, $\tilde k_{r,J_{\omega}^kL^k}$ is the sum of algebraic multiplicities of positive eigenvalues of $J_{\omega}^k L^k$, $\tilde k_{c,J_{\omega}^kL^k}$ is the sum of algebraic multiplicities of eigenvalues of $J_{\omega}^k L^k$ in the first quadrant,
$\tilde{k}_{i,J_{\omega}^kL^k}^{\leq0}$ is the total
number of non-positive dimensions of $\langle L^{k}\cdot,\cdot\rangle$
restricted to the generalized eigenspaces    of purely imaginary
eigenvalues of $J_{\omega}^{k}L^{k}$ with positive imaginary parts, and
$\tilde{k}_{0,J_{\omega}^kL^k}^{\leq0}$ is the number of non-positive dimensions of $\langle
L^{k}\cdot,\cdot\rangle$ restricted to the generalized zero eigenspace of $J_{\omega}^k L^k$ modulo $\ker L^{k}$.
Then
\begin{equation}
2k_{i,J_{\omega,k}L_{k}}^{\leq0}=2\tilde{k}_{i,J_{\omega}^{k}L^{k}}^{\leq0},\;2 k_{0,J_{\omega,k}L_{k}}^{\leq0}=\tilde k_{0,J_{\omega}^{k}L^{k}}^{\leq0},\;2k_{c,J_{\omega,k}L_{k}}=2\tilde{k}_{c,J_{\omega}^{k}L^{k}},\;\;2k_{r,J_{\omega,k}L_{k}}%
=\tilde{k}_{r,J_{\omega}^{k}L^{k}}, \label{219}%
\end{equation}
where 
 ${k}%
_{r,J_{\omega,k}L_{k}}$ is the sum of algebraic multiplicities of positive eigenvalues of
$J_{\omega,k}L_{k}$, ${k}_{c,J_{\omega,k}L_{k}}$ is the sum of algebraic multiplicities
of eigenvalues of $J_{\omega,k}L_{k}$ in the first and the fourth
quadrants, ${k}_{i,J_{\omega,k}L_{k}}^{\leq0}$ is the total number of non-positive dimensions of
$\langle{L}_{k}\cdot,\cdot\rangle$ restricted to the
generalized eigenspaces of nonzero purely imaginary eigenvalues of
$J_{\omega,k}L_{k}$, and ${k}_{0,J_{\omega,k}L_{k}}^{\leq0}$ is the number of non-positive dimensions of $\langle
L_{k}\cdot,\cdot\rangle$ restricted to the generalized zero eigenspace of $J_{\omega,k} L_k$ modulo $\ker L_{k}$.
By \eqref{index-for real}-\eqref{219}, we have the index formula
\begin{align*}
 k_{i,J_{\omega,k}L_{k}}^{\leq0}+ k_{0,J_{\omega,k}L_{k}}^{\leq0}+ k_{c,J_{\omega,k}L_{k}}+ k_{r,J_{\omega,k}L_{k}}=n^-(L_k)
\end{align*}
for the complex operator $J_{\omega,k}L_{k}$.
For $k=\pm1$,
\begin{align*}
\sigma(L_k)=\left\{{1\over12}-{1\over l(l+1)}\bigg|l\geq2\right\}.
\end{align*}
For $|k|\geq2$,
\begin{align*}
\sigma(L_k)=\left\{{1\over12}-{1\over l(l+1)}\bigg|l\geq k\right\}.
\end{align*}
Then $$n^-(L_k)=1\quad\text{ for }k=\pm1,\pm2,$$  and $n^-(L_k)=0$ for $|k|\geq3$.
Thus, we only consider $k=1,2$ and the index formula becomes
\begin{align}\label{index-for complex2}
 k_{i,J_{\omega,k}L_{k}}^{\leq0}+ k_{0,J_{\omega,k}L_{k}}^{\leq0}+ k_{c,J_{\omega,k}L_{k}}+ k_{r,J_{\omega,k}L_{k}}=1.
\end{align}
Let
$$X_e^k=\{\Upsilon\in X^k|\Upsilon\text{ is even}\},\quad X_o^k=\{\Upsilon\in X^k|\Upsilon\text{ is odd}\}.$$
Note that $J_{\omega,k}$ and $L_k$ are invariant for the parity decomposition  in the sense that
\begin{align*}
J_{\omega,k}:(X_e^k)^*\cap D(J_{\omega,k})\to X_e^k,\; L_k:X_e^k\to (X_e^k)^*,\; J_{\omega,k}:(X_o^k)^*\cap D(J_{\omega,k})\to X_o^k,\; L_k:X_o^k\to (X_o^k)^*.
\end{align*}
Since
\begin{align}\label{L1e2ok3nonnegative}
L_1|_{X_e^1}\geq0,\;\; L_2|_{X_o^2}\geq0,\;\;\text{and}\;\; L_k\geq0,\;\;k\geq3,
\end{align}
and
\begin{align}\label{L1o2enegative}
n^-(L_1|_{X_o^1})=1\;\;\text{and}\;\;n^-(L_2|_{X_e^2})=1,
\end{align}
by \eqref{index-for complex2} we get the index formulae
\begin{align}\label{index formula 1o1} k_{i,J_{\omega,1}L_1|_{X_o^1}}^{\leq0}+ k_{0,J_{\omega,1}L_1|_{X_o^1}}^{\leq0}+ k_{c,J_{\omega,1}L_1|_{X_o^1}}+ k_{r,J_{\omega,1}L_1|_{X_o^1}}=1,\\\label{index formula 1o2}
  k_{i,J_{\omega,2}L_2|_{X_e^2}}^{\leq0}+ k_{0,J_{\omega,2}L_2|_{X_e^2}}^{\leq0}+ k_{c,J_{\omega,2}L_2|_{X_e^2}}+ k_{r,J_{\omega,2}L_2|_{X_e^2}}=1\,
\end{align}
 for   $J_{\omega,1}L_1|_{X_o^1}$ and  $J_{\omega,2}L_2|_{X_e^2}$.
 From the index formulae, the linear stability/instability of the 3-jet
 is
reduced to determine $ k_{i,J_{\omega,1}L_1|_{X_o^1}}^{\leq0}+ k_{0,J_{\omega,1}L_1|_{X_o^1}}^{\leq0}$ and
$k_{i,J_{\omega,2}L_2|_{X_e^2}}^{\leq0}+ k_{0,J_{\omega,2}L_2|_{X_e^2}}^{\leq0}$. Namely, if
 \begin{align*}
 \text{both}\quad k_{i,J_{\omega,1}L_1|_{X_o^1}}^{\leq0}+ k_{0,J_{\omega,1}L_1|_{X_o^1}}^{\leq0}=1\quad\text{ and } \quad k_{i,J_{\omega,2}L_2|_{X_e^2}}^{\leq0}+ k_{0,J_{\omega,2}L_2|_{X_e^2}}^{\leq0}=1,
 \end{align*}
 then the zonal flow $\tilde \Psi_{\omega}$   is spectrally stable for  $\omega=0$ and  thus, the $3$-jet
 is  spectrally stable for this $\omega$; if
 \begin{align*}
 \text{either}\quad k_{i,J_{\omega,1}L_1|_{X_o^1}}^{\leq0}+ k_{0,J_{\omega,1}L_1|_{X_o^1}}^{\leq0}=0\quad\text{ or } \quad k_{i,J_{\omega,2}L_2|_{X_e^2}}^{\leq0}+ k_{0,J_{\omega,2}L_2|_{X_e^2}}^{\leq0}=0,
 \end{align*}
  then
   the zonal flow $\tilde \Psi_{\omega}$   is linearly unstable for  $\omega=0$ and  thus, the $3$-jet is  linearly unstable for this $\omega$.

\subsection{Neutral modes: scope of the neutral speeds}\label{Neutral modes scope of the neutral speeds}
It is usually difficult to calculate
 the above indices in \eqref{index formula 1o1}-\eqref{index formula 1o2}, because we need to find the purely imaginary eigenvalues of $J_{\omega,k}L_k$ and calculate the signature of the corresponding energy quadratic form $\langle
L_{k}\cdot,\cdot\rangle$, where $k=1,2$. First, let us look for the purely imaginary eigenvalues of $J_{\omega,k}L_k$. Suppose $\lambda=-ik(c-c_\omega)$ is such an eigenvalue of $J_{\omega,k}L_k$ with corresponding eigenfunction $\Upsilon\in L^2(-1,1)$.
  Then
  \begin{align*}
  J_{\omega,k}L_k\Upsilon=-ik(c-c_\omega)\Upsilon \;\;\Longrightarrow\;\;
  \tilde \Upsilon_{\omega}'\left({1\over 12}+ \Delta_k^{-1}\right)\Upsilon=(c-c_\omega)\Upsilon.
  \end{align*}
  Let $\Phi=\Delta_k^{-1}\Upsilon$. Then direct computation implies that $\Phi$ solves the following Rayleigh equation
 \begin{align}\nonumber
 &\Delta_k\Phi-{\tilde \Upsilon_{\omega}'\over \tilde \Psi_{\omega}'+c}\Phi= ((1-s^2)\Phi')'-{k^2\over 1-s^2}\Phi-{\tilde \Upsilon_{\omega}'\over \tilde \Psi_{\omega}'+c}\Phi\\\label{Rayleigh-type equation}
 =&((1-s^2)\Phi')'-{k^2\over 1-s^2}\Phi-{-12(15s^2-3)+2\omega\over15s^2-3-\omega +c}\Phi=0, \quad \Delta_k\Phi\in L^2(-1,1),
 \end{align}
 where $c\in \mathbb{R}$ since $\lambda\in i\mathbb{R}$.
 By Lemma 2.4.1 and (2.4.9) in \cite{Skiba2017}, we have $\Delta_k\Phi\in L^2(-1,1)\Longrightarrow \nabla_k\Phi\in L^2(-1,1)\Longrightarrow\Phi (\pm1)=0$.
 We call the quadruple $(c,k,\omega,\Phi)$ a neutral mode if $\Delta_k\Phi\in L^2(-1,1)$ and $\Phi$ is a non-trivial solution of \eqref{Rayleigh-type equation} with $c\in \mathbb{R}$. Here, $c$ is called the neutral speed.
To compute the indices   $k_{i,J_{\omega,1}L_1|_{X_o^1}}^{\leq0}+ k_{0,J_{\omega,1}L_1|_{X_o^1}}^{\leq0}$ and  $k_{i,J_{\omega,2}L_2|_{X_e^2}}^{\leq0}+ k_{0,J_{\omega,2}L_2|_{X_e^2}}^{\leq0}$ in \eqref{index formula 1o1}-\eqref{index formula 1o2}, we determine the scope of $c\in \mathbb{R}$ such that $(c,k,\omega,\Phi)$ is a neutral mode.

 \begin{Remark}\label{neutral-imaginary}
 If we find a neutral mode $(c,k,\omega,\Phi)$, then $-ik(c-c_\omega)$ is a purely imaginary eigenvalue of $J_{\omega,k}L_k$ with the eigenfunction $\Delta_k\Phi$. To compute the index  $k_{i,J_{\omega,1}L_1|_{X_o^1}}^{\leq0}+ k_{0,J_{\omega,1}L_1|_{X_o^1}}^{\leq0}$ and  $k_{i,J_{\omega,2}L_2|_{X_e^2}}^{\leq0}+ k_{0,J_{\omega,2}L_2|_{X_e^2}}^{\leq0}$, it is thus a first step to determine for which $c\in \mathbb{R}$, $(c,k,\omega,\Phi)$ is a neutral mode.
 \end{Remark}

 In this subsection, we determine an interval  such that if $(c,k,\omega,\Phi)$ is a neutral mode, then  $c$ must lie in the interval.
 More precisely, we will prove  that
 \begin{itemize}
 \item $\omega\in(-3,12]\Rightarrow c\in\text{Ran}(-\tilde \Psi_{\omega}')^\circ=(-12+\omega,3+\omega)$, see Lemmas \ref{tilde-psi-change-sign} and \ref{case-omega=12} (i).

\item $\omega= -3\Rightarrow c\in (-15,0]$, see Lemma \ref{case-omega=12} (ii).

\item   $\omega\in(-\infty,-3)\Rightarrow c\in\text{Ran}(-\tilde \Psi_{\omega}')^\circ\cup[3+\omega,0]=(-12+\omega,0]$, see Lemma \ref{tilde Psi0 does not change sign}.

 \item $\omega\in(12,\infty)\Rightarrow c\in[0,-12+\omega]\cup\text{Ran}(-\tilde \Psi_{\omega}')^\circ=[0,3+\omega)$, see Lemma \ref{tilde Psi0 does not change sign}.
 \end{itemize}

 We first study the case that $\tilde \Psi_{\omega}'$ changes sign (or equivalently, $\omega\in\text{Ran}(\Psi_0')^\circ=(-3,12)$). Here, we use $\text{Ran}(\Psi_0')^\circ$ to denote the open interval $(\min( \Psi_0'),\max( \Psi_0'))$.
 \begin{Lemma}\label{tilde-psi-change-sign}
 If $\tilde \Psi_{\omega}'$ changes sign $(i.e.\;\omega\in(-3,12))$, then for any neutral mode $(c,k,\omega,\Phi)$ with $k\neq0$, $c$ must lie in $\text{Ran}(-\tilde \Psi_{\omega}')^\circ=(-12+\omega,3+\omega)$.
 \end{Lemma}
 \begin{proof} We divide the proof into two steps.

 \noindent{\bf Step 1.} Prove that for any neutral mode $(c,k,\omega,\Phi)$, $c$ must lie in $\text{Ran}(-\tilde \Psi_{\omega}')=[-12+\omega,3+\omega]$.

Let $c\notin \text{Ran}(-\tilde \Psi_{\omega}')$ and  define
\begin{align}\label{def-R F}
R_\omega(s)=\tilde \Psi_{\omega}'(s)+c=15s^2-3-\omega+c,\quad F_\omega(s)={\Phi(s)\over R_\omega(s)}={\Phi(s)\over  15s^2-3-\omega+c}
\end{align}
for $s\in[-1,1]$. By  \eqref{Rayleigh-type equation}, we have
\begin{align}\label{Rayleigh-type equation-RF}
-R_\omega\left(((1-s^2)(R_\omega F_\omega)')'-{k^2R_\omega F_\omega\over (1-s^2)}\right)+R_\omega F_\omega((1-s^2)(R_\omega-c))''=0.
\end{align}
Motivated by \cite{Thuburn-Haynes1996}, we have
\begin{align}\label{identity}
&R_\omega F_\omega((1-s^2)R_\omega)''-R_\omega((1-s^2)(R_\omega F_\omega)')'\\
&+(1-s^2)^{-{1\over2}}(((1-s^2)^{-{1\over2}}F_\omega)'(1-s^2)^2R_\omega^2)'
=-{F_\omega R_\omega^2\over 1-s^2}.\nonumber
\end{align}
Inserting \eqref{identity} into \eqref{Rayleigh-type equation-RF}, we have
\begin{align*}
R_\omega F_\omega(-c(1-s^2))''-(1-s^2)^{-{1\over2}}(((1-s^2)^{-{1\over2}}F_\omega)'(1-s^2)^2R_\omega^2)'-{F_\omega R_\omega^2\over 1-s^2}+{k^2R_\omega^2F_\omega\over 1-s^2}=0.
\end{align*}
Thus, in terms of $F_\omega$,  \eqref{Rayleigh-type equation} becomes
 \begin{align}\label{Rayleigh-type equation-F}
 \left\{ \begin{array}{llll}
-(1-s^2)^{-{1\over2}}(((1-s^2)^{-{1\over2}}F_\omega)'(1-s^2)^2R_\omega^2)'+{(k^2-1)R_\omega^2F_\omega\over 1-s^2}+2cR_\omega F_\omega=0,\\
F_\omega(\pm1)=0.
\end{array}\right.
\end{align}
Multiplying \eqref{Rayleigh-type equation-F} by $F_\omega$ and integrating from $-1$ to $1$, we obtain from integration by parts that
\begin{align}\label{inte}
\int_{-1}^1\left(|((1-s^2)^{-{1\over2}}F_\omega)'|^2(1-s^2)^2R_\omega^2+{(k^2-1)R_\omega^2F_\omega^2\over 1-s^2}+2cR_\omega F_\omega^2\right)ds=0,
\end{align}
where the boundary term vanishes, since
\begin{align*}
&(1-s^2)^{-{1\over2}}F_\omega((1-s^2)^{-{1\over2}}F_\omega)'(1-s^2)^2R_\omega^2|_{s=\pm1}\\
=&(1-s^2)^{{3\over2}}R_\omega^2F_\omega\left(-{1\over2}(1-s^2)^{-{3\over2}}(-2s)F_\omega+(1-s^2)^{-{1\over2}}F_\omega'\right)|_{s=\pm1}\\
=&sR_\omega^2F_\omega^2+(1-s^2)R_\omega^2F_\omega F_\omega'|_{s=\pm1}\\
=&s\Phi^2+(1-s^2)\Phi\Phi'-(1-s^2){\Phi^2R_\omega'\over R_\omega}|_{s=\pm1}\\
=&0.
\end{align*}
Let
\begin{align*}
P_\omega=|((1-s^2)^{-{1\over2}}F_\omega)'|^2(1-s^2)^2+{(k^2-1)F_\omega^2\over 1-s^2}\geq0,\quad Q_\omega=F_\omega^2\geq0,
\end{align*}
where $k\neq0$.
Then \eqref{inte} becomes
\begin{align}\label{inte-t}
\int_{-1}^1(P_\omega R_\omega^2+2cR_\omega Q_\omega)ds=0.
\end{align}
Then we divide the discussion into two cases.

Case 1. $c<\min(-\tilde \Psi_{\omega}')$.

Since $\tilde \Psi_{\omega}'$ changes sign, we have $c<0$. Moreover, $R_\omega(s)=c+\tilde \Psi_{\omega}'(s)<0$ for $s\in[-1,1]$. Thus, $2cR_\omega>0$.
By \eqref{inte-t}, we have $F_\omega\equiv0$, which is a contradiction.

Case 2. $c>\max(-\tilde \Psi_{\omega}')$.

Since $\tilde \Psi_{\omega}'$ changes sign, we have $c>0$. Moreover, $R_\omega(s)=c+\tilde \Psi_{\omega}'(s)>0$ for $s\in[-1,1]$. Thus, $2cR_\omega>0$.
This implies that $F_\omega\equiv0$ again.

In sum, we have $c\in \text{Ran}(-\tilde \Psi_{\omega}')$.

 \noindent{\bf Step 2.} Prove that for any  neutral mode $(c,k,\omega,\Phi)$,  $c\neq-12+\omega$ and $c\neq3+\omega$.

 Case 1.  $c\neq-12+\omega$.

 Suppose that there exists a  neutral mode $(c,k,\omega,\Phi)$ with $c=-12+\omega$. We still define $R_\omega(s)$ and $F_\omega(s)$ as in \eqref{def-R F}. Then we have \eqref{Rayleigh-type equation-F}. After  multiplying \eqref{Rayleigh-type equation-F} by $F_\omega$ and integrating from $-1$ to $1$, the difference in analysis is to handle the boundary term when using the integration by parts.
 In this case, the boundary term is
 \begin{align}\nonumber
&(1-s^2)^{-{1\over2}}F_\omega((1-s^2)^{-{1\over2}}F_\omega)'(1-s^2)^2R_\omega^2|_{s=\pm1}\\\nonumber
=&s\Phi^2+(1-s^2)\Phi\Phi'-(1-s^2){\Phi^2R_\omega'\over R_\omega}|_{s=\pm1}\\\label{boundary term pm1}
=&0,
\end{align}
 since $R_\omega(s)=15s^2-15$. Thus, we still have similar contradiction as in Step 1.

 Case 2. $c\neq3+\omega$.

 Suppose that there exists a  neutral mode $(c,k,\omega,\Phi)$ with $c=3+\omega$.
 Note that $R_\omega(s)=15s^2$, which is different from Case 1 since the singularity comes from the point $0$ rather than $\pm1$.
 By \eqref{def-tilde-Upsilon0-derivative}, we have $\tilde \Upsilon_{\omega}'(0)=36+2\omega$. Thus, $\tilde \Upsilon_{\omega}'(0)\neq0$ for $\omega\in(-3,12)$.
 The first two terms in \eqref{Rayleigh-type equation} is in $L^2(-1,1)$, and thus, so does the last term $-{\tilde \Upsilon_{\omega}'\over 15s^2}\Phi$. This means that $\Phi(0)$ and $\Phi'(0)$ have to be $0$.
 After  multiplying \eqref{Rayleigh-type equation-F} by $F$ and now integrating from $0$ to $1$, let us look at  the boundary term at $0$:
 \begin{align*}
&(1-s^2)^{-{1\over2}}F_\omega((1-s^2)^{-{1\over2}}F_\omega)'(1-s^2)^2R_\omega^2|_{s=0}\\
=&s\Phi^2+(1-s^2)\Phi\Phi'-(1-s^2){\Phi^2R_\omega'\over R_\omega}|_{s=0}\\
=&0,
\end{align*}
 since $(1-s^2){\Phi^2R_\omega'\over R_\omega}={30\Phi^2s\over 15s^2}\to0$ as $s\to0+$.  Thus, similar contradiction in Step 1 appears.
 \end{proof}
 We consider the boundary cases $\omega= 12$ and $\omega=-3$.
 \begin{Lemma}\label{case-omega=12}
$(\rm{i})$ Let $\omega= 12$ and $k\neq0$. Then for any neutral mode $(c,k,12,\Phi)$, we have $c\in \text{Ran}(-\tilde \Psi_{\omega}')^\circ$ $=(0,15)$.

$(\rm{ii})$ Let $\omega= -3$ and $k\neq0$. Then for any neutral mode $(c,k,-3,\Phi)$, we have $c\in \text{Ran}(-\tilde \Psi_{\omega}')^\circ\cup\{0\}=(-15,0]$.
 \end{Lemma}
 \begin{Remark}
For $k=1$ and $\omega= -3$, the point $c=0$ is quite different, since there exists a neutral mode with $c=0$ (see Lemma \ref{ometa=3} for details).
 \end{Remark}
 \begin{proof}
  (i) Note that $\text{Ran}(-\tilde \Psi_{\omega}')=[0,15]$. By a similar argument to Step 1  and Case 2 of Step 2 in the proof of Lemma \ref{tilde-psi-change-sign}, we have
 $c\in[0,15)$  for any neutral mode $(c,k,12,\Phi)$. But the proof of $c\neq0$ is slightly different from Case 1 of Step 2 in the proof of Lemma  \ref{tilde-psi-change-sign}. Indeed, similar to \eqref{inte-t} we have
$
\int_{-1}^1P_\omega R_\omega^2ds=0$ for $c=0$. If $k\neq\pm1$, by the definition of $F_\omega$ we have $F_\omega\equiv0$, which is a contradiction. If $k=\pm1$, we have $ P_\omega(s)=|((1-s^2)^{-{1\over2}}F_\omega)'|^2(1-s^2)^2\equiv0$, which implies  $F_\omega(s)={\Phi(s)\over 15s^2-15}=C_0(1-s^2)^{1\over2}$ for some $C_0\in\mathbb{R}$.
Thus, $\Phi(s)=-15C_0(1-s^2)^{3\over2}=C_0P_3^3(s)$.  Then $\Delta_3\Phi=-12\Phi$ and
\begin{align*}
\Upsilon(\varphi,s)=e^{i\varphi}\Delta_1\Phi(s)=e^{i\varphi}\left({8\over1-s^2}\Phi(s)-12\Phi(s)\right)=-15C_0e^{i\varphi}(8(1-s^2)^{1\over2}-12(1-s^2)^{3\over2}).
\end{align*}
 $\Upsilon\in X$ implies that
\begin{align}\label{vorticity Y1-1 constraint}
\iint_{D_T}\Upsilon Y_1^{-1}d\varphi ds=-15\pi C_0\sqrt{3\over 2\pi}\int_{-1}^1(8(1-s^2)-12(1-s^2)^{2})ds=0.
\end{align}
Since $\int_{-1}^1(8(1-s^2)-12(1-s^2)^{2}) ds\neq0$, we have $C_0=0$ and $\Phi\equiv0$, which finishes the proof of $c\neq0$.

(ii) Note that $\text{Ran}(-\tilde \Psi_{\omega}')=[-15,0]$. The proof of $c\in(-15,0]$ for a neutral mode $(c,k,-3,\Phi)$ is similar as above.
\end{proof}
 Next, we consider  the case that $\tilde \Psi_{\omega}'$ does not change sign (i.e. $\omega\in(-\infty,-3)\cup(12,\infty)$).

\begin{Lemma}\label{tilde Psi0 does not change sign}
 $(\rm{1})$ If $\tilde \Psi_{\omega}'$ is positive $($i.e. $\omega\in(-\infty,-3)$$)$, then for any neutral mode $(c,k,\omega,\Phi)$ with $k\neq0$, $c$ must lie in $\text{Ran}(-\tilde \Psi_{\omega}')^\circ\cup[3+\omega,0]=(-12+\omega,0]$.

\medskip

\begin{center}
 \begin{tikzpicture}[scale=0.58]
 \draw [->](-12, 0)--(10, 0)node[right]{$c$};
       \node (a) at (5,-0.5) {\tiny$0$};
       \node (a) at (-2,-0.5) {\tiny$3+\omega$};
        \node (a) at (-10,-0.5) {\tiny$-12+\omega$};
       \draw  (-2, -0.1).. controls (-2, 0.1) and (-2, 0.1)..(-2, 0.1);
        \draw  (5, -0.1).. controls (5, 0.1) and (5, 0.1)..(5, 0.1);
       \draw  (-10, -0.1).. controls (-10, 0.1) and (-10, 0.1)..(-10, 0.1);
         \draw(-10, -0.1).. controls (-10, 0.1) and (-10, 0.1)..(-10, 0.1);

        \draw  [red][thick] (-2, 0).. controls (-2, 0) and (5, 0)..(5, 0);
       \node (a) at (-6,-0.5) {\tiny$\text{Ran}(-\tilde \Psi_{\omega}')$};
 \end{tikzpicture}
\end{center}
\vspace{-0.7cm}
\begin{figure}[ht]
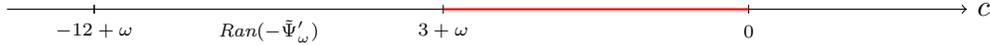

    \centering
    \caption{Neutral speeds for $\omega\in(-\infty,-3)$.
}
\label{Neutral speeds for omega-infty-3)}
\end{figure}
\vspace{-0.2cm}

 $(\rm{2})$ If $\tilde \Psi_{\omega}'$ is negative $($i.e. $\omega\in(12,\infty)$$)$, then for any neutral mode $(c,k,\omega,\Phi)$ with $k\neq0$, $c$ must lie in $[0,-12+\omega]\cup\text{Ran}(-\tilde \Psi_{\omega}')^\circ=[0,3+\omega)$.

\medskip

\begin{center}
 \begin{tikzpicture}[scale=0.58]
 \draw [->](-12, 0)--(10, 0)node[right]{$c$};
       \node (a) at (5,-0.5) {\tiny$3+\omega$};
       \node (a) at (-3,-0.5) {\tiny$-12+\omega$};
        \node (a) at (-10,-0.5) {\tiny$0$};
       \draw  (-3, -0.1).. controls (-3, 0.1) and (-3, 0.1)..(-3, 0.1);
        \draw  (5, -0.1).. controls (5, 0.1) and (5, 0.1)..(5, 0.1);
       \draw  (-10, -0.1).. controls (-10, 0.1) and (-10, 0.1)..(-10, 0.1);
         \draw(-10, -0.1).. controls (-10, 0.1) and (-10, 0.1)..(-10, 0.1);

        \draw  [red][thick] (-10, 0).. controls (-10, 0) and (-3, 0)..(-3, 0);
       \node (a) at (1,-0.5) {\tiny$\text{Ran}(-\tilde \Psi_{\omega}')$};
 \end{tikzpicture}
\end{center}

\vspace{-0.7cm}
\begin{figure}[ht]
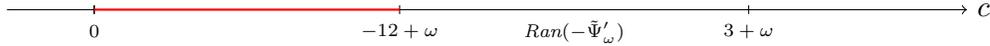

    \centering
    \caption{Neutral speeds for $\omega\in(12,\infty)$.
}
\label{Neutral speeds for omega-in12-infty}
\end{figure}
\vspace{-0.4cm}
\end{Lemma}

\begin{Remark}\label{rem tilde Psi0 does not change sign} Lemma \ref{tilde Psi0 does not change sign} can be rewritten in the following way.

 $(\rm{1})$ If  $\omega\in(-\infty,-3)$ and $(c,k,\omega,\Phi)$ is a neutral mode with $c\notin\text{Ran}(-\tilde \Psi_{\omega}')^\circ$, $k\neq0$, then $c\in[3+\omega,0]$.

 $(\rm{2})$ If  $\omega\in(12,\infty)$ and  $(c,k,\omega,\Phi)$ is a neutral mode with $c\notin\text{Ran}(-\tilde \Psi_{\omega}')^\circ$, $k\neq0$, then $c\in[0,-12+\omega]$.
\end{Remark}

\begin{proof}
(1) Let $c\leq\min(-\tilde \Psi_{\omega}')=-12+\omega$ or $c>0$. Suppose that $(c,k,\omega,\Phi)$ is a neutral mode.  Similar to Step 1 (for $c<-12+\omega$ or $c>0$) and Case 1 of Step 2 (for $c=-12+\omega$) in the proof of Lemma \ref{tilde-psi-change-sign}, we have \eqref{inte-t}.

If $c\leq-12+\omega$, then $c<0$. Moreover, $R_\omega(s)=\tilde \Psi_{\omega}'(s)+c<0$ for  $s\in[-1,1]$ if $c<-12+\omega$ and $R_\omega(s)=15s^2-15<0$ for  $s\in(-1,1)$ if $c=-12+\omega$. Thus, $2cR_\omega>0$ for $s\in(-1,1)$ and $F_\omega\equiv0$.

If $c>0$, then $R_\omega(s)=\tilde \Psi_{\omega}'(s)+c>0$ for  $s\in[-1,1]$. Again, we have $2cR_\omega>0$ and $F_\omega\equiv0$.

(2) The proof is similar to (1).
\end{proof}

\subsection{The resonant neutral modes}\label{The neutral modes with c in Ran tilde Psi 0 prime}

Let $\omega\in\mathbb{R}$ and  recall that  $\text{Ran} (-\tilde \Psi_{\omega}')^\circ=(-12+\omega,3+\omega)$.
In this subsection,  we study for which $c\in (-12+\omega,3+\omega)$,  $(c,k,\omega,\Phi)$ is a neutral mode.

Note that $-\tilde \Psi_{\omega}'(\pm1)=-12+\omega$ and $-\tilde \Psi_{\omega}'(0)=3+\omega$ are the endpoints of $\text{Ran} (-\tilde \Psi_{\omega}')$, and $\tilde \Psi_{\omega}''(s)=30s\neq0$ for $s\in(-1,0)\cup(0,1)$. Then for $c\in \text{Ran} (-\tilde \Psi_{\omega}')^\circ$, there exist two points $s_1<s_2$ in $(-1,0)\cup(0,1)$ solving $c+\tilde \Psi_{\omega}'(s_i)=0$ and $\tilde \Psi_{\omega}''(s_i)\neq0$ for $i=1,2$.  Let $s_0=-1$ and $s_{3}=1$. Then we have the following result.

\begin{Lemma}\label{inner c}
Let $\omega\in\mathbb{R}$, $k\neq0$, $c\in \text{Ran} (-\tilde \Psi_{\omega}')^\circ=(-12+\omega,3+\omega)$, $\{s_j\}_{j=0}^{3}$ be defined as above
 and $\Phi$ solves \eqref{Rayleigh-type equation}. Assume that one of the following conditions

$(\rm{i})$
  $\Phi$ is odd if $k=1$;

  $(\rm{ii})$ $\omega\in(-\infty,-3)\cup(12,\infty)$,\\
 holds.  Then  $\Phi$ can not vanish at $s_{i_0-1}$, $s_{i_0}$ and $s_{i_0+1}$ simultaneously unless it vanishes identically on at least one of $(s_{i_0-1},s_{i_0})$ and $(s_{i_0},s_{i_0+1})$, where $i_0=1,2$.
\end{Lemma}

\begin{proof}
Suppose that $\Phi(s_i)=0$ for $i=i_0-1,i_0,i_0+1$. Assume that $\Phi\not\equiv0$ on both intervals $(s_{i_0-1},s_{i_0})$ and $(s_{i_0},s_{i_0+1})$.
Here, if (i) holds, then we argue  by the following $6$ cases. If (ii) holds, then $3+\omega<0$ or $-12+\omega>0$, which implies $c\neq0$. We divide the proof  in Cases 1-4.

Case 1. $c>0$ and $\tilde \Psi_{\omega}''(s_{i_0})<0$.

Consider the interval $(s_{i_0-1},s_{i_0})$. Since $\tilde \Psi_{\omega}''(s_{i_0})<0$, we have $R_\omega(s)=\tilde \Psi_{\omega}'(s)+c>0$ on $(s_{i_0-1},s_{i_0})$. Let $\tilde s\in [s_{i_0-1},s_{i_0})$ be the nearest zero of $\Phi$ to $s_{i_0}$. Without loss of generality, assume that $\Phi>0$ on $(\tilde s,s_{i_0})$, $\Phi'(\tilde s)\geq0$ and $\Phi'(s_{i_0})\leq0$.
Similar to \eqref{Rayleigh-type equation-F}, we have
 \begin{align}\label{F-c}
\left(\left((1-s^2)^{-{1\over2}}{\Phi\over R_\omega}\right)'(1-s^2)^2R_\omega^2\right)'-{(k^2-1)R_\omega\Phi\over (1-s^2)^{1\over 2}}-2c(1-s^2)^{{1\over2}}\Phi=0,
\end{align}
where  $s\in(s_{i_0-1},s_{i_0})$.
Integrating \eqref{F-c} from $\tilde s$ to $s_{i_0}$, we have
 \begin{align}\label{int-F-c}
 \left(\left((1-s^2)^{-{1\over2}}{\Phi\over R_\omega}\right)'(1-s^2)^2R_\omega^2\right)\bigg|_{s=\tilde s}^{s_{i_0}}=\int_{\tilde s}^{s_{i_0}}\left({(k^2-1)R_\omega\Phi\over (1-s^2)^{1\over 2}}+2c(1-s^2)^{{1\over2}}\Phi\right)ds.
 \end{align}
 Direct computation implies
 \begin{align}\label{di-c}
 \left((1-s^2)^{-{1\over2}}{\Phi\over R_\omega}\right)'(1-s^2)^2R_\omega^2=s(1-s^2)^{1\over2}\Phi R_\omega +(1-s^2)^{3\over2}(\Phi'R_\omega-\Phi R_\omega').
 \end{align}
 Since $\Phi(s_{i_0})=R_\omega(s_{i_0})=0$, by \eqref{di-c} we have
 \begin{align}\label{s-i-0}
  \left((1-s^2)^{-{1\over2}}{\Phi\over R_\omega}\right)'(1-s^2)^2R_\omega^2\bigg|_{s=s_{i_0}}=0.
 \end{align}

 If $\tilde s=s_{i_0-1}$, then
 $\Phi(s_{i_0-1})=0$, and moreover, $R_\omega(s_{i_0-1})=0$ if $i_0-1\neq0$ and $1-s_{i_0-1}^2=0$ if $i_0-1=0$. Thus, by \eqref{di-c} we have
  \begin{align*}
  \left((1-s^2)^{-{1\over2}}{\Phi\over R_\omega}\right)'(1-s^2)^2R_\omega^2\bigg|_{s=s_{i_0-1}}=0.
 \end{align*}
 This implies that the LHS of \eqref{int-F-c} is zero. Noting that $k\neq0$, $c>0$, $R_\omega>0$ and $\Phi\geq0$ on $(s_{i_0-1},s_{i_0})$, it then follows from the RHS of \eqref{int-F-c} that $\Phi\equiv0$ on $(s_{i_0-1},s_{i_0})$, which is a contradiction.

If $\tilde s>s_{i_0-1}$, then by the fact that $\Phi(\tilde s)=0$, we infer from \eqref{s-i-0} and \eqref{di-c} that
\begin{align*}
 \text{LHS of \eqref{int-F-c}}=&-\left((1-s^2)^{-{1\over2}}{\Phi\over R_\omega}\right)'(1-s^2)^2R_\omega^2\bigg|_{s=\tilde s}\\
 =&-s(1-s^2)^{1\over2}\Phi R_\omega -(1-s^2)^{3\over2}(\Phi'R_\omega-\Phi R_\omega')\bigg|_{s=\tilde s}\\
= &-(1-s^2)^{3\over2}\Phi'R_\omega\bigg|_{s=\tilde s}\leq0.
\end{align*}
However, since $k\neq0$, $c>0$, $R_\omega>0$ and $\Phi\geq0$ on $(\tilde s,s_{i_0})$, we have
$
 \text{RHS of \eqref{int-F-c}}\geq0$.
Then $\Phi\equiv0$ on $(\tilde s,s_{i_0})$, and thus, on  $(s_{i_0-1},s_{i_0})$, which is a contradiction.

Case 2. $c>0$ and $\tilde \Psi_{\omega}''(s_{i_0})>0$.

In this case, we consider the interval $(s_{i_0},s_{i_0+1})$.  Since $\tilde \Psi_{\omega}''(s_{i_0})>0$, $R_\omega(s)=\tilde \Psi_{\omega}'(s)+c>0$ on $(s_{i_0},s_{i_0+1})$. Let $\tilde s\in (s_{i_0},s_{i_0+1}]$ be the nearest zero of $\Phi$ to $s_{i_0}$. Assume that $\Phi>0$ on $(s_{i_0},\tilde s)$, $\Phi'(s_{i_0})\geq0$ and $\Phi'(\tilde s)\leq0$.
Integrating \eqref{F-c} from $s_{i_0}$ to $\tilde s$, we have
 \begin{align}\label{int-F-c-2}
 \left(\left((1-s^2)^{-{1\over2}}{\Phi\over R_\omega}\right)'(1-s^2)^2R_\omega^2\right)\bigg|_{s=s_{i_0}}^{\tilde s}=\int_{s_{i_0}}^{\tilde s}\left({(k^2-1)R_\omega\Phi\over (1-s^2)^{1\over 2}}+2c(1-s^2)^{{1\over2}}\Phi\right)ds.
 \end{align}
 At the point $s_{i_0}$, similar to \eqref{s-i-0}, we have
 \begin{align}\label{s-i-0-2}
  \left((1-s^2)^{-{1\over2}}{\Phi\over R_\omega}\right)'(1-s^2)^2R_\omega^2\bigg|_{s=s_{i_0}}=0.
 \end{align}

 If $\tilde s=s_{i_0+1}$, then
 $\Phi(s_{i_0+1})=0$, and moreover, $R_\omega(s_{i_0+1})=0$ if $i_0\neq 2$ and $1-s_{i_0+1}^2=0$ if $i_0=2$. Thus, by \eqref{di-c} we have
  \begin{align*}
  \left((1-s^2)^{-{1\over2}}{\Phi\over R_\omega}\right)'(1-s^2)^2R_\omega^2\bigg|_{s=s_{i_0+1}}=0.
 \end{align*}
 This implies that the LHS of \eqref{int-F-c-2} is zero. Noting that $k\neq0$, $c>0$, $R_\omega>0$ and $\Phi\geq0$ on $(s_{i_0},s_{i_0+1})$, we have by \eqref{int-F-c-2} that $\Phi\equiv0$ on $(s_{i_0},s_{i_0+1})$.

If $\tilde s<s_{i_0+1}$, then by the fact that $\Phi(\tilde s)=0$, we infer from \eqref{s-i-0-2} and \eqref{di-c} that
\begin{align*}
 \text{LHS of \eqref{int-F-c-2}}=&\left((1-s^2)^{-{1\over2}}{\Phi\over R_\omega}\right)'(1-s^2)^2R_\omega^2\bigg|_{s=\tilde s}
= (1-s^2)^{3\over2}\Phi'R_\omega\bigg|_{s=\tilde s}\leq0.
\end{align*}
However, the
$
 \text{RHS of \eqref{int-F-c-2}}\geq0$.
Then $\Phi\equiv0$ on   $(s_{i_0},s_{i_0+1})$.

Case 3. $c<0$ and $\tilde \Psi_{\omega}''(s_{i_0})<0$.

In this case, we consider the interval $(s_{i_0},s_{i_0+1})$.  Then $R_\omega(s)=\tilde \Psi_{\omega}'(s)+c<0$ on $(s_{i_0},s_{i_0+1})$. Let $\tilde s\in (s_{i_0},s_{i_0+1}]$ be the nearest zero of $\Phi$ to $s_{i_0}$. Assume that $\Phi>0$ on $(s_{i_0},\tilde s)$, $\Phi'(s_{i_0})\geq0$ and $\Phi'(\tilde s)\leq0$.
Integrating \eqref{F-c} from $s_{i_0}$ to $\tilde s$, we again have \eqref{int-F-c-2}.
If $\tilde s=s_{i_0+1}$, similar to Case 2, the LHS of \eqref{int-F-c-2} is zero. Noting that $k\neq0$, $c<0$, $R_\omega<0$ and $\Phi\geq0$ on $(s_{i_0},s_{i_0+1})$, we have by \eqref{int-F-c-2} that $\Phi\equiv0$ on $(s_{i_0},s_{i_0+1})$.
If $\tilde s<s_{i_0+1}$, then the
$
 \text{LHS of \eqref{int-F-c-2}}= (1-s^2)^{3\over2}\Phi'R_\omega|_{s=\tilde s}\geq0.
$
However, the
$
 \text{RHS of \eqref{int-F-c-2}}\leq0$ since $c<0$, $R_\omega<0$ and $\Phi\geq0$ on $(s_{i_0},\tilde s)$.
Then $\Phi\equiv0$ on   $(s_{i_0},s_{i_0+1})$.

Case 4. $c<0$ and $\tilde \Psi_{\omega}''(s_{i_0})>0$.

In this case, we consider the interval $(s_{i_0-1},s_{i_0})$.  Then $R_\omega(s)=\tilde \Psi_{\omega}'(s)+c<0$ on $(s_{i_0-1},s_{i_0})$. Let $\tilde s\in [s_{i_0-1},s_{i_0})$ be the nearest zero of $\Phi$ to $s_{i_0}$. Assume that $\Phi>0$ on $(\tilde s, s_{i_0})$, $\Phi'(\tilde s)\geq0$ and $\Phi'( s_{i_0})\leq0$.
Integrating \eqref{F-c} from  $\tilde s$ to $s_{i_0}$, we again have \eqref{int-F-c}.
If $\tilde s=s_{i_0-1}$, similar to Case 1, the LHS of \eqref{int-F-c} is zero. Noting that $k\neq0$, $c<0$, $R_\omega<0$ and $\Phi\geq0$ on $(s_{i_0-1},s_{i_0})$, we have by \eqref{int-F-c} that $\Phi\equiv0$ on $(s_{i_0-1},s_{i_0})$.
If $\tilde s>s_{i_0-1}$, then the
$
 \text{LHS of \eqref{int-F-c}}= -(1-s^2)^{3\over2}\Phi'R_\omega|_{s=\tilde s}\geq0.
$
However, the
$
 \text{RHS of \eqref{int-F-c}}\leq0$ since $k\neq0$, $c<0$, $R_\omega<0$ and $\Phi\geq0$ on $(\tilde s,s_{i_0})$.
Then $\Phi\equiv0$ on   $(s_{i_0-1},s_{i_0})$.

Case 5. $c=0$ and  $k\neq1$.

 Since $k\neq1$, the proof is a repetition of Cases $1$-$4$.

Case 6. $c=0$ and $k=1$.

In this case, we consider the interval $[s_1,s_2]$. By \eqref{F-c}, we have
\begin{align*}
\left((1-s^2)^{-{1\over2}}{\Phi\over R_\omega}\right)'(1-s^2)^2R_\omega^2=C_1\quad\text{on}\quad[s_1,s_2]
\end{align*}
for some $C_1\in\mathbb{R}$. By
\eqref{di-c} and the fact that $\Phi(s_{1})=R_\omega(s_{1})=0$, we have $\left((1-s^2)^{-{1\over2}}{\Phi\over R_\omega}\right)'(1-s^2)^2R_\omega^2\bigg|_{s=s_{1}}=0$, and thus, $C_1=0$. Then $(1-s^2)^{-{1\over2}}{\Phi\over R_\omega}=C_2$ on $[s_1,s_2]$ for some $C_2\in\mathbb{R}$. This implies that $\Phi(s)=C_2(15s^2-3-\omega)(1-s^2)^{1\over2}$ on $[s_1,s_2]$. Note that $s_2=-s_1$.  Since $\Phi$ is odd, we have $C_2=0$ and $\Phi\equiv0$ on $[s_1,s_2]$.
\end{proof}

Next, we prove a uniqueness result for the initial value problem of a singular ODE.

\begin{Lemma}\label{initial ode}
Let $\omega\in\mathbb{R}$, $k\neq0$, $c\in \text{Ran} (-\tilde \Psi_{\omega}')^\circ=(-12+\omega,3+\omega)$ and $\{s_j\}_{j=0}^{3}$ be defined as above. Let $1\leq i_0\leq 2$.
 If $\Phi$ solves the equation in \eqref{Rayleigh-type equation} and $\Phi\in C^1(s_{i_0-1},s_{i_0+1})$ satisfies that $\Phi(s_{i_0})=\Phi'(s_{i_0})=0$, then $\Phi\equiv0$ on $(s_{i_0-1},s_{i_0+1})$.
\end{Lemma}

\begin{proof}
First, we consider the interval  $[s_{i_0},s_{i_0+1}).$ Since $\tilde \Psi_{\omega}''(s_{i_0})\neq0$ and $s_{i_0}\in(-1,1)$, we can choose $\delta_0>0$ satisfying that  there exist  $C_1, C_2, C_3>0$ such that $C_1<| \tilde \Psi_{\omega}''(s)|<C_2$ and  $C_3< |1-s^2|$ for $s\in[s_{i_0},s_{i_0}+\delta_0]$. Let $\xi=(1-s^2)\Phi'$. Then we have by \eqref{Rayleigh-type equation} that
\begin{equation*}
\left\{
\begin{array}
[c]{l}%
\Phi'={1\over 1-s^2}\xi,\\
\xi'={k^2\over 1-s^2}\Phi+{\tilde \Upsilon_{\omega}'\over \tilde \Psi_{\omega}'+c}\Phi,
\end{array}
\right.
\end{equation*}
with the initial data $\Phi(s_{i_0})=\xi(s_{i_0})=0$. Let  $s\in [s_{i_0},s_{i_0}+\delta_0]$. For any $\tilde s\in[s_{i_0},s]$, we have
\begin{equation*}
|\Phi(\tilde s)|\leq\int_{s_{i_0}}^{\tilde s}{1\over 1- \tau^2}|\xi(\tau)|d\tau
\leq {1\over C_3}( \tilde s-s_{i_0})\|\xi\|_{L^{\infty}(s_{i_0},\tilde s)}\leq{\delta_0\over C_3}\|\xi\|_{L^{\infty}(s_{i_0},\tilde s)},
\end{equation*}
and
\begin{equation*}
\left|{\tilde \Upsilon_{\omega}'(\tilde s)\over \tilde \Psi_{\omega}'(\tilde s)+c}\Phi(\tilde s)\right|
\leq C\left|{\Phi(\tilde s)-\Phi(s_{i_0})\over \tilde \Psi_{\omega}'(\tilde s)-\tilde \Psi_{\omega}'(s_{i_0})}\right|\leq{\frac{C\|\Phi'\|_{L^{\infty}(s_{i_0},\tilde s)}}{C_1}}.
\end{equation*}
Then
\begin{align*}
\|\xi\|_{L^{\infty}(s_{i_0},s)}   \leq&\int_{s_{i_0}}^{s}\left( \left|{k^2\over 1-\tilde s^2}\Phi(\tilde s)\right|+\left\vert {\tilde \Upsilon_{\omega}'(\tilde s)\over \tilde \Psi_{\omega}'(\tilde s)+c}\Phi(\tilde s)\right\vert
\right)  d\tilde s  \\
\leq&\int_{s_{i_0}}^{s}\left({k^2\over C_3}|\Phi(\tilde s)|+{\frac{C\|\Phi'\|_{L^{\infty}(s_{i_0},\tilde s)}}{C_1}}\right)d\tilde s\\
\leq &\int_{s_{i_0}}^{s}\left({k^2\delta_0\over C_3^2}\|\xi\|_{L^{\infty}(s_{i_0},\tilde s)}+{\frac{C}{C_1C_3}}\|\xi\|_{L^{\infty}(s_{i_0},\tilde s)}\right)d\tilde s,
\end{align*}
where we used $\|\Phi'\|_{L^{\infty}(s_{i_0},\tilde s)}=\|{1\over 1-s^2}\xi\|_{L^{\infty}(s_{i_0},\tilde s)}\leq {1\over C_3}\|\xi\|_{L^{\infty}(s_{i_0},\tilde s)}$.
By Gr\"{o}nwall inequality, we have $\xi\equiv0$ and $\Phi
\equiv0\ $ on $[s_{i_0},s_{i_0}+\delta_{0}]$ (thus, on $[s_{i_0},s_{i_0+1}]$).

The proof of $\Phi\equiv0$ on $(s_{i_0-1},s_{i_0})$ is similar as above.
\end{proof}

The $3$-jet is spectrally  stable for $\omega\in(-\infty,-18]\cup[72,\infty)$ by the Rayleigh's criterion. So, we only need to consider $\omega\in(-18,72)$.
Note that $c_\omega={5\over6}\omega\in\text{Ran} (-\tilde \Psi_{\omega}')^\circ=(-12+\omega,3+\omega)$ for $\omega\in(-18,72)$. Now, we determine $c$ for a neutral mode $(c,k,\omega,\Phi)$ if $c\in \text{Ran} (-\tilde \Psi_{\omega}')^\circ$.

\begin{Theorem} \label{Psi prime change sign c}
$(\rm{i})$ Let $\omega\in(-18,72)$, $k\neq0$ and  $(c,k,\omega,\Phi)$ is a  neutral mode with $c\in\text{Ran} (-\tilde \Psi_{\omega}')^\circ=(-12+\omega,3+\omega)$. Assume that  $\Phi$ is odd if $k=1$. Then $c$ must be $c_\omega$.

$(\rm{ii})$ Let $\omega\in(-\infty,-18)\cup(72,\infty)$ and  $k\neq0$. Then there exist no neutral modes $(c,k,\omega,\Phi)$ such that $c\in\text{Ran} (-\tilde \Psi_{\omega}')^\circ=(-12+\omega,3+\omega)$.
\end{Theorem}
\begin{proof}
We claim that if $c\in \text{Ran} (-\tilde \Psi_{\omega}')^\circ$ satisfies that $(c,k,\omega,\Phi)$ is a  neutral mode, then there exists $1\leq i_1\leq 2$ such that $\Phi(s_{i_1})\neq 0$. Here, $\{s_i\}_{i=1}^{2}$ are defined before Lemma \ref{inner c}. In fact, suppose that $\Phi(s_{i})=0$ for all $1\leq i\leq 2$. Now, we fix $1\leq i_0\leq 2$.
Then Lemma \ref{inner c} tells us that $\Phi\equiv0$ on at least one of $(s_{i_0-1},s_{i_0})$ and $(s_{i_0},s_{i_0+1})$. Since $\Delta_k\Phi\in L^2(-1,1)$, we have $\Phi\in C^1((s_{i_0-1},s_{i_0+1}))$. By Lemma \ref{initial ode}, we have $\Phi\equiv0$ on $(s_{i_0-1},s_{i_0+1})$. By the arbitrary choice of $i_0$, we have $\Phi\equiv0$ on $(-1,1)$.

(i) For the above $s_{i_1}$, if $c\neq c_\omega$, we have by \eqref{c omega k} that
\begin{align*}
\tilde \Upsilon_{\omega}'(s_{i_1})=-12(\tilde \Psi_{\omega}'(s_{i_1})+c_\omega)=-12(-c+c_\omega)\neq 0.
\end{align*}
By the fact that  $\Phi(s_{i_1})\neq 0$ and $\tilde \Upsilon_{\omega}'(s_{i_1})\neq 0$, it follows from \eqref{Rayleigh-type equation} that
\begin{align}\label{Laplace k Phi neq0}
 ((1-s^2)\Phi')'-{k^2\over 1-s^2}\Phi={\tilde \Upsilon_{\omega}'\over \tilde \Psi_{\omega}'+c}\Phi\notin L_{loc}^2
 \end{align}
 near $s_{i_1}$. This contradicts $\Delta_k\Phi\in L^2(-1,1)$. Thus, $c= c_\omega$.

 (ii) Since $\omega\in(-\infty,-18)\cup(72,\infty)$, we have  $\tilde \Upsilon_{\omega}'(s_{i_1})\neq0$. But $\Phi(s_{i_1})\neq 0$ and $\tilde \Psi_{\omega}'(s_{i_0})+c=0$.  Similar to \eqref{Laplace k Phi neq0}, we have $\Delta_k\Phi\notin  L_{loc}^2$. This is a contradiction.
\end{proof}

For $\omega\in(-18,72)$ and the  neutral modes $(c_\omega,k,\omega,\Phi_k)$ with $k=1,2$, it follows from Remark \ref{neutral-imaginary} that the corresponding imaginary eigenvalue of $J_{\omega,k}L_k$ is $-ik(c-c_\omega)=0$.
To study the indices $k_{0,J_{\omega,1}L_1|_{X_o^1}}^{\leq0}$ and $k_{0,J_{\omega,2}L_2|_{X_e^2}}^{\leq0}$ in
\eqref{index formula 1o1} and \eqref{index formula 1o2}, by the definition we need to study the generalized kernels  of
$J_{\omega,1}L_1|_{X_o^1}$ and $J_{\omega,2}L_2|_{X_e^2}$. It turns out that the kernels of the above two operators are  trivial.

\begin{Lemma}\label{kernel JL}
Let $\omega\in(-18,72)$. Then $\ker (J_{\omega,1}L_1|_{X_o^1})=\ker(J_{\omega,2}L_2|_{X_e^2})=\{0\}$.
\end{Lemma}
\begin{proof}
Let  $\Upsilon_1 \in\ker (J_{\omega,1}L_1|_{X_o^1})$ and $\Phi_1=\Delta_1^{-1}\Upsilon_1$. By \eqref{c omega k} and \eqref{Rayleigh-type equation},  $\Phi_1$ solves
\begin{align*}
((1-s^2)\Phi')'-{1\over 1-s^2}\Phi+12\Phi=0, \quad
 \Delta_1\Phi\in L^2(-1,1).
 \end{align*}
 This implies that
 \begin{align}\label{kernel 1 odd}
 \Phi_1(s)=C_0P_3^1(s)={3\over2}C_0(1-5s^2)(1-s^2)^{1\over2}
 \end{align}
for some $C_0\in\mathbb{R}$.
Thus, $\Upsilon_1=\Delta_1\Phi_1$ is even. But $\Upsilon_1 \in\ker (J_{\omega,1}L_1|_{X_o^1})$ is also odd. Thus, $C_0=0$ and  $\Upsilon_1\equiv0$.

Let  $\Upsilon_2 \in\ker (J_{\omega,2}L_2|_{X_e^2})$ and $\Phi_2=\Delta_2^{-1}\Upsilon_2$. Then $\Phi_2$ solves
\begin{align*}
((1-s^2)\Phi')'-{4\over 1-s^2}\Phi+12\Phi=0, \quad
 \Delta_2\Phi\in L^2(-1,1).
 \end{align*}
 This implies that
 \begin{align}\label{kernel 2 even}
 \Phi_2(s)=C_1P_3^2(s)=15C_1s(1-s^2)
 \end{align}
for some $C_1\in\mathbb{R}$.
Thus, $\Upsilon_2=\Delta_2\Phi_2\in L^2(-1,1)$ is odd. But $\Upsilon_2 \in\ker (J_{\omega,2}L_2|_{X_e^2})$ is also even. Thus, $C_1=0$ and $\Upsilon_2\equiv0$.
\end{proof}

If we do not restrict $J_{\omega,1}L_1$ in the odd subspace and $J_{\omega,2}L_2$ in the even subspace, then by \eqref{kernel 1 odd}-\eqref{kernel 2 even}, $\ker(J_{\omega,k}L_k)$ is nontrivial for $\omega\in(-18,72)$ and $k=1,2.$ In the following remark,  we, however, prove that the generalized kernels of $J_{\omega,k}L_k$   coincide with  $\ker(L_k)$.

\begin{Remark}\label{kernel-J12L12}
$(1)$ Let $\omega\in(-18,72)$ and $E_0^2$ be the generalized kernel of $J_{\omega,2}L_2$. Then we claim that $E_0^2=\ker (L_2)$.

First, by \eqref{kernel 1 odd} we have $\ker (L_2)=\text{span}\{P_3^2(s)\}$.
Suppose that there exists a generalized eigenfunction  $\Upsilon\in X^2=\oplus_{l=2}^\infty\text{span}\{ P_l^2(s)\}$  such that
\begin{align*}
J_{\omega,2}L_2 \Upsilon=-2i\tilde \Upsilon_{\omega}'\left({1\over 12}+ \Delta_2^{-1}\right)\Upsilon=P_3^2(s)=15s(1-s^2)\in\ker (L_2).
\end{align*}
Then by \eqref{def-tilde-Upsilon0-derivative} and $\omega\in(-18,72)$, we have
\begin{align*}
\left({1\over 12}+ \Delta_2^{-1}\right)\Upsilon={P_3^2(s)\over -2i\tilde \Upsilon_{\omega}'}={15s(1-s^2)\over -2i(-180s^2+36+2\omega)}\notin L_{loc}^2.
\end{align*}

$(2)$ Let $\omega\in(-18,72)$  and $E_0^1$ be the generalized kernel of $J_{\omega,1}L_1$. Then we claim that  $E_0^1=\ker (L_1)$.

By \eqref{kernel 2 even}, we have $\ker (L_1)=\text{span}\{P_3^1(s)\}$.
Suppose that there exists a generalized eigenfunction  $\Upsilon\in X^1=\oplus_{l=2}^\infty\text{span}\{ P_l^1(s)\}$ such that
\begin{align*}
J_{\omega,1}L_1 \Upsilon=-i\tilde \Upsilon_{\omega}'\left({1\over 12}+ \Delta_1^{-1}\right)\Upsilon=P_3^1(s)={3\over2}(1-5s^2)(1-s^2)^{1\over2}\in\ker (L_1).
\end{align*}
Then
\begin{align}\label{gker1locl2}
\left({1\over 12}+ \Delta_1^{-1}\right)\Upsilon={P_3^1(s)\over -i\tilde \Upsilon_{\omega}'}={{3\over2}(1-5s^2)(1-s^2)^{1\over2}\over -i(-180s^2+36+2\omega)}.
\end{align}
If $(-18,72)\ni\omega\neq0$, then by \eqref{gker1locl2}, $\left({1\over 12}+ \Delta_1^{-1}\right)\Upsilon\notin L_{loc}^2$. If $\omega=0$, then
\begin{align}\label{gker1locl2contradict1}
\left({1\over 12}+ \Delta_1^{-1}\right)\Upsilon=-{1\over 24i}(1-s^2)^{1\over2}\in\text{span}\{P_1^1(s)\}.
\end{align}
Since $\Upsilon\in X^1=\oplus_{l=2}^\infty\text{span}\{ P_l^1(s)\}$, it can be written as $\Upsilon(s)=\sum_{l=2}^\infty a_l P_l^1(s)$.  Then
\begin{align}\nonumber
\left({1\over 12}+ \Delta_1^{-1}\right)\Upsilon
=&{1\over 12}\sum_{l=2}^\infty a_l P_l^1(s)-\sum_{l=2}^\infty(l(l+1))^{-1} a_l P_l^1(s)\\
\in& X^1=\oplus_{l=2}^\infty\text{span}\{ P_l^1(s)\}.\label{gker1locl2contradict2}
\end{align}
Then \eqref{gker1locl2contradict1} contradicts \eqref{gker1locl2contradict2}.
\end{Remark}

\section{The critical rotation rate in the positive half-line}\label{The critical rotation rate for the positive half-line}

In this section, we prove Theorem \ref{positive half-line critical rotation rate}, that is, the critical rotation rate  for linear stability/instability of the $3$-jet  in the positive half-line is $\omega_{cr}^+={99\over2}$.

\subsection{Linear instability for $\omega\in(-3,12]$}
With the preparatory work in Subsections \ref{subsection-Hamiltonian structure of the linearized equations and index formulae}-\ref{The neutral modes with c in Ran tilde Psi 0 prime},
 we are ready to show linear instability of the $3$-jet  for  $\omega\in(-3,12]$.
For convenience, here we include the proof of linear instability of the  $3$-jet  for the negative rotation rate $\omega\in(-3,0)$.
\begin{Theorem}\label{linear instability}
Let $k=1,2$.
If $\omega\in(-3,12]$, then the $3$-jet is linearly unstable.
\end{Theorem}
\begin{proof}
Let $\omega\in(-3,12]$.
By Lemmas \ref{tilde-psi-change-sign} and \ref{case-omega=12} (i),  $c\in\text{Ran}(-\tilde \Psi_{\omega}')^\circ=(-12+\omega,3+\omega)$ for any neutral mode $(c,k,\omega,\Phi)$.
By Theorem \ref{Psi prime change sign c}, if $\Phi$ is odd for $k=1$,
  then $c=c_\omega$. This proves that $k_{i,J_{\omega,1}L_1|_{X_o^1}}^{\leq0}=k_{i,J_{\omega,2}L_2|_{X_e^2}}^{\leq0}=0$.
  By Lemma \ref{kernel JL}, $k_{0,J_{\omega,1}L_1|_{X_o^1}}^{\leq0}=k_{0,J_{\omega,2}L_2|_{X_e^2}}^{\leq0}=0$.
By the index formulae \eqref{index formula 1o1} and \eqref{index formula 1o2}, we have
 $ k_{c,J_{\omega,1}L_1|_{X_o^1}}+ k_{r,J_{\omega,1}L_1|_{X_o^1}}=1$ and
 $k_{c,J_{\omega,2}L_2|_{X_e^2}}+ k_{r,J_{\omega,2}L_2|_{X_e^2}}=1$.
This implies that  the $3$-jet  is linearly unstable for both $k=1$ and $k=2$.
\end{proof}

\subsection{Proof of the positive critical rotation rate ${99\over2}$ for the first Fourier mode}
We prove that the  critical rotation rate in the positive half-line  is ${99\over2}$ for the $1$'st Fourier mode.

Let $\omega\in\left(12,72\right)$ and $k\in\{1,2\}$.
By Lemma \ref{tilde Psi0 does not change sign} (2), $c$ must be in $[0,-12+\omega]\cup\text{Ran}(-\tilde \Psi_{\omega}')^\circ$ for any neutral mode $(c,k,\omega,\Phi)$.
By Theorem \ref{Psi prime change sign c},
 if $(c,k,\omega,\Phi)$ is a  neutral mode with $c\in\text{Ran} (-\tilde \Psi_{\omega}')^\circ$ (where $\Phi$ is odd provided that $k=1$), then $c=c_\omega$.
 The neutral modes $(c,k,\omega,\Phi)$ with  $c=c_\omega$ correspond to zero eigenvalues of $J_{\omega,k}L_k$. By Lemma \ref{kernel JL},  $k_{0,J_{\omega,1}L_1|_{X_o^1}}^{\leq0}=k_{0,J_{\omega,2}L_2|_{X_e^2}}^{\leq0}=0$.
Thus, it reduces to study the indices $k_{i,J_{\omega,1}L_1|_{X_o^1}}^{\leq0}$ and $k_{i,J_{\omega,2}L_2|_{X_e^2}}^{\leq0}$ in \eqref{index formula 1o1} and \eqref{index formula 1o2}. So
we need  to determine for which $c\in[0,-12+\omega]$, $(c,k,\omega,\Phi)$ is a neutral mode.
Let us first introduce a spectral parameter $\mu$ in the Rayleigh equation \eqref{Rayleigh-type equation}. By \eqref{def-tilde-Upsilon0-derivative},
\begin{align*}
-{\tilde \Upsilon_{\omega}'\over \tilde \Psi_{\omega}'+c}=-{-12(15s^2-3)+2\omega\over 15s^2-3-\omega+c}
=12-{-10\omega+12c\over 15s^2-3-\omega+c}=12-{2\omega+12\mu\over 15s^2-3+\mu},
 \end{align*}
where
\begin{align*}
\mu=-\omega+c.
\end{align*}
Then \eqref{Rayleigh-type equation} can be rewritten as
 \begin{align}\label{Rayleigh-type equation 12}
 ((1-s^2)\Phi')'-{k^2\over 1-s^2}\Phi-{2\omega+12\mu\over 15s^2-3+\mu}\Phi=-12\Phi, \quad
 \Delta_k\Phi\in L^2(-1,1).
 \end{align}
 Now, we modify the Rayleigh equation \eqref{Rayleigh-type equation} to the ODE eigenvalue problem
\begin{align}\label{Rayleigh-type equation lambda}
 ((1-s^2)\Phi')'-{k^2\over 1-s^2}\Phi-{2\omega+12\mu\over 15s^2-3+\mu}\Phi=\lambda\Phi, \quad
 \Delta_k\Phi\in L^2(-1,1),
 \end{align}
 where $\lambda$ is the spectral parameter (actually, we use $\lambda$ for $k=1$ and $\tilde \lambda$ for $k=2$ in the following). To study the eigenvalues of \eqref{Rayleigh-type equation lambda}, we need the following compact embedding lemma.
 \begin{Lemma}\label{compact embedding}
$(\rm{i})$ The space
 \begin{align}\label{def-tilde X}
 \tilde X=\left\{\Phi\bigg| \int_{-1}^1\left((1-s^2)|\Phi'|^2+{1\over 1-s^2}|\Phi|^2\right) ds<\infty\right\}
 \end{align}
 is compactly embedded in $L^2(-1,1)$.

$(\rm{ii})$ The space
 \begin{align}\label{def-X-omega-mu}
X_{\omega,\mu}=\left\{\Phi\bigg| \int_{-1}^1\left((1-s^2)|\Phi'|^2+{1\over 1-s^2}|\Phi|^2+{2\omega+12\mu\over 15s^2-3+\mu}|\Phi|^2\right) ds<\infty\right\}
 \end{align}
 is compactly embedded in $L^2(-1,1)$, where $\omega\in\mathbb{R}$ and $\mu\leq-12$.
 \end{Lemma}
 \begin{proof}  (i) is a direct consequence of Theorem 2.9  in \cite{Hebey2000}.

(ii) The proof  follows from (i)  and
 \begin{align*}
 \left|\int_{-1}^1{2\omega+12\mu\over 15s^2-3+\mu}|\Phi|^2 ds\right|\leq C\int_{-1}^1{1\over 1-s^2}|\Phi|^2 ds
 \end{align*}
 for $\Phi\in X_{\omega,\mu}$ and $\mu\leq-12$.
 \end{proof}

Now, we consider the first Fourier mode and fix $k=1$.
 The main result in this subsection states as follows.
 \begin{Theorem}\label{k=1 positive half-line critical rotation rate}
Let $k=1$. Then the $3$-jet is  linearly unstable  for $\omega\in\left(12,{99\over2}\right)$ and spectrally  stable for $\omega\in\left[{99\over2},72\right)$.
\end{Theorem}

By Theorem \ref{linear instability}, the $3$-jet is  linearly unstable  for $\omega\in\left(0,12\right]$. By Rayleigh's criterion, the $3$-jet is spectrally stable  for $\omega\in\left[72,\infty\right)$. Combining the above results and Theorem \ref{k=1 positive half-line critical rotation rate}, we  rigorously prove that the critical rotation rate of the $3$-jet for the positive half-line for $k=1$ is
$\omega={99\over2}$, which confirms the numerical result in \cite{Sasaki-Takehiro-Yamada2012}.

Instead of directly providing in the proof, we first discuss the main ideas of the approach in the following remark.

\begin{Remark}\label{ideas thm k=1 positive half-line critical rotation rate and k=2 positive half-line critical rotation rate}
 The ideas in the proof of the above  theorem
  are as follows.
 First,  since $\omega>12$,  by Lemma \ref{tilde Psi0 does not change sign} we have $\lambda_{1}(\mu,\omega)\neq-12$ for $\mu\geq 3$. This means that we only need to study  $\lambda_{1}(\mu,\omega)$  with $\mu\leq -12$. For the endpoint case $\mu=-12$, the Rayleigh equation \eqref{Rayleigh-type equation lambda} has  no singularity in $(-1,1)$ when $\omega\in[12,72]$.  For $\omega\in[12,72]$, the principal eigenvalues  $\lambda_{1}(-12,\omega)$  can be solved explicitly using a transformation \eqref{transformation mu=-12} and Gegenbauer polynomials. Moreover, $\lambda_1(-12,{99\over2})=-12$ and $\lambda_1(-12,\omega)$ is increasing and continuous  on $\omega\in[12,72]$, see Lemmas \ref{ometa=12}-\ref{principla eigenvalue monotonicity omega}. We further need a delicate analysis  to study the spectral left-continuity of $\lambda_1(\cdot,\omega)$ at $\mu=-12$, see Lemma \ref{continuity of the principal eigenvalue mu -12}.
This, along with $\lim_{\mu\to-\infty}\lambda_1(\mu,\omega)=-18$ in Lemma \ref{asymptotic behavior principal eigenvalue lim mu -infty}, essentially gives a neutral mode with desired signature of the quadratic form $\langle L_1\cdot,\cdot\rangle$. Thus, $k_{i,J_{\omega,1}L_1|_{X_o^1}}^{\leq0}=1$ for $\omega\in\left[{99\over2},72\right)$. This proves  spectral stability for $\omega\in\left[{99\over2},72\right)$.
For the instability part, a key observation is that $\lambda_1(\mu,{99\over2})<-12$ for $\mu<-12$. This, combining with the monotonicity of $\lambda_1(\mu,\cdot)$ on $\omega$, implies that there are no neutral modes for $\omega\in(12,{99\over2})$.
\end{Remark}
Let $\mu\in(-\infty,-12]$ and $\omega\in\mathbb{R}$.
By \eqref{L1e2ok3nonnegative}-\eqref{L1o2enegative} and the index formula \eqref{index formula 1o1}, we only need to consider the space of odd functions
 \begin{align}\label{def-X-omega-mu-o}
X_{\omega,\mu,o}=\left\{\Phi\in X_{\omega,\mu}| \Phi \text{ is odd}\right\}.
 \end{align}
 By Lemma \ref{compact embedding}, all the eigenvalues of the eigenvalue problem
  \begin{align}\label{Rayleigh-type equation lambda k=1}
  ((1-s^2)\Phi')'-{1\over 1-s^2}\Phi-{2\omega+12\mu\over 15s^2-3+\mu}\Phi=\lambda\Phi, \quad\Delta_1\Phi\in L^2(-1,1),
 \end{align}
 (restricted to the space $X_{\omega,\mu,o}$) are arranged in a sequence  $-\infty<\cdots\leq \lambda_{n}(\mu,\omega)\leq \cdots \leq \lambda_{1}(\mu,\omega)$, which can be defined by
 \begin{align}\label{def-lambda-n}
\lambda_{n}(\mu,\omega)=& \sup_{\Phi \in X_{\omega,\mu,o}, (\Phi, \Phi_{i})_{L^2} = 0, i = 1, 2, \cdots, n-1}{\int_{-1}^1\left(-(1-s^2)|\Phi'|^2-{1\over 1-s^2}|\Phi|^2-{2\omega+12\mu\over 15s^2-3+\mu}|\Phi|^2\right) ds\over\int_{-1}^1|\Phi|^2ds},
\end{align}
where the supremum for $\lambda_{n}(\mu,\omega)$ is attained at $\Phi_{n} \in  X_{\omega,\mu,o}$.
Here, $\lambda_{1}(\mu,\omega)$ is called the principal (i.e. maximal) eigenvalue of \eqref{Rayleigh-type equation lambda k=1}.
\begin{Remark}
One can define the minimal eigenvalue of
\begin{align*}
  -((1-s^2)\Phi')'+{1\over 1-s^2}\Phi+{2\omega+12\mu\over 15s^2-3+\mu}\Phi=\nu\Phi, \quad\Delta_1\Phi\in L^2(-1,1)
 \end{align*}
 to be the principal eigenvalue, and  its  variational expression is given by
 \begin{align*}
\nu_{n}(\mu,\omega)=& \inf_{\Phi \in X_{\omega,\mu,o}, (\Phi, \Phi_{i})_{L^2} = 0, i = 1, 2, \cdots, n-1}{\int_{-1}^1\left((1-s^2)|\Phi'|^2+{1\over 1-s^2}|\Phi|^2+{2\omega+12\mu\over 15s^2-3+\mu}|\Phi|^2\right) ds\over\int_{-1}^1|\Phi|^2ds}.
\end{align*}
 This is more commonly used in the Sturm-Liouville theory. Our choice of $\lambda_n(\mu,\omega)=- \nu_{n}(\mu,\omega)$ is just to facilitate the analysis of our problem.
\end{Remark}

Let $\omega\in(12,72)$. Then $(c,1,\omega,\Phi)$ is a neutral mode with $c\in[0,-12+\omega]$ if and only if the eigenvalue problem \eqref{Rayleigh-type equation lambda k=1} has an eigenvalue $\lambda_{n_0}(\mu,\omega)=-12$ for some $n_0\geq1$ and   $\Phi$ is a corresponding eigenfunction, where $\mu=-\omega+c$. Thus, to look for a neutral mode with $\omega\in(12,72)$, we study whether  $\lambda_{n}(\mu,\omega)=-12$  has a solution $\mu\in[-\omega,-12]$ and $n\geq1$.
Once we can find  $\mu_1\in[-\omega,-12]$ so that $(c_1,1,\omega,\Phi_{\mu_1,\omega})$ is a neutral mode with $\mu_1=-\omega+c_1$, we obtain an imaginary eigenvalue $-i(c_1-c_\omega)$ for $J_{\omega,1}L_1$ with a corresponding eigenfunction $\Upsilon_{\mu_1,\omega}=\Delta_1\Phi_{\mu_1,\omega}$. To compute $k_{i,J_{\omega,1}L_1|_{X_o^1}}^{\leq0}$, by its definition
 we need to compute the quadratic form $\langle L_1\cdot,\cdot\rangle$ restricted to the
generalized eigenspace of $-i(c_1-c_\omega)$. We give a method to compute  $\langle L_1\Upsilon_{\mu_1,\omega},\Upsilon_{\mu_1,\omega}\rangle$.

\begin{Lemma}\label{quadratic form-computation}
$(\rm{i})$ Let $\omega\geq12$ and  $(c_{1},1,\omega,\Phi_{\mu_1,\omega})$ be
a  neutral mode, where $c_1\leq-12+\omega$ and $\mu_1=-\omega+c_{1}$. Then $\lambda_{n_0}(\mu_1,\omega)=-12$
for  some $n_{0}\geq1$ and
\begin{align}\label{equality-L-form-derivative}%
\langle L_{1}\Upsilon_{\mu_{1},\omega},\Upsilon_{\mu_{1},\omega}\rangle
=&(c_1-c_\omega)\int_{-1}^1{\tilde \Upsilon_{\omega}'\over (\tilde \Psi_{\omega}'+c_1)^2}|\Phi_{\mu_{1},\omega}|^2ds\\\nonumber
=&(c_1-c_\omega)\int_{-1}^1{-12(15s^2-3)+2\omega \over (15s^2-3+\mu_1)^2}|\Phi_{\mu_{1},\omega}|^2ds,
\end{align}
where $\Upsilon_{\mu_{1},\omega}=\Delta_1\Phi_{\mu_{1},\omega}$ and $c_\omega={5\over 6}\omega$.

$(\rm{ii})$ Under the assumptions of $(\rm{i})$, if $c_1<-12+\omega$ and $\|\Phi_{\mu_{1},\omega}\|_{L^2(-1,1)}=1$, then
\begin{equation}
\langle L_{1}\Upsilon_{\mu_{1},\omega},\Upsilon_{\mu_{1},\omega}\rangle
=(c_1-c_\omega)\partial_\mu \lambda_{n_0}(\mu_1,\omega).
\label{equality-L-form-derivative2}%
\end{equation}
\end{Lemma}
\begin{proof} (i)
By the definition of $L_1$, we have
\begin{align*}
\langle L_{1}\Upsilon_{\mu_{1},\omega},\Upsilon_{\mu_{1},\omega}\rangle=\int_{-1}^1\left({1\over 12}\Upsilon_{\mu_{1},\omega}+\Phi_{\mu_{1},\omega} \right)\Upsilon_{\mu_{1},\omega}ds.
\end{align*}
By \eqref{Rayleigh-type equation}, $\Phi_{\mu_1,\omega}$ satisfies
\begin{align*}
\Upsilon_{\mu_1,\omega}-{\tilde \Upsilon_{\omega}'\over \tilde \Psi_{\omega}'+c_1}\Phi_{\mu_1,\omega}=\Delta_1\Phi_{\mu_1,\omega}-{\tilde \Upsilon_{\omega}'\over \tilde \Psi_{\omega}'+c_1}\Phi_{\mu_1,\omega}=0.
\end{align*}
Then
\begin{align*}
{\Phi_{\mu_1,\omega}\over \Upsilon_{\mu_1,\omega}}={ \tilde \Psi_{\omega}'+c_1\over\tilde \Upsilon_{\omega}'}={ \tilde \Psi_{\omega}'+c_\omega-c_\omega+c_1\over\tilde \Upsilon_{\omega}'}=-{1\over12}-{ c_\omega-c_1\over\tilde \Upsilon_{\omega}'}.
\end{align*}
Thus,
$
{1\over12}\Upsilon_{\mu_1,\omega}+{\Phi_{\mu_1,\omega}}=-{ c_\omega-c_1\over\tilde \Upsilon_{\omega}'}\Upsilon_{\mu_1,\omega}
$
 and
 \begin{align}\label{comp-lk}
\langle L_{1}\Upsilon_{\mu_1,\omega},\Upsilon_{\mu_1,\omega}\rangle=(c_1-c_\omega)\int_{-1}^1{1\over\tilde \Upsilon_{\omega}'}\Upsilon_{\mu_1,\omega}^2ds=(c_1-c_\omega)\int_{-1}^1{\tilde \Upsilon_{\omega}'\over (\tilde \Psi_{\omega}'+c_1)^2}\Phi_{\mu_1,\omega}^2ds.
\end{align}

(ii)
For the $n_0$-th eigenvalue $\lambda=\lambda_{n_0}(\mu,\omega)$ of \eqref{Rayleigh-type equation lambda k=1} with
$\mu$ near $\mu_1$ (i.e. $c$ near $c_1$), there exists an eigenfunction   $\Phi_{\mu,\omega}$
with $\Vert\Phi_{\mu,\omega}\Vert_{L^{2}(-1,1)}=1$.
Then
\begin{align*}
&\int_{-1}^1\left(-(1-s^2)\Phi_{\mu,\omega}'\Phi_{\mu_1,\omega}'-{1\over 1-s^2}\Phi_{\mu,\omega}\Phi_{\mu_1,\omega}-{2\omega+12\mu\over 15s^2-3+\mu}\Phi_{\mu,\omega}\Phi_{\mu_1,\omega}\right)ds\\
=&\lambda_{n_0}(\mu,\omega)\int_{-1}^1\Phi_{\mu,\omega}\Phi_{\mu_1,\omega}ds,\\
&\int_{-1}^1\left(-(1-s^2)\Phi_{\mu_1,\omega}'\Phi_{\mu,\omega}'-{1\over 1-s^2}\Phi_{\mu_1,\omega}\Phi_{\mu,\omega}-{2\omega+12\mu_1\over 15s^2-3+\mu_1}\Phi_{\mu_1,\omega}\Phi_{\mu,\omega}\right)ds\\
=&\lambda_{n_0}(\mu_1,\omega)\int_{-1}^1\Phi_{\mu_1,\omega}\Phi_{\mu,\omega}ds.
\end{align*}
Thus,
\begin{align*}
(\lambda_{n_0}(\mu,\omega)-\lambda_{n_0}(\mu_1,\omega))\int_{-1}^1\Phi_{\mu_1,\omega}\Phi_{\mu,\omega}ds=\int_{-1}^1\left(-{2\omega+12\mu\over 15s^2-3+\mu}+{2\omega+12\mu_1\over 15s^2-3+\mu_1}\right)\Phi_{\mu,\omega}\Phi_{\mu_1,\omega} ds.
\end{align*}
Since $\mu_1<-12$, it follows from \eqref{Rayleigh-type equation lambda k=1} that $e^{i\varphi}\Phi_{\mu,\omega}$, $|\mu-\mu_1|\ll1$, have a uniform $H_2^2(\mathbb{S}^2)$ bound, and thus,
$\Phi_{\mu,\omega}\to\Phi_{\mu_1,\omega}$ in $L^2(-1,1)$ as $\mu\to\mu_1$ (see Lemma \ref{lem-differential calculus} (ii)).
Then
\begin{align}\label{comp-eigenvalue derivative}
\partial_\mu\lambda_{n_0}(\mu_1,\omega)=\int_{-1}^1{\tilde \Upsilon_{\omega}'\over (\tilde \Psi_{\omega}'+c_1)^2}|\Phi_{\mu_1,\omega}|^2 ds.
\end{align}
Combining \eqref{comp-lk} and \eqref{comp-eigenvalue derivative}, we get \eqref{equality-L-form-derivative2}.
\end{proof}

By the numerical result in \cite{Sasaki-Takehiro-Yamada2012}, the critical rotation rate  of linear instability/stability for the positive half-line is $\omega_{cr}^+={99\over2}$.
In the next lemma, we show that when $k=1$, the $3$-jet is spectrally stable for  the critical point ${99\over2}$.

\begin{Lemma}\label{ometa=12}
$(\rm{i})$ Let $k=1$. Then the $3$-jet is spectrally stable for $\omega={99\over2}$.

$(\rm{ii})$ For the eigenvalue problem \eqref{Rayleigh-type equation lambda k=1}, the principal eigenvalue is $\lambda_1\left(-12,{99\over2}\right)=-12$ with a corresponding eigenfunction $P_3^2$ satisfying $ \langle L_{1}\Delta_1P_3^2,\Delta_1P_3^2\rangle<0$.

\end{Lemma}

\begin{proof} (i)
Note that for any neutral mode $(c,k,\omega,\Phi)$ with $c\neq c_\omega$, we have $c\in[0,-12+\omega]$ by Lemma \ref{tilde Psi0 does not change sign} and Theorem \ref{Psi prime change sign c}, where $\Phi$ is odd.
 The motivation is from investigating for which $\omega$, $c=-12+\omega$ (i.e. $\mu=-12$) is neutral.
Putting  $\mu=-12$ and  $\lambda=-12$ into \eqref{Rayleigh-type equation lambda k=1}, we have
\begin{align}\label{c-12+omega}
   ((1-s^2)\Phi')'+{-1-{144-2\omega\over15}\over 1-s^2}\Phi+12\Phi=0, \quad
\Delta_1 \Phi\in L^2(-1,1).
 \end{align}
If $-1-{144-2\omega\over15}=-m^2$ with $m=0,1,2,3$, then \eqref{c-12+omega} can be solvable.

For $m=0\Rightarrow\omega={159\over 2}$ and
$m=1\Rightarrow\omega=72$, the $3$-jet is spectrally stable by Rayleigh's criterion.

For $m=2$, we have $\omega={99\over2}$, $c={75\over2}$ and $c_\omega={165\over4}$. The equation \eqref{c-12+omega} has a solution $P_3^2(s)=15s(1-s^2)$. Since it is an odd function, $\Upsilon(\varphi,s)=e^{i\varphi}\Delta_1P_3^2(s)$ satisfies the constraints $\iint_{D_T} \Upsilon d\varphi ds=0$ and $ \iint_{D_T}\Upsilon Y_1^md\varphi ds=0,  m=0,\pm1$.
Moreover,
by Lemma \ref{quadratic form-computation} (i), we have
\begin{equation*}
\langle L_{1}\Delta_1P_3^2,\Delta_1P_3^2\rangle
=225\left({75\over2}-{165\over4}\right)\int_{-1}^1{-180s^2+135\over (15s^2-15)^2}s^2(1-s^2)^2ds=-{135\over2}<0.
\end{equation*}
Then
 $k_{i,J_{\omega,1}L_1|_{X_o^1}}^{\leq0} =1$ and by the index formula  \eqref{index formula 1o1}, we have
 $ k_{c,J_{\omega,1}L_1|_{X_o^1}}+ k_{r,J_{\omega,1}L_1|_{X_o^1}}=0.$ Thus, the $3$-jet is spectrally stable for $k=1$.

 For $m=3$, we have $\omega=12$, $c=0$ and $c_\omega=10$. The equation \eqref{c-12+omega} has a solution $\Phi(s)=P_3^3(s)=-15(1-s^2)^{3\over2}$. By \eqref{vorticity Y1-1 constraint}, this solution does not satisfy $\iint_{D_T}\Upsilon Y_1^{-1}d\varphi ds=0$ for $\Upsilon=e^{i\varphi}\Delta_1P_3^3$.

 (ii) When $\mu=-12$ and $\omega={99\over2}$, the eigenvalue problem \eqref{Rayleigh-type equation lambda k=1} becomes
\begin{align*}
 ((1-s^2)\Phi')'-{4\over 1-s^2}\Phi=\lambda\Phi, \quad
\Delta_1 \Phi\in L^2(-1,1).
 \end{align*}
 This is the general Legendre equation of order $2$. Since the eigenvalue problem is restricted to the odd space $X_{\omega,\mu,o}$,
 the principal eigenvalue is $\lambda_1\left(-12,{99\over2}\right)=-12$ with a corresponding eigenfunction $P_3^2$.
\end{proof}
We regard the  the principal eigenvalue $\lambda_1(\mu,\omega)$ as a function of $\mu\in(-\infty,-12]$ and $\omega\in[12,72]$. Now, we study some  properties of this function $\lambda_1(\mu,\omega)$ as mentioned in Remark \ref{ideas thm k=1 positive half-line critical rotation rate and k=2 positive half-line critical rotation rate}, which play an important role in computing the index $k_{i,J_{\omega,1}L_1|_{X_o^1}}^{\leq0}$.  The first property is the monotonicity of  $\lambda_1(\mu,\omega)$ with respect to $\omega$, and explicit eigenpairs for $\mu=-12$.
\begin{Lemma}\label{principla eigenvalue monotonicity omega}
$(\rm{i})$  For $\mu\in (-\infty,-12)$,
\begin{align*}
\partial_\omega\lambda_1(\mu,\omega)=-\int_{-1}^1{2\over15s^2-3+\mu}|\Phi_{\mu,\omega}|^2 ds>0, \quad \omega\in[12,72],
\end{align*}
where $\Phi_{\mu,\omega}$ is a $L^2$ normalized eigenfunction of $\lambda_1(\mu,\omega)$ for the eigenvalue problem \eqref{Rayleigh-type equation lambda k=1}.

$(\rm{ii})$  For $\mu=-12$,
\begin{align}\label{principal eigenvalue -12 omega}
\lambda_1(-12,\omega)=-\left(1+\sqrt{1-{2\omega-144\over 15}}\right)\left(2+\sqrt{1-{2\omega-144\over 15}}\right)\in[-20,-6]
\end{align}
with $\omega\in[12,72]$,
and a corresponding eigenfunction is given by
\begin{align}\label{eigenfunction of principal eigenvalue -12 omega}
\Phi_{-12,\omega}(s)=(1-s^2)^{{\sqrt{1-{2\omega-144\over 15}}\over2}}s, \quad s\in [-1,1].
\end{align}

$(\rm{iii})$  For $\mu\in (-\infty,-12]$,
$\lambda_1(\mu,\cdot)$ is increasing on $\omega\in[12,72]$.
\end{Lemma}
\begin{proof}
The proof of (i) is similar to that of \eqref{comp-eigenvalue derivative}, and thus we omit it.

(ii) For $\mu=-12$, the eigenvalue problem \eqref{Rayleigh-type equation lambda k=1} becomes
\begin{align}\label{c-12+omega lambda}
 ((1-s^2)\Phi')'-{1-{2\omega-144\over15}\over 1-s^2}\Phi=\lambda\Phi, \quad \Delta_1\Phi\in L^2(-1,1).
 \end{align}
 Let
 \begin{align}\label{transformation mu=-12}
 \Phi(s)=(1-s^2)^{{\sqrt{1-{2\omega-144\over 15}}\over2}}\phi(s), \quad s\in[-1,1].
 \end{align}
 By \eqref{c-12+omega lambda},  $\phi$ solves the equation
 \begin{align*}
 (1-s^2)\phi''-2\left(\sqrt{1-{2\omega-144\over 15}}+1\right)s\phi'+\left(-1+{2\omega-144\over15}-\sqrt{1-{2\omega-144\over15}}-\lambda\right)\phi=0.
 \end{align*}
 Let
 \begin{align*}
  \lambda=-\left(n+\sqrt{1-{2\omega-144\over 15}}\right)\left(n+\sqrt{1-{2\omega-144\over 15}} +1\right),
\end{align*}
where $n\geq0$.
Then this is the   Gegenbauer
 differential equation
\begin{equation}\label{Gegenbauer differential equation}
(1-s^2)\phi''-(2\beta+1)s\phi'+n\left(n+2\beta\right)\phi =0
\end{equation}
with $\beta=\sqrt{1-{2\omega-144\over 15}}+{1\over2}\in[{3\over2},{7\over2}]$.  Using the
   the
Gegenbauer polynomials (see, e.g.  \cite{Suetin2001}), which are the solutions of \eqref{Gegenbauer differential equation}, we obtain that
   the first eigenvalue of \eqref{c-12+omega lambda} is given by
\eqref{principal eigenvalue -12 omega}
with a corresponding eigenfunction
\begin{align*}
  (1-s^2)^{{\sqrt{1-{2\omega-144\over 15}}\over2}}C_1^\beta(s), \quad s\in[-1,1],
\end{align*}
where $C_1^\beta(s)$ is the Gegenbauer polynomial of the first order and we use the fact that the eigenfunction is  odd.
By the definition of  $C_1^\beta(s)$, up to a constant we have
$C_1^\beta(s)=(1-s^2)^{-\beta+{1\over2}}{d\over ds}((1-s^2)^{\beta+{1\over2}})=-(2\beta+1)s,$ which gives \eqref{eigenfunction of principal eigenvalue -12 omega}.

(iii) is a direct consequence of (i)-(ii).
\end{proof}

Next, we study the left-continuity of the principal eigenvalue $\lambda_1(\mu,\omega)$ of \eqref{Rayleigh-type equation lambda k=1} as $\mu\to-12^-$.

\begin{Lemma}\label{continuity of the principal eigenvalue mu -12}
For $\omega\in(12,72)$, we have
\begin{align}\label{mu lim -12}
\lim_{\mu\to-12^-}\lambda_1(\mu,\omega)=\lambda_1(-12,\omega).
\end{align}
\end{Lemma}

\begin{proof} Let $\omega\in(12,72)$.
For any $\Phi\in \tilde X$ (defined in \eqref{def-tilde X}), we have
\begin{align*}
&\left|\left({2\omega+12\mu\over 15s^2-3+\mu}-{2\omega-144\over 15s^2-15}\right)\Phi^2\right|=\left|{(12+\mu)(12(15s^2-3)-2\omega)\over (15s^2-3+\mu)(15s^2-15)}\Phi^2\right|\\
=&{|(12+\mu)(12(15s^2-3)-2\omega)|\over ((15-15s^2)-(12+\mu))(15-15s^2)}\Phi^2\leq {|12(15s^2-3)-2\omega|\over 15-15s^2}\Phi^2
\end{align*}
for  $\mu\in(-\infty,-12)$.
Since
\begin{align*}
\int_{-1}^1 {|12(15s^2-3)-2\omega|\over 15-15s^2}\Phi^2ds\leq C \int_{-1}^1 {1\over 1-s^2}\Phi^2ds,
\end{align*}
by the Lebesgue's Dominated Convergence Theorem we have
\begin{align}\label{potential limit}
\lim_{\mu\to-12^-}\int_{-1}^1{2\omega+12\mu\over 15s^2-3+\mu}\Phi^2ds=\int_{-1}^1{2\omega-144\over 15s^2-15}\Phi^2ds.
\end{align}
Let $\Phi_{\mu,\omega}$ be a $L^2$ normalized eigenfunction of $\lambda_1(\mu,\omega)$ for $\mu\in(-\infty,-12]$. By \eqref{def-lambda-n}, we have
\begin{align*}
\lambda_1(\mu,\omega)\geq \int_{-1}^1\left(-(1-s^2)|\Phi_{-12,\omega}'|^2-{1\over 1-s^2}|\Phi_{-12,\omega}|^2-{2\omega+12\mu\over 15s^2-3+\mu}|\Phi_{-12,\omega}|^2\right) ds
\end{align*}
for  $\mu\in(-\infty,-12)$. By \eqref{potential limit}, we have
\begin{align}\nonumber
\liminf_{\mu\to-12^-}\lambda_1(\mu,\omega)\geq& \int_{-1}^1\left(-(1-s^2)|\Phi_{-12,\omega}'|^2-{1\over 1-s^2}|\Phi_{-12,\omega}|^2-{2\omega-144\over 15s^2-15}|\Phi_{-12,\omega}|^2\right) ds\\\label{liminf}
=&\lambda_1(-12,\omega).
\end{align}
On the other hand, we will prove that
\begin{align}\label{limsup}
\lambda_1(-12,\omega)\geq\limsup_{\mu\to-12^-}\lambda_1(\mu,\omega).
\end{align}
To this end, we first prove that there exist $\delta, C>0$ such that
\begin{align}\label{H1bound}
\int_{-1}^1\left((1-s^2)|\Phi_{\mu,\omega}'|^2+{1\over 1-s^2}|\Phi_{\mu,\omega}|^2\right)ds\leq C
\end{align}
uniformly for $\mu\in(-12-\delta,-12)$.
In fact, by \eqref{def-lambda-n} and \eqref{liminf}, there exists $\delta>0$ such that
\begin{align}\label{principal eigenvalue bound}
\lambda_1(-12,\omega)-{1\over2}<\lambda_1(\mu,\omega)\leq0,\quad \forall\; \mu\in(-12-\delta,-12).
\end{align}
Since $\Phi_{\mu,\omega}$ solves
\eqref{Rayleigh-type equation lambda k=1} with $\lambda=\lambda_1(\mu,\omega)$, we have
\begin{align*}
\int_{-1}^1\left((1-s^2)|\Phi_{\mu,\omega}'|^2+{1\over 1-s^2}|\Phi_{\mu,\omega}|^2+{2\omega+12\mu\over 15s^2-3+\mu}|\Phi_{\mu,\omega}|^2\right) ds=-\lambda_1(\mu,\omega),
\end{align*}
which, along with \eqref{principal eigenvalue bound} and ${2\omega+12\mu\over 15s^2-3+\mu}>0$ for $\omega\in(12,72)$ and $\mu<-12$, gives \eqref{H1bound}.
Then
there exists $\Phi_{*,\omega}\in \tilde X$ such that $\Phi_{\mu,\omega}\rightharpoonup \Phi_{*,\omega}$ in $\tilde X$ and $\Phi_{\mu,\omega}\to \Phi_{*,\omega}$ in $L^2(-1,1)$ as $\mu\to-12^{-}$ by Lemma 2.4.6 in \cite{Skiba2017}. Thus, $\|\Phi_{*,\omega}\|_{L^2(-1,1)}=\lim_{\mu\to-12^-}\|\Phi_{\mu,\omega}\|_{L^2(-1,1)}=1$.
For any subinterval $[a,b]\subset(-1,1)$, $\|\Phi_{\mu,\omega}\|_{H^1(a,b)}\leq C$ uniformly for $\mu\in(-12-\delta,-12)$. Then  $\Phi_{\mu,\omega}\to\Phi_{*,\omega}$  in $C^0([a,b])$ as $\mu\to-12^{-}$ by the compactness of $H^1(a,b)\hookrightarrow C^0([a,b])$. Thus,
$
{2\omega+12\mu\over 15s^2-3+\mu}|\Phi_{\mu,\omega}(s)|^2\to{2\omega-144\over 15s^2-15}|\Phi_{*,\omega}(s)|^2$ pointwise on $(-1,1)$ as $\mu\to-12^-$.
 Since
${2\omega+12\mu\over 15s^2-3+\mu}|\Phi_{\mu,\omega}(s)|^2\geq0$ on $(-1,1)$ for $\mu\in(-12-\delta,-12)$, by Fatou's Lemma we have
\begin{align}\label{potential term mu liminf -12}
\int_{-1}^1{2\omega-144\over 15s^2-15}|\Phi_{*,\omega}(s)|^2 ds\leq \liminf_{\mu\to-12^-}\int_{-1}^1{2\omega+12\mu\over 15s^2-3+\mu}|\Phi_{\mu,\omega}(s)|^2 ds.
\end{align}
By \eqref{H1bound}, we have
\begin{align}\nonumber
&\int_{-1}^1\left((1-s^2)|\Phi_{*,\omega}'|^2+{1\over 1-s^2}|\Phi_{*,\omega}|^2\right)ds\\\label{weak convergence H1}
\leq& \liminf_{\mu\to-12^-}\int_{-1}^1\left((1-s^2)|\Phi_{\mu,\omega}'|^2+{1\over 1-s^2}|\Phi_{\mu,\omega}|^2\right)ds.
\end{align}
By \eqref{potential term mu liminf -12} and \eqref{weak convergence H1}, we have
\begin{align*}
\lambda_1(-12,\omega)\geq&\int_{-1}^1\left(-(1-s^2)|\Phi_{*,\omega}'|^2-{1\over 1-s^2}|\Phi_{*,\omega}|^2-{2\omega-144\over 15s^2-15}|\Phi_{*,\omega}|^2\right) ds\\
\geq&\limsup_{\mu\to-12^-}\int_{-1}^1\left(-(1-s^2)|\Phi_{\mu,\omega}'|^2-{1\over 1-s^2}|\Phi_{\mu,\omega}|^2\right)ds\\
&+\limsup_{\mu\to-12^-}\int_{-1}^1-{2\omega+12\mu\over 15s^2-3+\mu}|\Phi_{\mu,\omega}|^2 ds\\
\geq&\limsup_{\mu\to-12^-}\int_{-1}^1\left(-(1-s^2)|\Phi_{\mu,\omega}'|^2-{1\over 1-s^2}|\Phi_{\mu,\omega}|^2-{2\omega+12\mu\over 15s^2-3+\mu}|\Phi_{\mu,\omega}|^2\right) ds\\
=&\limsup_{\mu\to-12^-}\lambda_1(\mu,\omega).
\end{align*}
This proves \eqref{limsup}. Combining \eqref{liminf} and \eqref{limsup}, we obtain \eqref{mu lim -12}.
\end{proof}
Then we consider the asymptotic behavior of the principal eigenvalue $\lambda_1(\mu,\omega)$ as $\mu\to-\infty$.
\begin{Lemma}\label{asymptotic behavior principal eigenvalue lim mu -infty}
For $\omega\in(12,72)$, we have
\begin{align}\label{principal eigenvalue lim mu -infty}
\lim_{\mu\to-\infty}\lambda_1(\mu,\omega)=-18.
\end{align}
\end{Lemma}
\begin{proof}
Let $\Phi_{\mu,\omega}$ be a $L^2$ normalized eigenfunction of $\lambda_1(\mu,\omega)$ for $\mu<-12$.
For any $\epsilon>0$, there exists $M>0$ such that ${1\over |15s^2-3+\mu|}\leq {\epsilon\over 180}$ for $\mu<-M$ and $s\in[-1,1]$.
Then
\begin{align*}
\left|\int_{-1}^1\left({2\omega+12\mu\over 15s^2-3+\mu}-12\right)|\Phi_{\mu,\omega}|^2ds\right|\leq&\int_{-1}^1{|-180s^2+36+2\omega|\over |15s^2-3+\mu|}|\Phi_{\mu,\omega}|^2ds
\leq\epsilon\int_{-1}^1|\Phi_{\mu,\omega}|^2ds=\epsilon
\end{align*}
for  $\mu<-M$. Thus,
\begin{align}\label{potential term lim mu -infty}
\lim_{\mu\to-\infty}\int_{-1}^1{2\omega+12\mu\over 15s^2-3+\mu}|\Phi_{\mu,\omega}|^2ds=12.
\end{align}
By the definition of $X_{\omega,\mu,o}$ in \eqref{def-X-omega-mu-o}, we have
\begin{align*}
 \inf_{\Phi \in X_{\omega,\mu,o},\|\Phi\|_{L^2(-1,1)}=1}\int_{-1}^1\left((1-s^2)|\Phi'|^2+{1\over 1-s^2}|\Phi|^2\right) ds\geq6.
 \end{align*}
 Then
 \begin{align*}
 \lambda_1(\mu,\omega)=&\int_{-1}^1\left(-(1-s^2)|\Phi_{\mu,\omega}'|^2-{1\over 1-s^2}|\Phi_{\mu,\omega}|^2-{2\omega+12\mu\over 15s^2-3+\mu}|\Phi_{\mu,\omega}|^2\right) ds\\
 \leq&-6-\int_{-1}^1{2\omega+12\mu\over 15s^2-3+\mu}|\Phi_{\mu,\omega}|^2 ds,
 \end{align*}
which, along with \eqref{potential term lim mu -infty}, implies that
 \begin{align*}
 \limsup_{\mu\to-\infty}\lambda_1(\mu,\omega)
 \leq&-6-12=-18.
 \end{align*}
 Let $\Phi_{-\infty}={1\over \|P_2^1\|_{L^2(-1,1)}}P_2^1$.
 Then
 \begin{align*}
 \lambda_1(\mu,\omega)\geq&\int_{-1}^1\left(-(1-s^2)|\Phi_{-\infty}'|^2-{1\over 1-s^2}|\Phi_{-\infty}|^2-{2\omega+12\mu\over 15s^2-3+\mu}|\Phi_{-\infty}|^2\right) ds\\
 =&-6-\int_{-1}^1{2\omega+12\mu\over 15s^2-3+\mu}|\Phi_{-\infty}|^2 ds.
 \end{align*}
 Thus,
  \begin{align*}
\liminf_{\mu\to-\infty} \lambda_1(\mu,\omega)\geq-6-\lim_{\mu\to-\infty}\int_{-1}^1{2\omega+12\mu\over 15s^2-3+\mu}|\Phi_{-\infty}|^2 ds
 =-6-12=-18.
 \end{align*}
 This proves \eqref{principal eigenvalue lim mu -infty}.
\end{proof}
Now, we are ready to prove Theorem \ref{k=1 positive half-line critical rotation rate}.

\begin{proof}[Proof of Theorem \ref{k=1 positive half-line critical rotation rate}]
First, we prove that the $3$-jet is spectrally stable for $\omega\in\left[{99\over2},72\right)$. The case of $\omega={99\over2}$ has been proved in Lemma \ref{ometa=12} (i). Thus, we  consider $\omega\in\left({99\over2},72\right)$.
By Lemma \ref{ometa=12} (ii)  and Lemma \ref{principla eigenvalue monotonicity omega} (iii), we have $\lambda_1(-12,\omega)>\lambda_1(-12,{99\over2})=-12$ for $\omega\in\left({99\over2},72\right)$. This, along with Lemmas \ref{continuity of the principal eigenvalue mu -12} and \ref{asymptotic behavior principal eigenvalue lim mu -infty}, implies that there exists $\mu_{1,\omega}\in(-\infty,-12)$ such that
\begin{align}\label{lambda1-mu1-omega-12}
\lambda_1(\mu_{1,\omega},\omega)=-12 \quad \text{and}\quad \partial_\mu\lambda_1(\mu_{1,\omega},\omega)\geq0,
\end{align}
where  $\lambda_1(\cdot,\omega)$ is differentiable on $\mu\in(-\infty,-12)$. See the blue eigenvalue curves in Fig. \ref{fig-eigenvalue curves1}. Thus, there exists a neutral mode $(c_{1,\omega},1,\omega,\Phi_{\mu_{1,\omega},\omega,1})$ with $c_{1,\omega}=\mu_{1,\omega}+\omega<-12+\omega$, where $\|\Phi_{\mu_{1,\omega},\omega,1}\|_{L^2(-1,1)}=1$.
Note that $c_{1,\omega}-c_\omega<-12+\omega-{5\over6}\omega=-12+{1\over 6}\omega<0$ for $\omega\in({99\over 2},72)$.
By Lemma \ref{quadratic form-computation} (ii), we have
\begin{equation*}
\langle L_{1}\Upsilon_{\mu_{1,\omega},\omega,1},\Upsilon_{\mu_{1,\omega},\omega,1}\rangle
=(c_{1,\omega}-c_\omega)\partial_\mu \lambda_{1}(\mu_{1,\omega},\omega)\leq 0,
\end{equation*}
where $\Upsilon_{\mu_{1,\omega},\omega,1}=\Delta_1\Phi_{\mu_{1,\omega},\omega,1}$.
Thus, $k_{i,J_{\omega,1}L_1|_{X_o^1}}^{\leq0}=1$. By the index formula
\eqref{index formula 1o1}, we have $k_{c,J_{\omega,1}L_1|_{X_o^1}}+ k_{r,J_{\omega,1}L_1|_{X_o^1}}=0$. This proves
 spectral stability of the $3$-jet  for $\omega\in\left({99\over2},72\right)$.

 Next, we prove that the $3$-jet is linearly unstable for $\omega\in\left(12,{99\over2}\right)$. For $\omega={99\over2}$, we claim that
 \begin{align}\label{lambda1-omega-99over2}
 \lambda_1\left(\mu,{99\over2}\right)<-12\quad\text{for}\quad \mu\in(-\infty,-12)\quad\text{and}\quad
 \lambda_1\left(-12,{99\over2}\right)=-12.
 \end{align}
 See the red eigenvalue curve in Fig. \ref{fig-eigenvalue curves1}.
 In fact, it follows from Lemma \ref{ometa=12} and its proof that
  $\lambda_1(-12,{99\over2})=-12$,  and $P_3^2$ is one of its eigenfunctions, which yields   a neutral mode $({75\over2},1,{99\over2},P_3^2)$ with  the quadratic form satisfying
 \begin{align}\label{signature of quadratic form mu -12}
\langle L_{1}\Delta_1P_3^2,\Delta_1P_3^2\rangle
<0.\end{align}
 Suppose that there exists $\tilde \mu_1\in(-\infty,-12)$ such that  $\lambda_1(\tilde \mu_1,{99\over2})\geq-12$.
Since $
\lim_{\mu\to-\infty}\lambda_1(\mu,{99\over2})$ $=-18$ by Lemma \ref{asymptotic behavior principal eigenvalue lim mu -infty}, there exists $\tilde\mu_2\in(-\infty,\tilde \mu_1]$ such that
\begin{align*}
\lambda_1\left(\tilde\mu_2,{99\over2}\right)=-12\quad\text{and}\quad\partial_\mu\lambda_1\left(\tilde\mu_2,{99\over2}\right)\geq0.
\end{align*}
\begin{center}
 \begin{tikzpicture}[scale=0.58]
 \draw [->](-15, 0)--(1, 0)node[right]{$\mu$};
 \draw [->](0,-10)--(0,2) node[above]{$\lambda_1(\mu,\omega)$};
\draw [blue] (-15, -7.75).. controls (-10, -6.8) and (-7, -4)..(-5,-3.5);
\draw [blue] (-15, -7.8).. controls (-10, -7) and (-7, -4.8)..(-5,-4.4);
\draw [blue] (-15, -7.85).. controls (-10, -7.5) and (-7, -5.2)..(-5,-4.7);
\draw [red] (-15, -7.9).. controls (-10, -7.8) and (-7, -5.1)..(-5,-5);
\draw [green] (-15, -7.92).. controls (-10, -7.85) and (-7, -5.6)..(-5,-5.4);
\draw [green] (-15, -7.95).. controls (-10, -7.9) and (-7, -6.2)..(-5,-6);
\draw [green] (-15, -7.98).. controls (-10, -7.95) and (-7, -7)..(-5,-6.7);
\path   (0, -5)  edge [-,dotted](-15,-5) [line width=0.4pt];
\path   (0, -8)  edge [-,dotted](-15,-8) [line width=0.4pt];
\path   (-5, 2)  edge [-,dotted](-5,-10) [line width=0.4pt];
\node[red] (a) at (-5,-5) {$\bullet$};
\node[blue] (a) at (-5.82,-5) {$\bullet$};
\node[blue] (a) at (-6.68,-5) {$\bullet$};
\node[blue] (a) at (-8,-5) {$\bullet$};
         \node (a) at (0.5,-5) {\tiny$-12$};
        \node (a) at (0.5,-8) {\tiny$-18$};
        \node (a) at (-5,0.2) {\tiny$-12$};
 \end{tikzpicture}
\end{center}\vspace{-0.5cm}
\begin{figure}[ht]
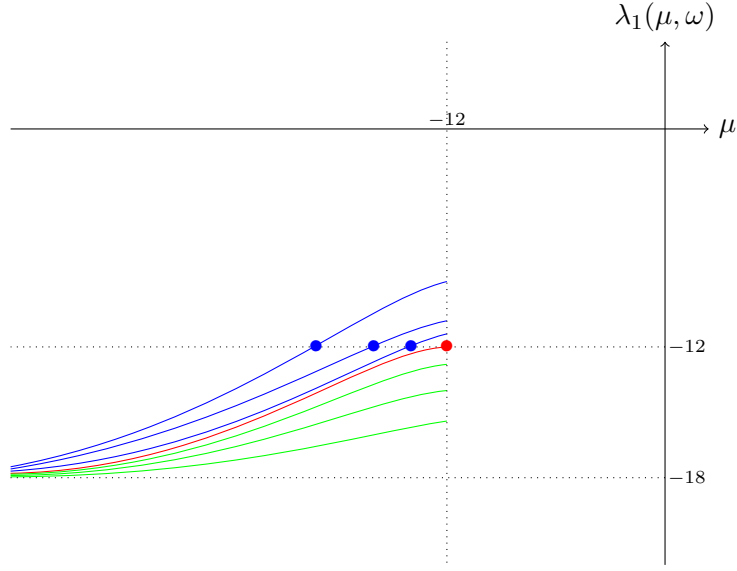

    \centering
 \caption{
The red eigenvalue curve is $\lambda_1(\cdot,\omega)$ with $\omega={99\over2}$, the blue eigenvalue curves are $\lambda_1(\cdot,\omega)$ with $\omega\in({99\over2},72)$, and the green eigenvalue curves are $\lambda_1(\cdot,\omega)$ with $\omega\in(12,{99\over2})$. The blue bold points are $(\mu_{1,\omega},\lambda_1(\mu_{1,\omega},\omega))$ for different $\omega\in({99\over2},72)$. The red bold point is $\left((-12,\lambda_1\left(-12,{99\over2}\right)=-12\right)$.
}
\label{fig-eigenvalue curves1}
\end{figure}
\vspace{-0.1cm}
Note that $\tilde c_2-c_\omega<-12+\omega-{5\over6}\omega=-12+{1\over6}\omega<0$, where $\tilde c_2=\tilde\mu_2+\omega$. Let
$\Phi_{\tilde\mu_2,{99\over2}}$ be a $L^2$ normalized eigenfunction of  $\lambda_1\left(\tilde\mu_2,{99\over2}\right)$. By Lemma \ref{quadratic form-computation} (ii), the quadratic form has signature
\begin{align}\label{signature of quadratic form mu1}
\langle L_{1}\Upsilon_{\tilde\mu_{2},{99\over2}},\Upsilon_{\tilde\mu_{2},{99\over2}}\rangle
=(\tilde c_2-c_\omega)\partial_\mu\lambda_1\left(\tilde\mu_2,{99\over2}\right)\leq 0,
\end{align}
 where $\Upsilon_{\tilde\mu_{2},{99\over2}}=\Delta_1\Phi_{\tilde\mu_2,{99\over2}}$.
Combining \eqref{signature of quadratic form mu -12} and \eqref{signature of quadratic form mu1}, we have
$ k_{i,J_{\omega,1}L_1|_{X_o^1}}^{\leq0}\geq2$, which contradicts
 \eqref{index formula 1o1}. Thus, $
 \lambda_1(\mu,{99\over2})<-12$ for  $\mu\in(-\infty,-12)$.

By Lemma \ref{principla eigenvalue monotonicity omega} (iii),
$\lambda_1(\mu,\cdot)$ is increasing on $\omega\in(12,{99\over2}]$ for $\mu\in (-\infty,-12]$. This, along with \eqref{lambda1-omega-99over2}, yields
\begin{align}\label{lambda1muomegalessthan-12}
\lambda_1(\mu,\omega)<\lambda_1\left(\mu,{99\over2}\right)\leq -12,\quad \forall\quad \mu\in(-\infty,-12],\;\;\omega\in\left(12,{99\over2}\right).
\end{align}
See the green eigenvalue curves in Fig. \ref{fig-eigenvalue curves1}.
Thus, $\lambda_1(-\omega+c,\omega)< -12$ for $c\in(-\infty,-12+\omega]$ and $\omega\in(12,{99\over2})$. Since $\lambda_1(-\omega+c,\omega)$ is the maximal eigenvalue of \eqref{Rayleigh-type equation lambda k=1} with $\mu=-\omega+c$, we obtain that there exist no neutral modes $(c,1,\omega,\Phi)$ with $c\leq-12+\omega$ for $\omega\in(12,{99\over2})$. This, along with  Lemma \ref{tilde Psi0 does not change sign} (2), implies that $c\in\text{Ran} (-\tilde \Psi_{\omega}')^\circ$ for any neutral mode $(c,1,\omega,\Phi)$. It then follows from  Theorem \ref{Psi prime change sign c} that $c=c_\omega$. Thus, $k_{i,J_{\omega,1}L_1|_{X_o^1}}^{\leq0}=0$. By Lemma \ref{kernel JL}, we have $k_{0,J_{\omega,1}L_1|_{X_o^1}}^{\leq0}=0$.
By the index formula \eqref{index formula 1o1}, we have
$ k_{c,J_{\omega,1}L_1|_{X_o^1}}+ k_{r,J_{\omega,1}L_1|_{X_o^1}}=1$. This proves  linear instability of the $3$-jet for $\omega\in\left(12,{99\over2}\right)$.
\end{proof}
\begin{Corollary}\label{uniqueness1}
Let $\omega\in\left({99\over2},72\right)$ and $k=1$. Then there exists
a unique $\mu_{1,\omega}\in(-\infty,-12)$ such that  $(c_{1,\omega},1,\omega,\Phi_{\mu_{1,\omega},\omega,1})$ is a  neutral mode, where $c_{1,\omega}=\mu_{1,\omega}+\omega$.
Moreover, $\langle L_{1}\Upsilon_{\mu_{1,\omega},\omega,1},$ $\Upsilon_{\mu_{1,\omega},\omega,1}\rangle$ $
\leq 0$, where $\Upsilon_{\mu_{1,\omega},\omega,1}=\Delta_1\Phi_{\mu_{1,\omega},\omega,1}$.
\end{Corollary}
\begin{proof}
The existence of $\mu_{1,\omega}$ is proved in \eqref{lambda1-mu1-omega-12}. Now we prove the uniqueness. If there exists another $\hat \mu_{1,\omega}\in(-\infty,-12)$ such that $\hat\mu_{1,\omega}\neq\mu_{1,\omega}$ and
$(\hat c_{1,\omega},1,\omega,\Phi_{\hat \mu_{1,\omega},\omega,1})$ is a  neutral mode with $\hat c_{1,\omega}=\hat \mu_{1,\omega}+\omega$, $\|\Phi_{\hat \mu_{1,\omega},\omega,1}\|_{L^2(-1,1)}=1$, then we will get a contradiction.
  In fact, if $\Phi_{\hat \mu_{1,\omega},\omega,1}$ is odd, by the  index formula \eqref{index formula 1o1} and Lemma \ref{quadratic form-computation} (ii), we have
\begin{equation*}
\langle L_{1}\Upsilon_{\hat\mu_{1,\omega},\omega,1},\Upsilon_{\hat\mu_{1,\omega},\omega,1}\rangle
=(\hat c_{1,\omega}-c_\omega)\partial_\mu \lambda_{1}(\hat \mu_{1,\omega},\omega)> 0,
\end{equation*}
where $\Upsilon_{\hat\mu_{1,\omega},\omega,1}=\Delta_1\Phi_{\hat\mu_{1,\omega},\omega,1}$.
Since
$\hat c_{1,\omega}-c_\omega<0$, we have $\partial_\mu \lambda_{1}(\hat \mu_{1,\omega},\omega)<0$. If $\hat \mu_{1,\omega}<\mu_{1,\omega}$, then
$\lambda_{1}( \mu,\omega)>-12$ for $\mu<\hat \mu_{1,\omega}$ sufficiently close to $\hat\mu_{1,\omega}$. Since
$
\lim_{\mu\to-\infty}\lambda_1(\mu,\omega)=-18$ by Lemma \ref{asymptotic behavior principal eigenvalue lim mu -infty}, there exists $\hat\mu_{2,\omega}\in (-\infty,\hat\mu_{1,\omega})$ such that $\lambda_1(\hat\mu_{2,\omega},\omega)=-12$ and  $\partial_\mu\lambda_1(\hat\mu_{2,\omega},\omega)\geq0$. Then by Lemma \ref{quadratic form-computation} (ii), we have $k_{i,J_{\omega,1}L_1|_{X_o^1}}^{\leq0}\geq2$, which contradicts the index formula \eqref{index formula 1o1}. If $\hat \mu_{1,\omega}>\mu_{1,\omega}$, then
$\lambda_{1}( \mu,\omega)<-12$ for $\mu>\hat \mu_{1,\omega}$ sufficiently close to $\hat\mu_{1,\omega}$. Since
$
\lim_{\mu\to-12^-}\lambda_1(\mu,\omega)=\lambda_1(-12,\omega)>-12$
by Lemma \ref{continuity of the principal eigenvalue mu -12}, there exists $\hat\mu_{3,\omega}\in (\hat\mu_{1,\omega},-12)$ such that $\lambda_1(\hat\mu_{3,\omega},\omega)=-12$ and  $\partial_\mu\lambda_1(\hat\mu_{3,\omega},\omega)\geq0$,  which again contradicts  \eqref{index formula 1o1}. If $\Phi_{\hat \mu_{1,\omega},\omega,1}$ is even, similarly we will get a contradiction  due to \eqref{L1e2ok3nonnegative}.
\end{proof}
\subsection{Proof of the positive critical rotation rate ${69\over2}$  for the second Fourier mode}
We prove that the  critical rotation rate in the positive half-line is ${69\over2}$ for the $2$'nd Fourier mode.
 \begin{Theorem}\label{k=2 positive half-line critical rotation rate}
Let $k=2$. Then the $3$-jet is  linearly unstable  for $\omega\in\left(12,{69\over2}\right)$ and spectrally  stable for $\omega\in\left[{69\over2},72\right)$.
\end{Theorem}

By Theorem \ref{linear instability} and Rayleigh's criterion, the $3$-jet is  linearly unstable  for $\omega\in\left(0,12\right]$ and spectrally stable  for $\omega\in\left[72,\infty\right)$ for $k=2$. This, along with Theorem \ref{k=2 positive half-line critical rotation rate}, implies that the critical rotation rate for the positive half-line for $k=2$ is ${69\over2}$.

Let us first consider the instability part, the  motivation  of which is from the integral identity \eqref{inte}.

\begin{Lemma}\label{instability part k=2}
Let $k=2$. Then the $3$-jet is  linearly unstable  for $\omega\in\left(12,{69\over2}\right)$.
\end{Lemma}

\begin{proof}
For any neutral mode $(c,2,\omega,\Phi)$, we have $c\in [0,3+\omega)$ by Lemma \ref{tilde Psi0 does not change sign} (2). If $c\in\text{Ran}(-\tilde \Psi_{\omega}')^\circ=(-12+\omega,3+\omega)$, we have $c=c_\omega$ by Theorem \ref{Psi prime change sign c}.
Then $k_{0,J_{\omega,2}L_2|_{X_e^2}}^{\leq0}=0$ by Lemma \ref{kernel JL}. Thus, it suffices to prove that  there exist no $c\in[0,-12+\omega]$ such that  $(c,2,\omega,\Phi)$ is a neutral mode.

Suppose that there exists a neutral mode $(c,2,\omega,\Phi)$ with $c\in[0,-12+\omega]$ for $\omega\in\left(12,{69\over2}\right)$.  We define $R_\omega(s)$ and $F_\omega(s)$ as in \eqref{def-R F}.
Then we get the ODE system \eqref{Rayleigh-type equation-F} and the integral identity \eqref{inte}, where we use \eqref{boundary term pm1} to handle the boundary terms from integration by parts for $c=-12+\omega$. On the one hand, by \eqref{inte}, we have
\begin{align}\label{integrate contr1}
\int_{-1}^1\left({3R_\omega^2F_\omega^2\over 1-s^2}+2cR_\omega F_\omega^2\right)ds\leq 0.
\end{align}
Noting that $0\leq c\leq -12+\omega<{45\over2}$, we have $3-{2c\over15}>3-{2\over15}\cdot{45\over2}=0.$ Since $R_\omega(s)=15s^2-3-\omega+c\leq 15s^2-15$,   we have, on the other hand, that
\begin{align}\nonumber
\int_{-1}^1\left({3R_\omega^2F_\omega^2\over 1-s^2}+2cR_\omega F_\omega^2\right)ds=&\int_{-1}^1\left({3\over 1-s^2}+{2c\over R}\right)R_\omega^2F_\omega^2ds\\\nonumber
\geq&\int_{-1}^1\left({3\over 1-s^2}+{2c\over 15s^2-15}\right)R_\omega^2F_\omega^2ds\\\label{integrate contr2}
=&\left({3}-{2c\over 15}\right)\int_{-1}^1{\Phi^2\over 1-s^2}ds>0,
\end{align}
where we use $\Phi\not\equiv0$. Then \eqref{integrate contr1} contradicts \eqref{integrate contr2}.

Thus, $k_{i,J_{\omega,2}L_2|_{X_e^2}}^{\leq0}=0$ and $k_{c,J_{\omega,2}L_2|_{X_e^2}}+ k_{r,J_{\omega,2}L_2|_{X_e^2}}=1$ by \eqref{index formula 1o2}. This proves linear instability of the $3$-jet for $\omega\in\left(12,{69\over2}\right)$.
\end{proof}

To prove   the stability part  for $\omega\in\left[{69\over2},72\right)$,
we study the eigenvalue problem
  \begin{align}\label{Rayleigh-type equation lambda k=2}
  ((1-s^2)\Phi')'-{4\over 1-s^2}\Phi-{2\omega+12\mu\over 15s^2-3+\mu}\Phi=\tilde\lambda\Phi,\quad \Delta_2\Phi\in L^2(-1,1).
 \end{align}
restricted to the space
 \begin{align}\label{def-X-omega-mu-e k=2}
X_{\omega,\mu,e}=\left\{\Phi\in X_{\omega,\mu}| \Phi \text{ is even}\right\}.
 \end{align}
Here, $X_{\omega,\mu}$ is defined in \eqref{def-X-omega-mu}, and  we note that the $3$-jet is spectrally stable  for $k=2$ when restricting to the  space of odd functions
by \eqref{L1e2ok3nonnegative}.

 By Lemma \ref{compact embedding}, the eigenvalue problem
 \eqref{Rayleigh-type equation lambda k=2} has a sequence eigenvalues $-\infty<\cdots\leq\tilde \lambda_{n}(\mu,\omega)\leq \cdots \leq \tilde\lambda_{1}(\mu,\omega)$, which can be defined by
 \begin{align}\label{def-tilde-lambda-n}
\tilde\lambda_{n}(\mu,\omega)=& \sup_{\Phi \in X_{\omega,\mu,e}, (\Phi, \tilde\Phi_{i})_{L^2} = 0, i = 1, 2, \cdots, n-1}{\int_{-1}^1\left(-(1-s^2)|\Phi'|^2-{4\over 1-s^2}|\Phi|^2-{2\omega+12\mu\over 15s^2-3+\mu}|\Phi|^2\right) ds\over\int_{-1}^1|\Phi|^2ds},
\end{align}
where the supremum for $\tilde\lambda_{n}(\mu,\omega)$ is attained at $\tilde\Phi_{n} \in  X_{\omega,\mu,e}$.
Similar  to Lemma  \ref{quadratic form-computation},
    we give a formula to compute  $\langle L_2\Upsilon,\Upsilon\rangle$ for a neutral mode $(c,2,\omega,\Phi)$ with $c\in[0,-12+\omega]$, where  $\Upsilon=\Delta_2\Phi$.

\begin{Lemma}\label{quadratic form-computation k=2}
$(\rm{i})$ Let $\omega\geq12$ and  $(c_{2},2,\omega,\Phi_{\mu_2,\omega})$ be
a  neutral mode, where $c_2\leq-12+\omega$ and $\mu_2=-\omega+c_{2}$. Then $\tilde \lambda_{n_0}(\mu_2,\omega)=-12$
for  some $n_{0}\geq1$ and
\begin{equation}
\langle L_{2}\Upsilon_{\mu_{2},\omega},\Upsilon_{\mu_{2},\omega}\rangle
=(c_2-c_\omega)\int_{-1}^1{-12(15s^2-3)+2\omega \over (15s^2-3+\mu_2)^2}|\Phi_{\mu_{2},\omega}|^2ds,
\label{equality-L-form-derivative k=2}%
\end{equation}
where $\Upsilon_{\mu_{2},\omega}=\Delta_2\Phi_{\mu_{2},\omega}$ and $c_\omega={5\over6} \omega$.

$(\rm{ii})$ Under the assumptions of $(\rm{i})$, if $c_2<-12+\omega$ and $\|\Phi_{\mu_{2},\omega}\|_{L^2(-1,1)}=1$, then
\begin{equation}
\langle L_{2}\Upsilon_{\mu_{2},\omega},\Upsilon_{\mu_{2},\omega}\rangle
=(c_2-c_\omega)\partial_\mu \tilde \lambda_{n_0}(\mu_2,\omega),
\label{equality-L-form-derivative2 k=2}%
\end{equation}
\end{Lemma}

Then we prove the spectral stability of the critical rotation rate $\omega={69\over2}$.

\begin{Lemma}\label{critical rotation rate 69 over 2 ometa=-12}
Let $k=2$. Then the $3$-jet  is spectrally  stable for $\omega={69\over2}$.
\end{Lemma}
\begin{proof}
We determine for which $\omega$, $c=-12+\omega$ is neutral.
Putting  $\mu=-\omega+c=-12$ and  $k=2$ into \eqref{Rayleigh-type equation 12}, we have
\begin{align}\label{c-12+omega k=2}
  ((1-s^2)\Phi')'+{-4-{144-2\omega\over15}\over 1-s^2}\Phi+12\Phi=0, \quad \Delta_2\Phi\in L^2(-1,1).
 \end{align}
If $-4-{144-2\omega\over15}=-m^2$ and $m=0,1,2,3$, then \eqref{c-12+omega k=2} can be solvable.

Noting that for $m=0\Rightarrow\omega={102}>72$, $m=1\Rightarrow\omega={189\over2}>72$ and $m=2\Rightarrow\omega=72$, the $3$-jet is spectrally stable  by Rayleigh's criterion.

For $m=3\Rightarrow\omega={69\over2}<72$,  $c={45\over2}$ and $c_\omega={115\over4}$. The equation \eqref{c-12+omega k=2} has a solution $P_3^3(s)=-15(1-s^2)^{3\over2}$.  Let $\Upsilon(\varphi,s)=e^{i2\varphi}\Delta_2P_3^3(s)$. Due to different frequencies, $\Upsilon$ satisfies the constraints $\iint_{D_T} \Upsilon d\varphi ds=0, \iint_{D_T}\Upsilon Y_1^md\varphi ds=0,  m=0,\pm1$.
By Lemma \ref{quadratic form-computation k=2} (i), we have
\begin{equation*}
\langle L_{2}\Delta_2P_3^3,\Delta_2P_3^3\rangle
=225\cdot\left({45\over2}-{115\over4}\right)\int_{-1}^1{-180s^2+105\over (15s^2-15)^2}(1-s^2)^3ds=-575<0.
\end{equation*}
Then
 $k_{i,J_{\omega,2}L_2|_{X_e^2}}^{\leq0} =1$ since $\Upsilon$ is even.  By the index formula  \eqref{index formula 1o2}, we have
 $ k_{c,J_{\omega,2}L_2|_{X_e^2}}+ k_{r,J_{\omega,2}L_2|_{X_e^2}}=0.$ Thus, the $3$-jet is spectrally stable for $\omega={69\over2}$ and $k=2$.
\end{proof}

For $\mu=-12$, we compute the explicit values of the principal eigenvalues $\tilde\lambda_1(\mu,\omega)$ with corresponding eigenfunctions.
\begin{Lemma}\label{explicit values of the principal eigenvalues}
The principal eigenvalue $\tilde \lambda_1(-12,\omega)$ is
\begin{align}\label{principal eigenvalue -12 omega k=2}
\tilde\lambda_1(-12,\omega)=-\sqrt{4-{2\omega-144\over 15}}\left(1+\sqrt{4-{2\omega-144\over 15}}\right)\in[-2\sqrt{3}-12,-6], \quad\omega\in[12,72],
\end{align}
and a corresponding eigenfunction is given by
\begin{align}\label{eigenfunction of principal eigenvalue -12 omega k=2}
\tilde\Phi_{-12,\omega}(s)=(1-s^2)^{{\sqrt{4-{2\omega-144\over 15}}\over2}}, \quad s\in [-1,1].
\end{align}
Consequently,
$\tilde\lambda_1(-12,\cdot)$ is increasing on $\omega\in[12,72]$.
\end{Lemma}
\begin{proof}
 Let
 \begin{align*}
 \Phi(s)=(1-s^2)^{{\sqrt{4-{2\omega-144\over 15}}\over2}}\phi(s), \quad s\in[-1,1].
 \end{align*}
If $\Phi$ solves \eqref{Rayleigh-type equation lambda k=2} with $\mu=-12$, then  $\phi$ solves the equation
 \begin{align*}
 (1-s^2)\phi''-2\left(\sqrt{4-{2\omega-144\over 15}}+1\right)s\phi'+\left(-4+{2\omega-144\over15}-\sqrt{4-{2\omega-144\over15}}-\tilde\lambda\right)\phi=0.
 \end{align*}
This is a
Gegenbauer equation. Using the Gegenbauer polynomials, we know that
 \begin{align*}
  \tilde\lambda=-\left(n+\sqrt{4-{2\omega-144\over 15}}\right)\left(n+\sqrt{4-{2\omega-144\over 15}} +1\right),
\end{align*}
and $\Phi(s)=(1-s^2)^{{\sqrt{4-{2\omega-144\over 15}}\over2}}C_n^\beta(s)$ solves  \eqref{Rayleigh-type equation lambda k=2} with $\mu=-12$,
where $C_n^\beta(s)$ is the Gegenbauer polynomial with $n\geq0$, $\beta=\sqrt{4-{2\omega-144\over 15}}+{1\over2}\in[{5\over2},2\sqrt{3}+{1\over2}]$.  Since the eigenfunction is even, \eqref{principal eigenvalue -12 omega k=2} and \eqref{eigenfunction of principal eigenvalue -12 omega k=2} are obtained by taking $n=0$.
\end{proof}

Now, we prove   Theorem \ref{k=2 positive half-line critical rotation rate}.
\begin{proof}[Proof of Theorem \ref{k=2 positive half-line critical rotation rate}]
The instability part is proved in Lemma \ref{instability part k=2}. Now, we consider the stability part.
By Lemma \ref{critical rotation rate 69 over 2 ometa=-12}, it suffices to prove spectral stability  of the $3$-jet for $\omega\in\left({69\over2},72\right)$.
By Lemma \ref{explicit values of the principal eigenvalues}, we have $\tilde\lambda_1(-12,\omega)>\tilde\lambda_1(-12,{69\over2})=-12$ for $\omega\in\left({69\over2},72\right)$.
 Similar to \eqref{liminf}, we have
$
\liminf_{\mu\to-12^-}\tilde \lambda_1(\mu,\omega)\geq\tilde \lambda_1(-12,\omega)>-12.
$
Similar to \eqref{principal eigenvalue lim mu -infty}, we have
$
\lim_{\mu\to-\infty}\tilde\lambda_1(\mu,\omega)=-18.
$
Thus, there exists $\mu_{2,\omega}\in(-\infty,-12)$ such that
$
\tilde \lambda_1(\mu_{2,\omega},\omega)=-12 $ and  $\partial_\mu\tilde\lambda_1(\mu_{2,\omega},\omega)\geq0.
$
So, there exists a neutral mode $(c_{2,\omega},2,\omega,\Phi_{\mu_{2,\omega},\omega,2})$ with $c_{2,\omega}=\mu_{2,\omega}+\omega<-12+\omega$, where $\|\Phi_{\mu_{2,\omega},\omega,2}\|_{L^2(-1,1)}=1$.
Note that $c_{2,\omega}-c_\omega<-12+\omega-{5\over6}\omega=-12+{1\over 6}\omega<0$ for $\omega\in({69\over 2},72)$.
By Lemma \ref{quadratic form-computation k=2} (ii), we have
$
\langle L_{2}\Upsilon_{\mu_{2,\omega},\omega,2},\Upsilon_{\mu_{2,\omega},\omega,2}\rangle
=(c_{2,\omega}-c_\omega)\partial_\mu \tilde\lambda_{1}(\mu_{2,\omega},\omega)\leq 0,
$
where $\Upsilon_{\mu_{2,\omega},\omega,2}=\Delta_2\Phi_{\mu_{2,\omega},\omega,2}$.
Thus, $k_{i,J_{\omega,2}L_2|_{X_e^2}}^{\leq0}=1$. By the index formula
\eqref{index formula 1o2}, we have $k_{c,J_{\omega,2}L_2|_{X_e^2}}+ k_{r,J_{\omega,2}L_2|_{X_e^2}}=0$. This proves
 spectral stability of the $3$-jet  for $\omega\in\left({69\over2},72\right)$.
 \end{proof}
Similar to Corollary \ref{uniqueness1}, we have the following result.
 \begin{Corollary}\label{uniqueness3}
Let $\omega\in\left({69\over2},72\right)$ and $k=2$. Then there exists
a unique $\mu_{2,\omega}\in(-\infty,-12)$ such that  $(c_{2,\omega},2,\omega,\Phi_{\mu_{2,\omega},\omega,2})$ is a  neutral mode, where $c_{2,\omega}=\mu_{2,\omega}+\omega$. Moreover,  $
\langle L_{2}\Upsilon_{\mu_{2,\omega},\omega,2},$ $\Upsilon_{\mu_{2,\omega},\omega,2}\rangle$ $
\leq 0$, where $\Upsilon_{\mu_{2,\omega},\omega,2}=\Delta_2\Phi_{\mu_{2,\omega},\omega,2}$.
\end{Corollary}

 \section{The critical rotation rate for the negative half-line}\label{The critical rotation rate in the negative half-line}

In this section, we consider the critical rotation rate in the negative half-line, and  prove Theorem \ref{negative half-line critical rotation rate}. 
Note that linear instability of the  $3$-jet in the case of $\omega\in(-3,0]$ is proved in Theorem \ref{linear instability} for $k=1, 2$. Thus, we only need to consider $\omega\in\left(-18,-3\right]$.

\subsection{Proof of the negative critical rotation rate $-3$ for the first Fourier mode} The main result in this subsection states as follows.

 \begin{Theorem}\label{k=1 negative half-line}
Let $k=1$. Then the $3$-jet is  spectrally  stable for $\omega\in\left(-18,-3\right]$.
\end{Theorem}
Before going into the details,
let us first discuss the ideas in the proof of the above theorem.
\begin{Remark}\label{ideas in the proof of k=1 negative half-line}
The  ideas in the proofs of
Theorems
\ref{k=1 positive half-line critical rotation rate} and
\ref{k=2 positive half-line critical rotation rate}
 can not be applied to prove Theorems
\ref{k=1 negative half-line}.
 First,  by Lemma \ref{tilde Psi0 does not change sign} (1) and Lemma \ref{case-omega=12} (ii), it suffices to  study  $\lambda_{1}(\mu,\omega)$ and $\tilde\lambda_{1}(\mu,\omega)$ with $\mu\geq 3$. An important difference from Theorems
\ref{k=1 positive half-line critical rotation rate} and
\ref{k=2 positive half-line critical rotation rate} is that  for the endpoint case $\mu=3$, the Rayleigh equation \eqref{Rayleigh-type equation lambda} has   singularity at $s=0$ when $\omega\in(-18,-3]$, and  \eqref{Rayleigh-type equation lambda} has  no singularity in $(-1,1)$ when $\omega=-18$.
 For $\omega\in(-18,-3]$, the principal eigenvalues  $\lambda_{1}(3,\omega)$ and $\tilde\lambda_{1}(3,\omega)$ are subtle to  be solved due to the singularity at $s=0$.
For $\omega=-3$, we benefit again from \eqref{inte} with $k=1$ and $c=0$, which gives a nontrivial solution
$\text{sign}(s)s^2(1-s^2)^{1\over2}$ with $\lambda_1(3,-3)=-12$. Moreover, $\lambda_1(3,-18)=-6$ and    $s(1-s^2)^{1\over2}$ is a corresponding eigenfunction. This motivates us to conjecture that the eigenfunction for $\lambda_{1}(3,\omega)$ with $\omega\in[-18,-3]$ has the form of $s^a(1-s^2)^{1\over2}$, $a\in[1,2]$.
 Using this form, we explicitly solve the eigenvalue $\lambda_{1}(3,\omega)$  in \eqref{eigenvalue-3-omega} for $\omega\in[-18,-3]$.
 Note that $\lambda_1(3,\omega)>\lambda_1(3,-3)=-12$ for $\omega\in[-18,-3)$.
  This, along with the
 asymptotic behavior of $\lambda_{1}(\cdot,\omega)$ as $\mu\to 3^+$ or $\infty$ in Lemma \ref{asymptotic behavior of the principal eigenvalues as mu-to 3 or infty lemma},
 proves spectral stability for $\omega\in\left[-18,-3\right]$ ($k=1$).
\end{Remark}

We first consider $k=1$ and $\omega=-3$. By the numerical result in Fig. 3 of \cite{Taylor2016}, the critical rotation rate  is $\omega=-3$ in our notation. Now,
we give a rigorous proof for $\omega=-3$. To this end, we need the following lemma, the proof of which is similar to Lemma \ref{quadratic form-computation}.
 \begin{Lemma}\label{quadratic form-computation k=1 negative}
$(\rm{i})$ Let $\omega\leq-3$ and  $(c_{1},1,\omega,\Phi_{\mu_1,\omega})$ be
a  neutral mode, where $c_1\geq3+\omega$ and $\mu_1=-\omega+c_{1}$. Then $\lambda_{n_0}(\mu_1,\omega)=-12$
for  some $n_{0}\geq1$ and
\eqref{equality-L-form-derivative}
holds.

$(\rm{ii})$ Under the assumptions of $(\rm{i})$, if $c_1>3+\omega$ and $\|\Phi_{\mu_{1},\omega}\|_{L^2(-1,1)}=1$, then
\eqref{equality-L-form-derivative2}
 holds.
\end{Lemma}
For $\omega=-3$, we have the following result.
\begin{Lemma}\label{ometa=3}
Let $k=1$. Then the $3$-jet is spectrally  stable for $\omega=-3$.
\end{Lemma}
\begin{proof}
The proof is motivated by \eqref{inte} for $k=1$, $c=0$, $R_\omega(s)=\tilde \Psi_{\omega}'(s)=15s^2$ and $F_\omega(s)={\Phi(s)\over \tilde \Psi_{\omega}'(s)+c}={\Phi(s)\over 15s^2}.$
In this case, \eqref{Rayleigh-type equation-F} becomes
\begin{align}\label{Rayleigh-type equation-F-omega-3kc}
-(1-s^2)^{-{1\over2}}(((1-s^2)^{-{1\over2}}F_\omega)'(1-s^2)^2R_\omega^2)'=0.
\end{align}
Multiplying \eqref{Rayleigh-type equation-F-omega-3kc} by $F_\omega$ and integrating from $0$ to $1$, we have
\begin{align*}
\int_{0}^1\left(|((1-s^2)^{-{1\over2}}F_\omega)'|^2(1-s^2)^2R_\omega^2\right)ds=0.
\end{align*}
This implies that $F_\omega(s)={\Phi(s)\over 15s^2}=C_0(1-s^2)^{1\over2}$ on $(0,1)$, where $C_0\in\mathbb{R}$. Taking $C_0={1\over 15}$, we have $\Phi(s)=s^2(1-s^2)^{1\over2}$ on $(0,1)$. Then by \eqref{Rayleigh-type equation-F-omega-3kc}, $\Phi$ solves \eqref{Rayleigh-type equation} on $(0,1)$ for $k=1$, $c=0$ and $\omega=-3$. This can also be checked directly. Indeed,
 \begin{align*}
& ((1-s^2)\Phi')'=2(1-s^2)^{3\over2}-9s^2(1-s^2)^{1\over2}+s^4(1-s^2)^{-{1\over2}},\\
 &-{1\over 1-s^2}\Phi-{\tilde \Upsilon_{\omega}'\over \tilde \Psi_{\omega}'}\Phi=-s^2(1-s^2)^{-{1\over2}}-(-12s^2+2)(1-s^2)^{1\over2}\\
 &=-s^2(1-s^2)^{-{1\over2}}-12(1-s^2)^{3\over2}+10(1-s^2)^{3\over2}+10s^2(1-s^2)^{1\over2}\\
 &=-2(1-s^2)^{3\over2}+9s^2(1-s^2)^{1\over2}+s^2(1-s^2)^{1\over2}-s^2(1-s^2)^{-{1\over2}}\\
 &=-2(1-s^2)^{3\over2}+9s^2(1-s^2)^{1\over2}-s^4(1-s^2)^{-{1\over2}}.
 \end{align*}
 Now, we construct a function
 \begin{align*}
 \Phi_1(s)=\text{sign}(s) |s|^2(1-s^2)^{1\over2}, \quad s\in[-1,1].
 \end{align*}
 Then   $\Phi_1(0)=\Phi_1(0+)=\Phi_1'(0)=\Phi_1'(0+)=0$, and
 \begin{align*}
 \Upsilon_1(\varphi,s)=e^{i\varphi}\Delta_1\Phi_1(s)=
e^{i\varphi}\text{sign}(s)(-12s^2+2)(1-s^2)^{1\over2}, \quad s\in[-1,1].
 \end{align*}
 Then $\Upsilon_1\in L^2(D_T)$.
 Since $\Phi_1$ is an odd function, we have
 \begin{align*}
\iint_{D_T} \Upsilon_1 d\varphi ds=0,
\iint_{D_T}\Upsilon_1 Y_1^md\varphi ds=0, \quad m=0,\pm1.
\end{align*}
For $\omega=-3$, we have $c_\omega=-{5\over2}$. By Lemma \ref{quadratic form-computation k=1 negative}, for the neutral mode $(0,1,\omega,\Phi_1)$, we have
\begin{equation*}
\langle L_{1}\Delta_1\Phi_1,\Delta_1\Phi_1\rangle
={5\over2}\int_{-1}^1{-180s^2+30\over 225s^4}s^4(1-s^2)ds=-{4\over45}<0.
\end{equation*}
This implies that
$ k_{i,J_{\omega,1}L_1|_{X_o^1}}^{\leq0}=1$, and according to the index formula \eqref{index formula 1o1}, we have
 $k_{c,J_{\omega,1}L_1|_{X_o^1}}+ k_{r,J_{\omega,1}L_1|_{X_o^1}}=0.$ Thus, the $3$-jet is spectrally stable for $\omega=-3$ and $k=1$.
\end{proof}

Now, we  consider $k=1$ and $\omega\in(-18,-3)$.
To compute the indices $k_{i,J_{\omega,1}L_1|_{X_o^1}}^{\leq0}$  in \eqref{index formula 1o1},
we need  to determine for which $c\in[3+\omega,0]$, $(c,1,\omega,\Phi)$ is a neutral mode. To this end, we study the eigenvalues of the
 Rayleigh system
  \begin{align}\label{Rayleigh-type equation lambda k=1 negative}
 ((1-s^2)\Phi')'-{1\over 1-s^2}\Phi-{2\omega+12\mu\over 15s^2-3+\mu}\Phi=\lambda\Phi, \quad
 \Delta_1\Phi\in L^2(-1,1)
 \end{align}
in $
X_{\omega,\mu,o}=\left\{\Phi\in X_{\omega,\mu}| \Phi \text{ is odd}\right\}
$
for $\mu\in[3,\infty)$, where $X_{\omega,\mu}$ is defined in \eqref{def-X-omega-mu}.
Since $X_{\omega,\mu,o}$ is compactly embedded in $L^2(-1,1)$, all the eigenvalues of the eigenvalue problem
\eqref{Rayleigh-type equation lambda k=1 negative}
 (restricted to the space $X_{\omega,\mu,o}$) are arranged in a sequence  $-\infty<\cdots\leq \lambda_{n}(\mu,\omega)\leq \cdots \leq \lambda_{1}(\mu,\omega)$, which has the expressions \eqref{def-lambda-n}.

For $\mu=3$, we give the exact values of the principal eigenvalues of \eqref{Rayleigh-type equation lambda k=1 negative}.
\begin{Lemma}\label{principal eigenvalues of Rayleigh-type equation lambda k=1 negative}
For $\omega\in[-18,-3]$, we have
\begin{align}\label{eigenvalue-3-omega}
\lambda_1(3,\omega)=-{2\omega+96\over 15}-2\sqrt{{8\omega+159\over 15}}\in[-12,-6],
\end{align}
with a corresponding eigenfunction
\begin{align}\label{eigenfunction-3-omega}
\Phi_{3,\omega}(s)={\rm{sign}}(s)|s|^{1+\sqrt{{8\omega+159\over 15}}\over2}(1-s^2)^{1\over2},\quad s\in[-1,1].
\end{align}
In particular, $\lambda_1(3,-3)=-12$, $\lambda_1(3,-18)=-6$, and $\lambda_1(3,\omega)$ is decreasing on $\omega\in[-18,-3]$.
\end{Lemma}

\begin{proof}
Let $\mu=3$. Then
  \eqref{Rayleigh-type equation lambda k=1 negative} becomes
  \begin{align}\label{Rayleigh-type equation lambda k=1 negative mu=3}
  ((1-s^2)\Phi')'-{1\over 1-s^2}\Phi-{2\omega+36\over 15s^2}\Phi=\lambda\Phi, \quad
 \Delta_1\Phi\in L^2(-1,1).
 \end{align}
By the proof of Lemma \ref{ometa=3}, for $\omega=-3$, $-12$ is an eigenvalue of \eqref{Rayleigh-type equation lambda k=1 negative mu=3} with a corresponding eigenfunction ${\rm{sign}}(s)s^2(1-s^2)^{1\over 2}$, $s\in[-1,1]$. For $\omega=-18$, $-6$ is  an eigenvalue of \eqref{Rayleigh-type equation lambda k=1 negative mu=3} with a corresponding eigenfunction $s(1-s^2)^{1\over 2}$, $s\in[-1,1]$. This motivates us to insert $\Phi(s)=s^a(1-s^2)^{1\over2}$, $s\in(0,1)$, into \eqref{Rayleigh-type equation lambda k=1 negative mu=3} with $a\in[1,2]$. By comparing the coefficients of $s^{a-2}(1-s^2)^{1\over2}$ and $s^{a}(1-s^2)^{1\over2}$, we have
 \begin{align*}
 \left\{ \begin{array}{llll} a^2-a={2\omega+36\over15}, \\
 a^2+3a+\lambda+2=0. \end{array}\right.
 \end{align*}
Thus,
\begin{align*}a={1+\sqrt{1+{4(2\omega+36)\over15}}\over2}={1+\sqrt{{8\omega+159\over 15}}\over2},
\end{align*}
and
\begin{align}\label{principal-eigenvalue-3-omega-computation}
\lambda=-a^2-3a-2=-{2\omega+36\over 15}-4a-2=-{2\omega+96\over 15}-2\sqrt{{8\omega+159\over 15}}.
 \end{align}
 Since we only consider odd functions, we choose $\Phi_{3,\omega}$ as given in \eqref{eigenfunction-3-omega}. By the equation in \eqref{Rayleigh-type equation lambda k=1 negative mu=3} and $a\in(1,2]$, we have
 $\Delta_1\Phi_{3,\omega}\in L^2(-1,1)$  for $\omega\in(-18,-3]$. Thus, $\lambda$ in \eqref{principal-eigenvalue-3-omega-computation}
 is an eigenvalue of \eqref{Rayleigh-type equation lambda k=1 negative mu=3} with an eigenfunction $\Phi_{3,\omega}$. For $\omega=-18$, $-6$ is clearly the principal eigenvalue of \eqref{Rayleigh-type equation lambda k=1 negative mu=3}.  To prove that $\lambda$ in \eqref{principal-eigenvalue-3-omega-computation} is the principal eigenvalue of \eqref{Rayleigh-type equation lambda k=1 negative mu=3} for $\omega\in(-18,-3]$,  we first note that ${|\Phi(s)|^2\over s}\big|_{s=0}=0$ for any $\Phi\in X_{\omega,3,o}$. Integrating by parts implies
\begin{align}\label{proof-principal eigenvalue}
&\left\|\sqrt{1-s^2}\Phi'-\Phi{\sqrt{1-s^2}\Phi_{3,\omega}'\over \Phi_{3,\omega}}\right\|_{L^2(-1,1)}^2\\\nonumber
=&\int_{-1}^1\left( (1-s^2)|\Phi'|^2 -2\Phi\Phi'{(1-s^2)\Phi_{3,\omega}'\over \Phi_{3,\omega}}+|\Phi|^2{(1-s^2){|\Phi_{3,\omega}'|}^2\over |\Phi_{3,\omega}|^2}\right)ds\\\nonumber
=&\int_{-1}^1\left( (1-s^2)|\Phi'|^2 -(1-s^2)\Phi_{3,\omega}'\left({|\Phi|^2\over \Phi_{3,\omega}}\right)'\right)ds
\\\nonumber
=&\int_{-1}^1 (1-s^2)|\Phi'|^2 ds-(1-s^2){\Phi_{3,\omega}'|\Phi|^2\over\Phi_{3,\omega}}\bigg|_{-1}^0-(1-s^2){\Phi_{3,\omega}'|\Phi|^2\over\Phi_{3,\omega}}\bigg|_{0}^1+\int_{-1}^1{((1-s^2)\Phi_{3,\omega}')'\over \Phi_{3,\omega}}|\Phi|^2ds\\\nonumber
=&\int_{-1}^1\left( (1-s^2)|\Phi'|^2 +{((1-s^2)\Phi_{3,\omega}')'\over \Phi_{3,\omega}}|\Phi|^2\right)ds\\\nonumber
=&\int_{-1}^1\left( (1-s^2)|\Phi'|^2 +\left({1\over1-s^2}+{2\omega+36\over 15s^2}+\lambda\right)|\Phi|^2\right)ds\geq0
\end{align}
for any $\Phi\in X_{\omega,3,o}$.
Thus,
\begin{align*}
\lambda_1(3,\omega)=\lambda=-a^2-3a-2=\sup_{\Phi \in X_{\omega,3,o}}{\int_{-1}^1\left(-(1-s^2)|\Phi'|^2-{1\over 1-s^2}|\Phi|^2-{2\omega+36\over 15s^2}|\Phi|^2\right) ds\over\int_{-1}^1|\Phi|^2ds}
\end{align*}
is the principal eigenvalue of \eqref{Rayleigh-type equation lambda k=1 negative mu=3} for $\omega\in(-18,-3]$
and the supremum can be attained at $\Phi_{3,\omega}$.
\end{proof}

Next, we study the asymptotic behavior of the principal eigenvalues as $\mu\to 3^+$ or $\infty$.

\begin{Lemma}\label{asymptotic behavior of the principal eigenvalues as mu-to 3 or infty lemma}
Let $\omega\in(-18,-3)$. Then
\begin{align}\label{asymptotic behavior of the principal eigenvalues as mu-to 3 or infty}
\liminf_{\mu\to3^+}\lambda_1(\mu,\omega)\geq\lambda_1(3,\omega)>-12,\quad \lim_{\mu\to\infty}\lambda_1(\mu,\omega)=-18.
\end{align}
\end{Lemma}
\begin{proof}
We normalize  $\Phi_{3,\omega}$ in \eqref{eigenfunction-3-omega}  such that $\|\Phi_{3,\omega}\|_{L^2(-1,1)}=1$.
Since $0<\mu-3\leq 15s^2-3+\mu$, and
\begin{align*}
\left|\left({2\omega+12\mu\over 15s^2-3+\mu}-{2\omega+36\over 15s^2}\right)|\Phi_{3,\omega}|^2\right|=\left|{(\mu-3)(12(15s^2-3)-2\omega)\over (15s^2-3+\mu)15s^2}|\Phi_{3,\omega}|^2\right|\leq {C\over s^2}|\Phi_{3,\omega}|^2
\end{align*}
with $C$ independent of  $\mu\in(3,\infty)$, we have
\begin{align*}
\lim_{\mu\to3^+}\int_{-1}^1{2\omega+12\mu\over 15s^2-3+\mu}|\Phi_{3,\omega}|^2 ds=\int_{-1}^1{2\omega+36\over 15s^2}|\Phi_{3,\omega}|^2 ds.
\end{align*}
Taking the infimum limit as $\mu\to3^+$ in
\begin{align*}
\lambda_1(\mu,\omega)\geq \int_{-1}^1\left(-(1-s^2)|\Phi_{3,\omega}'|^2-{1\over 1-s^2}|\Phi_{3,\omega}|^2-{2\omega+12\mu\over 15s^2-3+\mu}|\Phi_{3,\omega}|^2 \right)ds,
\end{align*}
by Lemma \ref{principal eigenvalues of Rayleigh-type equation lambda k=1 negative} we have
\begin{align*}
&\liminf_{\mu\to3^+}\lambda_1(\mu,\omega)\geq \int_{-1}^1\left(-(1-s^2)|\Phi_{3,\omega}'|^2-{1\over 1-s^2}|\Phi_{3,\omega}|^2-{2\omega+36\over 15s^2}|\Phi_{3,\omega}|^2 \right)ds\\
=&\lambda_1(3,\omega)>\lambda_1(3,-3)=-12.
\end{align*}
Similar to \eqref{principal eigenvalue lim mu -infty}, we obtain the second limit in \eqref{asymptotic behavior of the principal eigenvalues as mu-to 3 or infty}.
\end{proof}

Now, we are ready to prove  Theorem \ref{k=1 negative half-line}.
\begin{proof}[Proof of Theorem \ref{k=1 negative half-line}]
For $\omega=-3$, spectral stability is proved in Lemma \ref{ometa=3}.
Let $\omega\in(-18,-3)$.
By Lemma \ref{asymptotic behavior of the principal eigenvalues as mu-to 3 or infty lemma}, there exists $\mu_{1,\omega}\in(3,\infty)$ such that $\lambda_1(\mu_{1,\omega},\omega)=-12$ and
$\partial_\mu\lambda_1(\mu_{1,\omega},\omega)\leq0$. Let $\Phi_{\mu_{1,\omega},\omega,1}$ be a $L^2$ normalized eigenfunction of $\lambda_1(\mu_{1,\omega},\omega)$ and $\Upsilon_{\mu_{1,\omega},\omega,1}=\Delta\Phi_{\mu_{1,\omega},\omega,1}$.
Since $c_{1,\omega}=\mu_{1,\omega}+\omega>3+\omega$, we have $c_{1,\omega}-c_{\omega}>3+\omega-{5\over 6}\omega=3+{1\over 6}\omega>0$ for $-18<\omega<-3$.
By Lemma \ref{quadratic form-computation k=1 negative} (ii), we have
\begin{align*}
\langle L_{1}\Upsilon_{\mu_{1,\omega},\omega,1},\Upsilon_{\mu_{1,\omega},\omega,1}\rangle
=(c_{1,\omega}-c_\omega)\partial_\mu \lambda_{1}(\mu_{1,\omega},\omega)\leq0.
\end{align*}
Thus, $k_{i,J_{\omega,1}L_1|_{X_o^1}}^{\leq0}=1$. By the index formula
\eqref{index formula 1o1}, the $3$-jet is
 spectrally stable    for $\omega\in(-18,-3)$.
\end{proof}
By the asymptotic behavior of $\lambda_1(\mu,\omega)$  as $\mu\to 3^+$ or $\infty$, and the index formula \eqref{index formula 1o1}, we have the following result.
\begin{Corollary}\label{uniqueness2}
Let $\omega\in\left(-18,-3\right)$ and $k=1$. Then there exists
a unique $\mu_{1,\omega}\in(3,\infty)$ such that  $(c_{1,\omega},1,\omega,\Phi_{\mu_{1,\omega},\omega,1})$ is a  neutral mode, where $c_{1,\omega}=\mu_{1,\omega}+\omega$. Moreover, $\langle L_{1}\Upsilon_{\mu_{1,\omega},\omega,1},$ $\Upsilon_{\mu_{1,\omega},\omega,1}\rangle
\leq0$, where $\Upsilon_{\mu_{1,\omega},\omega,1}=\Delta\Phi_{\mu_{1,\omega},\omega,1}$.
\end{Corollary}

\subsection{Proof of the negative critical rotation  rate $g^{-1}(-12)$ for the  second Fourier mode} Recall that the function $g$ is defined in \eqref{def-g}. In this subsection, we prove that the critical rotation rate  for the negative half-line  is $g^{-1}(-12)$ for the $2$'nd Fourier mode.

 \begin{Theorem}\label{k=2 negative half-line}
Let $k=2$. Then the $3$-jet is   linearly unstable  for $\omega\in\left(g^{-1}(-12),-3\right]$ and spectrally stable for  $\omega\in\left(-18,g^{-1}(-12)\right]$.
\end{Theorem}

Let us first discuss the ideas in the proof of Theorem \ref{k=2 negative half-line}.

\begin{Remark}\label{ideas in the proof of k=2 negative half-line}
For $k=2$, we can compute the eigenvalue $\tilde\lambda_{1}(3,\omega)$  in
 \eqref{eigenvalue-tilde-3-omega} for $\omega\in(-18,-3]$ like what we did for $k=1$ in Lemma \ref{principal eigenvalues of Rayleigh-type equation lambda k=1 negative}. For $\omega=-18$, $\tilde\lambda_{1}(3,-18)=-6$ can be solved directly using the Legendre polynomials.
 However, $\tilde \lambda_1(3,\omega)<-12$ for $\omega\in (-18,-3]$. This means that we can not obtain stability  using  the spectral right-continuity of $\tilde \lambda_1(\cdot,\omega)$ at $\mu=3$ and the asymptotic behavior near $\mu=\infty$. A key point is that since $\tilde \lambda_1(3,-18)=-6$ but $\lim_{\omega\to-18^+}\tilde \lambda_1(3,\omega)=-12$,  there is a lift-up jump of $\tilde \lambda_1(3,\omega)$ at $\omega=-18$. After a careful study on $\tilde\lambda_{1}(\mu,\omega)$, we obtain that $\sup_{\mu\in[3,\infty)}\tilde\lambda_{1}(\mu,\omega)>-12$  if $\omega$ is close to $-18$. This implies spectral stability for $\omega$ near $-18$ and motivates us to define the function $g$ in \eqref{def-g}, which is   decreasing and continuous on $\omega\in[-18,-3]$. Here, we use
$
\tilde \lambda_1(\mu,\omega)<-17
$ for  $\mu>183$ and $\omega\in[-18,-3]$ by Lemma \ref{principal eigenvalue mu larger than 183}.
Since $g(-18)=-6$ and $g(-3)<-12$, we have $g^{-1}(-12)\in(-18,-3)$. For $\omega\in(-18,g^{-1}(-12)]$, there exists a neutral mode  ($c>3+\omega$)  with desired signature of the quadratic form $\langle L_2\cdot,\cdot\rangle$, which implies
 $k_{i,J_{\omega,2}L_2|_{X_e^2}}^{\leq0}=1$ and
 spectral stability for $k=2$.  For $\omega\in (g^{-1}(-12),-3)$, there exist no neutral modes  with $c\geq3+\omega$, which implies instability for $k=2$.
 \end{Remark}

For $k=2$, we study the eigenvalues of the
 Rayleigh system
  \begin{align}\label{Rayleigh-type equation lambda k=2 negative}
  ((1-s^2)\Phi')'-{4\over 1-s^2}\Phi-{2\omega+12\mu\over 15s^2-3+\mu}\Phi=\tilde \lambda\Phi, \quad
 \Delta_2\Phi\in L^2(-1,1)
 \end{align}
in $
X_{\omega,\mu,e}=\left\{\Phi\in X_{\omega,\mu}| \Phi \text{ is even}\right\}
$
for $\mu\in[3,\infty)$, where $X_{\omega,\mu}$ is defined in \eqref{def-X-omega-mu}.
By the  compact embeddedness   $X_{\omega,\mu,e}\hookrightarrow L^2$, all the eigenvalues of the eigenvalue problem
\eqref{Rayleigh-type equation lambda k=2 negative}
  are arranged in a sequence  $-\infty<\cdots\leq \tilde \lambda_{n}(\mu,\omega)\leq \cdots \leq \tilde\lambda_{1}(\mu,\omega)$, which has the expressions \eqref{def-tilde-lambda-n}.

  For $\mu=3$, the principal eigenvalues of \eqref{Rayleigh-type equation lambda k=2 negative} for $\omega\in(-18,-3]$ is quite different from those for $\omega=-18$.

\begin{Lemma}\label{principal eigenvalues mu=3}
$($Lift-up jump of $\tilde \lambda_1(3,\omega)$ at $\omega=-18$$)$
$({\rm i})$ For $\omega\in(-18,-3]$, we have
\begin{align}\label{eigenvalue-tilde-3-omega}
\tilde\lambda_1(3,\omega)=-{2\omega+171\over 15}-3\sqrt{{8\omega+159\over 15}}\in[-20,-12),
\end{align}
with a corresponding eigenfunction
\begin{align}\label{def-tilde-Phi-3-omega}
\tilde\Phi_{3,\omega}(s)=|s|^{1+\sqrt{{8\omega+159\over 15}}\over2}(1-s^2), \quad s\in[-1,1].
\end{align}
In particular, $\tilde\lambda_1(3,-3)=-20$, $\lim_{\omega\to-18^+}\tilde\lambda_1(3,\omega)=-12$, and $\tilde\lambda_1(3,\omega)$ is decreasing on $\omega\in(-18,-3]$.

$({\rm ii})$ Consider $\omega=-18$. Then
\begin{align*}
\tilde\lambda_1(3,-18)=-6,
\end{align*}
with a corresponding eigenfunction
\begin{align}\label{eigenfunction tilde lambda13-18}
\tilde\Phi_{3,-18}(s)=1-s^2, \quad s\in[-1,1].
\end{align}

Consequently,  $\lim_{\omega\to-18^+}\tilde \lambda_1(3,\omega)=-12$ and $\tilde \lambda_1(3,-18)=-6$  yield a lift-up jump of $\tilde \lambda_1(3,\omega)$ at $\omega=-18$.
\end{Lemma}

\begin{proof}
(i) The
 Rayleigh equation \eqref{Rayleigh-type equation lambda k=2 negative} with $\mu=3$ is
  \begin{align}\label{Rayleigh-type equation lambda k=2 negative mu=3}
 ((1-s^2)\Phi')'-{4\over 1-s^2}\Phi-{2\omega+36\over 15s^2}\Phi=\tilde\lambda\Phi, \quad
 \Delta_2\Phi\in L^2(-1,1).
 \end{align}
For $\omega=-3$, the principal eigenvalue of \eqref{Rayleigh-type equation lambda k=2 negative mu=3} is $\tilde \lambda_1(3,-3)=-20$ with a corresponding eigenfunction to be $\tilde \Phi_{3,-3}(s)=s^2(1-s^2)$, $s\in[-1,1]$. Inserting $\Phi(s)=s^a(1-s^2)$, $a>0$, into \eqref{Rayleigh-type equation lambda k=2 negative mu=3} and comparing the coefficients of $s^{a-2}(1-s^2)$ and $s^a(1-s^2)$,
we have
 \begin{align*}
 \left\{ \begin{array}{llll} a^2-a={2\omega+36\over15}, \\
 a^2+5a+\tilde\lambda+6=0. \end{array}\right.
 \end{align*}
Thus,
\begin{align*}a={1+\sqrt{{8\omega+159\over 15}}\over2},
\end{align*}
and
\begin{align}\label{tilde lambda a}
\tilde\lambda=-a^2-5a-6=-{2\omega+36\over 15}-6a-6=-{2\omega+171\over 15}-3\sqrt{{8\omega+159\over 15}}.
 \end{align}
 Since we only consider even functions, we choose $\tilde \Phi_{3,\omega}$ as given in \eqref{def-tilde-Phi-3-omega}.
 By the equation in \eqref{Rayleigh-type equation lambda k=2 negative mu=3} and $a\in(1,2]$, we have $\Delta_2\tilde \Phi_{3,\omega}\in L^2(-1,1)$. Then
 $\tilde \lambda$ in \eqref{tilde lambda a} is an eigenvalue of \eqref{Rayleigh-type equation lambda k=2 negative mu=3} with an eigenfunction
 $\tilde\Phi_{3,\omega}$.
 For $\omega\in(-18,-3]$ and $\Phi\in X_{\omega,\mu,e}$, similar to \eqref{proof-principal eigenvalue} we have
\begin{align*}
&\left\|\sqrt{1-s^2}\Phi'-\Phi{\sqrt{1-s^2}\tilde\Phi_{3,\omega}'\over \tilde \Phi_{3,\omega}}\right\|_{L^2(-1,1)}^2
=\int_{-1}^1\left( (1-s^2)|\Phi'|^2 +{((1-s^2)\tilde\Phi_{3,\omega}')'\over \tilde\Phi_{3,\omega}}|\Phi|^2\right)ds\\
=&\int_{-1}^1\left( (1-s^2)|\Phi'|^2 +\left({4\over1-s^2}+{2\omega+36\over 15s^2}+\tilde\lambda\right)|\Phi|^2\right)ds\geq0.
\end{align*}
Thus,
$\tilde\lambda_1(3,\omega)=\tilde \lambda=-a^2-5a-6$
is the principal eigenvalue of \eqref{Rayleigh-type equation lambda k=2 negative mu=3} for $\omega\in(-18,-3]$.

(ii) For $\omega=-18$, \eqref{Rayleigh-type equation lambda k=2 negative mu=3} becomes
 \begin{align*}
 \left\{ \begin{array}{llll} ((1-s^2)\Phi')'-{4\over 1-s^2}\Phi=\tilde\lambda\Phi, \\
 \Phi(\pm1)=0. \end{array}\right.
 \end{align*}
 This is the general Legendre equation restricted in the space $X_{-18,3,e}$. The principal eigenvalue is clearly $\tilde \lambda_1(3,-18)=-6$.
\end{proof}
  For $\omega=-18$, we study the asymptotic behavior of $\tilde \lambda_1(\mu,-18)$ as $\mu\to 3^+$ or $\mu\to\infty$.

  \begin{Lemma}\label{asymptotic behavior of tilde lambda1 mu-18 lemma}
  For  $\omega=-18$, we have
  \begin{align}\label{asymptotic behavior of tilde lambda1 mu-18}
  \lim_{\mu\to3^+}\tilde \lambda_1(\mu,-18)=\tilde \lambda_1(3,-18)=-6,\quad\lim_{\mu\to\infty}\tilde \lambda_1(\mu,-18)=-18.
  \end{align}
  Moreover, $\partial_\mu\tilde\lambda_1(\mu,-18)<0$ for $\mu\in(3,\infty)$, and consequently,  $\tilde \lambda_1(\cdot,-18)$ is decreasing on $[3,\infty)$.
  \end{Lemma}
  \begin{proof}
  For $\Phi\in X_{\omega,\mu,e}$ and $\mu>3$, since
$
  \left|{-36+12\mu\over 15s^2-3+\mu}\Phi^2\right|\leq 12|\Phi|^2,
 $
  we have
  \begin{align*}
  \lim_{\mu\to3^+}\int_{-1}^1{-36+12\mu\over 15s^2-3+\mu}|\Phi|^2ds=0.
  \end{align*}
  Then by \eqref{def-tilde-lambda-n} we have
  \begin{align*}
  &\liminf_{\mu\to3^+}\tilde\lambda_1(\mu,-18)\\
  \geq&\liminf_{\mu\to3^+}\int_{-1}^1\left(-(1-s^2)|\tilde\Phi_{3,-18}'|^2-{4\over 1-s^2}|\tilde\Phi_{3,-18}|^2-{-36+12\mu\over 15s^2-3+\mu}|\tilde\Phi_{3,-18}|^2\right)ds\\
  =&\int_{-1}^1\left(-(1-s^2)|\tilde\Phi_{3,-18}'|^2-{4\over 1-s^2}|\tilde\Phi_{3,-18}|^2\right)ds=\tilde \lambda_1(3,-18)=-6,
  \end{align*}
  where we normalize  $\tilde\Phi_{3,-18}$ in \eqref{eigenfunction tilde lambda13-18} such that $\|\tilde\Phi_{3,-18}\|_{L^2(-1,1)}=1$.
 This implies that there exists $\delta>0$ such that
\begin{align*}
-{13\over2}<\tilde\lambda_1(\mu,\omega)<0,\quad \forall\; \mu\in(3,3+\delta).
\end{align*}
Since $\tilde\Phi_{\mu,-18}$ (which is a $L^2$ normalized eigenfunction of  $\tilde \lambda_1(\mu,-18)$)  solves
\eqref{Rayleigh-type equation lambda k=2 negative} with $\omega=-18$, $\tilde\lambda=\tilde\lambda_1(\mu,-18)$, we have
\begin{align*}
\int_{-1}^1\left((1-s^2)|\tilde\Phi_{\mu,-18}'|^2+{4\over 1-s^2}|\tilde\Phi_{\mu,-18}|^2+{-36+12\mu\over 15s^2-3+\mu}|\tilde\Phi_{\mu,-18}|^2\right) ds=-\tilde\lambda_1(\mu,-18).
\end{align*}
Noting that ${-36+12\mu\over 15s^2-3+\mu}>0$ for $\mu>3$, we have $
\int_{-1}^1\left((1-s^2)|\tilde\Phi_{\mu,-18}'|^2+{1\over 1-s^2}|\tilde\Phi_{\mu,-18}|^2\right) ds\leq C
$ uniformly for $\mu\in(3,3+\delta)$.
Then there exists $\tilde\Phi_0\in X_{3,-18,e}$ such that   $\tilde \Phi_{\mu,-18}\rightharpoonup\tilde \Phi_0$ in $X_{3,-18,e}$ and $\tilde \Phi_{\mu,-18}\to\tilde \Phi_0$ in $L^2(-1,1)$ by Lemma 2.4.6 in \cite{Skiba2017}. Thus, $\|\tilde\Phi_0\|_{L^2(-1,1)}=\lim_{\mu\to3^+}\|\tilde\Phi_{\mu,-18}\|_{L^2(-1,1)}=1$.
 Moreover, $\tilde \Phi_{\mu,-18}\to\tilde \Phi_0$
 in $C^0([a,b])$ by the compactness of $H^1([a,b])\hookrightarrow C^0([a,b])$ for any subinterval $[a,b]\subset(-1,1)$. Then
${-36+12\mu\over 15s^2-3+\mu}|\tilde\Phi_{\mu,-18}(s)|^2\to0$ for $[-1,1]\setminus\{0\}$ as $\mu\to3^+$.
By Fatou's Lemma, we have
\begin{align*}
0\leq \liminf_{\mu\to3^+}\int_{-1}^1{-36+12\mu\over 15s^2-3+\mu}|\tilde\Phi_{\mu,-18}(s)|^2 ds.
\end{align*}
Then
\begin{align*}
-6=\tilde\lambda_1(3,-18)\geq&\int_{-1}^1\left(-(1-s^2)|\tilde\Phi_{0}'|^2-{4\over 1-s^2}|\tilde \Phi_{0}|^2\right) ds\\
\geq&\limsup_{\mu\to-12^-}\int_{-1}^1\left(-(1-s^2)|\tilde\Phi_{\mu,-18}'|^2-{4\over 1-s^2}|\tilde\Phi_{\mu,-18}|^2-{-36+12\mu\over 15s^2-3+\mu}|\tilde\Phi_{\mu,-18}|^2\right) ds\\
=&\limsup_{\mu\to-12^-}\tilde\lambda_1(\mu,-18).
\end{align*}
This proves the first limit in \eqref{asymptotic behavior of tilde lambda1 mu-18}.

The proof of the second limit in \eqref{asymptotic behavior of tilde lambda1 mu-18} is  similar  to Lemma \ref{asymptotic behavior principal eigenvalue lim mu -infty}.

Direct computation implies that
\begin{align*}
\partial_\mu\tilde \lambda_1(\mu,-18)=-\int_{-1}^1{180s^2\over (15s^2-3+\mu)^2} |\tilde\Phi_{\mu,-18}|^2ds<0, \quad \mu\in(3,\infty),
\end{align*}
and thus, $\tilde \lambda_1(\cdot,-18)$ is decreasing on $[3,\infty)$.
  \end{proof}
For $\omega\in(-18,-3]$, we consider  the asymptotic behavior of $\tilde \lambda_1(\mu,\omega)$ as $\mu\to 3^+$ or $\mu\to\infty$.
  \begin{Lemma}\label{asymptotic behavior of tilde lambda1 mu-18-3 lemma}
  For $\omega\in(-18,-3]$, we have
  \begin{align}\label{upper bound for tilde lambda1}
  \tilde \lambda_1(\mu,\omega)<-6,\quad\mu\in(3,\infty),
  \end{align}
   and moreover,
  \begin{align}\label{limit tilde lambda 1 mu to 3 omega larger than -18}
  \lim_{\mu\to3^+}\tilde \lambda_1(\mu,\omega)=\tilde \lambda_1(3,\omega)\in[-20,-12),\quad \lim_{\mu\to\infty}\tilde \lambda_1(\mu,\omega)=-18.
  \end{align}
  \end{Lemma}
  \begin{proof}
  First, we prove \eqref{upper bound for tilde lambda1}. Let $\mu>3$ and $\omega_0=-6\mu$. Then $\omega_0<-18$ and \eqref{Rayleigh-type equation lambda k=2 negative} becomes
  \begin{align}\nonumber
 &((1-s^2)\Phi')'-{4\over 1-s^2}\Phi-{2\omega_0+12\mu\over 15s^2-3+\mu}\Phi\\\label{modified Rayleigh equation omega=-6 mu}
 =&((1-s^2)\Phi')'-{4\over 1-s^2}\Phi=\tilde \lambda\Phi, \quad
\Delta_2 \Phi\in L^2(-1,1).
 \end{align}
  Thus, the principal eigenvalue  of \eqref{modified Rayleigh equation omega=-6 mu} is  $\tilde \lambda_1(\mu,\omega_0)=-6$. For $\mu>3$, we have
  \begin{align*}
  \partial_\omega\tilde \lambda_1(\mu,\omega)=-\int_{-1}^1{2\over 15s^2-3+\mu}|\tilde\Phi_{\mu,\omega}|^2ds<0,
  \end{align*}
  and thus, $\tilde \lambda_1(\mu,\cdot)$ is decreasing on $\omega\in\mathbb{R}$, where $\tilde\Phi_{\mu,\omega}$ is a $L^2$ normalized eigenfunction of $\tilde \lambda_1(\mu,\omega)$.
  Since $\omega>-18>\omega_0$, we have $-6=\tilde \lambda_1(\mu,\omega_0)>\tilde \lambda_1(\mu,\omega)$. This proves  \eqref{upper bound for tilde lambda1}.

  Then we prove \eqref{limit tilde lambda 1 mu to 3 omega larger than -18}. Let us recall that $\tilde\Phi_{3,\omega}$ is given in \eqref{def-tilde-Phi-3-omega} and we normalize it such that $\|\tilde\Phi_{3,\omega}\|_{L^2(-1,1)}=1$. Then $0<\mu-3\leq 15s^2-3+\mu$ and
  \begin{align*}
  \left|{2\omega+12\mu\over 15s^2-3+\mu}|\tilde\Phi_{3,\omega}|^2-{2\omega+36\over 15s^2}|\tilde\Phi_{3,\omega}|^2\right|=\left|{(\mu-3)(12(15s^2-3)-2\omega)\over (15s^2-3+\mu)15s^2}|\tilde\Phi_{3,\omega}|^2\right|\leq {C\over s^2} |\tilde\Phi_{3,\omega}|^2,
  \end{align*}
  where $C$ is independent of $\mu\in(3,\infty)$.
Moreover, $\int_{-1}^1{1\over s^2} |\tilde\Phi_{3,\omega}|^2ds<\infty$ since ${1+\sqrt{{8\omega+159\over 15}}\over2}\in(1,2]$. Thus, $\lim_{\mu\to3^+}\int_{-1}^1{2\omega+12\mu\over 15s^2-3+\mu}|\tilde\Phi_{3,\omega}|^2ds=\int_{-1}^1{2\omega+36\over 15s^2}|\tilde\Phi_{3,\omega}|^2ds$ and
  \begin{align*}
  \tilde \lambda_1(\mu,\omega)\geq\int_{-1}^1\left(-(1-s^2)|\tilde\Phi_{3,\omega}'|^2-{4\over1-s^2}|\tilde\Phi_{3,\omega}|^2-{2\omega+12\mu\over 15s^2-3+\mu}|\tilde\Phi_{3,\omega}|^2\right)ds.
  \end{align*}
  Thus, $\liminf_{\mu\to3^+}\tilde \lambda_1(\mu,\omega)\geq\tilde \lambda_1(3,\omega)$.
  Choose $\delta>0$ such that $\tilde \lambda_1(3,\omega)-{1\over2}<\tilde\lambda_1(\mu,\omega)< -6$ for $\mu\in(3,3+\delta)$.
   Then
\begin{align*}
\int_{-1}^1\left((1-s^2)|\tilde\Phi_{\mu,\omega}'|^2+{4\over 1-s^2}|\tilde\Phi_{\mu,\omega}|^2+{2\omega+12\mu\over 15s^2-3+\mu}|\tilde\Phi_{\mu,\omega}|^2\right) ds=-\tilde\lambda_1(\mu,\omega)\leq C
\end{align*}
 uniformly for $\mu\in(3,3+\delta)$.
Then there exists $\tilde\Phi_{*,\omega}\in X_{3,\omega,e}$ such that   $\tilde \Phi_{\mu,\omega}\to\tilde \Phi_{*,\omega}$ in $L^2(-1,1)$ and $C^0([a,b])$ for any subinterval $[a,b]\subset(-1,1)$.
Thus, $\|\tilde\Phi_{*,\omega}\|_{L^2(-1,1)}=1$.
Then
${2\omega+12\mu\over 15s^2-3+\mu}|\tilde\Phi_{\mu,\omega}(s)|^2\to{2\omega+36\over 15s^2}|\tilde\Phi_{*,\omega}(s)|^2$ for $[-1,1]\setminus\{0\}$ as $\mu\to3^+$. By Fatou's Lemma, we have
\begin{align*}
\int_{-1}^1{2\omega+36\over 15s^2}|\tilde\Phi_{*,\omega}|^2 ds\leq \liminf_{\mu\to3^+}\int_{-1}^1{2\omega+12\mu\over 15s^2-3+\mu}|\tilde\Phi_{\mu,\omega}|^2 ds.
\end{align*}
Then
\begin{align*}
\tilde\lambda_1(3,\omega)\geq&\int_{-1}^1\left(-(1-s^2)|\tilde\Phi_{*,\omega}'|^2-{4\over 1-s^2}|\tilde\Phi_{*,\omega}|^2-{2\omega+36\over 15s^2}|\tilde\Phi_{*,\omega}|^2\right) ds\\
\geq&\limsup_{\mu\to3^+}\int_{-1}^1\left(-(1-s^2)|\tilde\Phi_{\mu,\omega}'|^2-{4\over 1-s^2}|\tilde\Phi_{\mu,\omega}|^2-{2\omega+12\mu\over 15s^2-3+\mu}|\tilde\Phi_{\mu,\omega}|^2\right) ds\\
=&\limsup_{\mu\to3^+}\tilde\lambda_1(\mu,\omega).
\end{align*}
The proof of the second limit in \eqref{limit tilde lambda 1 mu to 3 omega larger than -18} is similar to Lemma \ref{asymptotic behavior principal eigenvalue lim mu -infty}.
  \end{proof}
Now, we study the properties of the function
\begin{align*}
g(\omega)=\max_{\mu\in[3,183]}\tilde\lambda_1(\mu,\omega),\quad \omega\in[-18,-3].
\end{align*}
By Lemmas \ref{asymptotic behavior of tilde lambda1 mu-18 lemma}-\ref{asymptotic behavior of tilde lambda1 mu-18-3 lemma}, $g(\omega)$ is well-defined for $\omega\in[-18,-3]$. The definition of $g$ also relies on the following estimates of the principal eigenvalues.

\begin{Lemma}\label{principal eigenvalue mu larger than 183}
$({\rm i})$
 For $\mu>183$ and $\omega\in[-18,-3]$, we have
\begin{align*}
\tilde \lambda_1(\mu,\omega)<-17.
\end{align*}

$({\rm ii})$ For $\mu>33$ and $\omega\in[-18,-3]$, we have
\begin{align*}
\tilde \lambda_1(\mu,\omega)<-12.
\end{align*}

\end{Lemma}
\begin{proof}
 (i) There exists $\delta_\mu>0$ small enough such that
\begin{align*}
\left|{2\omega+12\mu\over 15s^2-3+\mu}-12\right|=\left|{2\omega-12(15s^2-3)\over 15s^2-3+\mu}\right|\leq{36+144\over \mu-3}<1-\delta_\mu, \quad s\in[-1,1]
\end{align*}
for $\mu>183$ and $\omega\in[-18,-3]$, which implies
\begin{align*}
-{2\omega+12\mu\over 15s^2-3+\mu}<-11-\delta_\mu,\quad s\in[-1,1].
\end{align*}
Thus,
\begin{align*}
\tilde \lambda_1 (\mu,\omega)=&
\sup_{\Phi \in X_{\omega,\mu,e},\|\Phi\|_{L^2(-1,1)}=1}\int_{-1}^1\left(-(1-s^2)|\Phi'|^2-{4\over 1-s^2}|\Phi|^2-{2\omega+12\mu\over 15s^2-3+\mu}|\Phi|^2\right) ds\\
\leq &\sup_{\Phi \in X_{\omega,\mu,e},\|\Phi\|_{L^2(-1,1)}=1}\int_{-1}^1\left(-(1-s^2)|\Phi'|^2-{4\over 1-s^2}|\Phi|^2\right) ds\\
&+\sup_{\Phi \in X_{\omega,\mu,e},\|\Phi\|_{L^2(-1,1)}=1}\int_{-1}^1\left(-{2\omega+12\mu\over 15s^2-3+\mu}|\Phi|^2 \right)ds\\
\leq& -6-11-\delta_\mu<-17
\end{align*}
for $\mu>183$ and $\omega\in[-18,-3]$.

 The proof of (ii) is similar  by observing that
there exists $\tilde \delta_\mu>0$ small enough such that
\begin{align*}
\left|{2\omega+12\mu\over 15s^2-3+\mu}-12\right|=\left|{2\omega-12(15s^2-3)\over 15s^2-3+\mu}\right|\leq{36+144\over \mu-3}<6-\tilde \delta_\mu
\end{align*}
for $\mu>33$ and $\omega\in[-18,-3]$.
\end{proof}
Similar to Lemma \ref{quadratic form-computation k=2}, we have the following result.
\begin{Lemma}\label{quadratic form-computation k=2 omega less than 3}
$(\rm{i})$ Let $\omega\leq -3$ and  $(c_{2},2,\omega,\Phi_{\mu_2,\omega})$ be
a  neutral mode, where $c_2\geq3+\omega$ and $\mu_2=-\omega+c_{2}$. Then $\tilde \lambda_{n_0}(\mu_2,\omega)=-12$
for  some $n_{0}\geq1$ and
\eqref{equality-L-form-derivative k=2} holds.

$(\rm{ii})$ Under the assumptions of $(\rm{i})$, if $c_2>3+\omega$ and $\|\Phi_{\mu_{2},\omega}\|_{L^2(-1,1)}=1$, then
\eqref{equality-L-form-derivative2 k=2} holds.
\end{Lemma}

The properties of the function $g$ are listed as follows.
\begin{Lemma}\label{properties of g}
$({\rm i})$ $g$ is decreasing  on $[-18,-3]$,

$({\rm ii})$ $g$ is continuous on $[-18,-3]$,

$({\rm iii})$  $g(-18)=-6$ and $g(-3)<-12$.\\
Consequently, there exists a unique $\omega_*\in(-18,-3)$ such that $\omega_*=g^{-1}(-12)$.
\end{Lemma}

\begin{proof}
(i) Let $-18\leq\omega_2<\omega_1\leq -3$. There exists $\mu_1\in[3,183]$ such that $g(\omega_1)=\tilde \lambda_1(\mu_1,\omega_1)$.
Note that if $\mu_1\in(3,183]$, then
  \begin{align*}
  \partial_\omega\tilde \lambda_1(\mu_1,\omega)=-\int_{-1}^1{2\over 15s^2-3+\mu_1}|\tilde\Phi_{\mu_1,\omega}|^2ds<0, \quad\omega\in[-18,-3],
  \end{align*}
  and thus, $\tilde \lambda_1(\mu_1,\cdot)$ is decreasing on $\omega\in [-18,-3]$, where $\tilde\Phi_{\mu_1,\omega}$ is a $L^2$ normalized eigenfunction of $\tilde \lambda_1(\mu_1,\omega)$. If $\mu_1=3$, then by Lemma \ref{principal eigenvalues mu=3}, $\tilde \lambda_1(\mu_1,\cdot)$ is also decreasing on $\omega\in [-18,-3]$.
Then $g(\omega_1)=\tilde \lambda_1(\mu_1,\omega_1)<\tilde \lambda_1(\mu_1,\omega_2)\leq g(\omega_2)$.

  (ii) By a similar argument to the first limit in \eqref{limit tilde lambda 1 mu to 3 omega larger than -18}, we have
  $ \lim_{(\mu,\omega)\to(3,\omega_0)}\tilde \lambda_1(\mu,\omega)=\tilde \lambda_1(3,\omega_0) $ for $\omega_0\in(-18,-3]$, where $\mu>3$.
  This implies that $\tilde \lambda_1$ is continuous on $(\mu,\omega)\in[3,183]\times (-18,-3]$.

  First, we prove that $g$ is continuous at $\omega_0\in(-18,-3]$. Choose $\delta_1>0$ such that $\omega_0\in(-18+\delta_1,-3]$. Then $\tilde \lambda_1$ is uniformly continuous on $(\mu,\omega)\in[3,183]\times [-18+\delta_1,-3]$. For any $\varepsilon>0$, there exists $\delta_2>0$ such that $|\tilde \lambda_1(\mu_1,\omega_1)-\tilde \lambda_1(\mu_2,\omega_2)|\leq \varepsilon$ for any $(\mu_1,\omega_1),(\mu_2,\omega_2)\in[3,183]\times [-18+\delta_1,-3]$ and $|\mu_1-\mu_2|+|\omega_1-\omega_2|<\delta_2$.
  Then for $\omega\in[-18+\delta_1,-3]$, $|\omega-\omega_0|<\delta_2$ and $\omega>\omega_0$, by the monotonicity of $g$ we have
  \begin{align*}
  0< g(\omega_0)-g(\omega)=\tilde \lambda_1(\mu_{\omega_0},\omega_0)-g(\omega)\leq \tilde \lambda_1(\mu_{\omega_0},\omega_0)-\tilde \lambda_1(\mu_{\omega_0},\omega)\leq \varepsilon,
  \end{align*}
  where $\mu_{\omega_0}\in[3,183]$, $g(\omega_0)= \tilde \lambda_1(\mu_{\omega_0},\omega_0)$, and we use $g(\omega)\geq\tilde \lambda_1(\mu_{\omega_0},\omega)$. For  $\omega\in[-18+\delta_1,-3]$, $|\omega-\omega_0|<\delta_2$ and $\omega<\omega_0$, we have
  \begin{align*}
  0< g(\omega)-g(\omega_0)=\tilde \lambda_1(\mu_\omega,\omega)-g(\omega_0)\leq \tilde \lambda_1(\mu_{\omega},\omega)-\tilde \lambda_1(\mu_{\omega},\omega_0)\leq \varepsilon,
  \end{align*}
   where $\mu_{\omega}\in[3,183]$, $g(\omega)= \tilde \lambda_1(\mu_{\omega},\omega)$, and we use the uniform continuity of $\tilde \lambda_1$ on $(\mu,\omega)\in[3,183]\times [-18+\delta_1,-3]$. This proves the continuity of $g$ on $(-18,-3]$.

   Then we prove that $g$ is continuous at $\omega_0=-18$. By Lemma \ref{asymptotic behavior of tilde lambda1 mu-18 lemma}, $\tilde \lambda_1(\cdot,-18)$ is decreasing and continuous on $\mu\in[3,\infty)$. For any $\varepsilon>0$, there exists $\mu_0\in(3,183)$ such that
   \begin{align*}
   0<-6-\tilde \lambda_1(\mu_0,-18)\leq {\varepsilon\over2}.
    \end{align*}
    Note that $\tilde \lambda_1(\mu_0,\cdot)$ is continuous on $\omega\in\mathbb{R}$. Then there exists $\omega_1\in(-18,-3)$ such that
    \begin{align*}|\tilde \lambda_1(\mu_0,-18)-\tilde \lambda_1(\mu_0,\omega_1)|\leq {\varepsilon\over 2}.\end{align*}
Thus,
\begin{align*}
   0<-6-\tilde \lambda_1(\mu_0,\omega_1)\leq {\varepsilon},
    \end{align*}
where we use  $\tilde \lambda_1(\mu_0,\omega_1)<-6$  by Lemma \ref{asymptotic behavior of tilde lambda1 mu-18-3 lemma}. Thus,
    \begin{align*}
   0<-6-g(\omega_1)\leq {\varepsilon},
    \end{align*}
where we use the monotonicity of $g$ and $g(-18)=-6$ (see (iii)). Let $\tilde\delta_1=\omega_1+18>0$. By the monotonicity of $g$ again, we have
\begin{align*}
0<-6-g(\omega)<-6-g(\omega_1)\leq {\varepsilon}
\end{align*}
for $-18<\omega<\omega_1\Longleftrightarrow0<\omega+18<\tilde\delta_1=\omega_1+18$. This proves the right continuity of $g$ at $-18$.

Finally, we prove (iii). By Lemma \ref{asymptotic behavior of tilde lambda1 mu-18 lemma}, we have $g(-18)=\tilde \lambda_1(3,-18)=-6$. To prove that $g(-3)<-12$, we infer from Lemma \ref{case-omega=12} (ii) that there are no neutral modes $(c,2,\omega,\Phi)$ with $c\in(0,\infty)$ and $\omega=-3$. Thus, $\tilde \lambda_1(\mu,-3)\neq-12$ for $\mu=3+c\in(3,\infty)$. Since $\lim_{\mu\to\infty}\tilde \lambda_1(\mu,-3)=-18$ by Lemma \ref{asymptotic behavior of tilde lambda1 mu-18-3 lemma}, we have
$\tilde \lambda_1(\mu,-3)<-12$ for $\mu\in(3,\infty)$. By Lemma \ref{principal eigenvalues mu=3}, we have $\tilde \lambda_1(3,-3)=-20$. Moreover,
$\lim_{\mu\to3^+}\tilde \lambda_1(\mu,-3)=\tilde \lambda_1(3,-3)=-20$ by Lemma \ref{asymptotic behavior of tilde lambda1 mu-18-3 lemma}. Thus, $g(-3)=\max_{\mu\in[3,183]}\tilde\lambda_1(\mu,-3)<-12$.
\end{proof}

Now, we are in a  position to prove Theorem \ref{k=2 negative half-line}.

\begin{proof}[Proof of Theorem \ref{k=2 negative half-line}]
Let $\omega\in\left(-18,g^{-1}(-12)\right]$. Then $g(\omega)\geq-12$ by Lemma \ref{properties of g} (i). There exists $\mu_\omega\in[3,183]$ such that $g(\omega)=\tilde \lambda_1(\mu_\omega,\omega)\geq-12$. Moreover, $\tilde \lambda_1(183,\omega)\leq -17$ by Lemma \ref{principal eigenvalue mu larger than 183} (i). Thus, there exists $ \mu_{2,\omega}\in[\mu_\omega,183]$ such that
\begin{align}\label{tilde-lambda-1-mu_2-omega}
\tilde \lambda_1(  \mu_{2,\omega},\omega)=-12 \quad \text{and}\quad \partial_\mu\tilde\lambda_1(  \mu_{2,\omega},\omega)\leq0.
\end{align}
For $\omega\in(-18,g^{-1}(-12))$, see the blue eigenvalue curves in Fig. \ref{fig-eigenvalue curves2}, where the brown bold points are $(\mu_{2,\omega},\tilde \lambda_1(  \mu_{2,\omega},\omega))$. For $\omega=g^{-1}(-12)$, see the red eigenvalue curve in Fig. \ref{fig-eigenvalue curves2}, where the red bold point is $(\mu_{2,\omega},\tilde \lambda_1(  \mu_{2,\omega},\omega))$ and $\mu_{2,\omega}=\mu_\omega$.
Let $c_{2,\omega}=\omega+ \mu_{2,\omega}$. Then $ c_{2,\omega}-c_\omega=\omega+ \mu_{2,\omega}-{5\over6}\omega={1\over 6}\omega+\mu_{2,\omega}>-3+3=0.$
Let $\Phi_{ \mu_{2,\omega},\omega,2}$ be a $L^2$ normalized eigenfunction of $\tilde \lambda_1( \mu_{2,\omega},\omega)$ and $\Upsilon_{ \mu_{2,\omega},\omega,2}=\Delta\Phi_{\mu_{2,\omega},\omega,2}$. By Lemma \ref{quadratic form-computation k=2 omega less than 3} (ii), we have
\begin{equation*}
\langle L_{2}\Upsilon_{ \mu_{2,\omega},\omega,2},\Upsilon_{ \mu_{2,\omega},\omega,2}\rangle
=( c_{2,\omega}-c_\omega)\partial_\mu \tilde \lambda_{1}( \mu_{2,\omega},\omega)\leq 0.
\end{equation*}
Thus, $k_{i,J_{\omega,2}L_2|_{X_e^2}}^{\leq0}=1$. By the index formula
\eqref{index formula 1o2}, the $3$-jet is
 spectrally stable    for $\omega\in\left(-18,g^{-1}(-12)\right]$.

 Let $\omega\in\left(g^{-1}(-12),-3\right]$. Then $g(\omega)<-12$ by Lemma \ref{properties of g} (i). Thus,  $\tilde \lambda_1(\mu,\omega)<-12$ for $\mu\in[3,183]$. This, along with Lemma \ref{principal eigenvalue mu larger than 183} (i), implies that
 \begin{align*}
 \tilde \lambda_1(\mu,\omega)<-12, \quad\forall\;\;\; \mu\in[3,\infty).
 \end{align*}
 See the green eigenvalue curves in Fig. \ref{fig-eigenvalue curves2}.

\begin{center}
 \begin{tikzpicture}[scale=0.58]
 \draw [->](-1, 0)--(15, 0)node[right]{$\mu$};
 \draw [->](0,-12)--(0,2) node[above]{$\tilde\lambda_1(\mu,\omega)$};
\draw [black] (1.58, -3.16).. controls (4, -6.8) and (12, -9)..(15,-9.1);
\draw [blue] (1.58, -6.7).. controls (1.6, -0.95) and (4, -9)..(15,-9.15);
\draw [blue] (1.58, -7.4).. controls (2.5, -1.5) and (4, -9.15)..(15,-9.21);
\draw [blue] (1.58, -7.95).. controls (2.9, -2) and (4, -9.24)..(15,-9.26);
\draw [red] (1.58, -8.4).. controls (2.8, -3.25) and (4.9, -9.3)..(15,-9.32);
\draw [green] (1.58, -9).. controls (3.2, -3.8) and (5.3, -9.37)..(15,-9.38);
\draw [green] (1.58, -9.5).. controls (3.6, -4.5) and (5.8, -9.43)..(15,-9.44);
\draw [green] (1.58, -10).. controls (4, -5) and (6.2, -9.46)..(15,-9.47);
\node[brown] (a) at (5.15,-6.32) {$\bullet$};
\node[brown] (a) at (4.84,-6.32) {$\bullet$};
\node[brown] (a) at (4.5,-6.32) {$\bullet$};
\node[red] (a) at (3.08,-6.32) {$\bullet$};
\node[blue] (a) at (2.1,-6.32) {$\bullet$};
\node[blue] (a) at (1.84,-6.32) {$\bullet$};
\node[blue] (a) at (1.58,-6.32) {$\bullet$};
\path   (0, -3.16)  edge [-,dotted](15,-3.16) [line width=0.4pt];
\path   (1.58, 2)  edge [-,dotted](1.58,-12) [line width=0.4pt];
\path   (0, -6.32)  edge [-,dotted](15,-6.32) [line width=0.4pt];
\path   (0, -9.48)  edge [-,dotted](15,-9.48) [line width=0.4pt];
         \node (a) at (-0.5,-3.16) {\tiny$-6$};
        \node (a) at (-0.5,-6.32) {\tiny$-12$};
        \node (a) at (-0.5,-9.48) {\tiny$-18$};
        \node (a) at (1.72,0.2) {\tiny$3$};
 \end{tikzpicture}
\end{center}\vspace{-0.5cm}
\begin{figure}[ht]
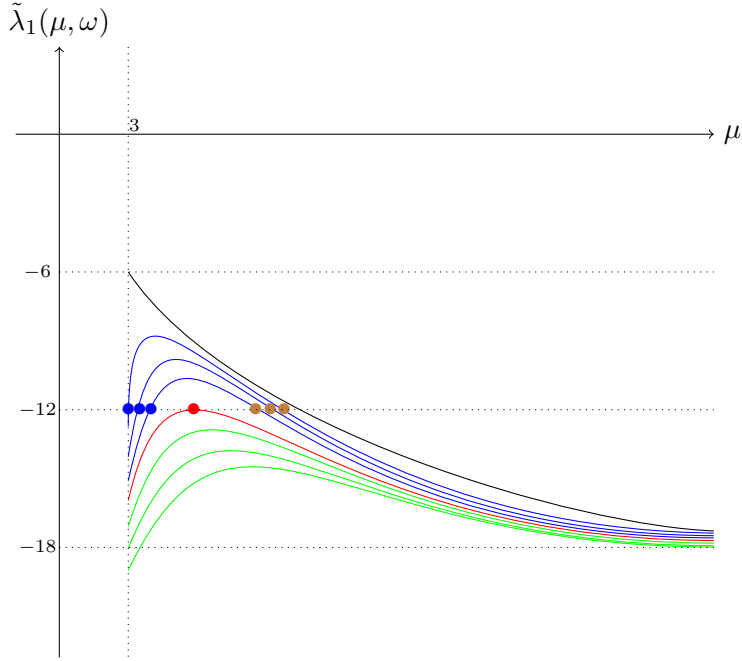

    \centering
 \caption{
The black eigenvalue curve is $\tilde\lambda_1(\cdot,\omega)$ with $\omega=-18$, the red eigenvalue curve is $\tilde\lambda_1(\cdot,\omega)$ with $\omega=g^{-1}(-12)$, the blue eigenvalue curves are $\tilde\lambda_1(\cdot,\omega)$ with $\omega\in(-18,g^{-1}(-12))$, and the green eigenvalue curves are $\tilde\lambda_1(\cdot,\omega)$ with $\omega\in(g^{-1}(-12),-3]$. The brown bold points and the blue bold points are $(\mu_{2,\omega},\tilde\lambda_1(\mu_{2,\omega},\omega))$ and $(\mu_{3,\omega},\tilde\lambda_1(\mu_{3,\omega},\omega))$ for different $\omega\in(-18,g^{-1}(-12))$ in Corollary \ref{uniqueness4}. The red bold point is $(\mu_{g^{-1}(-12)},\tilde\lambda_1(  \mu_{g^{-1}(-12)},g^{-1}(-12)))$.
}
\label{fig-eigenvalue curves2}
\end{figure}
\vspace{-0.2cm}
 Then there exist no neutral modes $(c,2,\omega,\Phi)$ with $c\geq3+\omega$. By  Lemma \ref{tilde Psi0 does not change sign} (1) for $\omega\in\left(g^{-1}(-12),-3\right)$ and Lemma \ref{case-omega=12} (ii) for $\omega=-3$, we have $c\in\text{Ran} (-\tilde \Psi_{\omega}')^\circ$ for any neutral mode $(c,2,\omega,\Phi)$. Then   $c=c_\omega$ by Theorem \ref{Psi prime change sign c}. Thus, $k_{i,J_{\omega,2}L_2|_{X_e^2}}^{\leq0}=0$. By Lemma \ref{kernel JL}, we have $k_{0,J_{\omega,2}L_2|_{X_e^2}}^{\leq0}=0$.
By the index formula \eqref{index formula 1o2}, we have
$ k_{c,J_{\omega,2}L_2|_{X_e^2}}+ k_{r,J_{\omega,2}L_2|_{X_e^2}}=1$. This proves  linear instability of the $3$-jet for $\omega\in\left(g^{-1}(-12),-3\right]$.
\end{proof}
\begin{Corollary}\label{uniqueness4}
Let $\omega\in(-18,$ $g^{-1}(-12))$ and $k=2$. Then there exist
exactly two $\mu_{2,\omega},\mu_{3,\omega}\in(3,\infty)$ such that $\mu_{2,\omega}\neq\mu_{3,\omega}$ and  $(c_{j,\omega},2,\omega,\Phi_{\mu_{j,\omega},\omega,2})$ is a  neutral mode, where $c_{j,\omega}=\mu_{j,\omega}+\omega$,  $j=2,3$.
Moreover,
\begin{align}\label{L-signature23}
\langle L_{2}\Upsilon_{ \mu_{2,\omega},\omega,2},\Upsilon_{ \mu_{2,\omega},\omega,2}\rangle\leq 0,\quad\langle L_{2}\Upsilon_{ \mu_{3,\omega},\omega,2},\Upsilon_{ \mu_{3,\omega},\omega,2}\rangle> 0
\end{align}
where $\Upsilon_{ \mu_{j,\omega},\omega,2}=\Delta\Phi_{\mu_{j,\omega},\omega,2}$.
\end{Corollary}

\begin{proof} The existence of $\mu_{2,\omega}$ is proved in \eqref{tilde-lambda-1-mu_2-omega}. By Lemma \ref{properties of g} (i), we have
$g(\omega)>-12$ and there exists $\mu_\omega\in[3,183]$ such that $g(\omega)=\tilde \lambda_1(\mu_\omega,\omega)>-12$.
By Lemma \ref{asymptotic behavior of tilde lambda1 mu-18-3 lemma}, $
  \lim_{\mu\to3^+}\tilde \lambda_1(\mu,\omega)$ $=\tilde \lambda_1(3,\omega)<-12$. Thus, there exists $ \mu_{3,\omega}\in(3,\mu_\omega)$ such that
$
\tilde \lambda_1(  \mu_{3,\omega},\omega)=-12$.  This gives a neutral mode $(c_{3,\omega},2,\omega,\Phi_{\mu_{3,\omega},\omega,2})$. Note that $\langle L_{2}\Upsilon_{ \mu_{3,\omega},\omega,2},\Upsilon_{ \mu_{3,\omega},\omega,2}\rangle> 0$ and $\partial_\mu\tilde\lambda_1(  \mu_{3,\omega},\omega)>0$  due to the index formula \eqref{index formula 1o2}.
The proof of no other neutral modes with $\mu\in(3,\infty)$ is similar to that of  the uniqueness in Corollary \ref{uniqueness1}.
\end{proof}

Finally, we emphasize the geometric curvature effects on the stability of zonal flows on the sphere, and this leads to some differences   with the flat geometry (the $\beta$-plane approximations). Recall that
the $\beta$-plane equation in the vorticity form is
\begin{align*}
\partial_{t}\gamma+(-\partial_y\psi\partial_x+\partial_x\psi\partial_y)\gamma+\beta \partial_x\psi%
=0
\end{align*}
 in a channel $\bbT_{2\pi}\times[-1,1]$
with non-permeable  boundary condition on $y=\pm1$, where $\beta\in\mathbb{R}$, $\psi$ is the stream function and $\gamma=\Delta\psi$ is  the vorticity.

\begin{Remark}\label{geometric curvature effects and differences with the flat geometry}
Let us first present two examples.
\begin{itemize}
 \item
{\it Spherical geometry:} Consider the fixed sphere (i.e. $\omega=0$) and the zonal flow with stream function $\Psi_*(s)=5s^3-3s-a s$, where $a\in(-18,-3)\cup({99\over 2},72)$. Let  $\mathcal{L}_{*,\omega}$ be the linearized vorticity operator around the zonal flow and  $\mathcal{L}_{*,\omega,k}$ be  the projection of $\mathcal{L}_{*,\omega}$ on the $k$'th Fourier mode.
By Corollaries \ref{uniqueness1} and  \ref{uniqueness2}, $\mathcal{L}_{*,0,1}$ has a purely imaginary  isolated  eigenvalue $-ic_{1,a}=-i(\mu_{1,a}+a)\notin  Ran(i\Psi_*')=\sigma_{e}(\mathcal{L}_{*,0,1})$.
In other words, non-resonant neutral modes  do  exist even if the sphere is fixed.
Moreover, by Theorem \ref{k=2 negative half-line}, the stability boundary at $a=g^{-1}(-12)$ is composed of a  non-resonant neutral mode,
see Fig. \ref{fig-eigenvalue curves2}.
 \item
{\it Flat geometry:} Consider the non-rotational case (i.e. $\beta=0$) and any shear flow with stream function $\psi_*\in C^3$.
 Let  $\mathbb{L}_{*,\beta}$ be the linearized vorticity operator around the zonal flow and  $\mathbb{L}_{*,\beta,k}$ be  the projection of $\mathbb{L}_{*,\beta}$ on the $k$'th Fourier mode.
It is known that $\mathbb{L}_{*,0,k}$ has only essential spectra $\sigma_{e}(\mathbb{L}_{*,0,k})=Ran(ik\psi_*')$ and no isolated eigenvalues in the imaginary axis for   $k\neq0$ (see, for example, \cite{Howard1961}). This means that
 non-resonant neutral modes do not exist in the non-rotational case $\beta=0$. 
\end{itemize}
In  the flat geometry, only in the rotational case
($\beta \neq 0$), non-resonant neutral modes exist and do serve as some parts of the stability boundary  for a Kolmogorov flow \cite{lin-yang-zhu20}. The rotation effects trigger the emergence of non-resonant neutral modes and their role as the stability boundary in the planar $\beta$-plane model.
In the spherical geometry,
however,
even on a fixed sphere, non-resonant neutral modes do
 appear and  act as the stability boundary in the above example.
The geometric curvature effects, rather than the rotational effects,
are responsible for the appearance of non-resonant neutral modes and their role as stability boundary in the spherical model.
\end{Remark}

\subsection{Consistency  with previous numerical calculations}
By Theorem \ref{negative half-line critical rotation rate}, the critical rotation rate is $\omega_{cr}^-=g^{-1}(-12)$  for the negative half-line.
The function $g$ is defined in \eqref{def-g}, which is based on the principal eigenvalues of a modified Rayleigh equation
\eqref{Rayleigh-type equation thm 1.2}
in the space $X_{\omega,\mu,e}$.
In this subsection, we use Matlab to calculate the principal eigenvalues of  \eqref{Rayleigh-type equation thm 1.2} to find the value of $g^{-1}(-12)$.
 Our calculation reveals that $g^{-1}(-12)\approx-16.0735$, which is very close to the numerical critical rotation rate $-16.0732$ in \cite{Sasaki-Takehiro-Yamada2012}. This shows that
  our analytical  critical rotation rate $g^{-1}(-12)$ is consistent with the previous numerical results in \cite{Sasaki-Takehiro-Yamada2012}.

 Indeed, we compute $\tilde \lambda_1(\mu,\omega_0)$ for $\omega_0=-16.07354$. By Lemma \ref{principal eigenvalue mu larger than 183} (ii), we have $\tilde \lambda_1(\mu,\omega_0)<-12$ if $\mu>33$. Thus, we only need to consider $\tilde \lambda_1(\mu,\omega_0)$ with $\mu\in[3,33]$. The tool we use to compute the principal  eigenvalues $\tilde \lambda_1(\mu,\omega_0)$ of \eqref{Rayleigh-type equation thm 1.2} numerically is  Matslise (a Matlab package). The graph of $\tilde \lambda_1(\cdot,\omega_0)$ as a function of $\mu\in[3,33]$ is given in Fig. \ref{Numerical computations of the principal eigenvalues} $(a)$ and the $\mu$-value such that the maximum of $\tilde \lambda_1(\cdot,\omega_0)$ is attained is between  $3.2$ and $3.3$.
  So we focus on the interval $\mu\in[3.1,3.5]$ and calculate the eigenvalues $\tilde \lambda_1(\mu,\omega_0)$ with better accuracy, see Fig. \ref{Numerical computations of the principal eigenvalues} $(b)$. The error is almost within $10^{-7}$.
  
  \begin{figure}[ht]
      \centering
	\includegraphics[scale = 0.25]{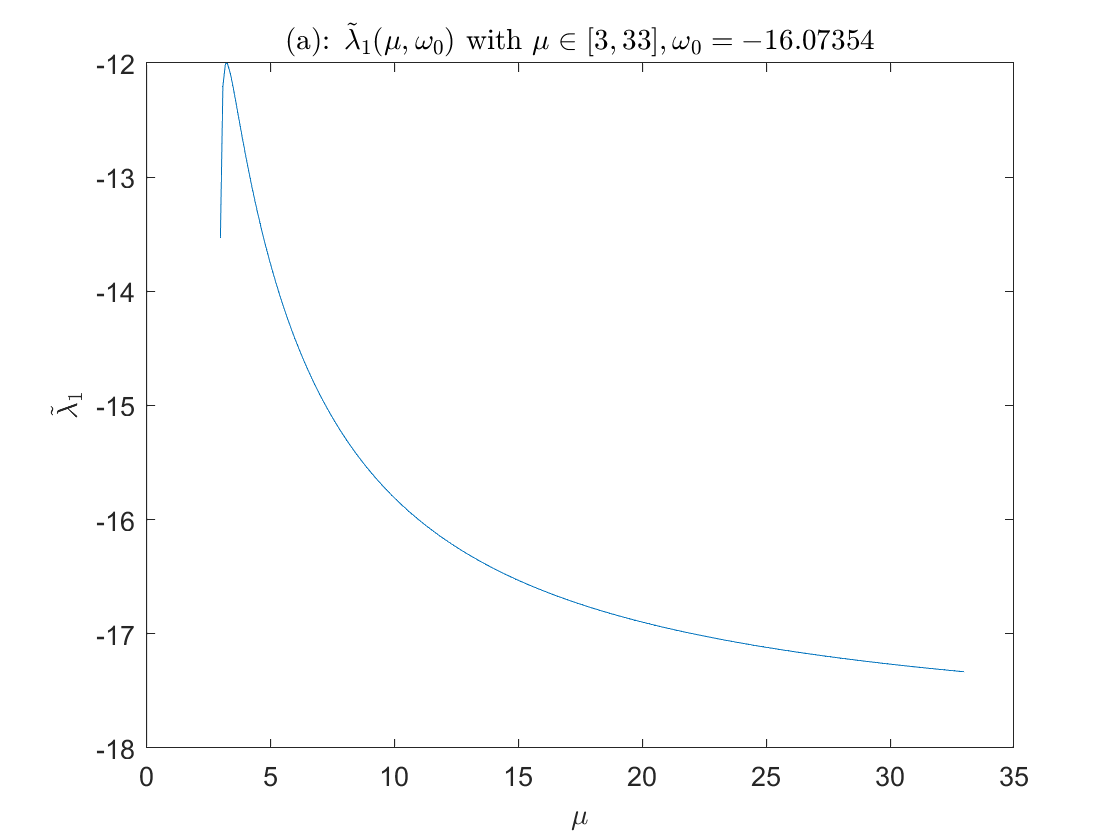}
        \includegraphics[scale = 0.25]{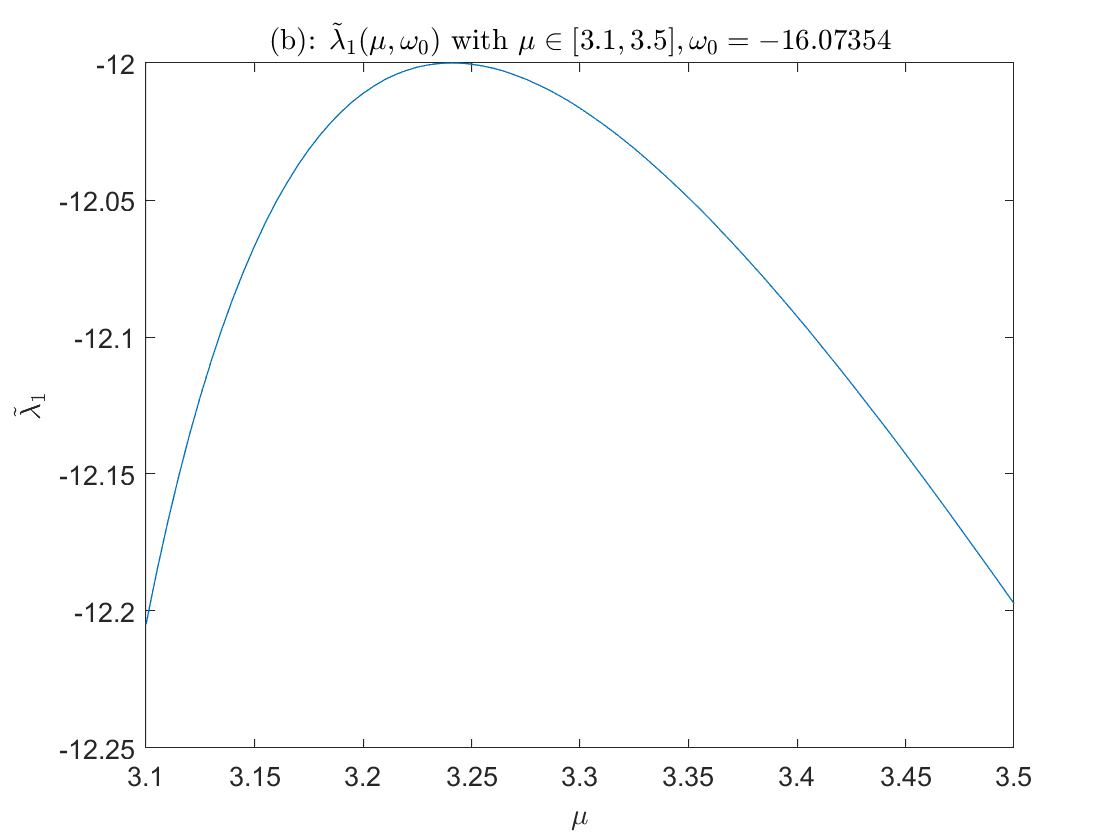}
        \caption{Numerical computations of the principal eigenvalues $\tilde \lambda_1$.}
	\label{Numerical computations of the principal eigenvalues}
     \end{figure}

 We list the values of $\tilde \lambda_1(\mu,\omega_0)$ with $\mu\in[3.233,3.248]$ in Table \ref{tbl:Tbl1}.
 
 \begin{table}[h!]
\centering

\begin{tabular}{| c | c|c|c|}
\hline
  $\mu$ &   $\tilde \lambda_1(\mu,\omega_0)$ & $\mu$ &   $\tilde \lambda_1(\mu,\omega_0)$ \\
\hline
3.233&-12.00038017&3.241&-12.00000589\\
3.234  &-12.00029226&3.242&-12.00001148\\
3.235&-12.00021565&3.243&-12.00002835\\
3.236&-12.00015114&3.244&-12.00005638\\
3.237&-12.00009860&3.245&-12.00009546\\
3.238&-12.00005793&3.246&-12.00014551\\
3.239&-12.00002900&3.247&-12.00020638\\
3.24&-12.00001169&3.248&-12.00027800\\
\hline
\end{tabular}
\vspace{0.2cm}
\caption{$\tilde \lambda_1(\mu,\omega_0)$ with $\mu\in[3.233,3.248]$, $\omega_0=-16.07354$.}
\label{tbl:Tbl1}
\end{table}

On the one hand, by Table \ref{tbl:Tbl1} we numerically conclude that
 \begin{align*}
g(-16.07354)=\max_{\mu\in[3,183]}\tilde\lambda_1(\mu,-16.07354)\approx\tilde\lambda_1(3.241,-16.07354)\approx-12.00000589<-12.
  \end{align*}
  On the other hand,  we have $\tilde\lambda_1(3.241,-16.07355)\approx-11.99999399$ by numerical computation, and thus, \begin{align*}g(-16.07355)=\max_{\mu\in[3,183]}\tilde\lambda_1(\mu,-16.07355)\geq\tilde\lambda_1(3.241,-16.07355)\approx-11.99999399>-12.
  \end{align*}
Since $g$ is decreasing and continuous on $\omega\in[-18,-3]$ by Lemma \ref{properties of g},  the critical rotation rate $g^{-1}(-12)$ is between $-16.07354$ and $-16.07355$. This shows that $g^{-1}(-12)\approx-16.0735$ if we  keep four digits after the decimal point.

Next, we give an application of the above eigenvalue computations.
By Remark \ref{neutral-imaginary}, existence of the  two neutral modes  $(\mu_{j,\omega}+\omega,2,\omega,\Phi_{\mu_{j,\omega},\omega,2}), j=2,3,$ in Corollary \ref{uniqueness4}
implies that there  are two purely imaginary isolated eigenvalues $-2i(\mu_{j,\omega}+{1\over 6}\omega)$, $j=2,3$  of $J_{\omega,2}L_2$ with  eigenfunctions $\Upsilon_{ \mu_{j,\omega},\omega,2}=\Delta_2\Phi_{\mu_{j,\omega},\omega,2}$. Correspondingly, $-2i\mu_{j,\omega}$, $j=2,3$  are two purely imaginary isolated eigenvalues of
$\mathcal{L}_{\omega,2}$, where $\mathcal{L}_{\omega,2}$ 
is the projection of  $\mathcal{L}_{\omega}$ on the second Fourier mode. 
In the following remark, we study the Krein signatures   of the two eigenvalues.

\begin{Remark}\label{krein signature-rem}
For $\omega$ smaller than but close enough to $g^{-1}(-12)$, we provide a computer-assistant proof for the opposite  Krein signatures   of the two purely imaginary isolated eigenvalues $-2i\mu_{j,\omega}$, $j=2,3$ of $\mathcal{L}_{\omega,2}$. Indeed, based on \eqref{L-signature23}, it suffices to show that
\begin{align}\label{L2neq0}
\langle L_{2}\Upsilon_{ \mu_{2,\omega},\omega,2},\Upsilon_{ \mu_{2,\omega},\omega,2}\rangle\neq 0
\end{align}
for any $\omega$ smaller than but close enough to $g^{-1}(-12)$. Assume that \eqref{L2neq0} holds true. By Lemma 3.4 in \cite{lin2022instability}, we have $\Upsilon_{ \mu_{j,\omega},\omega,2}\notin Ran(\mathcal{L}_{\omega,2}+2i\mu_{j,\omega})$ and thus $-2i\mu_{j,\omega}$, $j=2,3$ are simple eigenvalues of $\mathcal{L}_{\omega,2}$. For the  simple non-zero eigenvalue $-2i\mu_{j,\omega}$, the Krein  signature is defined by the sign of $\langle L_{2}\Upsilon_{ \mu_{j,\omega},\omega,2},\Upsilon_{ \mu_{j,\omega},\omega,2}\rangle$ for every $j=2,3$ (see \cite{MacKay1987,Kapitula-Kevrekidis-Sandstede2004}). Then \eqref{L-signature23}--\eqref{L2neq0} imply that the two purely imaginary isolated eigenvalues $-2i\mu_{j,\omega}$, $j=2,3$ of $\mathcal{L}_{\omega,2}$ have opposite  Krein signatures.

Now, we prove \eqref{L2neq0}. Suppose that there exist a sequence $\{\omega_n\}_{n=1}^\infty$ with $\omega_n\to (g^{-1}(-12))^-$ such that
$\langle L_{2}\Upsilon_{ \mu_{2,\omega_n},\omega_n,2},\Upsilon_{ \mu_{2,\omega_n},\omega_n,2}\rangle=0$ for $n\geq1$. By Lemma \ref{quadratic form-computation k=2 omega less than 3} (ii) and $ c_{2,\omega_n}-c_{\omega_n}=\mu_{2,\omega_n}+\omega_n-{5\over6}\omega_n=\mu_{2,\omega_n}+{1\over6}\omega_n\neq0$, we have $\partial_\mu\tilde\lambda_1(  \mu_{2,\omega_n},\omega_n)=0$. By \eqref{limit tilde lambda 1 mu to 3 omega larger than -18} and Corollary \ref{uniqueness4}, we have $\tilde\lambda_1(  \mu,\omega_n)\leq-12$ for $\mu\in[3,\mu_{3,\omega_n}]\cup[\mu_{2,\omega_n},\infty)$. By Lemma \ref{properties of g} (i), we have $g(\omega_n)>-12$ and thus there exists $\mu_{\omega_n}\in(\mu_{3,\omega_n},\mu_{2,\omega_n})$ such that $\tilde\lambda_1(\mu_{\omega_n},\omega_n)=g(\omega_n)$ and $\partial_\mu\tilde\lambda_1(\mu_{\omega_n},\omega_n)=0$. Then
\begin{align*}
&\partial_\mu\tilde\lambda_1(  \mu_{2,\omega_n},\omega_n)=\partial_\mu\tilde\lambda_1(\mu_{\omega_n},\omega_n)=0\\
\Longrightarrow&\exists\;\;\mu_{4,\omega_n}\in(\mu_{\omega_n},\mu_{2,\omega_n})
\;\;\text{such that}\;\;\partial_\mu^2\tilde\lambda_1(  \mu_{4,\omega_n},\omega_n)=0
\end{align*}
since $\tilde \lambda_1\in C^2((3,\infty)\times\mathbb{R})$.
For $\omega=g^{-1}(-12)$, the index formula \eqref{index formula 1o2} ensures that there exists a unique $\mu_{g^{-1}(-12)}\in(3,\infty)$ such that $\tilde\lambda_1(  \mu_{g^{-1}(-12)},g^{-1}(-12))=-12$. Moreover, $\mu_{3,\omega_n},\mu_{2,\omega_n}\to\mu_{g^{-1}(-12)}$ and thus $\mu_{4,\omega_n}\to\mu_{g^{-1}(-12)}$  as $n\to\infty$. Therefore,
\begin{align}\label{second mu derivative of tilde lambda 1}\partial_\mu^2\tilde\lambda_1(  \mu_{g^{-1}(-12)},g^{-1}(-12))=\lim_{n\to\infty}\partial_\mu^2\tilde\lambda_1(  \mu_{4,\omega_n},\omega_n)=0.\end{align}

On the other hand, with the numeric data in Table \ref{tbl:Tbl1}, we compute the second  $\mu$-derivative of $\tilde\lambda_1$ at $(  \mu_{g^{-1}(-12)},g^{-1}(-12))$ by the finite difference approximations.
Noting that $g^{-1}(-12)$ is between $-16.07354$ and $-16.07355$, we take $g^{-1}(-12)\approx-16.07354$, and $ \mu_{g^{-1}(-12)}\approx3.241$ by Table \ref{tbl:Tbl1}. The finite difference approximations of $\partial_\mu^2\tilde\lambda_1(  \mu_{g^{-1}(-12)},g^{-1}(-12))$ is then given by
\begin{align}\nonumber
&a(\mu)\triangleq\delta_\mu^2[\tilde \lambda_1](3.241,-16.07354)\\\label{finite difference approximations of second derivative}
=&{\tilde \lambda_1(3.241+\mu,-16.07354)+\tilde \lambda_1(3.241-\mu,-16.07354)-2\tilde \lambda_1(3.241,-16.07354)\over \mu^2}.
\end{align}
By the formula \eqref{finite difference approximations of second derivative} and the data in Table \ref{tbl:Tbl1}, we compute the finite difference approximations $a(\mu)$ and list them in Table \ref{tbl:Tbl2}.
 \begin{table}[h!]
\centering
\begin{tabular}{| c | c|c|c|c|c|c|c|c|}
\hline
  $\mu$ &   $0.007$ &$0.006$ &$0.005$ &$0.004$ &$0.003$ &$0.002$ &$0.001$\\
\hline
$a(\mu)$&   $-11.3976$ &$-11.3958$ &$-11.3948$ &$-11.3925$ &$-11.3922$ &$-11.3925$ &$-11.39$\\
\hline
\end{tabular}
\vspace{0.2cm}
\caption{Finite difference approximations of $\partial_\mu^2\tilde\lambda_1(  \mu_{g^{-1}(-12)},g^{-1}(-12))$.}
\label{tbl:Tbl2}
\end{table}
Thus, $\partial_\mu^2\tilde\lambda_1(  \mu_{g^{-1}(-12)},g^{-1}(-12))\approx-11.39$, which contradicts \eqref{second mu derivative of tilde lambda 1}. This proves \eqref{L2neq0}.
\end{Remark}

\section{Invariant subspace decomposition and exponential trichotomy}

It is natural to ask what exact role of $E_1$ plays in the spectral analysis of $J_{\omega}L$. This is not straightforward as one can verify that $E_1$ is not an invariant subspace for $J_{\omega}L$ in the case of $\omega\neq0$. In fact, we prove that
a combination of $Y_1^{\pm1}$ with $Y_3^{\pm1}$  provide two purely imaginary eigenvalues $\pm i\omega$ of $\mathcal{L}_\omega$.
This also leads to an invariant subspace decomposition  for the operator $J_{\omega}L$.

Instead of restricting into the space $X$ (defined in \eqref{def-space-X}), we consider the linearized operator $J_{\omega}L$ (defined in \eqref{def-ham-JL}) in the whole space  $L_0^2(\mathbb{S}^2)$, which consists of functions in $L^2(\mathbb{S}^2)$ with zero mean.
For a subspace $Z\subset  L_0^2(\mathbb{S}^2)$, we denote $Z_{-}=\{\Upsilon\in Z:\langle L\Upsilon,\Upsilon \rangle<0\}$.
Direct computation gives
\begin{align*}
L_0^2(\mathbb{S}^2)_{-}&=\text{span}\{Y_1^0,  Y_1^{\pm1}, Y_2^0,Y_2^{\pm1},Y_2^{\pm2}\},
\end{align*}
and
\begin{align*}
 X_-&=
\text{span}\{Y_2^0,Y_2^{\pm1},Y_2^{\pm2}\}.
\end{align*}
Thus,
\begin{align*}
n^-(L)=8\quad\text{and}\quad n^-(L|_{X})=5.
\end{align*}
Moreover,
\begin{align*}
\ker(L)=\ker(L|_{X})&=\text{span}\{Y_3^0, Y_3^{\pm1},Y_3^{\pm2},Y_3^{\pm3}\}.
\end{align*}
Thus,
\begin{align*}
\dim\ker(L)=\dim\ker(L|_{X})=7.
\end{align*}
Recall that $E_1=\text{span}\{Y_1^0, Y_1^{\pm1}\}$, $L_0^2(\mathbb{S}^2)=E_1\oplus X$ and $X$ is invariant for the linearized operator $J_{\omega}L$ by \eqref{invariant-1}-\eqref{invariant-2}.
However, $E_1$ is not invariant for the operator $J_{\omega}L$ since
\begin{align}\label{JLE1nonzonal}
J_{\omega}L (Y_1^{\pm1})=&-( \Upsilon_0'+2\omega)\partial_\varphi\left({1\over 12}+ \Delta^{-1}\right)(Y_1^{\pm1})\\\nonumber
=&-( -12(15s^2-3)+2\omega)(\pm i)\left({1\over12}-{1\over2}\right)(Y_1^{\pm1})\\\nonumber
=&-{5i\over24}\sqrt{3\over 2\pi}e^{\pm i\varphi}( -180s^2+36+2\omega)(1-s^2)^{1\over2}\notin E_1
\end{align}
for any $\omega\in \mathbb{R}$, where we recall that $Y_1^{\pm1}(\varphi,s)={\mp}{1\over2}\sqrt{3\over2\pi}e^{\pm i\varphi}\sqrt{1-s^2}$. To obtain a suitable invariant subspace decomposition of  $L_0^2(\mathbb{S}^2)$ for the operator $J_{\omega}L$, we study what kind of role  $Y_1^{\pm1}$ play. 
In the next lemma, we will see that a combination of $Y_1^{\pm1}$ with $Y_3^{\pm1}$ will provide two exact eigenvalues of $J_{\omega}L$  on the imaginary axis.

\begin{Lemma}\label{E1JL exact eigenvalues}
Consider $\omega\neq0$ and $\omega=0$ separately.
\begin{itemize}
 \item[(i)]
If $\omega\neq0$, then $\pm{5\over6}\omega i$ are a pair of eigenvalues of $J_{\omega}L$ with
corresponding eigenfunctions $Y_1^{\pm1}-{72\over\omega}\sqrt{1\over14}Y_3^{\pm1}$.
Consequently, the pair of eigenvalues $\pm{5\over6}\omega i$ merges to zero as $\omega\to0^+$. Moreover,  an invariant subspace decomposition for $J_{\omega}L$ is
\begin{align*}
L_0^2(\mathbb{S}^2)=\text{span}\left\{Y_1^0,Y_1^{\pm1}-{72\over\omega}\sqrt{1\over14}Y_3^{\pm1}\right\}\dot{+}X.
\end{align*}
\item [(ii)]
If $\omega=0$, then
$Y_1^{\pm1}$ are in the  generalized kernel of $J_{\omega}L$.
Moreover, an invariant subspace decomposition for $J_{\omega}L$ is
 \begin{align*}
 L_0^2(\mathbb{S}^2)=E_1\dot{+}X,
 \end{align*}
  where $X$ is defined in \eqref{def-space-X}.
\end{itemize}
\end{Lemma}

\begin{Remark}\label{E1JL exact eigenvalues-rem}
Let $\omega\neq0$. By a direct computation, the linearized equation of ${\rm(\mathcal{E}_\omega)}$ around $\Upsilon_0(s)=\Delta \Psi_0(s)$ in the original frame  $(\varphi,s)$ is
$
\mathcal{L}_\omega=J_{\omega}L+{1\over 6}\omega\partial_\varphi.
$
By \eqref{eigenvalueJL5over6omegai}, we have
\begin{align*}
\mathcal{L}_\omega\left(Y_1^{\pm1}-{72\over\omega}\sqrt{1\over14}Y_3^{\pm1}\right)=&\left(J_{\omega}L\pm{1\over 6}\omega i\right)\left(Y_1^{\pm1}-{72\over\omega}\sqrt{1\over14}Y_3^{\pm1}\right)\\
=&\pm\omega i\left(Y_1^{\pm1}-{72\over\omega}\sqrt{1\over14}Y_3^{\pm1}\right).
\end{align*}
That is, $\pm\omega i$ are a pair of eigenvalues of $\mathcal{L}_\omega$, which also  merges to zero as $\omega\to0^+$.
\end{Remark}
\begin{proof}
$({\rm i})$
It suffices to  prove that
\begin{align}\label{JLY1pm1}
J_{\omega}L(Y_1^{\pm1})=\mp60\sqrt{1\over 14}iY_3^{\pm1}\pm{5\omega\over 6}i Y_1^{\pm1}
\end{align}
and
\begin{align}\label{eigenvalueJL5over6omegai}
J_{\omega}L\left(Y_1^{\pm1}-{72\over\omega}\sqrt{1\over14}Y_3^{\pm1}\right)=\pm{5\over6}\omega i\left(Y_1^{\pm1}-{72\over\omega}\sqrt{1\over14}Y_3^{\pm1}\right).
\end{align}
That is, $\pm{5\over6}\omega i$ are a pair of eigenvalues of $J_{\omega}L$.

Recall that $Y_3^{\pm1}(\varphi,s)=\pm{1\over8}\sqrt{21\over \pi}e^{\pm i\varphi}\sqrt{1-s^2}(1-5s^2)$. By \eqref{JLE1nonzonal}, we have
\begin{align*}
J_{\omega}L (Y_1^{\pm1})=&-{15\over 2}\sqrt{3\over 2\pi}ie^{\pm i\varphi}( 1-5s^2)(1-s^2)^{1\over2}-{5\over 12}\sqrt{3\over2\pi}\omega ie^{\pm i\varphi}(1-s^2)^{1\over2}\\
=&\mp60\sqrt{1\over 14}iY_3^{\pm1}\pm{5\omega\over 6}i Y_1^{\pm1}.
\end{align*}
Since $Y_3^{\pm1}\in\ker(L)$, we have
\begin{align*}
J_{\omega}L\left(Y_1^{\pm1}-{72\over\omega}\sqrt{1\over14}Y_3^{\pm1}\right)=&J_{\omega}L\left(Y_1^{\pm1}\right)=\pm{5\over6}\omega i\left(Y_1^{\pm1}-{72\over\omega}\sqrt{1\over14}Y_3^{\pm1}\right).
\end{align*}

$({\rm ii})$
In fact, we have
\begin{align*}
J_{\omega}L(Y_1^{\pm1})=\mp60\sqrt{1\over 14}iY_3^{\pm1}\quad {and}\quad (J_{\omega}L)^2(Y_1^{\pm1})=0,
\end{align*}
where the first equality is similar to  \eqref{JLY1pm1} and the second equality is due to the fact that $Y_3^{\pm1}\in \ker (J_{\omega}L)$.
\end{proof}

 Under a full perturbation which takes $E_1$ into account,
we now prove the exponential trichotomy of the semigroup $e^{tJ_{\omega}L}$.

 \begin{Proposition}\label{exponential trichotomy}
The linearized operator $J_{\omega}L$  generates a $C^0$ group $e^{tJ_{\omega}L}$ on $L_0^2(\mathbb{S}^2)$ and there exists a decomposition
\begin{align*}
L_0^2(\mathbb{S}^2)=E^u\oplus E^c\oplus E^s
\end{align*}
of closed subspaces $E^{u,s,c}$ with the following properties:

$(\rm{i})$ $E^c$, $E^u$ and $E^s$ are invariant under $e^{tJ_{\omega}L}$;

$(\rm{ii})$ $E^u$ ($E^s$) only consists of eigenfunctions corresponding to eigenvalues of $J_{\omega}L$ with   positive (negative) real part and
\begin{align*}
\dim(E^u)=\dim(E^s)=\left\{
\begin{array}{llll}
4,&\omega\in(-3,{69\over2}),\\
2,&\omega\in(g^{-1}(-12),-3]\cup[{69\over2},{99\over2}),\\
0,&\omega\notin(g^{-1}(-12),{99\over2});
\end{array}
\right.
\end{align*}

$(\rm{iii})$ The quadratic form $\langle L\cdot,\cdot\rangle$ vanishes on $E^{u,s}$, but is
non-degenerate on $E^u \oplus E^s$, and
\begin{align*}
E^c = \{u \in L_0^2(\mathbb{S}^2) | \langle Lu, v\rangle =0, \;\; \forall\; v \in E^s \oplus E^u\};
\end{align*}

$(\rm{iv})$ For $\lambda_u=\min\{Re (\lambda)|\lambda\in\sigma(J_{\omega}L),Re(\lambda)>0\}$, there exist $C>0$ and $0\leq k_0\leq 1+2(n^-(L)-\dim (E^u))$ such that
\begin{align*}
&|e^{tJ_{\omega}L}|_{E^s}|\leq C(1+t^{\dim(E^s)-1}) e^{-\lambda_ut},\quad t\geq0,\\
&|e^{tJ_{\omega}L}|_{E^u}|\leq C(1+|t|^{\dim(E^u)-1}) e^{\lambda_ut},\quad t\leq0,\\
&|e^{tJ_{\omega}L}|_{E^c}|\leq C(1+|t|^{k_0}), \quad  t\in\mathbb{R}.
\end{align*}
\end{Proposition}

\begin{proof} [Proof of Proposition \ref{exponential trichotomy}]
Note that all the unstable eigenfunctions of $J_{\omega}L$ satisfy the constraints in $X$ and are therefore in $X$.
Theorems \ref{linear instability}, \ref{k=1 positive half-line critical rotation rate} \ref{k=2 positive half-line critical rotation rate},
\ref{k=1 negative half-line} and \ref{k=2 negative half-line}, along with \eqref{L1e2ok3nonnegative}-\eqref{L1o2enegative}, imply that
$\dim(E^u)=\dim(E^s)=4$
for $\omega\in(-3,{69\over2})$, $\dim(E^u)=\dim(E^s)=2$ for $\omega\in(g^{-1}(-12),-3]\cup[{69\over2},{99\over2})$, and
$\dim(E^u)=\dim(E^s)=0$ for $\omega\notin(g^{-1}(-12),{99\over2})$.
Then Proposition \ref{exponential trichotomy} follows from Theorem 2.2 in \cite{lin2022instability}.
\end{proof}

From the perspective of dynamical systems, it is natural to ask  whether there are local stable/unstable manifolds  near the linearly unstable $3$-jet, which will
provide a more accurate characterization of the nonlinear local dynamics. Proposition \ref{exponential trichotomy} can be viewed as such a  characterization   at the linear level. On a flat geometry, local invariant  manifolds near a linearly unstable shear flow is constructed in \cite{lin-zeng13}.  Pursuit of local unstable manifolds on the setting of a general rotating surface is a problem  for future work.

 \section{Nonlinear orbital instability of general steady flows}
 In this section, we prove that linear instability implies nonlinear orbital instability   for general steady flows.
 
 First, we revisit  the Euler equations on the global sphere.
 We regard $\mathbb{S}^2=\{\mathbf{x}=(x,y,z)\in\mathbb{R}^3|x^2+y^2+z^2=1\}$ as a Riemannian manifold equipped with the Riemannian metric induced by the
Euclidean metric of $\mathbb{R}^3$.
The Euler equation in the velocity form (see \cite{Pedlosky1971}) is
\begin{align}\label{velocity equation-original form}
\partial_t\mathbf{v}+\nabla_{\mathbf{v}}\mathbf{v}+2\omega\chi J\mathbf{v}+\nabla p=0,\quad {\rm div}(\mathbf{v})=0,
\end{align}
where the vector field $\mathbf{v}$ is the velocity, $p$ is the pressure, $\chi(\mathbf{x})=\mathbf{e}_3\cdot\nu(\mathbf{x})=z$, $\nu(\mathbf{x})$ is the unit outward pointing normal to $\mathbb{S}^2$ at $\mathbf{x}$, and $J:T_{\mathbf{x}}\mathbb{S}^2\to T_{\mathbf{x}}\mathbb{S}^2$  is the  counterclockwise rotation by $\pi/2$.
Then the vorticity is $\Omega=curl(\mathbf{v})$ and the Euler equation in the vorticity form is
\begin{align}\label{vorticity equation-original form}
\partial_t(\Omega+2\omega\chi)+\nabla_{\mathbf{v}}(\Omega+2\omega\chi)=0.
\end{align}
The stream function  $\psi$ has zero mean and satisfies  $\mathbf{v}=J\nabla\psi=\nabla^{\bot}\psi$ and $\Omega=\Delta\psi$.
Consider the  coordinates $(\varphi, s)$  with a chart the  chart $(\mathbb{S}^2\setminus\Gamma,\zeta)$: 
 \begin{align}\label{chart1}
 \begin{array}{rrrl}
 \zeta:&\mathbf{x}=(\cos(\varphi)\sqrt{1-s^2},\sin(\varphi)\sqrt{1-s^2},s)&\mapsto&(\varphi,s),\\
 &\mathbb{S}^2\setminus\Gamma&\to&(-\pi,\pi)\times(-1,1),
 \end{array}
 \end{align}
where $\Gamma=\{\mathbf{x}=(-\sqrt{1-s^2},0,s)|s\in[-1,1]\}$. As a natural extension, we supplementarily define
$
 \zeta(-\sqrt{1-s^2},0,s)=(\pi,s)$ for
$\mathbf{x}=(-\sqrt{1-s^2},0,s)\in\Gamma\setminus\{N,S\}$ (i.e. $\{\varphi=\pi, s\neq\pm1\}$), where $N, S$ denote the North and South poles.
In the  coordinates $(\varphi,s)$, the poles $N, S$ are stretched into two boundary lines $\{s=\pm1\}$.
Let
$
\Omega(\mathbf{x},t)=\Upsilon(\varphi,s,t)$ and $ \psi(\mathbf{x},t)=\Psi(\varphi,s,t).
$
Then the vorticity equation  \eqref{vorticity equation-original form} becomes
$
{\rm(\mathcal{E}_\omega)}$.

\subsection{Differential calculus on $\mathbb{S}^2$ and   the  averaging Lyapunov exponent} To study nonlinear orbital instability of general steady flows, we need the following geometrical  preparations about differential calculus on $\mathbb{S}^2$ and  the degeneracy of the  averaging Lyapunov exponent of the flow generated by the steady velocity field.
\subsubsection{Differential calculus on $\mathbb{S}^2$}\label{Differential calculus on S2}
We identify a point $\mathbf{x}\in\mathbb{S}^2$ with its Cartesian coordinates $(x,y,z)$ in $\mathbb{R}^3$.
The  chart $(\mathbb{S}^2\setminus\Gamma,\zeta)$ in \eqref{chart1},
along with another geographic coordinates $(\tilde\varphi,\tilde s)$  with a chart $(\mathbb{S}^2\setminus\tilde\Gamma,\tilde\zeta)$
 \begin{align}\label{chart2}
 \begin{array}{rrrl}
 \tilde\zeta:&(-\cos(\tilde\varphi)\sqrt{1-\tilde s^2},\tilde s,\sin(\tilde\varphi)\sqrt{1-\tilde s^2})&\mapsto&(\tilde\varphi,\tilde s),\\
 &\mathbb{S}^2\setminus\tilde \Gamma&\to&(-\pi,\pi)\times(-1,1),
 \end{array}
 \end{align}
 gives a smooth manifold structure of $\mathbb{S}^2$,
 where $\tilde \Gamma=\{\mathbf{x}=(\sqrt{1-\tilde s^2},\tilde s,0)|\tilde s\in[-1,1]\}$.
  On the one hand, to avoid the singularity at the  poles $N, S$, we instead use two restricted charts $(\zeta^{-1}((-\pi,\pi)\times(-1+\kappa_0,1-\kappa_0)),\zeta)$ and $ (\tilde\zeta^{-1}((-\pi,\pi)\times(-1+\kappa_0,1-\kappa_0)),\tilde\zeta) $ to cover $\mathbb{S}^2$ for sufficiently small $\kappa_0>0$. On the other hand, when we consider the whole sphere and  the poles $N, S$  could be regarded as singular points, we use the full chart $(\mathbb{S}^2\setminus\Gamma,\zeta)$.
  We discuss  the  chart $(\mathbb{S}^2\setminus\Gamma,\zeta)$ below, and the chart $(\mathbb{S}^2\setminus\tilde \Gamma, \tilde \zeta)$ can be considered similarly.  The Riemannian metric of $\mathbb{S}^2$
is given by $g|_{(\mathbb{S}^2\setminus\Gamma,\zeta)}={1\over 1-s^2}ds^2+(1-s^2)d\varphi^2$.
Since $g_{11}=g(\partial_s,\partial_s)={1\over 1-s^2}$ and $g_{22}=g(\partial_\varphi,\partial_\varphi)=1-s^2$, we obtain  an orthonormal
basis $\{\mathbf{e}_s=\sqrt{1-s^2}\partial_s,\mathbf{e}_\varphi={1\over\sqrt{1-s^2}}\partial_\varphi\}$ of the tangent space $T\mathbb{S}^2$.
  The Riemannian volume is given by $d\sigma_g=dsd\varphi$ since $(g_{ij})= \left( \begin{array}{cc} {1\over1-s^2} & 0 \\ 0& { 1-s^2} \end{array} \right)$. For a vector field $\mathbf{u}=u^1\mathbf{e}_s+u^2\mathbf{e}_\varphi=u^1\sqrt{1-s^2}\partial_s+{u^2\over \sqrt{1-s^2}}\partial_\varphi$, the directional derivative along $\mathbf{u}$ of a scalar-valued function $f$ is $\mathbf{u}\cdot\nabla f=\nabla_{\mathbf{u}}f=u^1\sqrt{1-s^2}\partial_s f+{u^2\over \sqrt{1-s^2}}\partial_\varphi f$.
   Let $D$ be the Levi-Civita connection on the Riemannian manifold $(\mathbb{S}^2,g)$.
   For a vector field $\mathbf{u}$, the divergence of $\mathbf{u}$ is defined as ${\rm{div}}(\mathbf{u})={\rm trace}(D\mathbf{u})$.
   In the local coordinates $(\varphi,s)$,
 ${\rm{div}}(\mathbf{u})={1\over \det(g_{ij})}\partial_s(\det(g_{ij})\sqrt{1-s^2}u^1)+{1\over \det(g_{ij})}\partial_\varphi\left(\det(g_{ij}){u^2\over \sqrt{1-s^2}}\right)=\partial_s(\sqrt{1-s^2}u^1)+\partial_\varphi\left({u^2\over \sqrt{1-s^2}}\right)$ for
 $\mathbf{u} =u^1\mathbf{e}_s+u^2\mathbf{e}_\varphi=u^1\sqrt{1-s^2}\partial_s+{u^2\over \sqrt{1-s^2}}\partial_\varphi$.
    The gradient of a scalar-valued function $f$ is defined as a vector field, denoted by $\nabla f$, satisfying  $g(\nabla f, \mathbf{w})=df(\mathbf{w})$
   for any smooth vector field $\mathbf{w}$, where $df$ is the differential defined as $df(\xi)=\xi(f)$ for any $\xi\in T\mathbb{S}^2$.
    Since $(g^{ij})= \left( \begin{array}{cc} 1-s^2 & 0 \\ 0& {1\over 1-s^2} \end{array} \right)$ in the local coordinates $(\varphi,s)$,  $\nabla f=\partial_sfg^{11}\partial_s+\partial_\varphi f g^{22}\partial_\varphi=(1-s^2)\partial_s f\partial_s+{1\over 1-s^2}\partial_\varphi f\partial_\varphi=\sqrt{1-s^2}\partial_s f\mathbf{e}_s+{1\over \sqrt{1-s^2}}\partial_\varphi f\mathbf{e}_\varphi$.
     The  orthogonal of  gradient is
     $\nabla^{\bot}f=J\nabla f$,
    where $J=\left( \begin{array}{cc} 0& 1 \\ -1& 0 \end{array} \right)$ is the counterclockwise rotation by $\pi/2$ in the basis  $(\mathbf{e}_s,\mathbf{e}_\varphi)$.
     The Laplace-Beltrami operator is
  $\Delta f=({\rm div}\circ\nabla) f$ and $\Delta f=\partial_s((1-s^2)\partial_s f)+\partial_\varphi^2\left({f\over 1-s^2}\right)$ in the local coordinates $(\varphi,s)$. The Christoffel symbols are explicitly given by $\Gamma_{11}^1={s\over 1-s^2}$, $\Gamma_{22}^1=s(1-s)^2$,
 $\Gamma_{21}^1=\Gamma_{12}^1=\Gamma_{11}^2=\Gamma_{22}^2=0$ and $\Gamma_{21}^2=\Gamma_{12}^2=-{s\over 1-s^2}$.

 Then we introduce the Sobolev spaces on $\mathbb{S}^2$. For the theory of Sobolev spaces on general compact Riemannian manifold, the readers are referred to \cite{Aubin1982,Aubin1998,Hebey1996,Hebey2000}. For  $f\in C^\infty(\mathbb{S}^2)$, we define
 $|\nabla^0f|=|f|, |\nabla^1f|=\left(g^{11}(\partial_s f)^2+g^{22}(\partial_\varphi f)^2\right)^{1\over2}=\left((1-s^2)(\partial_s f)^2+{1\over 1-s^2}(\partial_\varphi f)^2\right)^{1\over2},$ and $|\nabla^2f|=\left((g^{11}(\nabla^2 f)_{11})^2+2g^{11}g^{22}((\nabla^2 f)_{12})^2+(g^{22}(\nabla^2 f)_{22})^2\right)^{1\over2}=\bigg((1-s^2)^2(\partial_{s}^2f-{s\over 1-s^2}\partial_s f)^2+2(\partial_s\partial_\varphi f+{s\over 1-s^2}\partial_\varphi f)^2+{1\over (1-s^2)^2}(\partial_\varphi^2 f-s(1-s)^2\partial_s f)^2\bigg)^{1\over2}$, where
 $\nabla^k f$ denotes the $k$-th covariant derivative of $f$ for $k=0,1,2$, $(\nabla^2 f)_{11}=\partial_{s}^2f-\Gamma_{11}^1\partial_s f$, $(\nabla^2 f)_{12}=(\nabla^2 f)_{21}=\partial_s\partial_\varphi f-\Gamma_{12}^2\partial_\varphi f$ and $(\nabla^2 f)_{22}=\partial_\varphi^2 f-\Gamma_{22}^1\partial_s f$. For $f\in C^\infty(\mathbb{S}^2)$, set $\|\nabla^k f\|_{L^p(\mathbb{S}^2)}= \bigg(\int_{\mathbb{S}^2}|\nabla^kf|^pd\sigma_g\bigg)^{1\over p}$ and $\|f\|_{H_k^p(\mathbb{S}^2)}=\sum_{j=0}^k\|\nabla^j f\|_{L^p(\mathbb{S}^2)}$ for $p\geq1$ and $k=0,1,2$.
The Sobolev space $H_{k}^p(\mathbb{S}^2)$ is defined by the completion of $C^{\infty}(\mathbb{S}^2)$ with respect to $\|\cdot\|_{H_k^p}$ for $k=0,1,2$.
For a vector field $\mathbf{u}=u^1\mathbf{e}_s+u^2\mathbf{e}_\varphi$, we define a scalar-valued function $|\mathbf{u}|=(g(\mathbf{u},\mathbf{u}))^{1\over2}=(|u^1|^2+|u^2|^2)^{1\over2}$, and we say that $\mathbf{u}\in H_k^p(T\mathbb{S}^2)$ if $|\mathbf{u}|\in H_k^p(\mathbb{S}^2)$. Define
 $\|\mathbf{u}\|_{H_k^p(T\mathbb{S}^2)}=\big\||\mathbf{u}|\big\|_{H_k^p(\mathbb{S}^2)}$.
For $f\in C^m(\mathbb{S}^2)$, the norm is $\|f\|_{C^m(\mathbb{S}^2)}=\sum_{j=0}^m\max_{\mathbf{x}\in\mathbb{S}^2}|(\nabla^j f)(\mathbf{x})|$.
We need the following lemma.
\begin{Lemma}\label{lem-differential calculus}
{\rm(i)} (Sobolev embedding) $H_2^2(\mathbb{S}^2)$ is embedded in $H_1^p(\mathbb{S}^2)$ for any $1\leq p<\infty$.

{\rm(ii)} (Compact embedding) Let $q_0\geq2$. Then the embedding of $H_2^{q_0}(\mathbb{S}^2)$ in $H_1^p(\mathbb{S}^2)$ is compact for any $p\geq1$.

{\rm(iii)} (Poincar${\rm\acute{e}}$ inequality) Let $p\geq1$. Then for any $\psi\in H_1^p(\mathbb{S}^2)$ satisfying $\int_{\mathbb{S}^2}\psi d\sigma_{g}=0$, we have
\begin{align}\label{Poincare inequality nonlinear instability}
\|\psi\|_{L^p(\mathbb{S}^2)}\leq C\|\nabla \psi\|_{L^p(T\mathbb{S}^2)}.
\end{align}

{\rm(iv)} (Vector field version of Hodge decomposition) $L^p(T\mathbb{S}^2)$ has the following direct sum decomposition:
\begin{align*}
L^p(T\mathbb{S}^2)=H_p(\mathbb{S}^2)\oplus G_p(\mathbb{S}^2),
\end{align*}
where $H_p(\mathbb{S}^2)=\{\nabla^\bot\psi|\psi\in H_{1}^p(\mathbb{S}^2)\}$ and $G_p(\mathbb{S}^2)=\{\nabla\phi|\phi\in H_{1}^p(\mathbb{S}^2)\}$.

{\rm(v)} ($L^p$ boundedness of the Riesz transform and the Leray projection) For any $1<p<\infty$, the Riesz transform $\nabla\Delta^{-{1\over2}}$ is a bounded operator from $\{\psi\in L^p(\mathbb{S}^2)|\int_{\mathbb{S}^2}\psi d\sigma_g=0\}$ to $L^p(T\mathbb{S}^2)$, and the Leray projection $P_p=Id-\nabla\Delta^{-1}\nabla\cdot$ is a bounded operator from $L^p(T\mathbb{S}^2)$ to $H_p(\mathbb{S}^2)$.

{\rm (vi)} (Estimate $L^p$ norm of a solenoidal vector field by duality)
Let $\nabla^{\bot}\hat\psi\in H_p(\mathbb{S}^2)$ for $1<p<\infty$. If $curl (\nabla^{\perp}\hat\psi)=-{\rm div} (J\nabla^{\bot}\hat\psi)=\Delta\hat\psi\in L^p(\mathbb{S}^2)$, then
\begin{align*}
\|\nabla^{\bot}\hat\psi\|_{L^p(T\mathbb{S}^2)}\leq C_1\sup_{\psi\in H_1^{p'}(\mathbb{S}^2),\|\psi\|_{ H_1^{p'}(\mathbb{S}^2)}=1}\left|\int_{\mathbb{S}^2}\psi\Delta\hat\psi d\sigma_g\right|
\end{align*}
for some $C_1>0$, where $p'$ is the   H\"{o}lder conjugate number  of $p$.
\end{Lemma}
\begin{proof}
(i) For $1\leq q<2$, we have $H_2^2(\mathbb{S}^2)\subset H_2^q(\mathbb{S}^2)$.  By Theorem 2.6 in \cite{Hebey2000}, $H_2^q(\mathbb{S}^2)\subset H_1^p(\mathbb{S}^2)$ for $1\leq q<2$ and ${1\over p}={1\over q}-{1\over2}$. If $1\leq q<2$ runs through $[1,2)$, then $p$ runs through $[2,\infty)$. Thus,  $H_2^2(\mathbb{S}^2)\subset H_2^q(\mathbb{S}^2)\subset H_1^p(\mathbb{S}^2)$ for $1\leq p<\infty$.

(ii) By Theorem 2.9 (i) in \cite{Hebey2000} (with $j=m=1$, $q=2$), the embedding of $H_2^2(\mathbb{S}^2)$ in $H_1^p(\mathbb{S}^2)$ is compact for any $p\geq1$. Moreover, $H_2^{q_0}(\mathbb{S}^2)\subset H_2^2(\mathbb{S}^2)$ since $q_0>2$. Thus, the embedding of $H_2^{q_0}(\mathbb{S}^2)$ in $H_1^p(\mathbb{S}^2)$ is compact.

(iii) First, the inequality  \eqref{Poincare inequality nonlinear instability} for $p\in[1,2)$ is obtained  by the Poincar${\rm\acute{e}}$ inequality in Theorem 2.10 of  \cite{Hebey2000}. Then for $p\geq2$, by the Sobolev-Poincar${\rm\acute{e}}$ inequality in Theorem 2.11 of \cite{Hebey2000},  we have
$
\|\psi\|_{L^p(\mathbb{S}^2)}\leq C\|\nabla \psi\|_{L^q(T\mathbb{S}^2)}
$, where ${1\over q}={1\over2}+{1\over p}\Rightarrow q\in[1,2)$.
Thus, \eqref{Poincare inequality nonlinear instability} holds true due to the fact that $q<2\leq p$.

(iv) Let $d$ be the exterior differential operator, $d^*$ be the $L^2$-adjoint of $d$ with respect to the Riemannian volume measure $d\sigma_g$,  $L^p(\wedge^k T^*\mathbb{S}^2)$ be the $L^p$ space of $k$-forms on $(\mathbb{S}^2,g)$, ${\rm H}_{k,p}(\mathbb{S}^2)$ be the space of $L^p$-harmonic $k$-forms on $(\mathbb{S}^2,g)$, and $W^{1,p}(\wedge^jT^*\mathbb{S}^2)=\{\varpi\in L^p(\wedge^j T^*\mathbb{S}^2)|\|\varpi\|_{L^p}$ $+\|d\varpi\|_{L^p}+\|d^*\varpi\|_{L^p}<\infty\}$ for $j=0,2$.
 By Proposition 6.5 in \cite{Scott1995} (see also Theorem 1.2 in \cite{Li2009}),  $L^p(\wedge^1 T^*\mathbb{S}^2)$ has the Hodge direct sum decomposition $L^p(\wedge^1 T^*\mathbb{S}^2)={\rm H}_{1,p}(\mathbb{S}^2)\oplus dW^{1,p}(\mathbb{S}^2)\oplus d^*W^{1,p}(\wedge^2T^*\mathbb{S}^2)$.
 For the cohomology class containing $\mathbb{S}^2$, the $1$st Betti number is $0$, where the definition of  Betti number can be found in (8.49) of \cite{Taylor1996}. By the Hodge theory, the dimension of ${\rm H}_{1,p}(\mathbb{S}^2)$ is the same with the $1$st Betti number $0$
 (see \cite{Cheng-Mahalov2013}), that is, ${\rm H}_{1,p}(\mathbb{S}^2)=\{0\}$.
 Thus, for any $1$-form $\varpi\in L^p(\wedge^1 T^*\mathbb{S}^2)$, there exist two scalar-valued functions $\phi_\varpi, \psi_\varpi\in H_1^p(\mathbb{S}^2)$ such that $\varpi=d\phi_\varpi-*d\psi_\varpi$, where $*$ is the Hodge star operator.
 Using the musical isomorphism between $T\mathbb{S}^2$ and $T^*\mathbb{S}^2$ in (5.9)-(5.10) of \cite{Cheng-Mahalov2013}, we get the vector field version of Hodge decomposition: for any $\mathbf{u}\in L^p(T\mathbb{S}^2)$, there exist $\nabla\phi_{\mathbf{u}}\in G_p(\mathbb{S}^2)$ and $\nabla^\bot\psi_{\mathbf{u}}\in H_p(\mathbb{S}^2)$ such that $\mathbf{u}=\nabla\phi_{\mathbf{u}}+\nabla^\bot\psi_{\mathbf{u}}$.

 (v) Let $1< p< \infty$. The  boundedness of $\nabla\Delta^{-{1\over2}}$ follows  from the boundedness of $d\Delta^{-{1\over2}}$ from $L^p(\mathbb{S}^2)$ to $L^p(\wedge^1 T^*\mathbb{S}^2)$ in \cite{Seeley1967,Strichartz1983}. The adjoint of $d\Delta^{-{1\over2}}$ is $ \Delta^{-{1\over2}}d^*$, which is bounded from $L^p(\wedge^1 T^*\mathbb{S}^2)$ to $L^p(\mathbb{S}^2)$. Thus, $d\Delta d^*$ is bounded from $L^p(\wedge^1 T^*\mathbb{S}^2)$ to $L^p(\wedge^1 T^*\mathbb{S}^2)$. Combining with the musical isomorphism between $T\mathbb{S}^2$ and $T^*\mathbb{S}^2$, we infer that the Leray projection $P_p=Id-\nabla\Delta^{-1}\nabla\cdot$ is a bounded operator from $L^p(T\mathbb{S}^2)$ to $H_p(\mathbb{S}^2)$.

 (vi) By (iv), for any vector field $\mathbf{u}\in L^{p'}(T\mathbb{S}^2)$ with $\|\mathbf{u}\|_{L^{p'}(T\mathbb{S}^2)}=1$,
 there exist $\nabla^{\bot}\psi_{\mathbf{u}}\in H_{p'}(\mathbb{S}^2)$ and $\nabla\phi_{\mathbf{u}}\in G_{p'}(\mathbb{S}^2)$ such that
$\mathbf{u}=\nabla^{\bot}\psi_{\mathbf{u}}+\nabla\phi_{\mathbf{u}}$. By (v), $\|\nabla^{\bot}\psi_{\mathbf{u}}\|_{L^{p'}(T\mathbb{S}^2)}=\|P_{p'}\mathbf{u}\|_{L^{p'}(T\mathbb{S}^2)}\leq \|P_{p'}\|_{op}$. Moreover, for
$\psi\in  H_1^{p'}(\mathbb{S}^2)$ satisfying $\|\nabla^{\bot}\psi\|_{L^{p'}(T\mathbb{S}^2)}=1$,
by (iii) we have $\|\psi\|_{L^{p'}(\mathbb{S}^2)}\leq  C \|\nabla^{\bot}\psi\|_{L^{p'}(T\mathbb{S}^2)}\leq C$ and thus, $\|\psi\|_{H_1^{p'}(\mathbb{S}^2)}\leq \tilde C_1$ for some $\tilde C_1>0$.
 Combining these facts, we have
\begin{align*}
\|\nabla^{\bot}\hat\psi\|_{L^p(T\mathbb{S}^2)}=&\sup_{\mathbf{u}\in L^{p'}(T\mathbb{S}^2),
\|\mathbf{u}\|_{ L^{p'}(T\mathbb{S}^2)}=1}\int_{\mathbb{S}^2}\nabla^{\bot}\hat\psi\cdot\mathbf{u} d\sigma_g\\
=&\sup_{\mathbf{u}\in L^{p'}(T\mathbb{S}^2),\|\mathbf{u}\|_{ L^{p'}(T\mathbb{S}^2)}=1}\int_{\mathbb{S}^2}\nabla^{\bot}\hat\psi\cdot\nabla^{\bot}\psi_{\mathbf{u}} d\sigma_g\\
\leq&\sup_{\nabla^{\bot}\psi\in H_{p'}(\mathbb{S}^2),\|\nabla^{\bot}\psi\|_{ L^{p'}(T\mathbb{S}^2)}\leq \|P_{p'}\|_{op}}\int_{\mathbb{S}^2}\nabla^{\bot}\hat\psi\cdot\nabla^{\bot}\psi d\sigma_g\\
=&\|P_{p'}\|_{op}\sup_{\nabla^{\bot}\psi\in H_{p'}(\mathbb{S}^2),\|\nabla^{\bot}\psi\|_{ L^{p'}(T\mathbb{S}^2)}=1}\int_{\mathbb{S}^2}\nabla^{\bot}\hat\psi\cdot\nabla^{\bot}\psi d\sigma_g\\
\leq&\|P_{p'}\|_{op}\sup_{\psi\in H_1^{p'}(\mathbb{S}^2),\|\psi\|_{H_1^{p'}(\mathbb{S}^2)}\leq \tilde C_1}\int_{\mathbb{S}^2}\nabla\hat\psi\cdot\nabla\psi d\sigma_g\\
\leq&\tilde C_1\|P_{p'}\|_{op}\sup_{\psi\in H_1^{p'}(\mathbb{S}^2),\|\psi\|_{H_1^{p'}(\mathbb{S}^2)}=1}\left|\int_{\mathbb{S}^2}\Delta\hat\psi\psi d\sigma_g\right|,
\end{align*}
where we used $\int_{\mathbb{S}^2}\nabla^\bot\hat\psi\cdot\nabla\phi_{\mathbf{u}} d\sigma_g=-\int_{\mathbb{S}^2}{\rm div}(\nabla^\bot\hat\psi)\phi_{\mathbf{u}} d\sigma_g=0$ in the second equality.
The proof is complete by setting $C_1=\tilde C_1\|P_{p'}\|_{op}$.
 \end{proof}
\subsubsection{The averaging Lyapunov exponent}
Let ${\mathbf{u}}_G$ be a $C^1$ steady flow with finite
 stagnation points.
 The corresponding  stream function and vorticity are denoted by $\psi_G$ and $\Omega_G$.
Let ${\mathbf{X}}_G(t,\mathbf{x})$ be the particle  trajectory induced by the steady velocity field ${\mathbf{u}}_G$ satisfying
\begin{align}\label{the trajectory induced by the steady velocity field}
\left\{ \begin{array}{rcl} {d{\mathbf{X}}_G\over dt}&=&{\mathbf{u}}_G({\mathbf{X}}_G(t,\mathbf{x})),  \\ {\mathbf{X}}_G(0,\mathbf{x})&=&\mathbf{x}.\end{array} \right.
\end{align}
For the $C^1$ mapping ${\mathbf{X}}_G(t,\cdot):\mathbb{S}^2\to\mathbb{S}^2$, we denote its tangent mapping at $\mathbf{x}$ by $d{\mathbf{X}}_G(t,\mathbf{x}):T_{\mathbf{x}}\mathbb{S}^2\to T_{{\mathbf{X}}_G(t,\mathbf{x})}\mathbb{S}^2$. For $\xi\in T_{\mathbf{x}}\mathbb{S}^2$, it is defined by $(d{\mathbf{X}}_G(t,\mathbf{x})\xi)(f)=\xi(f\circ{\mathbf{X}}_G(t,\mathbf{x}))$, $\forall f\in C_{{\mathbf{X}}_G(t,\mathbf{x})}^\infty$.
The adjoint map  of $d{\mathbf{X}}_G(t,\mathbf{x})$ is denoted by $(d{\mathbf{X}}_G(t,\mathbf{x}))^*: T_{{\mathbf{X}}_G(t,\mathbf{x})}\mathbb{S}^2\to T_{\mathbf{x}}\mathbb{S}^2$, which is characterized by the relation
\begin{align*}
g_{{\mathbf{X}}_G(t,\mathbf{x})}(\xi_2,d{\mathbf{X}}_G(t,\mathbf{x})\xi_1)
=g_{\mathbf{x}}((d{\mathbf{X}}_G(t,\mathbf{x}))^*\xi_2,\xi_1)
\end{align*}
for any $\xi_1\in T_\mathbf{x}\mathbb{S}^2$ and $\xi_2\in T_{{\mathbf{X}}_G(t,\mathbf{x})}\mathbb{S}^2$.
Let
\begin{align}\label{def-tangent-mapping-norm}
|(d{\mathbf{X}}_G(t,\mathbf{x}))^*|_{op}=\sup_{0\neq\xi\in T_{{\mathbf{X}}_G(t,\mathbf{x})}\mathbb{S}^2}{\left(g_{\mathbf{x}}((d{\mathbf{X}}_G(t,\mathbf{x}))^*\xi,(d{\mathbf{X}}_G(t,\mathbf{x}))^*\xi)\right)^{1\over2}\over \left(g_{{\mathbf{X}}_G(t,\mathbf{x})}(\xi,\xi)\right)^{1\over2}}
\end{align}
for $\mathbf{x}\in\mathbb{S}^2$.
The Lyapunov exponent of the flow \eqref{the trajectory induced by the steady velocity field} generated by the steady flow ${\mathbf{u}}_G$
is defined by
\begin{align}\label{def-Lyapunov exponent}
\mu=&\sup_{\mathbf{x}\in\mathbb{S}^2}\lim_{|t|\to\infty}{1\over |t|}\ln\left(|(d{\mathbf{X}}_G(t,\mathbf{x}))^*|_{op}\right),
\end{align}
and the averaging Lyapunov exponent is defined by
\begin{align}\label{def-averaging Lyapunov exponent}
\mu_{\rm av}=\lim_{|t|\to\infty}{1\over |t|}\ln\left(\int_{\mathbb{S}^2}|(d{\mathbf{X}}_G(t,\mathbf{x}))^*|_{op}d\sigma_g\right).
\end{align}
Note that   $\mu>0$ only if the stream function $\psi_G$ has at least one nondegenerate saddle point. Moreover, $|(d{\mathbf{X}}_G(t,\mathbf{x}))^*|_{op}$ has exponent growth only on the set of all nondegenerate saddle points of $\psi_G$ and the trajectories connecting them, while has only linear growth on other trajectories. Then we prove that the averaging Lyapunov exponent is zero.
\begin{Lemma}\label{averaging Lyapunov exponent is zero}
For a $C^1$ steady flow ${\mathbf{u}}_G$ 
with finite
 stagnation points,
we have $\mu_{\rm av}=0$.
\end{Lemma}
\begin{proof}
If ${\mathbf{u}}_G$ is  zonal, then $\psi_G$ has no nondegenerate saddle points, and thus, $|(d{\mathbf{X}}_G(t,\mathbf{x}))^*|_{op}$ has only linear growth on all trajectories. This implies $\mu_{\rm av}=\mu=0$.

If ${\mathbf{u}}_G$ is non-zonal, then we only need to study the growth of the integral $  \int|(d{\mathbf{X}}_G(t,\mathbf{x}))^*|_{op}d\sigma_g$ in the region near nondegenerate saddle points and the trajectories connecting them. The number of nondegenerate saddle points of $\psi_G$ is finite since such points are isolated and $\mathbb{S}^2$ is compact. We denote the nondegenerate saddle points by $\mathbf{x}_1,\cdots,\mathbf{x}_q$. We only consider the point $\mathbf{x}_1$ below since other points can be treated similarly.
 Define $D_{\varepsilon_1}^1=\{\mathbf{x}\in\mathbb{S}^2|\psi_G(\mathbf{x}_1)-\varepsilon_1<\psi_G(\mathbf{x})<\psi_G(\mathbf{x}_1)+\varepsilon_1\}$, where $\varepsilon_1>0$ is small enough so that $\psi_G(\mathbf{x}_1)$ is the only critical value in the interval   $[\psi_G(\mathbf{x}_1)-\varepsilon_1,\psi_G(\mathbf{x}_1)+\varepsilon_1]$ and there exists $\mathbf{z}_1\in \mathbb{S}^2$ such that $\mathbf{z}_1,-\mathbf{z}_1\notin \overline{\{{\mathbf{X}}_G(t,\mathbf{x})|\mathbf{x}\in\overline{D_{\varepsilon_1}^1},t\in\mathbb{R}\}}$.
We rotate the original Cartesian coordinate system $O$-$x y z$ so that in the new coordinate system $O$-$\hat x \hat y \hat z$,  the direction from $-\mathbf{z}_1$ to $\mathbf{z}_1$ is the positive direction of the vertical axis $\hat z$.
  With $\mathbf{z}_1$ and $-\mathbf{z}_1$ as the North and South poles, we can establish the latitude-longitude spherical coordinates $(\hat\varphi,\hat\theta)\in(-\pi,\pi)\times(-{\pi\over2},{\pi\over2})$ on $\mathbb{S}^2$. By setting $\hat s=\sin(\hat\theta)$, we instead use  the geographic coordinates $(\hat\varphi,\hat s)\in(-\pi,\pi]\times(-{\pi\over2},{\pi\over2})$, where we supplementarily define $\{\hat\varphi=\pi,\hat s\neq\pm1\}$  by a natural extension.
 Define $\hat\zeta(\mathbf{x})=\hat\zeta((\hat x,\hat y,\hat z))=\hat\zeta((\cos(\hat\varphi)\sqrt{1-\hat s^2},\sin(\hat \varphi)\sqrt{1-\hat s^2},\hat s))\triangleq(\hat \varphi,\hat s)$ from $\mathbb{S}^2\setminus\{\mathbf{z}_1,-\mathbf{z}_1\}$ to $(-\pi,\pi]\times(-1,1)$.
   For $\mathbf{x}=\hat\zeta^{-1}((\hat\varphi,\hat s))\in \overline{D_{\varepsilon_1}^1}$, the particle trajectory $\hat\zeta({\mathbf{X}}_G(t,\mathbf{x}))=(\hat \varphi_L(t;\hat\varphi,\hat s),\hat s_L(t;\hat\varphi,\hat s))$, $t\in\mathbb{R}$, does not touch the poles $\mathbf{z}_1,-\mathbf{z}_1$ and fully lies  inside $(-\pi,\pi]\times(-{\pi\over2},{\pi\over2})$.
 For $\mathbf{x}=\hat\zeta^{-1}((\hat\varphi,\hat s))\in \overline{D_{\varepsilon_1}^1}$, $\hat\xi=\hat\xi^1\partial_{\hat s}+\hat\xi^2\partial_{\hat\varphi}\in T_{\mathbf{x}}\mathbb{S}^2$ and $d{\mathbf{X}}_G(t,\mathbf{x})\hat\xi=(d{\mathbf{X}}_G(t,\mathbf{x})\hat\xi)^1\partial_{\hat s_L}+(d{\mathbf{X}}_G(t,\mathbf{x})\hat \xi)^2\partial_{\hat\varphi_L}\in T_{{\mathbf{X}}_G(t,\mathbf{x})}\mathbb{S}^2$, by Lemma 1.29 in \cite{Grigoryan2024} we have
\begin{align*}
\left( \begin{array}{cc} (d{\mathbf{X}}_G(t,\mathbf{x})\hat\xi)^1 \\  (d{\mathbf{X}}_G(t,\mathbf{x})\hat\xi)^2 \end{array} \right)=\left( \begin{array}{cc}
 \partial_{\hat s} \hat s_L(t,\hat \varphi,\hat s)&\partial_{\hat\varphi} \hat s_L(t,\hat\varphi, \hat s) \\  \partial_{\hat s}\hat \varphi_L(t,\hat\varphi,\hat s) &\partial_{\hat\varphi}{\hat\varphi}_L(t,\hat\varphi,\hat s) \end{array} \right)\left( \begin{array}{cc} \hat\xi^1\\ \hat\xi^2 \end{array} \right)\triangleq\mathcal{J}(t,\hat\varphi,\hat s)\left( \begin{array}{cc} \hat\xi^1\\ \hat\xi^2 \end{array} \right).
\end{align*}
By the local coordinate expression $\mathbf{x}=(\hat x,\hat y,\hat z)=( \cos(\hat\varphi)\sqrt{1-\hat s^2},\sin(\hat\varphi)\sqrt{1-\hat s^2},\hat s)$, we can compute the Riemannian metric of $\mathbb{S}^2$ to be
$\hat g|_{\mathbb{S}^2\setminus\{\hat\varphi=\pi\}}={1\over 1-\hat s^2}d\hat s^2+(1-\hat s^2)d\hat \varphi^2$. The orthonormal
basis of the tangent space $T\mathbb{S}^2$ is  $\{\mathbf{e}_{\hat s}=\sqrt{1-\hat s^2}\partial_{\hat s},\mathbf{e}_{\hat \varphi}={1\over\sqrt{1-\hat s^2}}\partial_{\hat \varphi}\}$.
Since
\begin{align*}
\hat g_{{\mathbf{X}}_G(t,\mathbf{x})}(\hat\xi_2,d{\mathbf{X}}_G(t,\mathbf{x})\hat\xi_1)
=\hat g_{\mathbf{x}}((d{\mathbf{X}}_G(t,\mathbf{x}))^*\hat\xi_2,\hat\xi_1)
\end{align*}
 for any $\hat\xi_1=\hat\xi_1^1\partial_{\hat s}+\hat\xi_1^2\partial_{\hat\varphi}={1\over \sqrt{1-\hat s^2}}\hat\xi_1^1 \mathbf{e}_{\hat s}+\sqrt{1-\hat s^2}\hat\xi_1^2\mathbf{e}_{\hat \varphi}\in T_\mathbf{x}\mathbb{S}^2$ and $\hat\xi_2=\hat\xi_2^1\partial_{\hat s_L}+\hat\xi_2^2\partial_{\hat\varphi_L}={1\over \sqrt{1-\hat s_L^2}}\hat\xi_2^1 \mathbf{e}_{\hat s_L}+\sqrt{1-\hat s_L^2}\hat\xi_2^2\mathbf{e}_{\hat \varphi_L}\in T_{{\mathbf{X}}_G(t,\mathbf{x})}\mathbb{S}^2$,
in the geographic coordinates $(\hat\varphi,\hat s)$ we have
\begin{align*}
&\left( \begin{array}{cc}  {1\over \sqrt{1-\hat s_L^2}}\hat\xi_2^1\\\sqrt{1-\hat s_L^2}\hat\xi_2^2 \end{array} \right)\cdot\left( \begin{array}{cc} {1\over\sqrt{1-\hat s_L^2}}&0\\0&\sqrt{1-\hat s_L^2}  \end{array} \right)\mathcal{J}(t,\hat\varphi,\hat s)\left( \begin{array}{cc} \hat\xi_1^1\\ \hat\xi_1^2\end{array} \right)\\
=&\left( \begin{array}{cc} {\sqrt{1-\hat s^2}}&0\\0&{1\over\sqrt{1-\hat s^2}}  \end{array} \right)(\mathcal{J}(t,\hat\varphi,\hat s))^*\left( \begin{array}{cc} {1\over{1-\hat s_L^2}}&0\\0&{1-\hat s_L^2}  \end{array} \right)\left( \begin{array}{cc} \hat\xi_2^1\\ \hat\xi_2^2\end{array} \right)\cdot\left( \begin{array}{cc}{1\over \sqrt{1-\hat s^2}}\hat\xi_1^1\\ \sqrt{1-\hat s^2}\hat\xi_1^2\end{array} \right),
\end{align*}
 where $(\hat\varphi_L,\hat s_L)=(\hat\varphi_L(t),\hat s_L(t))$, the left hand side is in the basis $\{\mathbf{e}_{\hat s_L}, \mathbf{e}_{\hat\varphi_L}\}$, and the right hand side is in the basis $\{\mathbf{e}_{\hat s}, \mathbf{e}_{\hat \varphi}\}$. Thus,
\begin{align*}
&(d{\mathbf{X}}_G(t,\mathbf{x}))^*\xi\\
=&\left( \begin{array}{cc} {\sqrt{1-\hat s^2}}&0\\0&{1\over\sqrt{1-\hat s^2}}  \end{array} \right)(\mathcal{J}(t,\hat\varphi,\hat s))^*\left( \begin{array}{cc} {1\over{1-\hat s_L^2}}&0\\0&{1-\hat s_L^2}  \end{array} \right)\left( \begin{array}{cc} \xi^1\\ \xi^2\end{array} \right)\cdot
\left( \begin{array}{cc}\mathbf{e}_{\hat s}\\\mathbf{e}_{\hat\varphi}\end{array} \right)\\
=&\left( \begin{array}{cc} {\sqrt{1-\hat s^2}}&0\\0&{1\over\sqrt{1-\hat s^2}}  \end{array} \right)(\mathcal{J}(t,\hat\varphi,\hat s))^*\left( \begin{array}{cc} {1\over{\sqrt{1-\hat s_L^2}}}&0\\0&\sqrt{1-\hat s_L^2}  \end{array} \right)\left( \begin{array}{cc} {1\over \sqrt{1-\hat s_L^2}}\xi^1\\ \sqrt{1-\hat s_L^2}\xi^2\end{array} \right)\cdot
\left( \begin{array}{cc}\mathbf{e}_{\hat s}\\\mathbf{e}_{\hat\varphi}\end{array} \right)
\end{align*}
for $\xi=\xi^1\partial_{\hat s_L}+\xi^2\partial_{\hat\varphi_L}={1\over \sqrt{1-\hat s_L^2}}\xi^1 \mathbf{e}_{\hat s_L}+\sqrt{1-\hat s_L^2}\xi^2\mathbf{e}_{\hat \varphi_L}\in T_{{\mathbf{X}}_G(t,\mathbf{x})}\mathbb{S}^2$. Since $\mathbb{S}^2$ is compact, there exists $C>1$ such that
${1\over C}g\leq \hat g \leq C g$ in the sense of bilinear forms.
By the choice of the North and South poles, we know that the distance between the two poles and $ \overline{\{{\mathbf{X}}_G(t,\mathbf{x})|\mathbf{x}\in\overline{D_{\varepsilon_1}^1},t\in\mathbb{R}\}}$ is positive. This implies that there exists $C>1$ such that
\begin{align*}
C^{-1}<{1\over {\sqrt{1-\hat s^2}}},{\sqrt{1-\hat s^2}},{1\over {\sqrt{1-\hat s_L(t)^2}}},{\sqrt{1-\hat s_L(t)^2}}<C
\end{align*}
uniformly for $\mathbf{x}=\hat\zeta^{-1}((\hat \varphi,\hat s))\in\overline{D_{\varepsilon_1}^1}$ and $t\in\mathbb{R}$,
where $(\hat \varphi_L(t),\hat s_L(t))=\hat\zeta({\mathbf{X}}_G(t,\mathbf{x}))$. Define $|\mathcal{J}(t,\hat\varphi,\hat s)|\triangleq
\left(| \partial_{\hat s} \hat s_L(t,\hat \varphi,\hat s)|^2+|\partial_{\hat\varphi} \hat s_L(t,\hat\varphi, \hat s)|^2 +|\partial_{\hat s}\hat \varphi_L(t,\hat\varphi,\hat s)|^2+|\partial_{\hat\varphi}{\hat\varphi}_L(t,\hat\varphi,\hat s)|^2\right)^{1\over2}$ for $\mathbf{x}=\hat\zeta^{-1}((\hat \varphi,\hat s))\in\mathbb{S}^2\setminus\{\mathbf{z}_1,-\mathbf{z}_1\}$ and $t\in\mathbb{R}$.
 Then
\begin{align*}
&\left(g_{\mathbf{x}}((d{\mathbf{X}}_G(t,\mathbf{x}))^*\xi,(d{\mathbf{X}}_G(t,\mathbf{x}))^*\xi)\right)^{1\over2}\leq C \left(\hat g_{\mathbf{x}}((d{\mathbf{X}}_G(t,\mathbf{x}))^*\xi,(d{\mathbf{X}}_G(t,\mathbf{x}))^*\xi)\right)^{1\over2}\\
=&C\bigg(\left( \begin{array}{cc} {\sqrt{1-\hat s^2}}&0\\0&{1\over\sqrt{1-\hat s^2}}  \end{array} \right)(\mathcal{J}(t,\hat\varphi,\hat s))^*\left( \begin{array}{cc} {1\over{1-\hat s_L^2}}&0\\0&{1-\hat s_L^2}  \end{array} \right)\left( \begin{array}{cc} \xi^1\\ \xi^2\end{array} \right)\cdot\\
&\left( \begin{array}{cc} {\sqrt{1-\hat s^2}}&0\\0&{1\over\sqrt{1-\hat s^2}}  \end{array} \right)(\mathcal{J}(t,\hat\varphi,\hat s))^*\left( \begin{array}{cc} {1\over{1-\hat s_L^2}}&0\\0&{1-\hat s_L^2}  \end{array} \right)\left( \begin{array}{cc} \xi^1\\ \xi^2\end{array} \right)\bigg)^{1\over2}\\
\leq&C\bigg((\mathcal{J}(t,\hat\varphi,\hat s))^*\left( \begin{array}{cc} \xi^1\\ \xi^2\end{array} \right)\cdot(\mathcal{J}(t,\hat\varphi,\hat s))^*\left( \begin{array}{cc} \xi^1\\ \xi^2\end{array} \right)\bigg)^{1\over2}
\leq C|\mathcal{J}(t,\hat\varphi,\hat s)|(|\xi^1|^2+|\xi^2|^2)^{1\over2}\\
\leq& C|\mathcal{J}(t,\hat\varphi,\hat s)|\left(\hat g_{{\mathbf{X}}_G(t,\mathbf{x})}(\xi,\xi)\right)^{1\over2}
\leq C|\mathcal{J}(t,\hat\varphi,\hat s)|\left( g_{{\mathbf{X}}_G(t,\mathbf{x})}(\xi,\xi)\right)^{1\over2}
\end{align*}
for any $\xi=\xi^1\partial_{\hat s_L}+\xi^2\partial_{\hat\varphi_L}\in T_{{\mathbf{X}}_G(t,\mathbf{x})}\mathbb{S}^2$, where the constant $C$ is independent of $\mathbf{x}=\hat\zeta^{-1}((\hat \varphi,\hat s))\in\overline{D_{\varepsilon_1}^1}$ and $t\in\mathbb{R}$.
Thus, by \eqref{def-tangent-mapping-norm} we have
\begin{align*}
|(d{\mathbf{X}}_G(t,\mathbf{x}))^*|_{op}\leq C|\mathcal{J}(t,\hat\varphi,\hat s)|
\end{align*}
uniformly for $\mathbf{x}=\hat\zeta^{-1}((\hat \varphi,\hat s))\in\overline{D_{\varepsilon_1}^1}$ and $t\in\mathbb{R}$.
In the geographic coordinates $(\hat\varphi,\hat s)$, by \eqref{the trajectory induced by the steady velocity field}  we know that the particle trajectory $\hat\zeta({\mathbf{X}}_G(t,\mathbf{x}))=(\hat\varphi_L(t),\hat s_L(t))$ satisfies
\begin{align}\label{the trajectory induced by the steady velocity field2}
\left\{ \begin{array}{rcl} {d\hat\varphi_L(t)\over dt}&=&-\partial_{\hat s}\psi_G(\hat\varphi_L(t),\hat s_L(t)),  \\
{d\hat s_L(t)\over dt}&=&\partial_{\hat \varphi}\psi_G(\hat\varphi_L(t),\hat s_L(t)),\\
(\hat\varphi_L(0),\hat s_L(0))&=&(\hat\varphi,\hat s)\end{array} \right.
\end{align}
for $\mathbf{x}=\hat\zeta^{-1}((\hat \varphi,\hat s))\in\overline{D_{\varepsilon_1}^1}$ and $t\in\mathbb{R}$.
Applying (2.1) in \cite{lin2004} to the trajectory $(\hat\varphi_L(t),\hat s_L(t))$ ruled by \eqref{the trajectory induced by the steady velocity field2}  in  $\overline{D_{\varepsilon_1}^1}$, we have
\begin{align*}
&\int_{\overline{D_{\varepsilon_1}^1}}|(d{\mathbf{X}}_G(t,\mathbf{x}))^*|_{op}d\sigma_{ g}\leq C\int_{\overline{D_{\varepsilon_1}^1}}|(d{\mathbf{X}}_G(t,\mathbf{x}))^*|_{op}d\sigma_{\hat g}\\
\leq& C\int_{\hat\zeta(\overline{D_{\varepsilon_1}^1})}\left|\mathcal{J}(t,\hat \varphi,\hat s)\right|d\hat sd\hat \varphi\leq C|t|+C
\end{align*}
for $t\in\mathbb{R}$.
For the other nondegenerate saddle points $\mathbf{x}_i$, $2\leq i\leq q$, we can similarly define  $D_{\varepsilon_i}^i$ as above for $\varepsilon_i>0$ small enough, and prove that
\begin{align*}
&\int_{\overline{D_{\varepsilon_i}^i}}|(d{\mathbf{X}}_G(t,\mathbf{x}))^*|_{op}d\sigma_{ g}\leq C|t|+C
\end{align*}
for $t\in\mathbb{R}$.
Since $\sup_{\mathbf{x}\in\mathbb{S}^2\setminus\cup_{i=1}^q\overline{D_{\varepsilon_i}^i}}\lim_{|t|\to\infty}{1\over |t|}\ln\left(|(d{\mathbf{X}}_G(t,\mathbf{x}))^*|_{op}\right)=0$, for any $\alpha>0$ there exists $C_\alpha>0$ such that $|(d{\mathbf{X}}_G(t,\mathbf{x}))^*|_{op}\leq C_\alpha e^{\alpha |t|}$ uniformly for $\mathbf{x}\in\mathbb{S}^2\setminus\cup_{i=1}^q\overline{D_{\varepsilon_i}^i}$ and $t\in\mathbb{R}$.
Then
\begin{align*}
\int_{\mathbb{S}^2}|(d{\mathbf{X}}_G(t,\mathbf{x}))^*|_{op}d\sigma_g=&\int_{\cup_{i=1}^q\overline{D_{\varepsilon_i}^i}}|(d{\mathbf{X}}_G(t,\mathbf{x}))^*|_{op}d\sigma_g+
\int_{\mathbb{S}^2\setminus\cup_{i=1}^q\overline{D_{\varepsilon_i}^i}}|(d{\mathbf{X}}_G(t,\mathbf{x}))^*|_{op}d\sigma_g\\
\leq& C|t|+C+C_\alpha |\mathbb{S}^2| e^{\alpha |t|},
\end{align*}
which implies $\mu_{{\rm av}}=0$ by the arbitrary choice of $\alpha>0$.
\end{proof}
\subsection{Regularity of the unstable mode}
The linearized vorticity equation around $\Omega_G$ is
\begin{align*}
\partial_t\Omega+\nabla_{\mathbf{v}}\Omega_G+\nabla_{{\mathbf{u}}_G}\Omega+\nabla_{\mathbf{v}}(2\omega\chi)=0,
\end{align*}
and the linearized velocity equation around ${\mathbf{u}}_G$ is
\begin{align*}
\partial_t\mathbf{v}+\nabla_{\mathbf{v}}{\mathbf{u}}_G+\nabla_{{\mathbf{u}}_G}\mathbf{v}+2\omega\chi J\mathbf{v}+\nabla p=0,
\end{align*}
where $\chi(\mathbf{x})=\mathbf{e}_3\cdot\nu(\mathbf{x})=z$, and $\nu(\mathbf{x})$ is the unit outward pointing normal to $\mathbb{S}^2$ at $\mathbf{x}$.
The linearized vorticity operator and the linearized velocity operator are denoted by
\begin{align}\label{linearized vorticity operator M0}
{\mathcal{L}}_G\Omega=-\nabla_{\mathbf{v}}\Omega_G-\nabla_{{\mathbf{u}}_G}\Omega-\nabla_{\mathbf{v}}(2\omega\chi)
\end{align}
with $\mathbf{v}=curl^{-1}\Omega=J\nabla\Delta^{-1}\Omega$, and
\begin{align*}
\mathcal{M}_G\mathbf{v}=-\nabla_{\mathbf{v}}{\mathbf{u}}_G-\nabla_{{\mathbf{u}}_G}\mathbf{v}-2\omega\chi J\mathbf{v}-\nabla p,
\end{align*}
respectively. Then we improve regularity of the unstable mode in the next lemma. This is an extension of the planar counterpart  in Theorem 1 (i) of \cite{lin2004} to the  case of sphere.

\begin{Lemma}\label{lem-regularity of eigenfunction}
For a $C^1$ steady flow $\mathbf{u}_{G}$ with finite
 stagnation points,, if the linearized velocity operator  $\mathcal{M}_G$ has an unstable eigenvalue $\lambda_1$ (${\rm Re}(\lambda_1)>0$) with an eigenfunction $\mathbf{v}_1\in L^2(T\mathbb{S}^2)$, then
$\Omega_1\in H_1^p\cap L^{p_1}(\mathbb{S}^2)$ for $1\leq p<b_0$ and $ p_1\geq1$, where $\Omega_1=curl\mathbf{v}_1$, and $b_0$ is defined in Theorem \ref{thm-nonlinear instability}.
\end{Lemma}
\begin{proof}
Since the fluid is incompressible, we have $curl \mathcal{M}_G={\mathcal{L}}_G curl$. Thus,
$\lambda_1\Omega_1=curl(\lambda_1\mathbf{v}_1)=curl\mathcal{M}_G\mathbf{v}_1={\mathcal{L}}_G curl\mathbf{v}_1={\mathcal{L}}_G\Omega_1$.
Then $\lambda_1$ is an eigenvalue of ${\mathcal{L}}_G$ with an eigenfunction $\Omega_1$.
In the vorticity form, the eigenpair $(\lambda_1,\Omega_1)$ satisfies
\begin{align}\label{vorticity-eigenfunction equation}
\lambda_1\Omega_1+\nabla_{{\mathbf{u}}_G}\Omega_1=-\nabla_{\mathbf{v}_1}\Omega_G-\nabla_{\mathbf{v}_1}(2\omega\chi).
\end{align}
Let $t>0$. Multiplying \eqref{vorticity-eigenfunction equation} by $e^{\lambda_1t}$, integrating  \eqref{vorticity-eigenfunction equation} along ${\mathbf{X}}_G(t)$  from $-\infty$ to $0$ and then changing the time variable $t$ to $-t$, we have
\begin{align}\label{Omega1}
\Omega_1(\mathbf{x})=-\int_0^\infty\left(e^{-\lambda_1t}\nabla_{\mathbf{v}_1}\Omega_G({\mathbf{X}}_G(-t))
+e^{-\lambda_1t}\nabla_{\mathbf{v}_1}(2\omega\chi({\mathbf{X}}_G(-t)))\right)dt.
\end{align}

We claim that for any $f\in L^1(\mathbb{S}^2)$,
 \begin{align}\label{integration change variableX0}
 \int_{\mathbb{S}^2}f({\mathbf{X}}_G(t,\mathbf{x}))d\sigma_g= \int_{\mathbb{S}^2}f(\mathbf{x})d\sigma_g,
 \end{align}
 where $t\in\mathbb{R}$.
In fact,
since ${\mathbf{u}}_G$ is  solenoidal,  we have $\det(d{\mathbf{X}}_G(t,\mathbf{x}))=1$, where $\det(d{\mathbf{X}}_G(t,\mathbf{x}))$ can be calculated  by local coordinates as follows but is independent of the coordinates. For example,
WLOG, we assume that $\mathbf{x}=\zeta^{-1}((\varphi,s))\in \mathbb{S}^2\setminus \Gamma$ and ${\mathbf{X}}_G(t,\mathbf{x})=\tilde\zeta^{-1}((\tilde \varphi_L(t,\varphi,s),\tilde s_L(t,$ $\varphi,s)))\in \mathbb{S}^2\setminus \tilde\Gamma$, where $\zeta$ and $\tilde \zeta$ are defined in \eqref{chart1} and \eqref{chart2}.
For $\xi=\xi^1\partial_s+\xi^2\partial_\varphi\in T_{\mathbf{x}}\mathbb{S}^2$ and $d{\mathbf{X}}_G(t,\mathbf{x})\xi=(d{\mathbf{X}}_G(t,\mathbf{x})\xi)^1\partial_{\tilde s_L}+(d{\mathbf{X}}_G(t,\mathbf{x})\xi)^2\partial_{\tilde \varphi_L}\in T_{{\mathbf{X}}_G(t,\mathbf{x})}\mathbb{S}^2$, by Lemma 1.29 in \cite{Grigoryan2024} we have
\begin{align*}
\left( \begin{array}{cc} (d{\mathbf{X}}_G(t,\mathbf{x})\xi)^1 \\  (d{\mathbf{X}}_G(t,\mathbf{x})\xi)^2 \end{array} \right)=\left( \begin{array}{cc}
 \partial_{s} \tilde s_L(t, \varphi, s)&\partial_{\varphi} \tilde s_L(t,\varphi,  s) \\  \partial_{ s}\tilde \varphi_L(t,\varphi, s) &\partial_{\varphi}{\tilde\varphi}_L(t,\varphi, s) \end{array} \right)\left( \begin{array}{cc} \xi^1\\ \xi^2 \end{array} \right)=
 \mathcal{\tilde J}(t, \varphi,s)\left( \begin{array}{cc} \xi^1\\ \xi^2 \end{array} \right).
\end{align*}
 Then
$\det(d{\mathbf{X}}_G(t,\mathbf{x}))=\det(\mathcal{\tilde J}(t, \varphi,s))=1$.
Thus, for any $f\in L^1(\mathbb{S}^2)$,
\begin{align*}
 \int_{\mathbb{S}^2}f(\mathbf{y})d\sigma_{g_\mathbf{y}}
 = \int_{{\mathbf{X}}_G(t,\mathbb{S}^2)}f({\mathbf{X}}_G(t,\mathbf{x}))|\det(d{\mathbf{X}}_G(t,\mathbf{x}))|d\sigma_{g_\mathbf{x}}= \int_{\mathbb{S}^2}f({\mathbf{X}}_G(t,\mathbf{x}))d\sigma_{g_\mathbf{x}},
\end{align*}
where $\mathbf{y}={\mathbf{X}}_G(t,\mathbf{x})$. This proves \eqref{integration change variableX0}.

Since $\mathbf{v}_1\in L^2(T\mathbb{S}^2)$, by \eqref{Omega1}-\eqref{integration change variableX0} we have
\begin{align}\label{vorticityL2}
\|\Omega_1\|_{L^2(\mathbb{S}^2)}\leq&\int_0^\infty e^{-{\rm Re}(\lambda_1)t}\|\nabla_{\mathbf{v}_1}\Omega_G({\mathbf{X}}_G(-t))
+\nabla_{\mathbf{v}_1}(2\omega\chi({\mathbf{X}}_G(-t)))\|_{L^2(\mathbb{S}^2)}dt\\\nonumber
\leq&\|\Omega_G+2\omega\chi\|_{C^1(\mathbb{S}^2)}\int_0^\infty e^{-{\rm Re}(\lambda_1)t}\|\mathbf{v}_1\|_{L^2(T\mathbb{S}^2)}dt\\\nonumber
=&{\|\Omega_G+2\omega\chi\|_{C^1(\mathbb{S}^2)}\over{\rm Re}(\lambda_1)}\|\mathbf{v}_1\|_{L^2(T\mathbb{S}^2)}.
\end{align}
Since $\Omega_1\in L^2(\mathbb{S}^2)$, by Lemma 3.1 in \cite{Cao-Wang-Zuo2023} we have $\Psi_1=\Delta^{-1}\Omega_1\in H_2^2(\mathbb{S}^2)$. By Lemma \ref{lem-differential calculus} (i), we have $\Psi_1\in H_1^{p_1}(\mathbb{S}^2)$ for $1\leq p_1<\infty$.
By a similar estimate in \eqref{vorticityL2}, we have
\begin{align*}
\|\Omega_1\|_{L^{p_1}(\mathbb{S}^2)}
\leq{\|\Omega_G+2\omega\chi\|_{C^1(\mathbb{S}^2)}\over{\rm Re}(\lambda_1)}\|\mathbf{v}_1\|_{L^{p_1}(T\mathbb{S}^2)}\leq{\|\Omega_G+2\omega\chi\|_{C^1(\mathbb{S}^2)}\over{\rm Re}(\lambda_1)}\|\Psi_1\|_{H_1^{p_1}(\mathbb{S}^2)}.
\end{align*}
Thus, $\Psi_1\in H_2^{p_1}(\mathbb{S}^2)$  for $1<p_1<\infty$ by Lemma 3.1 in \cite{Cao-Wang-Zuo2023}, and thus also for $p_1=1$. By Kato's inequality, we have $|\nabla|\nabla\Psi_1||\leq |\nabla^2\Psi_1|$, and thus,
\begin{align}\label{v1H1p1}
\|\mathbf{v}_1\|_{H_1^{p_1}(T\mathbb{S}^2)}=\|J\nabla\Psi_1\|_{H_1^{p_1}(T\mathbb{S}^2)}\leq \|\Psi_1\|_{H_2^{p_1}(\mathbb{S}^2)}.
\end{align}
By \eqref{Omega1}, we have
\begin{align}\label{grad-Omega1}
&\nabla\Omega_1=-\int_0^\infty\left(e^{-\lambda_1t}\nabla(\nabla_{\mathbf{v}_1}\Omega_G\circ{\mathbf{X}}_G(-t))+2\omega e^{-\lambda_1t}\nabla(\nabla_{\mathbf{v}_1}\chi\circ{\mathbf{X}}_G(-t))\right)dt\\\nonumber
=-\int_0^\infty&\left(e^{-\lambda_1t}(d{\mathbf{X}}_G(-t))^*\nabla\nabla_{\mathbf{v}_1}\Omega_G({\mathbf{X}}_G(-t))+2\omega e^{-\lambda_1t}(d{\mathbf{X}}_G(-t))^*\nabla\nabla_{\mathbf{v}_1}\chi({\mathbf{X}}_G(-t))\right)dt,
\end{align}
where $(d{\mathbf{X}}_G(-t))^*$ is adjoint of the tangent map $d{\mathbf{X}}_G(-t)$ at $\mathbf{x}$ with respect to the metric $g$.
First, we estimate $\|(d{\mathbf{X}}_G(-t))^*\nabla\nabla_{\mathbf{v}_1}\Omega_G({\mathbf{X}}_G(-t))\|_{L^p(T\mathbb{S}^2)}$.
For $p\in [1,b_0)$, choose $p_2\in(1,b_0)$ and $p_3\in(1,\infty) $ such that ${1\over p}={1\over p_2}+{1\over p_3}$. Then by \eqref{def-tangent-mapping-norm} we have
\begin{align}\label{estimate dX0gradgrad}
&\|(d{\mathbf{X}}_G(-t))^*\nabla\nabla_{\mathbf{v}_1}\Omega_G({\mathbf{X}}_G(-t))\|_{L^p(T\mathbb{S}^2)}\\\nonumber
\leq&\left\||(d{\mathbf{X}}_G(-t))^*|_{op}\left(g_{{\mathbf{X}}_G(-t)}(\nabla\nabla_{\mathbf{v}_1}\Omega_G({\mathbf{X}}_G(-t)),
\nabla\nabla_{\mathbf{v}_1}\Omega_G({\mathbf{X}}_G(-t)))\right)^{1\over2}\right\|_{L^p(\mathbb{S}^2)}\\\nonumber
\leq&\bigg\||(d{\mathbf{X}}_G(-t))^*|_{op}\bigg\|_{L^{p_2}(\mathbb{S}^2)}
\left\|\left(g_{{\mathbf{X}}_G(-t)}(\nabla\nabla_{\mathbf{v}_1}\Omega_G({\mathbf{X}}_G(-t)),
\nabla\nabla_{\mathbf{v}_1}\Omega_G({\mathbf{X}}_G(-t)))\right)^{1\over2}\right\|_{L^{p_3}(\mathbb{S}^2)}.
\end{align}
Note that $\mathbb{S}^2$ can be covered by the two charts $(M_1,\zeta_1)\triangleq(\zeta^{-1}((-\pi,\pi)\times(-1+\kappa_0,1-\kappa_0)),\zeta)$ and $ (M_2,\zeta_2)\triangleq (\tilde\zeta^{-1}((-\pi,\pi)\times(-1+\kappa_0,1-\kappa_0)),\tilde\zeta)$ for $\kappa_0>0$ small enough.
The components $g_{ij}^m$ of $g$ in $(M_m,\zeta_m)$ satisfy $d_1\delta_{ij}\leq g_{ij}^m\leq d_2\delta_{ij}$, $m=1,2$, as bilinear forms for some $d_2>d_1>0$.
Let $(\eta_m)$, $m=1,2$, be a smooth partition of unity subordinate to the covering $\{M_m\}_{m=1}^1$. Then by \eqref{integration change variableX0} and \eqref{v1H1p1} we have
\begin{align}\label{estimate dX0gradgrad1}
&\left\|\left(g_{{\mathbf{X}}_G(-t)}(\nabla\nabla_{\mathbf{v}_1}\Omega_G({\mathbf{X}}_G(-t)),
\nabla\nabla_{\mathbf{v}_1}\Omega_G({\mathbf{X}}_G(-t)))\right)^{1\over2}\right\|_{L^{p_3}(\mathbb{S}^2)}\\\nonumber
=&\|\nabla\nabla_{\mathbf{v}_1}\Omega_G(\mathbf{x})\|_{L^{p_3}(T\mathbb{S}^2)}
\leq C_{d_1,d_2}\sum_{m=1}^2\|\nabla((\eta_m
\nabla_{\mathbf{v}_1}\Omega_G)\circ\zeta_m^{-1})\|_{L^{p_3}(\zeta_m(M_m))}\\\nonumber
\leq& C_{d_1,d_2}\sum_{m=1}^2\|\nabla(\eta_m\circ\zeta_m^{-1})((
\nabla_{\mathbf{v}_1}\Omega_G)\circ\zeta_m^{-1})\|_{L^{p_3}(\zeta_m(M_m))}\\\nonumber
&+ C_{d_1,d_2}\sum_{m=1}^2\|\eta_m\circ\zeta_m^{-1}\nabla(\nabla_{\mathbf{v}_1}\Omega_G\circ\zeta_m^{-1})\|_{L^{p_3}(\zeta_m(M_m))}\\\nonumber
\leq& C_{d_1,d_2}\left(\|
\nabla_{\mathbf{v}_1}\Omega_G\|_{L^{p_3}(\mathbb{S}^2)}+\|\mathbf{v}_1\|_{H_1^{p_3}(T\mathbb{S}^2)}\|\Omega_G\|_{C^1(\mathbb{S}^2)}
+\|\mathbf{v}_1\|_{L^{p_3}(T\mathbb{S}^2)}\|\Omega_G\|_{C^2(\mathbb{S}^2)}\right)\\\nonumber
\leq&C_{d_1,d_2}\left(\|\Psi_1\|_{H_1^{p_3}(\mathbb{S}^2)}
\|\Omega_G\|_{C^1(\mathbb{S}^2)}+\|\Psi_1\|_{H_2^{p_3}(\mathbb{S}^2)}\|\Omega_G\|_{C^1(\mathbb{S}^2)}
+\|\Psi_1\|_{H_1^{p_3}(\mathbb{S}^2)}\|\Omega_G\|_{C^2(\mathbb{S}^2)}\right)\\\nonumber
\leq&C_{d_1,d_2}\|\Psi_1\|_{H_2^{p_3}(\mathbb{S}^2)}\|\Omega_G\|_{C^2(\mathbb{S}^2)}.
\end{align}
By the definition of $b_0$ and $p_2\in(1,b_0)$, we have $\epsilon_1\triangleq{\rm Re}(\lambda_1)-\mu\left(1-{1\over p_2}\right)>0$. Let $\epsilon_2,\epsilon_3>0$ such that $\epsilon_2+\epsilon_3<{\epsilon_1\over2}$. By \eqref{def-Lyapunov exponent}-\eqref{def-averaging Lyapunov exponent} and Lemma \ref{averaging Lyapunov exponent is zero}, there exist $C_{\epsilon_1}, C_{\epsilon_2}>0$ such that
\begin{align*}
\bigg||(d{\mathbf{X}}_G(-t,\mathbf{x}))^*|_{op}\bigg|\leq C_{\epsilon_2} e^{(\mu+\epsilon_2)t},
\end{align*}
and
\begin{align*}
\int_{\mathbb{S}^2}\bigg||(d{\mathbf{X}}_G(-t,\mathbf{x}))^*|_{op}\bigg|d\sigma_g\leq C_{\epsilon_3} e^{\epsilon_3t}
\end{align*}
for $t>0$.
Thus,
\begin{align}\label{estimate dX0gradgrad2}
\bigg\||(d{\mathbf{X}}_G(-t))^*|_{op}\bigg\|_{L^{p_2}(\mathbb{S}^2)}
\leq(C_{\epsilon_2} e^{(\mu+\epsilon_2)t})^{1-{1\over p_2}}(C_{\epsilon_3} e^{\epsilon_3t})^{{1\over p_2}}\leq C_{\ep_2,\ep_3}e^{({\rm Re}(\lambda_1)-{\epsilon_1\over 2})t}.
\end{align}
By \eqref{estimate dX0gradgrad}-\eqref{estimate dX0gradgrad2}, we have
\begin{align*}
\|(d{\mathbf{X}}_G(-t))^*\nabla\nabla_{\mathbf{v}_1}\Omega_G({\mathbf{X}}_G(-t))\|_{L^p(T\mathbb{S}^2)}\leq C_{d_1,d_2,\ep_2,\ep_3}e^{({\rm Re}(\lambda_1)-{\epsilon_1\over 2})t}\|\Psi_1\|_{H_2^{p_3}(\mathbb{S}^2)}\|\Omega_G\|_{C^2(\mathbb{S}^2)}.
\end{align*}
Similar to \eqref{estimate dX0gradgrad}-\eqref{estimate dX0gradgrad2}, we have
\begin{align*}
\|2\omega(d{\mathbf{X}}_G(-t))^*\nabla\nabla_{\mathbf{v}_1}\chi({\mathbf{X}}_G(-t))\|_{L^p(T\mathbb{S}^2)}\leq C_{\omega,d_1,d_2,\ep_2,\ep_3}e^{({\rm Re}(\lambda_1)-{\epsilon_1\over 2})t}\|\Psi_1\|_{H_2^{p_3}(\mathbb{S}^2)}\|\chi\|_{C^2(\mathbb{S}^2)}.
\end{align*}
Combining the above two inequalities and \eqref{grad-Omega1}, we have
\begin{align*}
\|\nabla\Omega_1\|_{L^p(T\mathbb{S}^2)}\leq&\int_0^\infty e^{-{\rm Re}(\lambda_1)t}\bigg(\|(d{\mathbf{X}}_G(-t))^*\nabla\nabla_{\mathbf{v}_1}\Omega_G({\mathbf{X}}_G(-t))\|_{L^p(T\mathbb{S}^2)}\\
&+\|2\omega(d{\mathbf{X}}_G(-t))^*\nabla\nabla_{\mathbf{v}_1}\chi({\mathbf{X}}_G(-t))\|_{L^p(T\mathbb{S}^2)}\bigg)dt\\
\leq&C_{\omega,d_1,d_2,\ep_1,\ep_2,\ep_3}\|\Psi_1\|_{H_2^{p_3}(\mathbb{S}^2)}(\|\Omega_G\|_{C^2(\mathbb{S}^2)}+\|\chi\|_{C^2(\mathbb{S}^2)}).
\end{align*}
Thus, $\Omega_1\in H_1^{p}(\mathbb{S}^2)$.
\end{proof}
\subsection{Proof of nonlinear instability of general steady flows}
 We  prove that linear instability implies
  nonlinear instability  for general steady flows.
  Orbital instability will be  discussed in the next subsection.

\begin{Theorem}[Linear to nonlinear instability of general steady flows] \label{thm-nonlinear instability-general steady flows} 
 For a $C^1$ steady flow $\mathbf{u}_{G}$ with finite
 stagnation points,
if it is linearly unstable in $L^2(T\mathbb{S}^2)$, then it is nonlinearly  unstable in the sense that
there exists $\epsilon_0>0$ such that for any $\delta>0$, there exists a solution $ \mathbf{u}_{\delta,G}$ to the nonlinear Euler equation and $t_{0}=O(|\ln\delta|)$ satisfying
\begin{align*}
\|\Omega_{\delta,G}(0)-\Omega_{G}\|_{L^{p_2}(\mathbb{S}^2)}+\|\nabla( \Omega_{\delta,G}(0)-\Omega_{G})\|_{L^{p_1}(T\mathbb{S}^2)}\leq \delta\quad\text{and}\quad
\|{\mathbf{u}}_{\delta,G}(t_0)-{\mathbf{u}}_G\|_{L^{p_0}(T\mathbb{S}^2)}\geq\epsilon_0,
\end{align*}
where  $\Omega_{G}=curl(\mathbf{u}_{G})$, $\Omega_{\delta,G}=curl(\mathbf{u}_{\delta,G})$, $p_0\in(1,\infty)$, $p_1\in[1,b_0)$, $p_2\in[1,\infty)$,
and
 $b_0$ is defined in Theorem \ref{thm-nonlinear instability}.
\end{Theorem}

\begin{proof}
Since $\mathbb{S}^2$ is compact, we have $L^{q_2}(T\mathbb{S}^2)\subset L^{q_1}(T\mathbb{S}^2)$ for $1\leq q_1<q_2$, and thus, it suffices to prove Theorem \ref{thm-nonlinear instability-general steady flows} for $p_0>1$ sufficiently close to $1$. Let   $\mathbf{v}_1\in L^2(T\mathbb{S}^2)$ be an eigenfunction of $\lambda_1$. Since $\mathbf{v}_1\in L^2(T\mathbb{S}^2)$,  by Lemma \ref{lem-regularity of eigenfunction} we have that $\Omega_1=curl\mathbf{v}_1\in  H_1^{p_1}\cap L^{p_2}(\mathbb{S}^2)$ for $1\leq p_1<b_0$ and $ p_2\geq1$. Let $\Omega_1=\Omega_1^r+i\Omega_1^i$ and $\mathbf{v}_1=\mathbf{v}_1^r+i\mathbf{v}_1^i$, where  $ \Omega_1^r$, $\mathbf{v}_1^r$ and $\Omega_1^i$, $\mathbf{v}_1^i$ are the real and imaginary parts of $\Omega_1$, $\mathbf{v}_1$, respectively.
We normalize $\Omega_1^r$ such that $\|\Omega_1^r\|_{L^{p_2}(\mathbb{S}^2)}+\|\nabla\Omega_1^r\|_{L^{p_1}(T\mathbb{S}^2)}=1$.
Here, the eigenfunction  $\Omega_1$ is chosen to be real-valued if $\lambda_1\in\mathbb{R}$.
Let $\Omega_{\delta,G}(t)$ and ${\mathbf{u}}_{\delta,G} (t)$ be the perturbed vorticity and velocity solving the nonlinear Euler equation  with the initial data $ \Omega_{\delta,G}(0)=\Omega_G+\delta\Omega_1^r$ and ${\mathbf{u}}_{\delta,G}(0)={\mathbf{u}}_G+\delta\mathbf{v}_1^r$, where $\delta\in(0,1)$ and $t>0$. Let $\Omega_\delta(t)=\Omega_{\delta,G}(t)-\Omega_G$, $ { \mathbf{u}}_\delta(t)={\mathbf{u}}_{\delta,G}(t)-{\mathbf{u}}_G$ and $\Psi_\delta(t)=\Delta^{-1}\Omega_\delta(t)$ be the perturbation of vorticity, velocity and stream function, respectively.  Then
$ { \mathbf{u}}_\delta(t)=curl^{-1}\Omega_\delta(t)$ and $\Omega_\delta(t)$ solves
\begin{align}\label{perturbation equation}
\partial_t\Omega_\delta=-\nabla_{{ \mathbf{u}}_\delta}\Omega_G-\nabla_{{\mathbf{u}}_G}\Omega_\delta-\nabla_{ { \mathbf{u}}_\delta}(2\omega\chi)-\nabla_{ { \mathbf{u}}_\delta}\Omega_\delta={\mathcal{L}}_G\Omega_\delta-\nabla_{{ \mathbf{u}}_\delta}\Omega_\delta,
\end{align}
where the linearized vorticity operator ${\mathcal{L}}_G$ is defined in \eqref{linearized vorticity operator M0}. By Duhamel's principle, we have
\begin{align*}
\Omega_\delta(t)=e^{t{\mathcal{L}}_G}\Omega_\delta(0)-\int_0^te^{(t-\tau) {\mathcal{L}}_G} (\nabla_{ { \mathbf{u}}_\delta}\Omega_\delta)(\tau)d\tau=\Omega_{\delta,L}(t)+\Omega_{\delta,N}(t)
\end{align*}
and
\begin{align*}
\mathbf{u}_\delta(t)=curl^{-1}\Omega_\delta(t)
=\nabla^{\bot}\Delta^{-1}\Omega_\delta(t)=\nabla^{\bot}\Delta^{-1}\Omega_{\delta,L}(t)+\nabla^{\bot}\Delta^{-1}\Omega_{\delta,N}(t)
=\mathbf{u}_{\delta,L}(t)+\mathbf{u}_{\delta,N}(t).
\end{align*}
By \eqref{v1H1p1}, $c_1\triangleq2\|\mathbf{v}_1\|_{L^{p_0}(T\mathbb{S}^2)}<\infty$.
We define
\begin{align}\label{def-T}
T\triangleq\sup\{t\geq0|\|\mathbf{u}_\delta(t)\|_{L^{p_0}(T\mathbb{S}^2)}\leq c_1\delta e^{{\rm Re}(\lambda_1) t}\}.
\end{align}
Then $\|\mathbf{u}_\delta(0)\|_{L^{p_0}(T\mathbb{S}^2)}=\delta\|\mathbf{v}_1^r\|_{L^{p_0}(T\mathbb{S}^2)}<c_1\delta$, and thus, $T>0$.
We divide into three steps to estimate $\|\mathbf{u}_{\delta,N}(t)\|_{L^{p_0}(T\mathbb{S}^2)}$.\\
{\bf Step 1.} Estimate $\|\mathbf{u}_{\delta,N}(t)\|_{L^{p_0}(T\mathbb{S}^2)}$ by duality for $0\leq t<T$.

First, we show that $\Omega_{\delta,N}(t)=curl\mathbf{u}_{\delta,N}(t)\in L^{p_0}(\mathbb{S}^2)$ for $0\leq t< T$. We rewrite \eqref{perturbation equation} as
\begin{align*}
\partial_t\Omega_\delta+\nabla_{{\mathbf{u}}_G+{ \mathbf{u}}_\delta}\Omega_\delta=-\nabla_{{ \mathbf{u}}_\delta}(\Omega_G+2\omega\chi).
\end{align*}
Let $\hat{\mathbf{X}}(t,\mathbf{x})$ be the trajectory induced by the perturbed negative velocity field $-({\mathbf{u}}_G+{ \mathbf{u}}_\delta)$ satisfying
\begin{align*}
\left\{ \begin{array}{rcl} {d\hat{\mathbf{X}}\over dt}&=&-({\mathbf{u}}_G+{ \mathbf{u}}_\delta)(t,\hat{\mathbf{X}}(t,\mathbf{x})),  \\ \hat{\mathbf{X}}(0,\mathbf{x})&=&\mathbf{x}.\end{array} \right.
\end{align*}
Then $e^{-t\nabla_{{\mathbf{u}}_G+{ \mathbf{u}}_\delta}}\Omega_{\delta}(0,\mathbf{x})=\Omega_{\delta}(0,\hat{\mathbf{X}}(t,\mathbf{x}))$ and by Duhamel's principle we have
\begin{align}\label{Omega delta exp-tra-per-flow}
\Omega_\delta(t,\mathbf{x})=\Omega_{\delta}(0,\hat{\mathbf{X}}(t,\mathbf{x}))-\int_0^t(\nabla_{{ \mathbf{u}}_\delta}(\Omega_G+2\omega\chi))(\tau,\hat{\mathbf{X}}(t-\tau,\mathbf{x}))d\tau.
\end{align}
Thus, by \eqref{def-T}, for $0\leq t<T$ we have
\begin{align}\label{Omega-delta-t-L-p0}
\|\Omega_\delta(t)\|_{L^{p_0}(\mathbb{S}^2)}\leq\|\Omega_{\delta}(0)\|_{L^{p_0}(\mathbb{S}^2)}+\|\Omega_G+2\omega\chi\|_{C^1(\mathbb{S}^2)}\int_0^t\|{ \mathbf{u}}_\delta(\tau)\|_{L^{p_0}(T\mathbb{S}^2)} d\tau<\infty.
\end{align}
The real part of  $e^{t{\mathcal{L}}_G}\Omega_1=e^{t\lambda_1}\Omega_1$ is $e^{t{\mathcal{L}}_G}\Omega_1^r=e^{t{\rm Re(\lambda_1)}}(\cos(t {\rm Im}(\lambda_1))\Omega_1^r-\sin(t{\rm Im}(\lambda_1))\Omega_1^i)$. Thus,
\begin{align}\label{Omega-delta-t-linear-L-p0}
&\|\Omega_{\delta,L}(t)\|_{L^{p_0}(\mathbb{S}^2)}=\|e^{t{\mathcal{L}}_G}\Omega_\delta(0)\|_{L^{p_0}(\mathbb{S}^2)}=\delta\| e^{t{\mathcal{L}}_G}\Omega_1^r\|_{L^{p_0}(\mathbb{S}^2)}\\\nonumber
=&e^{t{\rm Re(\lambda_1)}}\delta\|(\cos(t {\rm Im}(\lambda_1))\Omega_1^r-\sin(t{\rm Im}(\lambda_1))\Omega_1^i)\|_{L^{p_0}(\mathbb{S}^2)}\leq \delta e^{t{\rm Re(\lambda_1)}}\|\Omega_1\|_{L^{p_0}(\mathbb{S}^2)}<\infty
 \end{align}
 for $0\leq t<T$. Combining \eqref{Omega-delta-t-L-p0} and \eqref{Omega-delta-t-linear-L-p0}, we have $\Omega_{\delta,N}(t)\in L^{p_0}(\mathbb{S}^2)$ for $0\leq t< T$.

Since $\Omega_{\delta,N}(t)\in L^{p_0}(\mathbb{S}^2)$, we have by Lemma \ref{lem-differential calculus} (vi) that
\begin{align*}
&\|\mathbf{u}_{\delta,N}(t)\|_{L^{p_0}(T\mathbb{S}^2)}=\|\nabla^{\bot}\Delta^{-1}\Omega_{\delta,N}(t)\|_{L^{p_0}(T\mathbb{S}^2)}\\
\leq&C_1\sup_{\psi\in H_1^{p_0'}(\mathbb{S}^2),\|\psi\|_{ H_1^{p_0'}(\mathbb{S}^2)}=1}\left|\int_{\mathbb{S}^2}\psi\Omega_{\delta,N}(t) d\sigma_g\right|\\
=&C_1\sup_{\psi\in H_1^{p_0'}(\mathbb{S}^2),\|\psi\|_{ H_1^{p_0'}(\mathbb{S}^2)}=1}\left|\int_{\mathbb{S}^2}\psi\int_0^te^{(t-\tau) {\mathcal{L}}_G} (\nabla_{ { \mathbf{u}}_\delta}\Omega_\delta)(\tau)d\tau d\sigma_g\right|\\
=&C_1\sup_{\psi\in H_1^{p_0'}(\mathbb{S}^2),\|\psi\|_{ H_1^{p_0'}(\mathbb{S}^2)}=1}\left| \int_0^t\int_{\mathbb{S}^2}\psi e^{(t-\tau) {\mathcal{L}}_G} {\rm div}( \mathbf{u}_\delta\Omega_\delta)(\tau)d\sigma_gd\tau \right|\\
=&C_1\sup_{\psi\in H_1^{p_0'}(\mathbb{S}^2),\|\psi\|_{ H_1^{p_0'}(\mathbb{S}^2)}=1}\left| \int_0^t\int_{\mathbb{S}^2}\nabla(e^{(t-\tau) {\mathcal{L}}_G'}\psi) \cdot ( \mathbf{u}_\delta\Omega_\delta)(\tau)d\sigma_gd\tau \right|\\
\leq&C_1\sup_{\psi\in H_1^{p_0'}(\mathbb{S}^2),\|\psi\|_{ H_1^{p_0'}(\mathbb{S}^2)}=1} \int_0^t\|\nabla(e^{(t-\tau) {\mathcal{L}}_G'}\psi)\|_{L^{p_3}(T\mathbb{S}^2)}  \|\mathbf{u}_\delta(\tau)\|_{L^{p_4}(T\mathbb{S}^2)}\|\Omega_\delta(\tau)\|_{L^{p_4}(\mathbb{S}^2)}d\tau
\end{align*}
for $0\leq t< T$, where ${\mathcal{L}}_G'$ is the dual operator of ${\mathcal{L}}_G$, $p_3\in(1,b_0)$, $p_4$ satisfies ${1\over p_3}+{2\over p_4}=1$ and we used  ${\rm div}( \mathbf{u}_\delta\Omega_\delta)=\nabla_{\mathbf{u}_\delta}\Omega_\delta+\Omega_\delta{\rm div}(\mathbf{u}_\delta)=\nabla_{\mathbf{u}_\delta}\Omega_\delta$ in the second equality since $\mathbf{u}_\delta$ is solenoidal.\\
{\bf Step 2}.  Estimate  $\|\nabla(e^{t {\mathcal{L}}_G'}\psi)\|_{L^{p_3}(T\mathbb{S}^2)}$ by degeneracy of the averaging Lyapunov exponent.

We will prove that there exists $C_4>0$ such that
\begin{align*}
\|\nabla (e^{t{\mathcal{L}}_G'}\psi)\|_{L^{p_3}(T\mathbb{S}^2)}
\leq C_4e^{{3\over2}{\rm Re}(\lambda_1)t}\|\psi\|_{H_1^{p_0'}(\mathbb{S}^2)}
\end{align*}
for $\psi\in H_1^{p_0'}(\mathbb{S}^2)$, where $0\leq t<T$.

By \eqref{linearized vorticity operator M0},
 ${\mathcal{L}}_G=-{\mathbf{u}}_G\cdot\nabla+(\nabla^\bot(\Omega_G+2\omega\chi))\cdot\nabla\Delta^{-1}$ and ${\mathcal{L}}_G'={\mathbf{u}}_G\cdot\nabla-\Delta^{-1}(\nabla^\bot(\Omega_G+2\omega\chi))\cdot\nabla$.
For $\psi\in H_1^{p_0'}(\mathbb{S}^2)$, $\tilde \psi(t)\triangleq e^{t{\mathcal{L}}_G'}\psi$ solves $\partial_t\tilde \psi={\mathbf{u}}_G\cdot\nabla\tilde \psi-\Delta^{-1}(\nabla^\bot(\Omega_G+2\omega\chi))\cdot\nabla\tilde \psi$. By Duhamel's principle and $e^{t{\mathbf{u}}_G\cdot\nabla}\psi(\mathbf{x})=\psi({\mathbf{X}}_G(t,\mathbf{x}))$, we have
\begin{align*}
\tilde \psi(t,\mathbf{x})=\psi({\mathbf{X}}_G(t,\mathbf{x}))-\int_0^t\Delta^{-1}(\nabla^\bot(\Omega_G+2\omega\chi)\cdot\nabla\tilde \psi)(\tau,{\mathbf{X}}_G(t-\tau,\mathbf{x}))d\tau,
\end{align*}
and
\begin{align}\label{grad-tilde psi t x}
\nabla\tilde \psi(t,\mathbf{x})=&(d{\mathbf{X}}_G(t))^*\nabla\psi({\mathbf{X}}_G(t,\mathbf{x}))\\\nonumber
&-\int_0^t(d{\mathbf{X}}_G(t-\tau))^*\nabla\Delta^{-1}(\nabla^\bot(\Omega_G+2\omega\chi)\cdot\nabla\tilde \psi)(\tau,{\mathbf{X}}_G(t-\tau,\mathbf{x}))d\tau.
\end{align}
Since ${\mathcal{L}}_G'$ is a compact perturbation of ${\mathbf{u}}_G\cdot\nabla$, ${\mathcal{L}}_G'$ generates an isometry group in  $L^{p_0'}(\mathbb{S}^2)$ and   $\sigma_{\rm ess}({\mathcal{L}}_G')\subset i\mathbb{R}$. Moreover, $\bar{\lambda}_1$ is an eigenvalue of ${\mathcal{L}}_G'$ with the largest real part. Thus,
\begin{align}\label{group estimate}
\|\tilde \psi(t)\|_{L^{p_0'}(\mathbb{S}^2)}=\|e^{t{\mathcal{L}}_G'}\psi\|_{L^{p_0'}(\mathbb{S}^2)}\leq C_2e^{{3\over 2}{\rm Re}(\lambda_1)t}\|\psi\|_{L^{p_0'}(\mathbb{S}^2)}
\end{align}
for some $C_2>0$. By Lemma \ref{lem-differential calculus} (v) and $\tilde \psi(t)\in L^{p_0'}(\mathbb{S}^2)$, we have
$\nabla^{\bot}\Delta^{-1}(\nabla^\bot(\Omega_G+2\omega\chi)\cdot\nabla)\tilde \psi(t)=J(\nabla\Delta^{-{1\over2}})(\Delta^{-{1\over2}}$ $\nabla^\bot(\Omega_G+2\omega\chi)\cdot\nabla)\tilde \psi(t)\in L^{p_0'}(T\mathbb{S}^2)\cap H_{p_0'}(\mathbb{S}^2)$. Moreover, we infer from \eqref{grad-tilde psi t x} and $\psi\in H_1^{p_0'}(\mathbb{S}^2)$ that
\begin{align*}
\|\nabla\tilde \psi(t)\|_{L^{p_0'}(T\mathbb{S}^2)}\leq&\|(d{\mathbf{X}}_G(t))^*\nabla\psi({\mathbf{X}}_G(t))\|_{L^{p_0'}(T\mathbb{S}^2)}\\\nonumber
&+\int_0^t\|(d{\mathbf{X}}_G(t-\tau))^*\nabla\Delta^{-1}(\nabla^\bot(\Omega_G+2\omega\chi)\cdot\nabla\tilde \psi)(\tau,{\mathbf{X}}_G(t-\tau))\|_{L^{p_0'}(T\mathbb{S}^2)}d\tau\\
\leq& C_t\|\psi\|_{H_1^{p_0'}(\mathbb{S}^2)},
\end{align*}
and thus,
$curl(\nabla^{\bot}\Delta^{-1}(\nabla^\bot(\Omega_G+2\omega\chi)\cdot\nabla)\tilde \psi(t))=(\nabla^\bot(\Omega_G+2\omega\chi)\cdot\nabla)\tilde \psi(t)\in L^{p_0'}(\mathbb{S}^2)$.
Since $\nabla^{\bot}\Delta^{-1}(\nabla^\bot(\Omega_G+2\omega\chi)\cdot\nabla)\tilde \psi(t)\in H_{p_0'}(\mathbb{S}^2)$ and $curl(\nabla^{\bot}\Delta^{-1}(\nabla^\bot(\Omega_G+2\omega\chi)\cdot\nabla)\tilde \psi(t))\in L^{p_0'}(\mathbb{S}^2)$,
by Lemma \ref{lem-differential calculus} (vi) we have
\begin{align}\label{integrand1}
&\|\nabla\Delta^{-1}(\nabla^\bot(\Omega_G+2\omega\chi)\cdot\nabla)\tilde \psi(t)\|_{L^{p_0'}(T\mathbb{S}^2)}\\\nonumber
=&\|\nabla^{\bot}\Delta^{-1}(\nabla^\bot(\Omega_G+2\omega\chi)\cdot\nabla)\tilde \psi(t)\|_{L^{p_0'}(T\mathbb{S}^2)}\\\nonumber
\leq&C_1\sup_{\phi\in H_1^{p_0}(\mathbb{S}^2),\|\phi\|_{ H_1^{p_0}(\mathbb{S}^2)}=1}\left|\int_{\mathbb{S}^2}\phi(\nabla^\bot(\Omega_G+2\omega\chi)\cdot\nabla)\tilde \psi(t) d\sigma_g\right|\\\nonumber
=&C_1\sup_{\phi\in H_1^{p_0}(\mathbb{S}^2),\|\phi\|_{ H_1^{p_0}(\mathbb{S}^2)}=1}\left|\int_{\mathbb{S}^2}(\nabla^\bot(\Omega_G+2\omega\chi)\cdot\nabla)\phi\tilde \psi(t) d\sigma_g\right|\\\nonumber
\leq& C_1\sup_{\phi\in H_1^{p_0}(\mathbb{S}^2),\|\phi\|_{ H_1^{p_0}(\mathbb{S}^2)}=1}\|\Omega_G+2\omega\chi\|_{C^1(\mathbb{S}^2)}\|\nabla\phi\|_{L^{p_0}(T\mathbb{S}^2)}\|\tilde \psi(t) \|_{L^{p_0'}(\mathbb{S}^2)}\\\nonumber
\leq& C_1C_2e^{{3\over 2}{\rm Re}(\lambda_1)t}\|\Omega_G+2\omega\chi\|_{C^1(\mathbb{S}^2)}\|\psi\|_{L^{p_0'}(\mathbb{S}^2)},
\end{align}
where we used \eqref{group estimate} in the last inequality.
Since $p_3\in (1,b_0)$ and $p_0>1$ is sufficiently close to $1$, we can choose $p_5\in(1,b_0)$ such that ${1\over p_3}={1\over p_0'}+{1\over p_5}$. Using the degeneracy of the averaging Lyapunov exponent and
by a similar argument  to \eqref{estimate dX0gradgrad} and \eqref{estimate dX0gradgrad2}, we have
\begin{align}\label{integrand2}
&\|(d{\mathbf{X}}_G(t-\tau))^*\nabla\Delta^{-1}(\nabla^\bot(\Omega_G+2\omega\chi)\cdot\nabla\tilde \psi)(\tau,{\mathbf{X}}_G(t-\tau))\|_{L^{p_3}(T\mathbb{S}^2)}\\\nonumber
\leq&
\bigg\||(d{\mathbf{X}}_G(t-\tau))^*|_{op}\bigg\|_{L^{p_5}(\mathbb{S}^2)}\|\nabla\Delta^{-1}(\nabla^\bot(\Omega_G+2\omega\chi)\cdot\nabla\tilde \psi)(\tau)\|_{L^{p_0'}(T\mathbb{S}^2)}\\\nonumber
\leq& C_{3}e^{({\rm Re}(\lambda_1)-{\epsilon_1\over 2})(t-\tau)}\|\nabla\Delta^{-1}(\nabla^\bot(\Omega_G+2\omega\chi)\cdot\nabla\tilde \psi)(\tau)\|_{L^{p_0'}(T\mathbb{S}^2)},
\end{align}
where $\epsilon_1={\rm Re}(\lambda_1)-\mu(1-{1\over p_5})>0$. Similarly, we have
\begin{align}\label{integrand3}
&\|(d{\mathbf{X}}_G(t))^*\nabla\psi({\mathbf{X}}_G(t))\|_{L^{p_3}(T\mathbb{S}^2)}\leq C_{3}e^{({\rm Re}(\lambda_1)-{\epsilon_1\over 2})t}\|\psi\|_{H_1^{p_0'}(\mathbb{S}^2)}.
\end{align}
Combining \eqref{grad-tilde psi t x} and \eqref{integrand1}-\eqref{integrand3}, we have
\begin{align*}
&\|\nabla(e^{t {\mathcal{L}}_G'}\psi)\|_{L^{p_3}(T\mathbb{S}^2)}
=\|\nabla\tilde \psi(t)\|_{L^{p_3}(T\mathbb{S}^2)}\\
\leq&\|(d{\mathbf{X}}_G(t))^*\nabla\psi({\mathbf{X}}_G(t))\|_{L^{p_3}(T\mathbb{S}^2)}\\\nonumber
&+\int_0^t\|(d{\mathbf{X}}_G(t-\tau))^*\nabla\Delta^{-1}(\nabla^\bot(\Omega_G+2\omega\chi)\cdot\nabla\tilde \psi)(\tau,{\mathbf{X}}_G(t-\tau))\|_{L^{p_3}(T\mathbb{S}^2)}d\tau\\
\leq&C_{3}e^{({\rm Re}(\lambda_1)-{\epsilon_1\over 2})t}\|\psi\|_{H_1^{p_0'}(\mathbb{S}^2)}\\
&+C_1C_2C_{3}e^{({\rm Re}(\lambda_1)-{\epsilon_1\over 2})t}\int_0^te^{-({\rm Re}(\lambda_1)-{\epsilon_1\over 2})\tau+{{3\over 2}{\rm Re}(\lambda_1)\tau}}\|\Omega_G+2\omega\chi\|_{C^1(\mathbb{S}^2)}\|\psi\|_{L^{p_0'}(\mathbb{S}^2)}d\tau\\
\leq&C_4e^{{3\over2}{\rm Re}(\lambda_1)t}\|\psi\|_{H_1^{p_0'}(\mathbb{S}^2)}
\end{align*}
for some $C_4>0$.
\\
{\bf Step 3.} Estimate  $\|\mathbf{u}_\delta(t)\|_{L^{p_4}(T\mathbb{S}^2)}$ and $\|\Omega_\delta(t)\|_{L^{p_4}(\mathbb{S}^2)}$ via bootstrap from velocity to vorticity for $0\leq t<T$.

Note that $p_4>p_0$ since $p_0>1$ is sufficiently close to $1$. By bootstrap from $\|\mathbf{u}_\delta(t)\|_{L^{p_0}(T\mathbb{S}^2)}\leq c_1\delta e^{{\rm Re}(\lambda_1) t}$ (see \eqref{def-T}), we will prove that
\begin{align}\label{bootstrap}
\|\Omega_\delta(t)\|_{L^{p_4}(\mathbb{S}^2)},\;\; \|\mathbf{u}_\delta(t)\|_{L^{p_4}(T\mathbb{S}^2)}\leq C_5\delta e^{{\rm Re}(\lambda_1)t}
\end{align}
for some $C_5>0$, where  $0\leq t<T$.

By Lemma \ref{lem-differential calculus} (v),
 $curl^{-1}=J\nabla\Delta^{-1}:\Omega_\delta(t)\to\mathbf{u}_\delta(t)$ is bounded from $L_0^p(\mathbb{S}^2)$ to $L^p(T\mathbb{S}^2)$. Thus, it suffices to prove the part for $\|\Omega_\delta(t)\|_{L^{p_4}(\mathbb{S}^2)}$ in \eqref{bootstrap}.
 By \eqref{Omega delta exp-tra-per-flow}, we have
 \begin{align}\label{basic estimate Omega delta Lp4}
 \|\Omega_\delta(t)\|_{L^{p_4}(\mathbb{S}^2)}\leq\|\Omega_{\delta}(0)\|_{L^{p_4}(\mathbb{S}^2)}+\|\Omega_G+2\omega\chi\|_{C^1(\mathbb{S}^2)}\int_0^t\|{ \mathbf{u}}_\delta(\tau)\|_{L^{p_4}(T\mathbb{S}^2)} d\tau,
 \end{align}
 where $0\leq t<T$.
 By Lemma  \ref{lem-differential calculus} (ii) and $p_4={2\over 1-{1\over p_3}}>2$, the embedding of $H_2^{p_4}(\mathbb{S}^2)$ in $H_1^{p_4}(\mathbb{S}^2)$ is compact. Thus, for any $\alpha>0$, there exists $C_\alpha>0$ such that
 \begin{align}\label{estimate u delta Lp4}
 &\|\mathbf{u}_\delta(t)\|_{L^{p_4}(T\mathbb{S}^2)}=\|\nabla\Psi_\delta(t)\|_{L^{p_4}(T\mathbb{S}^2)}\leq \|\Psi_\delta(t)\|_{H_1^{p_4}(\mathbb{S}^2)}\\\nonumber
 \leq&\alpha\|\Psi_\delta(t)\|_{H_2^{p_4}(\mathbb{S}^2)}+C_\alpha\|\Psi_\delta(t)\|_{H_1^{p_0}(\mathbb{S}^2)}\leq \alpha C_6\|\Omega_\delta(t)\|_{L^{p_4}(\mathbb{S}^2)}+C_\alpha C \|\nabla\Psi_\delta(t)\|_{L^{p_0}(T\mathbb{S}^2)}\\\nonumber
 =&\alpha C_6\|\Omega_\delta(t)\|_{L^{p_4}(\mathbb{S}^2)}+C_\alpha C \|\mathbf{u}_\delta(t)\|_{L^{p_0}(T\mathbb{S}^2)},\end{align}
 where $0\leq t<T$, and we used $p_4>p_0$ in the second inequality, and used  Lemma 3.1 of \cite{Cao-Wang-Zuo2023}  and Lemma \ref{lem-differential calculus} (iii) in the third inequality. Inserting \eqref{estimate u delta Lp4} into \eqref{basic estimate Omega delta Lp4} and noticing $\|\mathbf{u}_\delta(t)\|_{L^{p_0}(T\mathbb{S}^2)}\leq c_1\delta e^{{\rm Re}(\lambda_1) t}$ for $0\leq t<T$, we have
 \begin{align*}
 \|\Omega_\delta(t)\|_{L^{p_4}(\mathbb{S}^2)}\leq&\|\Omega_{\delta}(0)\|_{L^{p_4}(\mathbb{S}^2)}+\alpha C_7\int_0^t\|\Omega_\delta(\tau)\|_{L^{p_4}(\mathbb{S}^2)} d\tau+C_8\delta e^{{\rm Re }\lambda_1t},
 \end{align*}
 where $C_7=C_6\|\Omega_G+2\omega\chi\|_{C^1(\mathbb{S}^2)}$ and $C_8=C_\alpha C c_1\|\Omega_G+2\omega\chi\|_{C^1(\mathbb{S}^2)}{1\over {\rm Re(\lambda_1)}}$. Fix $\kappa_1\in(0, {\rm Re}(\lambda_1))$. Let $\alpha>0$ be small enough such that $\alpha C_7< {\rm Re}(\lambda_1)-\kappa_1$. Then for $0\leq t<T$, we have
 \begin{align*}
 \|\Omega_\delta(t)\|_{L^{p_4}(\mathbb{S}^2)}\leq&(\delta\|\Omega_1^r\|_{L^{p_4}(\mathbb{S}^2)}+C_8\delta e^{{\rm Re }\lambda_1t})+({\rm Re}(\lambda_1)-\kappa_1)\int_0^t\|\Omega_\delta(\tau)\|_{L^{p_4}(\mathbb{S}^2)} d\tau.
 \end{align*}
 By Gr${\rm \ddot{o}}$nwall's inequality, we have
 \begin{align*}
 \|\Omega_\delta(t)\|_{L^{p_4}(\mathbb{S}^2)}\leq&(\delta\|\Omega_1^r\|_{L^{p_4}(\mathbb{S}^2)}+C_8\delta e^{{\rm Re }\lambda_1t})+
 ({\rm Re}(\lambda_1)-\kappa_1)\delta\|\Omega_1^r\|_{L^{p_4}(\mathbb{S}^2)}\int_0^te^{({\rm Re}(\lambda_1)-\kappa_1)(t-\tau)} d\tau\\
 &+({\rm Re}(\lambda_1)-\kappa_1)C_8\delta\int_0^te^{{\rm Re }\lambda_1\tau}e^{({\rm Re}(\lambda_1)-\kappa_1)(t-\tau)} d\tau\\
 \leq&(\delta\|\Omega_1^r\|_{L^{p_4}(\mathbb{S}^2)}+C_8\delta e^{{\rm Re }\lambda_1t})+
 \delta\|\Omega_1^r\|_{L^{p_4}(\mathbb{S}^2)}e^{({\rm Re}(\lambda_1)-\kappa_1)t}
 +{{\rm Re}(\lambda_1)-\kappa_1\over \kappa_1}C_8\delta e^{{\rm Re}(\lambda_1) t}\\
 \leq & C_5\delta e^{{\rm Re}(\lambda_1)t}
 \end{align*}
 for some $C_5>0$, where $0\leq t<T$.

 In summary, by Steps 1-3 we have
 \begin{align}\label{nonlinear estimate 0 T}
&\|\mathbf{u}_{\delta,N}(t)\|_{L^{p_0}(T\mathbb{S}^2)}\\\nonumber
\leq& C_1\sup_{\psi\in H_1^{p_0'}(\mathbb{S}^2),\|\psi\|_{ H_1^{p_0'}(\mathbb{S}^2)}=1} \int_0^t\|\nabla(e^{(t-\tau) {\mathcal{L}}_G'}\psi)\|_{L^{p_3}(T\mathbb{S}^2)}  \|\mathbf{u}_\delta(\tau)\|_{L^{p_4}(T\mathbb{S}^2)}\|\Omega_\delta(\tau)\|_{L^{p_4}(\mathbb{S}^2)}d\tau\\\nonumber
\leq &C_1 C_4C_5^2\sup_{\psi\in H_1^{p_0'}(\mathbb{S}^2),\|\psi\|_{ H_1^{p_0'}(\mathbb{S}^2)}=1}\int_0^te^{{3\over2}{\rm Re}(\lambda_1)(t-\tau)}\|\psi\|_{H_1^{p_0'}(\mathbb{S}^2)} (\delta e^{{\rm Re}(\lambda_1)\tau})^2 d\tau\\\nonumber
\leq& C_9(\delta e^{{\rm Re}(\lambda_1)t})^2
\end{align}
for some $C_9>0$, where $0\leq t<T$.
If $\lambda_1$ is real, then $\|e^{t\mathcal{M}_G}\mathbf{v}_1\|_{L^{p_0}(T\mathbb{S}^2)}=e^{t{\lambda_1}}c_0$, where we define $c_0\triangleq\|\mathbf{v}_1\|_{L^{p_0}(T\mathbb{S}^2)}$ in this case. If $\lambda_1$ is non-real, then
the real part of  $e^{t\mathcal{M}_G}\mathbf{v}_1=e^{t\lambda_1}\mathbf{v}_1$ is $e^{t\mathcal{M}_G}\mathbf{v}_1^r
=e^{t{\rm Re(\lambda_1)}}(\cos(t {\rm Im}(\lambda_1))\mathbf{v}_1^r-\sin(t{\rm Im}(\lambda_1))\mathbf{v}_1^i)$, and
 $\|e^{t\mathcal{M}_G}\mathbf{v}_1^r\|_{L^{p_0}(T\mathbb{S}^2)}\geq  e^{t{\rm Re(\lambda_1)}}$ $c_0$, where we define $c_0\triangleq\min_{\theta\in[0,2\pi)}\|\cos(\theta)\mathbf{v}_1^r-\sin(\theta)\mathbf{v}_1^i\|_{L^{p_0}(T\mathbb{S}^2)}$ in this case. Moreover,
we have
$c_0>0$, since otherwise, $\lambda_1$ is real. Thus, in any case,
\begin{align}\label{etLOv1rc0}\|e^{t\mathcal{M}_G}\mathbf{v}_1^r\|_{L^{p_0}(T\mathbb{S}^2)}\geq c_0e^{t{\rm Re(\lambda_1)}}>0,\quad t\geq0.\end{align}
Let $\tilde\epsilon_0\triangleq{\min\{c_0,c_1\}\over 4C_9}$, where we recall that $c_1=2\|\mathbf{v}_1\|_{L^{p_0}(T\mathbb{S}^2)}$. Define $t_0={1\over {\rm Re}(\lambda_1)}\ln{\tilde\epsilon_0\over \delta}>0$ for $\delta\in(0,\tilde \epsilon_0)$. Then $t_0=O(|\ln \delta|)$ and $\tilde\ep_0=\delta e^{t_0{\rm Re}(\lambda_1)}$. We claim that
\begin{align}\label{t0 less than T}
t_0<T,
\end{align}
 where $T$ is defined in \eqref{def-T}. Suppose $t_0\geq T$.
 Since $curl \mathcal{M}_G={\mathcal{L}}_G curl$ implies $e^{t{\mathcal{L}}_G}curl=curl e^{t\mathcal{M}_G}$, we have
 \begin{align}\label{u delta linear term}
 &{ \mathbf{u}}_{\delta,L}(T)=curl^{-1}\Omega_{\delta,L}(T)=curl^{-1}e^{T{\mathcal{L}}_G}\delta\Omega_1^r\\\nonumber
 =&\delta e^{T\mathcal{M}_G}\mathbf{v}_1^r=\delta e^{T{\rm Re(\lambda_1)}}(\cos(T {\rm Im}(\lambda_1))\mathbf{v}_1^r-\sin(T{\rm Im}(\lambda_1))\mathbf{v}_1^i)
  \end{align}
  and thus, $\|{ \mathbf{u}}_{\delta,L}(T)\|_{L^{p_0}(T\mathbb{S}^2)}\leq \delta e^{T{\rm Re(\lambda_1)}}$ $\|\mathbf{v}_1\|_{L^{p_0}(T\mathbb{S}^2)}$.
 Then
\begin{align*}&\|{ \mathbf{u}}_\delta(T)\|_{L^{p_0}(T\mathbb{S}^2)}\leq\|{ \mathbf{u}}_{\delta,L}(T)\|_{L^{p_0}(T\mathbb{S}^2)}+\|{ \mathbf{u}}_{\delta,N}(T)\|_{L^{p_0}(T\mathbb{S}^2)}\\
 \leq& \delta e^{T{\rm Re}(\lambda_1) }\|\mathbf{v}_1\|_{L^{p_0}(T\mathbb{S}^2)}+C_9(\delta e^{T{\rm Re}(\lambda_1)})^2\leq \delta e^{T{\rm Re}(\lambda_1) }({1\over 2} c_1+C_9\tilde\ep_0)\leq {3\over 4}c_1\delta e^{T{\rm Re}(\lambda_1) }.\end{align*}
 This contradicts \eqref{def-T}. Thus, we can apply \eqref{nonlinear estimate 0 T}, along with \eqref{etLOv1rc0},  to obtain
\begin{align*}
\|{ \mathbf{u}}_\delta(t_0)\|_{L^{p_0}(T\mathbb{S}^2)}\geq&\|{ \mathbf{u}}_{\delta,L}(t_0)\|_{L^{p_0}(T\mathbb{S}^2)}-\|{ \mathbf{u}}_{\delta,N}(t_0)\|_{L^{p_0}(T\mathbb{S}^2)}\geq c_0\delta e^{t_0{\rm Re(\lambda_1)}}-C_9(\delta e^{t_0{\rm Re}(\lambda_1)})^2
\geq \ep_0
\end{align*}
with $\ep_0={3\over4}\tilde\ep_0c_0$.
\end{proof}

\subsection{Proof of nonlinear orbital instability of general steady flows}
Based on  nonlinear instability  in Theorem \ref{thm-nonlinear instability-general steady flows}, we turn to prove nonlinear orbital instability in Theorem \ref{thm-nonlinear instability}.

The classical  orbital instability of (relative) equilibria governed by a one-parameter group for an abstract Hamiltonian system is developed in \cite{gss1990}. The nonlinear terms of the Hamiltonian system in \cite{gss1990} have no loss of derivative, and the basic setting is on the Hilbert space, so it can not be applied to our problem.

\begin{proof}[Proof of Theorem \ref{thm-nonlinear instability}]
As in the proof of Theorem  \ref{thm-nonlinear instability-general steady flows},
we only need to prove Theorem \ref{thm-nonlinear instability} for $p_0>1$ sufficiently close to $1$.
 Since the translational orbit of a zonal flow is itself, we can assume that $\mathbf{u}_{G}$ is non-zonal.
Let $\mathbf{v}_{1}\in L^2(T\mathbb{S}^2)$ be an eigenfunction of  $\lambda_1$. By Lemma \ref{lem-regularity of eigenfunction}, $\Omega_{1}=curl(\mathbf{v}_{1})\in H_1^{p_1}\cap L^{p_2}(\mathbb{S}^2)$ for $1\leq p_1<b_0$ and $p_2\geq1$. 
The normalization of $\Omega_{1}^r$ and the definitions of $\Omega_{\delta,G}(t)$, ${\mathbf{u}}_{\delta,G} (t)$, $\Omega_{\delta}(t)$, $ { \mathbf{u}}_{\delta}(t)$, $ { \mathbf{u}}_{\delta,L}(t)$ and $ { \mathbf{u}}_{\delta,N}(t)$ are  the same  as in the proof of Theorem \ref{thm-nonlinear instability-general steady flows}.

Define
\begin{align*}
c_2\triangleq\left\{ \begin{array}{lll}
\min\limits_{\theta\in[0,2\pi)}\inf\limits_{\mathbf{w}\in\text{span}\{T'(0)\mathbf{u}_{G}\}}
\|(\cos(\theta)\mathbf{v}_{1}^r-\sin(\theta)\mathbf{v}_{1}^i)-\mathbf{w}\|_{L^{p_0}(T\mathbb{S}^2)},& \text{$\lambda_1$ is non-real},
 \\\inf\limits_{\mathbf{w}\in\text{span}\{T'(0)\mathbf{u}_{G}\}}\|\mathbf{v}_{1}-\mathbf{w}\|_{L^{p_0}(T\mathbb{S}^2)},&\text{$\lambda_1$ is real}.
 \end{array} \right.
\end{align*}
We claim that \begin{align}\label{c2 positive}
c_2>0.
\end{align}
In fact, if $\lambda_1$ is non-real,  acting $T'(0)$ on the steady-state equation
\begin{align*}
\nabla_{{\mathbf{u}}_{G}}\mathbf{u}_{G}+2\omega\chi J\mathbf{u}_{G}+\nabla p=0,
\end{align*}
we have
\begin{align}\label{kernel-G-RH}
\mathcal{M}_{G} T'(0)\mathbf{u}_{G}=0.
\end{align}
On the other hand, by separating the real and imaginary parts of the equation $\mathcal{M}_{G} \mathbf{v}_{1}=\lambda_1\mathbf{v}_{1}$, we have $\mathcal{M}_{G} \mathbf{v}_{1}^r={\rm Re}(\lambda_1) \mathbf{v}_{1}^r-{\rm Im}(\lambda_1)\mathbf{v}_{1}^i$ and $\mathcal{M}_{G}\mathbf{v}_{1}^i={\rm Im}(\lambda_1)\mathbf{v}_{1}^r+{\rm Re}(\lambda_1)\mathbf{v}_{1}^i$. Then
\begin{align*}
&\mathcal{M}_{G}(\cos(\theta)\mathbf{v}_{1}^r-\sin(\theta)\mathbf{v}_{1}^i)\\
=&(\cos(\theta){\rm Re}(\lambda_1)-\sin(\theta){\rm Im}(\lambda_1))\mathbf{v}_{1}^r-(\cos(\theta){\rm Im}(\lambda_1)+\sin(\theta){\rm Re}(\lambda_1))\mathbf{v}_{1}^i.
\end{align*}
Since $\mathbf{v}_{1}^r\neq0$ is linearly independent of $\mathbf{v}_{1}^i\neq0$ and $(\cos(\theta){\rm Re}(\lambda_1)-\sin(\theta){\rm Im}(\lambda_1))^2+(\cos(\theta){\rm Im}(\lambda_1)+\sin(\theta){\rm Re}(\lambda_1))^2\neq0$, we have $\mathcal{M}_{G}(\cos(\theta)\mathbf{v}_{1}^r-\sin(\theta)\mathbf{v}_{1}^i)\neq0$ for $\theta\in[0,2\pi)$.
This, along with \eqref{kernel-G-RH}, gives $\cos(\theta)\mathbf{v}_{1}^r-\sin(\theta)\mathbf{v}_{1}^i\notin\text{span}\{T'(0)\mathbf{u}_{G}\}$ for $\theta\in[0,2\pi)$ and proves \eqref{c2 positive} in the case that $\lambda_1$ is non-real.
If $\lambda_1$ is real, by \eqref{kernel-G-RH} and $\mathcal{M}_{G} \mathbf{v}_{1}=\lambda_1\mathbf{v}_{1}$ we have  $c_2>0.$

Define
\begin{align}\label{def-tilde c3-c3}
\tilde c_3\triangleq \min\limits_{\tau\in\left[-{T_{G}\over2},{T_{G}\over2}\right)}\left\|\int_0^1T(\xi\tau)T'(0)\mathbf{u}_{G}d\xi\right\|_{L^{p_0}(T\mathbb{S}^2)},
\quad c_3\triangleq{1\over\|T'(0)^2\mathbf{u}_{G}\|_{L^{p_0}(T\mathbb{S}^2)}}\tilde c_3^2,
\end{align}
where $T_{G}$ is the minimal zonal period of $\mathbf{u}_{G}$.
We prove that $\tilde c_3, c_3>0.$
It suffices to show that $\left\|\int_0^1T(\xi\tau)T'(0)\mathbf{u}_{G}d\xi\right\|_{L^{p_0}(T\mathbb{S}^2)}>0$ for all $\tau\in\left[-{T_{G}\over2},{T_{G}\over2}\right]$.
If $\tau=0$, then $\bigg\|\int_0^1T(\xi\tau)T'(0)$ $\mathbf{u}_{G}d\xi\bigg\|_{L^{p_0}(T\mathbb{S}^2)}=\left\|T'(0)\mathbf{u}_{G}\right\|_{L^{p_0}(T\mathbb{S}^2)}>0$. If $\tau\neq0$, suppose $\left\|\int_0^1T(\xi\tau)T'(0)\mathbf{u}_{G}d\xi\right\|_{L^{p_0}(T\mathbb{S}^2)}$ $=0$, then $|\tau|\left\|\int_0^1T(\xi\tau)T'(0)\mathbf{u}_{G}d\xi\right\|_{L^{p_0}(T\mathbb{S}^2)}=\|\mathbf{u}_{G}-T(\tau)\mathbf{u}_{G}\|_{L^{p_0}(T\mathbb{S}^2)}=0.$
Then the minimal zonal period of $\mathbf{u}_{G}$ is less than $T_{G}$, which is a contradiction. This proves $\tilde c_3, c_3>0$.

Let $\tilde\epsilon_1=\min\left\{{c_0\over 4C_9},{c_1\over 4C_9},{c_2\over 4C_9}, {c_2c_3\over 8c_1^2}\right\}>0$ and $t_1={1\over {\rm Re}(\lambda_1)}\ln{\tilde\epsilon_1\over \delta}>0$ for $\delta\in(0,\tilde \epsilon_1)$, where $c_1\triangleq2\|\mathbf{v}_1\|_{L^{p_0}(T\mathbb{S}^2)}$; $C_9$ is determined in \eqref{nonlinear estimate 0 T}; and $c_0=\|\mathbf{v}_{1}\|_{L^{p_0}(T\mathbb{S}^2)}$ if $\lambda_1$ is real, and
$c_0=\min_{\theta\in[0,2\pi)}\|\cos(\theta)\mathbf{v}_{1}^r-\sin(\theta)\mathbf{v}_{1}^i\|_{L^{p_0}(T\mathbb{S}^2)}$ if $\lambda_1$ is non-real.
Then  $\tilde\ep_1=\delta e^{t_1{\rm Re}(\lambda_1)}$. By  \eqref{t0 less than T} and $t_1\leq t_0$, we have
\begin{align}\label{u delta tf t1 lp0 estimate}
\|{ \mathbf{u}}_{\delta}(t)\|_{L^{p_0}(T\mathbb{S}^2)}\leq c_1\delta e^{{\rm Re}(\lambda_1)t}\leq c_1\tilde\epsilon_1,\quad 0\leq t\leq t_1.
\end{align}
Moreover, by \eqref{nonlinear estimate 0 T} we have
\begin{align}\label{u delta nonlinear term tf}
\|\mathbf{u}_{\delta,N}(t_1)\|_{L^{p_0}(T\mathbb{S}^2)}\leq C_9\delta e^{{\rm Re}(\lambda_1) t_1}\tilde\epsilon_1 \leq {c_2\over4}\delta e^{{\rm Re}(\lambda_1) t_1}.
\end{align}
Since $\cos(t_1{\rm Im}(\lambda_1))\mathbf{v}_{1}^r-\sin(t_1{\rm Im}(\lambda_1))\mathbf{v}_{1}^i\notin\text{span}\{T'(0)\mathbf{u}_{G}\}$, by Hahn-Banach Theorem and the fact that the dual space of $L^{p_0}(T\mathbb{S}^2)$ is $L^{p_0'}(T\mathbb{S}^2)$, there exists a function, namely, $(\cos(t_1{\rm Im}(\lambda_1)$ $)\mathbf{v}_{1}^r-\sin(t_1{\rm Im}(\lambda_1))\mathbf{v}_{1}^i)^{\bot}$, in $L^{p_0'}(T\mathbb{S}^2)$ such that
\begin{align}\label{cos vr sin vi perpendicular span derivative T0uRH}
\left\langle
\left(\cos(t_1{\rm Im}(\lambda_1))\mathbf{v}_{1}^r-\sin(t_1{\rm Im}(\lambda_1))\mathbf{v}_{1}^i\right)^{\bot},\mathbf{w}\right\rangle=0,\quad\forall\;\mathbf{w}\in\text{span}\{T'(0)\mathbf{u}_{G}\},
\end{align}
and
\begin{align}\label{cos vr sin vi perpendicular cos vr sin vi dual}
&\left\langle
\left(\cos(t_1{\rm Im}(\lambda_1))\mathbf{v}_{1}^r-\sin(t_1{\rm Im}(\lambda_1))\mathbf{v}_{1}^i\right)^{\bot},\cos(t_1{\rm Im}(\lambda_1))\mathbf{v}_{1}^r-\sin(t_1{\rm Im}(\lambda_1))\mathbf{v}_{1}^i\right\rangle\\
=&\inf\limits_{\mathbf{w}\in\text{span}\{T'(0)\mathbf{u}_{G}\}}
\|(\cos(t_1{\rm Im}(\lambda_1))\mathbf{v}_{1}^r-\sin(t_1{\rm Im}(\lambda_1))\mathbf{v}_{1}^i)-\mathbf{w}\|_{L^{p_0}(T\mathbb{S}^2)}\geq c_2,\nonumber
\end{align}
where
\begin{align}\label{cos vr sin vi perpendicular norm}
\left\|\left(\cos(t_1{\rm Im}(\lambda_1))\mathbf{v}_{1}^r-\sin(t_1{\rm Im}(\lambda_1))\mathbf{v}_{1}^i\right)^{\bot}\right\|_{L^{p_0'}(T\mathbb{S}^2)}=1,
\end{align}
and $\langle\cdot,\cdot\rangle$ is the pairing between $L^{p_0}(T\mathbb{S}^2)$ and $L^{p_0'}(T\mathbb{S}^2)$.
Similar to \eqref{u delta linear term}, we have ${ \mathbf{u}}_{\delta,L}(t_1)=\delta e^{{\rm Re}(\lambda_1)t_1}\left(\cos(t_1 {\rm Im}(\lambda_1))\mathbf{v}_{1}^r-\sin(t_1{\rm Im}(\lambda_1))\mathbf{v}_{1}^i\right)$.
Then by \eqref{cos vr sin vi perpendicular cos vr sin vi dual}, \eqref{u delta nonlinear term tf} and \eqref{cos vr sin vi perpendicular norm} we have
\begin{align}\label{u delta tf t1 cos vr sin vi perpendicular}
&\left|\left\langle \mathbf{u}_{\delta}(t_1),\left(\cos(t_1{\rm Im}(\lambda_1))\mathbf{v}_{1}^r-\sin(t_1{\rm Im}(\lambda_1))\mathbf{v}_{1}^i\right)^{\bot}\right\rangle\right|\\\nonumber
\geq&\left|\left\langle \mathbf{u}_{\delta,L}(t_1),\left(\cos(t_1{\rm Im}(\lambda_1))\mathbf{v}_{1}^r-\sin(t_1{\rm Im}(\lambda_1))\mathbf{v}_{1}^i\right)^{\bot}\right\rangle\right|\\\nonumber
&-\left|\left\langle \mathbf{u}_{\delta,N}(t_1),\left(\cos(t_1{\rm Im}(\lambda_1))\mathbf{v}_{1}^r-\sin(t_1{\rm Im}(\lambda_1))\mathbf{v}_{1}^i\right)^{\bot}\right\rangle\right|\\\nonumber
\geq&
\delta e^{{\rm Re}(\lambda_1)t_1}\bigg|\bigg\langle \cos(t_1 {\rm Im}(\lambda_1))\mathbf{v}_{1}^r-\sin(t_1{\rm Im}(\lambda_1))\mathbf{v}_{1}^i,\\\nonumber
&\left(\cos(t_1{\rm Im}(\lambda_1))\mathbf{v}_{1}^r-\sin(t_1{\rm Im}(\lambda_1))\mathbf{v}_{1}^i\right)^{\bot}\bigg\rangle\bigg|\\\nonumber
&- \|\mathbf{u}_{\delta,N}(t_1)\|_{L^{p_0}(T\mathbb{S}^2)}\left\|\left(\cos(t_1{\rm Im}(\lambda_1))\mathbf{v}_{1}^r-\sin(t_1{\rm Im}(\lambda_1))\mathbf{v}_{1}^i\right)^{\bot}\right\|_{L^{p_0'}(T\mathbb{S}^2)}\\\nonumber
\geq& c_2\delta e^{{\rm Re}(\lambda_1)t_1}-{c_2\over4}\delta e^{{\rm Re}(\lambda_1) t_1}\\\nonumber
=&{3\over 4}c_2\tilde \epsilon_1.
\end{align}
Recall that ${\mathbf{u}}_{\delta,G}(t)$ is the perturbed velocity solving \eqref{velocity equation-original form}. Then 
\begin{align*}
\Theta(t)\triangleq&\inf_{\tau\in\mathbb{R}}\|{\mathbf{u}}_{\delta,G}(t)-T(\tau){\mathbf{u}}_{G}\|_{L^{p_0}(T\mathbb{S}^2)}
=\inf_{\tau\in\mathbb{R}}\|{\mathbf{u}}_{\delta}(t)+\mathbf{u}_{G}-T(\tau){\mathbf{u}}_{G}\|_{L^{p_0}(T\mathbb{S}^2)}\\
=&\inf_{\tau\in\left[-{T_{G}\over2},{T_{G}\over2}\right)}\|{\mathbf{u}}_{\delta}(t)+\mathbf{u}_{G}-T(\tau){\mathbf{u}}_{G}\|_{L^{p_0}(T\mathbb{S}^2)}
\end{align*}
for $t\geq0$, where $T_{G}$ is the minimal zonal period of $\mathbf{u}_{G}$. For $t=t_1$, there exists $\tau_{t_1}\in\left[-{T_{G}\over2},{T_{G}\over2}\right)$ such that $\Theta(t_1)=\|{\mathbf{u}}_{\delta}(t_1)+\mathbf{u}_{G}-T(\tau_{t_1}){\mathbf{u}}_{G}\|_{L^{p_0}(T\mathbb{S}^2)}$.
By the definition of $\Theta(t_1)$ and \eqref{u delta tf t1 lp0 estimate}  we have
\begin{align*}
\Theta(t_1)\leq \|{\mathbf{u}}_{\delta}(t_1)\|_{L^{p_0}(T\mathbb{S}^2)}\leq c_1\tilde\epsilon_1,
\end{align*}
and thus
\begin{align}\label{uRH-T-tau-t1-uRH-upper bound}
\|\mathbf{u}_{G}-T(\tau_{t_1}){\mathbf{u}}_{G}\|_{L^{p_0}(T\mathbb{S}^2)}\leq2 c_1\tilde\epsilon_1.
\end{align}
Direct computation implies
\begin{align}\label{uRH-T-tau-t1-uRH-lower bound}
\|\mathbf{u}_{G}-T(\tau_{t_1}){\mathbf{u}}_{G}\|_{L^{p_0}(T\mathbb{S}^2)}=&|\tau_{t_1}|\left\|\int_0^1T(\xi\tau_{t_1})T'(0)\mathbf{u}_{G} d\xi\right\|_{L^{p_0}(T\mathbb{S}^2)}
\geq |\tau_{t_1}| \tilde c_3,
\end{align}
where $\tilde c_3$ is defined in \eqref{def-tilde c3-c3}.
Combining \eqref{uRH-T-tau-t1-uRH-upper bound} and \eqref{uRH-T-tau-t1-uRH-lower bound} we have
\begin{align}\label{tau-t1-bound}
|\tau_{t_1}| \leq {2c_1\tilde\epsilon_1\over \tilde c_3}.
\end{align}
By \eqref{cos vr sin vi perpendicular norm}, \eqref{cos vr sin vi perpendicular span derivative T0uRH}, \eqref{u delta tf t1 cos vr sin vi perpendicular}, \eqref{tau-t1-bound} and the definitions of $c_3, \tilde \epsilon_1,$   we have
\begin{align*}
&\Theta(t_1)=\|{\mathbf{u}}_{\delta}(t_1)+\mathbf{u}_{G}-T(\tau_{t_1}){\mathbf{u}}_{G}\|_{L^{p_0}(T\mathbb{S}^2)}\\
\geq&\left|\left\langle {\mathbf{u}}_{\delta}(t_1)+\mathbf{u}_{G}-T(\tau_{t_1}){\mathbf{u}}_{G},\left(\cos(t_1{\rm Im}(\lambda_1))\mathbf{v}_{1}^r-\sin(t_1{\rm Im}(\lambda_1))\mathbf{v}_{1}^i\right)^{\bot}\right\rangle\right|\\
=&\bigg|\bigg\langle {\mathbf{u}}_{\delta}(t_1)-\tau_{t_1}T'(0)\mathbf{u}_{G}-\tau_{t_1}^2\int_0^1T(\xi\tau_{t_1})T'(0)^2\mathbf{u}_{G}(1-\xi)d\xi,\\
&\left(\cos(t_1{\rm Im}(\lambda_1))\mathbf{v}_{1}^r-\sin(t_1{\rm Im}(\lambda_1))\mathbf{v}_{1}^i\right)^{\bot}\bigg\rangle\bigg|\\
=&\bigg|\bigg\langle {\mathbf{u}}_{\delta}(t_1)-\tau_{t_1}^2\int_0^1T(\xi\tau_{t_1})T'(0)^2\mathbf{u}_{G}(1-\xi)d\xi,
\left(\cos(t_1{\rm Im}(\lambda_1))\mathbf{v}_{1}^r-\sin(t_1{\rm Im}(\lambda_1))\mathbf{v}_{1}^i\right)^{\bot}\bigg\rangle\bigg|\\
\geq&{3\over 4}c_2\tilde \epsilon_1-\tau_{t_1}^2\bigg|\bigg\langle\int_0^1T(\xi\tau_{t_1})T'(0)^2\mathbf{u}_{G}(1-\xi)d\xi,
\left(\cos(t_1{\rm Im}(\lambda_1))\mathbf{v}_{1}^r-\sin(t_1{\rm Im}(\lambda_1))\mathbf{v}_{1}^i\right)^{\bot}\bigg\rangle\bigg|\\
\geq&{3\over 4}c_2\tilde \epsilon_1-\tau_{t_1}^2\left\|\int_0^1T(\xi\tau_{t_1})T'(0)^2\mathbf{u}_{G}(1-\xi)d\xi\right\|_{L^{p_0}(T\mathbb{S}^2)}\\
\geq&{3\over 4}c_2\tilde \epsilon_1-\tau_{t_1}^2\int_0^1\left\|T(\xi\tau_{t_1})T'(0)^2\mathbf{u}_{G}(1-\xi)\right\|_{L^{p_0}(T\mathbb{S}^2)}d\xi\\
\geq&{3\over 4}c_2\tilde \epsilon_1-{4c_1^2\tilde\epsilon_1^2\over \tilde c_3^2}\left\|T'(0)^2\mathbf{u}_{G}\right\|_{L^{p_0}(T\mathbb{S}^2)}\\
=&{3\over 4}c_2\tilde \epsilon_1-{4c_1^2\tilde\epsilon_1^2\over  c_3}\\
\geq&{3\over 4}c_2\tilde \epsilon_1-{4c_1^2\tilde\epsilon_1\over  c_3}\cdot{c_2c_3\over 8c_1^2}\\
=&{1\over 4} c_2\tilde \epsilon_1\triangleq\epsilon_1,
\end{align*}
where we use the integral remainder of the Taylor expansion of $\mathbf{u}_{G}-T(\tau_{t_1}){\mathbf{u}}_{G}$ in the first equality.
\end{proof}

\begin{Remark}
For a general  travelling wave $T(-ct) \mathbf{u}$, it appears to be a steady flow in a travelling frame of reference moving with the wave. So it is not difficult to extend Theorem \ref{thm-nonlinear instability} to general  travelling waves. Typical examples include Rossby-Haurwitz travelling waves and the Stuart vortices on the sphere (see \cite{Crowdy04, Constantin-Crowdy-Krishnamurthy-Wheeler2021}).
\end{Remark}

\section{Description of streamline patterns of travelling waves near the $3$-jet}

\subsection{Construction of travelling waves  near the $3$-jet}
In this subsection, we prove Theorem $\ref{travelling wave solutions cat eye unidirectional thm}$. We construct  unidirectional travelling waves near the $3$-jet  for
$\omega\in(-18,-3)\cup({69\over2},72)$ in Theorem \ref{travelling wave solutions thm}. Then we construct  cat's eyes travelling waves for $\omega\in(-18,72)$
and unidirectional travelling waves  for $\omega\in(-\infty,-18)\cup(72,\infty)$ in Lemma  \ref{travelling waves c omega}.

The   cat's eyes and unidirectional  travelling waves are defined as follows.
\begin{itemize}
\item A cat's eyes  travelling wave means that its streamlines  have at least a cat's eyes structure.

\item A unidirectional  travelling wave means that all its streamlines  are unidirectional.
\end{itemize}
For  a cat's eyes (resp. unidirectional) travelling wave $\Psi(\varphi-ct,s)$, the travelling speed $c$ is located  in the interior of $Ran(-\partial_s\Psi)$ (reps. outside $Ran(-\partial_s\Psi)$).

For $\omega\in \left(-18,-3\right)\cup\left({69\over2},72\right)$,
to ensure that the unidirectional  travelling waves form  curves, we have to study the bifurcation directly at the $3$-jet, not at its nearby zonal flows. This requires a delicate spectral analysis on the kernel of the linearization of the nonlinear functional and a  weak transversal condition, which is  the major difficulty in the construction.

\begin{Theorem}\label{travelling wave solutions thm}
{\rm(i)}
Let $\omega\in \left(-18,-3\right)\cup\left({99\over2},72\right)$. Then there exists a curve of unidirectional  travelling waves
$\{\Psi_{(\gamma),1}(\varphi-c_1(\gamma)t,s)||\gamma|\ll1\}$
  such that $c_1(\cdot)\in C^0$ and
\begin{align}\label{a curve of traveling wave solution 1 mode}
\|\Psi_{(\gamma),1}-\Psi_0(s)-\gamma\Phi_{\mu_{1,\omega},\omega,1}(s)\cos(\varphi)\|_{H_2^4(\mathbb{S}^2)}=o(|\gamma|),
\end{align}
where $\Phi_{\mu_{1,\omega},\omega,1}$  is the neutral solution in Corollaries \ref{uniqueness1} and \ref{uniqueness2}. Moreover, 
 if $\omega\in\left(-18,-3\right)$, then $ c_1(\gamma)>\max(-\Psi_0')$;
 if $\omega\in\left({99\over2},72\right)$, then $ c_1(\gamma)<\min(-\Psi_0')$.

{\rm(ii)}
Let $\omega\in \left(-18,g^{-1}(-12)\right)\cup\left({69\over2},72\right)$. Then there exists another curve of unidirectional  travelling waves
$\{\Psi_{(\gamma),2}(\varphi-c_2(\gamma)t,s)||\gamma|\ll1\}$
 such that $c_2(\cdot)\in C^0$ and
\begin{align}\label{a curve of traveling wave solution 2 mode0}
\|\Psi_{(\gamma),2}-\Psi_0(s)-\gamma\Phi_{\mu_{2,\omega},\omega,2}(s)\cos(2\varphi)\|_{H_2^4(\mathbb{S}^2)}=o(|\gamma|),
\end{align}
where $\Phi_{\mu_{2,\omega},\omega,2}$  is the  neutral solution in Corollaries \ref{uniqueness3} and \ref{uniqueness4}.
Moreover, 
 if $\omega\in\left(-18,g^{-1}(-12)\right)$, then $ c_2(\gamma)>\max(-\Psi_0')$;
 if $\omega\in\left({69\over2},72\right)$, then $ c_2(\gamma)<\min(-\Psi_0')$.

{\rm(iii)}
Let $\omega\in \left(-18,g^{-1}(-12)\right)$. Then there exists one more curve of travelling waves
$\{\Psi_{(\gamma),3}(\varphi-c_3(\gamma)t,s)||\gamma|\ll1\}$
  such that
 $c_3(\cdot)\in C^0$, $c_3(\gamma)>\max(-\Psi_0')$ and
\begin{align}\label{a curve of traveling wave solution 2 mode}
\|\Psi_{(\gamma),3}-\Psi_0(s)-\gamma\Phi_{\mu_{3,\omega},\omega,2}(s)\cos(2\varphi)\|_{H_2^4(\mathbb{S}^2)}=o(|\gamma|),
\end{align}
where $\Phi_{\mu_{3,\omega},\omega,2}$  is the neutral solution in Corollary \ref{uniqueness4}.
\end{Theorem}

We need the following analytic version of Implicit Function Theorem in \cite{Hormander07}.

\begin{Lemma}\label{analytic version of Implicit Function Theorem}
Let $f(w,z,r)$ be an analytic function in a neighborhood of $(w_0,z_0,r_0)\in \mathbb{C}^3$, $f(w_0,z_0,r_0)=0$ and $\partial_wf(w_0,z_0,r_0)\neq0$. Then $f(w,z,r)=0$ has a uniquely determined analytic solution ${\bf{w}}(z,r)$ in a neighborhood of $(z_0,r_0)$ such that ${\bf{w}}(z_0,r_0)=w_0$.
\end{Lemma}
Before proving  Theorem \ref{travelling wave solutions thm}, we briefly discuss some ideas in the following remark.

\begin{Remark}\label{remark on the proof of travelling wave solutions thm}
 Our approach to construct the travelling waves  near the $3$-jet relies on two components.

 {\rm ($ 1$)}
  The nonlinear functional \eqref{def-F-B-C} used in the bifurcation is analytic near the bifurcation point, which is important
 when the
 transversal crossing condition fails.
  Another important point is that in this case,
  we can get a weak transversal condition based on the index formula. The strong regularity of the nonlinear functional, together with the weak transversal condition, allows us to apply a degenerate local bifurcation theorem due to Kielh\"{o}fer \cite{Kielhofer1980} to construct nearby travelling waves, once the linearized operator has $1$-dimensional kernel.
In particular, this bifurcation ensures the travelling waves to form a curve.

    {\rm ($ 2 $)}
The kernel can be obtained from the neutral solutions in the $1$'st and  $2$'nd Fourier modes by Corollaries \ref{uniqueness1}, \ref{uniqueness3}, \ref{uniqueness2} and \ref{uniqueness4}.   It is, however, non-trivial to ensure that the dimension of the kernel can not be larger than one. In   Theorem \ref{travelling wave solutions thm} {\rm(i)}, the kernel is induced from the $1$'st mode, and we restrict the space of functions to be odd in $s$ and even in $\varphi$, which is proved to be preserved by the nonlinear functional. In  Theorem \ref{travelling wave solutions thm} {\rm(ii)}-{\rm(iii)},  the  kernel is induced from the $2$'nd mode, and we restrict the space of functions to be $\pi$-periodic and even in $\varphi$. Nevertheless,
it is still subtle to rule out kernels from the $0$'th mode  of the linearized operator. To this end, we use a method    based on the Frobenius method and the variation of parameters technique for  solving  ODEs.
\end{Remark}

Now, we prove Theorem \ref{travelling wave solutions thm}.
\begin{proof}[Proof of Theorem \ref{travelling wave solutions thm}]
Let $\{\Upsilon_1,\Upsilon_2\}=\partial_\varphi\Upsilon_1\partial_s\Upsilon_2-\partial_s\Upsilon_1\partial_\varphi\Upsilon_2$ be the Poisson bracket.
Then the nonlinear Euler equation  ${\rm(\mathcal{E}_\omega)}$ is rewritten as $\partial_t\Upsilon+\{\Psi,\Upsilon+2\omega s\}=0$, where $\Upsilon=\Delta\Psi$.
Thus,  $\Psi(\varphi-ct,s)$ is a solution of ${\rm(\mathcal{E}_\omega)}$  if and only if
\begin{align*}
\{\Psi+cs,\Upsilon+2\omega s\}=0.
\end{align*}
This means that if there exists a function $g\in C^1$ such that $\Upsilon+2\omega s=g(\Psi+cs)$, then $\Psi(\varphi-ct,s)$ solves ${\rm(\mathcal{E}_\omega)}$.
Let $\mu_{0,\omega}\in\{\mu_{1,\omega},\mu_{2,\omega},\mu_{3,\omega}\}$, where $\mu_{1,\omega},\mu_{2,\omega},\mu_{3,\omega}$ are given in Corollaries \ref{uniqueness1}, \ref{uniqueness3}, \ref{uniqueness2} and \ref{uniqueness4}.
For the $3$-jet, since $\mu_{0,\omega}\notin\text{Ran}(- \Psi_0')$, there exists $\delta_0>0$ such that $\Psi_0'(s)+\mu=15s^2-3+\mu\neq0 $ on $[-1,1]$ for $\mu\in[\mu_{0,\omega}-\delta_0,\mu_{0,\omega}+\delta_0]$. Consider the function $f(s,z,\mu)=z-\Psi_0(s)-\mu s=z-5s^3+3s-\mu s$ on the complex region $s\in\cup_{\tilde s\in[-1,1]}\{s:|s-\tilde s|<\tilde\delta_0 \}$, $\mu\in \cup_{\tilde\mu\in[\mu_{0,\omega}-\delta_0,\mu_{0,\omega}+\delta_0]}\{\mu:|\mu-\tilde \mu|<\tilde\delta_0 \}$ and $z\in
\cup_{\tilde z\in \{\Psi_0(s)+\mu s|s\in[-1,1]\}}\{z:|z-\tilde z|<\tilde\delta_0 \}$
for some $\tilde \delta_0>0$ small enough. Since $f(\hat s,\hat z,\hat \mu)=0$ and $\partial_s f(\hat s,\hat z,\hat \mu)=-\Psi_0'(\hat s)-\hat \mu\neq0$ for $\hat s\in[-1,1]$, $\hat \mu\in[\mu_{0,\omega}-\delta_0,\mu_{0,\omega}+\delta_0]$ and $\hat z=\Psi_0(\hat s)+\hat \mu \hat s$,  we infer from Lemma \ref{analytic version of Implicit Function Theorem} that there exists a unique  analytic  function $\mathbf{s}(z,\mu)$ defined on an open subset $\mathcal{U}$ in $\mathbb{C}^2$ containing $\{\Psi_0(s)+\mu s:s\in[-1,1]\}\times[\mu_{0,\omega}-\delta_0,\mu_{0,\omega}+\delta_0]$ such that
$z-5\mathbf{s}(z,\mu)^3+3\mathbf{s}(z,\mu)-\mu \mathbf{s}(z,\mu)=0$ for $(z,\mu)\in\mathcal{U}$, and $\mathbf{s}(\hat z,\hat\mu)=\hat s$.
Define
\begin{align*}
h(z,\mu)\triangleq & \Upsilon_0(\mathbf{s}(z,\mu))+2\omega\mathbf{s}(z,\mu) \quad(z,\mu)\in \mathcal{U}.
\end{align*}
Then $h$ is analytic in $\mathcal{U}$ and
\begin{align*}
h(z,\mu)=-12(5\mathbf{s}(z,\mu)^3-3\mathbf{s}(z,\mu))+2\omega\mathbf{s}(z,\mu)\\
=
-12z+(12\mu +2\omega) \mathbf{s}(z,\mu),\quad(z,\mu)\in \mathcal{U}.
\end{align*}
Since $s=\mathbf{s}(\Psi_0(s)+\mu s,\mu)$, we have
\begin{align}\nonumber
&\Delta\Psi_0(s)+2\omega s=\Upsilon_0(s)+2\omega s=\Upsilon_0(\mathbf{s}(\Psi_0(s)+\mu s,\mu))+2\omega\mathbf{s}(\Psi_0(s)+\mu s,\mu)\\\nonumber
=&-12(5\mathbf{s}(\Psi_0(s)+\mu s,\mu)^3-3\mathbf{s}(\Psi_0(s)+\mu s,\mu))+2\omega\mathbf{s}(\Psi_0(s)+\mu s,\mu)\\\label{3jetflowstream function vorticity}
=&h(\Psi_0(s)+\mu s,\mu)
\end{align}
for $ s\in[-1,1]$ and $ \mu\in[\mu_{0,\omega}-\delta_0,\mu_{0,\omega}+\delta_0]$.
To construct the curve of travelling waves near the $3$-jet, we study the elliptic equations
 \begin{align}\label{3jetflow perturbation of stream function vorticity}
&\Delta\Psi+2\omega s= h(\Psi+\mu s,\mu).
\end{align}
Let $\phi(\varphi,s)=\Psi(\varphi,s)-\Psi_0(s)$ be the perturbation of the stream function. By \eqref{3jetflowstream function vorticity}-\eqref{3jetflow perturbation of stream function vorticity}, we have
\begin{align}\label{perturbation elliptic equation}
\Delta \phi=h(\phi+\Psi_0+\mu s,\mu)-h(\Psi_0+\mu s,\mu).
\end{align}
\\
{\it Proof of (i)}. In this case, $\mu_{0,\omega}=\mu_{1,\omega}$.
Define the mapping
\begin{align}\label{def-F-B-C}
&F_o:B_o\times[\mu_{1,\omega}-\delta_0,\mu_{1,\omega}+\delta_0]\rightarrow C_o,\\\nonumber
 &\;\;\;\;\;\;\;\;\;\;\;\;\;\;\;\;\;\;\;\;\;\;\;\;\;\;\;\;\;\;\;\;\;(\phi,\mu)\mapsto\Delta \phi-(h(\phi+\Psi_0+\mu s,\mu)-h(\Psi_0+\mu s,\mu)),
\end{align}
where
\begin{align*}
B_o=&\{\phi\in H_2^4(\mathbb{S}^2)|
 \phi(2\pi-\varphi,s)=\phi(\varphi,s), \phi\text{ is odd in } s \text{ and } 2\pi\text{-periodic in }\varphi\},
\end{align*}
and
\begin{align*}
C_o=\left\{\phi\in H_2^2(\mathbb{S}^2)| \phi(2\pi-\varphi,s)=\phi(\varphi,s), \phi\text{ is odd in } s \text{ and } 2\pi\text{-periodic in }\varphi\right\}.
\end{align*}
Here, we do not distinguish between $\phi(\varphi,s)$ and $\phi(\zeta(\mathbf{x}))$.
Note that $F_o$ is well-defined since if $\phi$ is odd in $s$, then by the oddness of $\mathbf{s}$ on $z$, we have
\begin{align*}
h(\phi+\Psi_0+\mu s,\mu)-h(\Psi_0+\mu s,\mu)=-12\phi+(12\mu+2\omega)(\mathbf{s}(\phi+\Psi_0+\mu s,\mu)-\mathbf{s}(\Psi_0+\mu s,\mu))
\end{align*}
is also odd in $s$.
Moreover, $F_o(0,\mu)=0$ for $\mu\in[\mu_{1,\omega}-\delta_0,\mu_{1,\omega}+\delta_0]$. To look for the solutions of \eqref{perturbation elliptic equation} near $\phi=0$, we study the bifurcation of $F_o(\phi,\mu)=0$ near the solutions $(0,\mu_{1,\omega})$. By \eqref{3jetflowstream function vorticity}, the Fr\'{e}chet derivative of $F_o$ near $\phi=0$ is
\begin{align*}
\partial_\phi F_o(0,\mu)=\Delta-\partial_zh(\Psi_0+\mu s,\mu)=\Delta-{\Upsilon_0'+2\omega\over \Psi_0'+\mu}.
\end{align*}
To study the dimension of $\ker(\partial_\phi F_o(0,\mu_{1,\omega}))$, we decompose $\partial_\phi F_o(0,\mu_{1,\omega})$ into the Fourier modes
\begin{align}\label{kernel-fourier modes-decom}
\Delta_k-{\Upsilon_0'+2\omega\over \Psi_0'+\mu_{1,\omega}}=\Delta_k -{2\omega+12\mu_{1,\omega}\over 15s^2-3+\mu_{1,\omega}}+12,
\end{align}
where $k\in\mathbb{Z}$.
If $\omega\in({99\over2},72)$, then $\mu_{1,\omega}<-12$ by Corollary \ref{uniqueness1}, and $\min_{s\in[-1,1]}{2\omega+12\mu_{1,\omega}\over 15s^2-3+\mu_{1,\omega}}={2\omega+12\mu_{1,\omega}\over -3+\mu_{1,\omega}}>0$. If $\omega\in(-18,-3)$, then $\mu_{1,\omega}>3$ by Corollary \ref{uniqueness2}, and $\min_{s\in[-1,1]}{2\omega+12\mu_{1,\omega}\over 15s^2-3+\mu_{1,\omega}}={2\omega+12\mu_{1,\omega}\over 12+\mu_{1,\omega}}>0$.
Thus, in any case, there exists $C_0>0$ independent of $s\in[-1,1]$ such that
 \begin{align}\label{low bound 2 omega+12mu1over 15s2-3+mu_1}
 {2\omega+12\mu_{1,\omega}\over 15s^2-3+\mu_{1,\omega}}> C_0.
 \end{align}
For $|k|\geq3$, we have $-\Delta_k\geq |k|(|k|+1)\geq12$ and thus, by \eqref{kernel-fourier modes-decom}-\eqref{low bound 2 omega+12mu1over 15s2-3+mu_1}, we have
\begin{align*}
\Delta_k-{\Upsilon_0'+2\omega\over \Psi_0'+\mu_{1,\omega}}\leq - {2\omega+12\mu_{1,\omega}\over 15s^2-3+\mu_{1,\omega}}< -C_0<0.
\end{align*}
For $|k|=2$, since the functions in $B_o$ are odd in $s$, we have $-\Delta_k\geq 12$ and $\Delta_k-{\Upsilon_0'+2\omega\over \Psi_0'+\mu_{1,\omega}}\leq- {2\omega+12\mu_{1,\omega}\over 15s^2-3+\mu_{1,\omega}}< -C_0<0$. For $|k|=1$, by the neutral mode $(c_{1,\omega},1,\omega,\Phi_{\mu_{1,\omega},\omega,1})$ in Corollaries \ref{uniqueness1} and \ref{uniqueness2}, we have $\Phi_{\mu_{1,\omega},\omega,1}(s)\cos(\varphi)\in \ker(\partial_\phi F_o(0,\mu_{1,\omega}))$, where $c_{1,\omega}=\omega+\mu_{1,\omega}$. For $|k|=1$, there are no more contributions to $\ker(\partial_\phi F_o(0,\mu_{1,\omega}))$ since the second eigenvalue of the operator $\Delta_k-{2\omega+12\mu_{1,\omega}\over 15s^2-3+\mu_{1,\omega}}$ is less than $-20-C_0$, where we used \eqref{low bound 2 omega+12mu1over 15s2-3+mu_1} and $\Phi$ is odd in $s$.
For $k=0$, since the second eigenvalue of the operator $\Delta_0-{2\omega+12\mu_{1,\omega}\over 15s^2-3+\mu_{1,\omega}}$ is less than $-12-C_0$, we only need to consider the principal eigenvalue. Suppose that the principal eigenvalue of $\Delta_0-{2\omega+12\mu_{1,\omega}\over 15s^2-3+\mu_{1,\omega}}$ is $-12$, then its eigenfunction $\Phi_1\in H_2^1(\mathbb{S}^2)$ satisfies
\begin{align}\label{0 mode principal eigenvalue 12}
\int_{-1}^1\left((1-s^2)\Phi_1'\Phi_{\mu_{1,\omega},\omega,1}'+{2\omega+12\mu_{1,\omega}\over 15s^2-3+\mu_{1,\omega}}\Phi_1\Phi_{\mu_{1,\omega},\omega,1}\right)ds=12\int_{-1}^1\Phi_1\Phi_{\mu_{1,\omega},\omega,1}ds,
\end{align}
 $\Phi_1>0$ for $s\in(0,1]$, $\Phi_1<0$ for $s\in[-1,0)$ and $\Phi_1(0)=0$, where $\Phi_{\mu_{1,\omega},\omega,1}$ is the neutral solution in Corollaries \ref{uniqueness1} and \ref{uniqueness2}. Moreover,
 \begin{align}\label{1 mode principal eigenvalue 12}
&\int_{-1}^1\left((1-s^2)\Phi_{\mu_{1,\omega},\omega,1}'\Phi_1'+{1\over1-s^2}\Phi_{\mu_{1,\omega},\omega,1}\Phi_1+{2\omega+12\mu_{1,\omega}\over 15s^2-3+\mu_{1,\omega}}\Phi_{\mu_{1,\omega},\omega,1}\Phi_1\right)ds\\\nonumber
=&12\int_{-1}^1\Phi_{\mu_{1,\omega},\omega,1}\Phi_1ds,
\end{align}
 $\Phi_{\mu_{1,\omega},\omega,1}>0$ for $s\in(0,1)$, $\Phi_{\mu_{1,\omega},\omega,1}<0$ for $s\in(-1,0)$ and $\Phi_{\mu_{1,\omega},\omega,1}(0)=0$.
 Combining \eqref{0 mode principal eigenvalue 12}-\eqref{1 mode principal eigenvalue 12}, we have
 \begin{align*}
 \int_{-1}^1{1\over 1-s^2}\Phi_1\Phi_{\mu_{1,\omega},\omega,1}ds=0,
 \end{align*}
 which contradicts $\Phi_1\Phi_{\mu_{1,\omega},\omega,1}>0$ on $s\in(-1,0)\cup(0,1)$. Thus, the principal eigenvalue of $\Delta_0-{\Upsilon_0'+2\omega\over \Psi_0'+\mu_{1,\omega}}$ is not $-12$. In summary,
 \begin{align*}
\ker(\partial_\phi F_o(0,\mu_{1,\omega}))=\text{span}\{\Phi_{\mu_{1,\omega},\omega,1}(s)\cos(\varphi)\}.
 \end{align*}
 Direct computation implies that
 \begin{align*}
 \partial_\mu\partial_\phi F_o(0,\mu_{1,\omega})(\Phi_{\mu_{1,\omega},\omega,1}(s)\cos(\varphi))={-12(15s^2-3)+2\omega\over (15s^2-3+\mu_{1,\omega})^2}\Phi_{\mu_{1,\omega},\omega,1}(s)\cos(\varphi)
 \end{align*}
 and by \eqref{comp-eigenvalue derivative}, we have
 \begin{align}\nonumber
&\bigg(\Phi_{\mu_{1,\omega},\omega,1}(s)\cos(\varphi), \partial_\mu\partial_\phi F_o(0,\mu_{1,\omega})(\Phi_{\mu_{1,\omega},\omega,1}(s)\cos(\varphi))\bigg)_{L^2}\\\label{inner product phi cos partial mu phi Fo}
=&\int_{0}^{2\pi}\int_{-1}^1 {-12(15s^2-3)+2\omega\over (15s^2-3+\mu_{1,\omega})^2}(\Phi_{\mu_{1,\omega},\omega,1}(s)\cos(\varphi))^2dsd\varphi=\pi\partial_\mu\lambda_1(\mu_{1,\omega},\omega),
 \end{align}
 where $\lambda_1(\mu_{1,\omega},\omega)$ is the principal eigenvalue of \eqref{Rayleigh-type equation lambda k=1} with $\mu=\mu_{1,\omega}$.
 Then we divide the discussion into two cases.
In the case that $\partial_\mu\lambda_1(\mu_{1,\omega},\omega)\neq0$, by \eqref{inner product phi cos partial mu phi Fo} we have $\partial_\mu\partial_\phi F_o(0,\mu_{1,\omega})$ $(\Phi_{\mu_{1,\omega},\omega,1}(s)$ $\cos(\varphi))\notin Ran(\partial_\phi F_o(0,\mu_{1,\omega}))$. Then by Crandall-Rabinowitz local bifurcation theorem in \cite{CR71}, there exists a $C^1$ bifurcating curve
\begin{align*}
\{(\phi_{(\gamma),1},c_1(\gamma))|\gamma\in(-\delta,\delta),(\phi_{(0),1},c_1(0))=(0,\mu_{1,\omega})\}
\end{align*}
for some $\delta>0$ such that
 $F_o(\phi_{(\gamma),1},c(\gamma))=0, \gamma\in(-\delta,\delta)$, and
 \begin{align*}
 \phi_{(\gamma),1}(\varphi,s)=\gamma\Phi_{\mu_{1,\omega},\omega,1}(s)\cos(\varphi)+o(|\gamma|),
 \end{align*}
 which implies \eqref{a curve of traveling wave solution 1 mode}.

 In the other case that $\partial_\mu\lambda_1(\mu_{1,\omega},\omega)=0$, since $\lambda_1(\cdot,\omega)$ is analytic on a neighborhood of $\mu_{1,\omega}$ and $\lambda_1(\cdot,\omega)$ is not a constant,  there exists  $m_0>1$ such that $\partial_\mu^{j}\lambda_1(\mu_{1,\omega},\omega)=0$ for $1\leq j<m_0$ and $\partial_\mu^{m_0}\lambda_1(\mu_{1,\omega},\omega)\neq0$.
 We claim that $m_0$ is odd and only give its proof for $\omega\in\left({99\over2},72\right)$ since the proof for $\omega\in(-18,-3)$ is similar. Suppose that $m_0$ is even. If $\partial_\mu^{m_0}\lambda_1(\mu_{1,\omega},\omega)<0$, noting that $\lambda_1(\mu_{1,\omega},\omega)=-12$ by \eqref{lambda1-mu1-omega-12}, we have $\lambda_1(\mu,\omega)<-12$ for $\mu\neq\mu_{1,\omega}$ sufficiently close to $\mu_{1,\omega}$. This, along with
$
\lim_{\mu\to-12^-}\lambda_1(\mu,\omega)=\lambda_1(-12,\omega)>-12$ due to Lemma \ref{continuity of the principal eigenvalue mu -12} and \eqref{principal eigenvalue -12 omega}, implies that there exists $\tilde \mu_{1,\omega}\in(\mu_{1,\omega},-12)$ such that
$
\lambda_1(\tilde\mu_{1,\omega},\omega)=-12$.
This gives  a new neutral mode $(\tilde c_{1,\omega},1,\omega,\Phi_{\tilde \mu_{1,\omega},\omega,1})$ with
$\tilde c_{1,\omega}=\omega+\tilde\mu_{1,\omega}$. Then
 $\mu_{1,\omega}\neq\tilde \mu_{1,\omega}<-12$ contradicts the uniqueness in Corollary \ref{uniqueness1}. If $\partial_\mu^{m_0}\lambda_1(\mu_{1,\omega},\omega)>0$, then $\lambda_1(\mu,\omega)>-12$ for $\mu\neq\mu_{1,\omega}$ sufficiently close to $\mu_{1,\omega}$. This, along with
$
\lim_{\mu\to-\infty}\lambda_1(\mu,\omega)=-18$ due to Lemma \ref{asymptotic behavior principal eigenvalue lim mu -infty}, implies that there exists $\hat \mu_{1,\omega}\in(-\infty,\mu_{1,\omega})$ such that
$
\lambda_1(\hat\mu_{1,\omega},\omega)=-12$.
This gives  another new neutral mode $(\hat c_{1,\omega},1,\omega,\Phi_{\hat \mu_{1,\omega},\omega,1})$ with
$\hat c_{1,\omega}=\omega+\hat\mu_{1,\omega}$. Then
 $\mu_{1,\omega}\neq\hat \mu_{1,\omega}<-12$ again contradicts the uniqueness in  Corollary \ref{uniqueness1}.
Thus, $m_0$ is odd.

Moreover, $F_o$ is analytic in a neighborhood of $(0,\mu_{1,\omega})$. By Kielh\"{o}fer's degenerate local bifurcation theorem (see Theorems 5.2-5.3 in \cite{Kielhofer1980} or Theorem I.16.4 in \cite{Kielhofer2012}),
 there exists a $C^0$ bifurcating curve
\begin{align*}
\{(\phi_{(\gamma),1},c_1(\gamma))|\gamma\in(-\delta,\delta),(\phi_{(0),1},c_1(0))=(0,\mu_{1,\omega})\}
\end{align*}
for some $\delta>0$ such that
 $F_o(\phi_{(\gamma),1},c_1(\gamma))=0, \gamma\in(-\delta,\delta)$, and
$
 \phi_{(\gamma),1}(\varphi,s)=\gamma\Phi_{\mu_{1,\omega},\omega,1}(s)$ $\cos(\varphi)+o(|\gamma|).
 $
 \\
 {\it Proof of (ii)}.
Note that $F_e:B_e\to C_e$ can not be defined  similarly as in the proof of (i), since the image of $\phi$ (even in $s$) under $F_e$ is not necessarily even in $s$. Instead, we replace the condition that $\phi$ is even in $s$ to the condition that $\phi$ is $\pi$-periodic in $\varphi$. The new condition is preserved under $F_e$. More precisely, we
define the mapping
\begin{align*}
&F_e:B_e\times[\mu_{2,\omega}-\delta_0,\mu_{2,\omega}+\delta_0]\rightarrow C_e,\\\nonumber
 &\;\;\;\;\;\;\;\;\;\;\;\;\;\;\;\;\;\;\;\;\;\;\;\;\;\;\;\;\;\;\;\;\;(\phi,\mu)\mapsto\Delta \phi-(h(\phi+\Psi_0+\mu s,\mu)-h(\Psi_0+\mu s,\mu)),
\end{align*}
where $\mu_{2,\omega}$ is given in Corollaries \ref{uniqueness3} and \ref{uniqueness4},
\begin{align*}
B_e=\bigg\{\phi\in H_2^4(\mathbb{S}^2)|
 \phi(\pi-\varphi,s)=\phi(\varphi,s)  \text{ and } \pi\text{-periodic in }\varphi\bigg\},
\end{align*}
and
\begin{align*}
C_e=\left\{\phi\in H_2^2(\mathbb{S}^2)|  \phi(\pi-\varphi,s)=\phi(\varphi,s)\text{ and } \pi\text{-periodic in }\varphi\right\}.
\end{align*}
Here, we do not distinguish between $\phi(\varphi,s)$ and $\phi(\zeta(\mathbf{x}))$.
 We decompose $\partial_\phi F_e(0,\mu_{2,\omega})=\Delta-{\Upsilon_0'+2\omega\over \Psi_0'+\mu_{2,\omega}}$ into the Fourier modes
\begin{align*}
\Delta_{2k}-{\Upsilon_0'+2\omega\over \Psi_0'+\mu_{2,\omega}}=\Delta_{2k} -{2\omega+12\mu_{2,\omega}\over 15s^2-3+\mu_{2,\omega}}+12,
\end{align*}
where $k\in\mathbb{Z}$. Similar to \eqref{low bound 2 omega+12mu1over 15s2-3+mu_1},
there exists $C_0>0$ independent of $s\in[-1,1]$ such that
 \begin{align}\label{low bound 2 omega+12mu1over 15s2-3+mu_2}
 {2\omega+12\mu_{2,\omega}\over 15s^2-3+\mu_{2,\omega}}> C_0.
 \end{align}
For $|k|\geq2$, by \eqref{low bound 2 omega+12mu1over 15s2-3+mu_2} we have $-\Delta_{2k}\geq 2|k|(2|k|+1)\geq20$ and thus, $\Delta_{2k}-{\Upsilon_0'+2\omega\over \Psi_0'+\mu_{2,\omega}}< -8-C_0<0$.
 For $|k|=1$, by the neutral mode $(c_{2,\omega},2,\omega,\Phi_{\mu_{2,\omega},\omega,2})$ in Corollaries \ref{uniqueness3} and \ref{uniqueness4}, we have $\Phi_{\mu_{2,\omega},\omega,2}(s)\cos(2\varphi)\in \ker(\partial_\phi F_e(0,\mu_{2,\omega}))$, where $c_{2,\omega}=\omega+\mu_{2,\omega}$. For $|k|=1$, there are no more contributions to $\ker(\partial_\phi F_e(0,\mu_{2,\omega}))$ since the second eigenvalue of the operator $\Delta_{2k}-{2\omega+12\mu_{2,\omega}\over 15s^2-3+\mu_{2,\omega}}$ is less than $-12-C_0$.
For $k=0$, we note that for $n\geq4$, the $n$-th eigenvalue of the operator $\Delta_0-{2\omega+12\mu_{2,\omega}\over 15s^2-3+\mu_{2,\omega}}$  is less than $-12-C_0$.
The principal eigenvalue of $\Delta_0-{2\omega+12\mu_{2,\omega}\over 15s^2-3+\mu_{2,\omega}}$ can be ruled out from contributing to $\ker(\partial_\phi F_e(0,\mu_{2,\omega}))$  by  a similar way as  \eqref{0 mode principal eigenvalue 12}-\eqref{1 mode principal eigenvalue 12}.
How to rule out contributions of the second and the third eigenvalues of $\Delta_0-{2\omega+12\mu_{2,\omega}\over 15s^2-3+\mu_{2,\omega}}$   is much more subtle to be dealt with. We give an approach based on the Frobenius method and the  variation of parameters technique for  solving  ODEs. Here, we only give the proof of ruling out the contributions of  the third eigenvalue of $\Delta_0-{2\omega+12\mu_{2,\omega}\over 15s^2-3+\mu_{2,\omega}}$ to $\ker(\partial_\phi F_e(0,\mu_{2,\omega}))$, since the other proof is similar. Note that an eigenfunction of the third eigenvalue is even.
We rewrite the ODEs
\begin{align}\label{zero mode equation and second mode equation}
\Delta_0\Phi-{2\omega+12\mu_{2,\omega}\over 15s^2-3+\mu_{2,\omega}}\Phi=-12\Phi \quad\text{and}\quad\Delta_2\Phi-{2\omega+12\mu_{2,\omega}\over 15s^2-3+\mu_{2,\omega}}\Phi=-12\Phi
\end{align}
to
\begin{align}\label{homogeneous ode}
(1-s)^2\Phi''+{2s(s-1)\over s+1}\Phi'+{(s-1)(-12(15s^2-3)+2\omega)\over (s+1)(15s^2-3+\mu_{2,\omega})}\Phi=0
\end{align}
and
\begin{align}\label{inhomogeneous ode}
\Phi''-{2s\over 1-s^2}\Phi'-\left({-12(15s^2-3)+2\omega\over 15s^2-3+\mu_{2,\omega}}+{4\over 1-s^2}\right){1\over1-s^2}\Phi=0,
\end{align}
respectively.

 We will prove that if \eqref{homogeneous ode} admits a nontrivial even solution $\Phi_*\in L^2(\mathbb{S}^2)$, then
 any nontrivial even solution $\hat \Phi_*$ of \eqref{inhomogeneous ode} satisfies that  $|\hat \Phi_*(\pm1)|=\infty$.
 This contradicts the existence of the even neutral solution $\Phi_{\mu_{2,\omega},\omega,2}$ with $\Phi_{\mu_{2,\omega},\omega,2}(\pm1)=0$ by Corollaries \ref{uniqueness3} and \ref{uniqueness4}.

 In fact,
since
\begin{align*}
p(s)={2s\over s+1}\quad\text{and}\quad q(s)={(s-1)(-12(15s^2-3)+2\omega)\over (s+1)(15s^2-3+\mu_{2,\omega})}
\end{align*}
are analytic near $1$, $p(1)=1$ and $q(1)=0$, we infer
from the Frobenius method for solving ODEs that
 the indicial equation of \eqref{homogeneous ode} is
 \begin{align*}
 r(r-1)+r=r^2=0.
 \end{align*}
 Thus, a  pair of  linearly independent solutions $\Phi_{*,1}$ and $\Phi_{*,2}$ of \eqref{homogeneous ode} has the expressions
 \begin{align*}
 \Phi_{*,1}(s)=\sum_{j=0}^\infty a_j(s-1)^j, \quad \Phi_{*,2}(s)= \Phi_{*,1}(s)\ln|s-1|+\sum_{j=0}^\infty b_j(s-1)^j
 \end{align*}
 near the endpoint $1$, where $a_0\neq0$. Then $\Phi_{*,1}(1)=a_0\neq0$ and $|\Phi_{*,2}(1)|=\infty$. Since \eqref{homogeneous ode} admits a nontrivial even solution $\Phi_*$ with $\Phi_*\in L^2(\mathbb{S}^2)$, by the first equation in \eqref{zero mode equation and second mode equation} we have $\Phi_*\in H_2^2(\mathbb{S}^2)$. By Theorem 2.7 in \cite{Hebey2000}, $\Phi_*\in C^0(\mathbb{S}^2)$. Thus, $|\Phi_*(\pm1)|<\infty$. Then $\Phi_*= \Phi_{*,1}$ on $[-1,1]$ and $ \Phi_{*,1}(-1)=\Phi_{*,1}(1)\neq0$. Since $\Phi_{*,2}$ can be chosen as an odd function, up to a constant we have
 \begin{align}\label{Phi*2}
 \Phi_{*,2}(s)=\Phi_{*,1}(s)\int_0^s{1\over |\Phi_{*,1}(\tilde s)|^2}e^{\int_0^{\tilde s}{2\hat s\over 1-\hat s^2} d\hat s} d\tilde s=\Phi_{*,1}(s)\int_0^s{1\over |\Phi_{*,1}(\tilde s)|^2}{1\over 1-\tilde s^2} d\tilde s.
 \end{align}
This implies $|\Phi_{*,2}(\pm1)|=\infty$. Direct computation implies that the Wronskian is
\begin{align}\label{Wronskian}
\Phi_{*,1}(s)\Phi_{*,2}'(s)-\Phi_{*,2}(s)\Phi_{*,1}'(s)={1\over 1-s^2} \quad \text{for} \quad s\in(-1,1).
\end{align}
 The equation \eqref{homogeneous ode} has the form
 \begin{align}\label{homogeneous ode2}
\Phi''-{2s\over 1-s^2}\Phi'-{-12(15s^2-3)+2\omega\over (15s^2-3+\mu_{2,\omega})(1-s^2)}\Phi=0.
\end{align}
We regard \eqref{homogeneous ode2} as the homogeneous equation of \eqref{inhomogeneous ode}, and
 the term $-{4\over (1-s^2)^2}\Phi$ in \eqref{inhomogeneous ode} as the inhomogeneous term.
 By the method of variation of parameters and \eqref{Wronskian}, the nontrivial even solution $\hat \Phi_*$ of \eqref{inhomogeneous ode}
has the form
\begin{align*}
\hat \Phi_*(s)=C_1\Phi_{*,1}(s)+C_2\Phi_{*,2}(s)+\int_0^s(\Phi_{*,1}(\tilde s)\Phi_{*,2}(s)-\Phi_{*,1}(s)\Phi_{*,2}(\tilde s)){4\over  1-\tilde s^2}\hat \Phi_*(\tilde s) d\tilde s
\end{align*}
for some $C_1, C_2\in \mathbb{R}$,
where $s\in(-1,1)$.
Note that the terms $\hat \Phi_*, \Phi_{*,1}$ and $\int_0^s(\Phi_{*,1}(\tilde s)\Phi_{*,2}(s)-\Phi_{*,1}(s)\Phi_{*,2}(\tilde s)){4\over  1-\tilde s^2}\hat \Phi_*(\tilde s) d\tilde s$ are even functions, while  $\Phi_{*,2}$ is odd. Thus,  $C_2=0$. By \eqref{Phi*2}, we have
\begin{align}\nonumber
\hat \Phi_*(s)=&C_1\Phi_{*,1}(s)+\int_0^s\bigg(\Phi_{*,1}(\tilde s)\Phi_{*,1}(s)\int_0^s{1\over |\Phi_{*,1}(\hat s)|^2}{1\over 1-\hat s^2} d\hat s
\\&-\Phi_{*,1}(s)\Phi_{*,1}(\tilde s)\int_0^{\tilde s}{1\over |\Phi_{*,1}(\hat s)|^2}{1\over 1-\hat s^2} d\hat s\bigg){4\over  1-\tilde s^2}\hat \Phi_*(\tilde s) d\tilde s\nonumber\\\label{hat Phi*}
=&C_1\Phi_{*,1}(s)+\Phi_{*,1}(s)\int_0^s\bigg(\Phi_{*,1}(\tilde s)\int_{\tilde s}^s{1\over |\Phi_{*,1}(\hat s)|^2}{1\over 1-\hat s^2} d\hat s\bigg){4\over  1-\tilde s^2}\hat \Phi_*(\tilde s) d\tilde s.
\end{align}
Let
\begin{align*}
\Phi=(1-s^2)\tilde \Phi
\end{align*}
in
 \eqref{inhomogeneous ode}.
 Then \eqref{inhomogeneous ode} is transformed to
\begin{align}\label{inhomogeneous ode2}
&(1-s)^2\tilde\Phi''+{6s(s-1)\over 1+s}\tilde\Phi'\\\nonumber
+&\left(-2+{4s^2\over1-s^2}-{-12(15s^2-3)+2\omega\over 15s^2-3+\mu_{2,\omega}}-{4\over 1-s^2}\right){1-s\over 1+s}\tilde\Phi=0.
\end{align}
Noting that
\begin{align*}
\tilde p(s)={6s\over 1+s}\quad\text{and}\quad \tilde q(s)=\left(-2+{4s^2\over1-s^2}-{-12(15s^2-3)+2\omega\over 15s^2-3+\mu_{2,\omega}}-{4\over 1-s^2}\right){1-s\over 1+s}
\end{align*}
are analytic near $1$, $\tilde p(1)=3$ and $\tilde q(1)=0$, we know that
 the indicial equation of \eqref{inhomogeneous ode2} is
$
 r(r-1)+3r=r^2+2r=0.$
 Thus, $r=0$ or $r=-2$, and a  pair of  linearly independent solutions $\tilde \Phi_{*,1}$ and $\tilde \Phi_{*,2}$ of \eqref{inhomogeneous ode2} is
 \begin{align*}
 \tilde\Phi_{*,1}(s)=\sum_{j=0}^\infty \tilde a_j(s-1)^j, \quad \tilde \Phi_{*,2}(s)= \tilde a\tilde\Phi_{*,1}(s)\ln|s-1|+(s-1)^{-2}\sum_{j=0}^\infty \tilde b_j(s-1)^j
 \end{align*}
  near the endpoint $1$, where $\tilde a_0\neq0$ and $\tilde a\in\mathbb{R}$.
  Moreover, there exist $\tilde C_1\in\mathbb{R}$ and $\tilde C_2\neq0$ such that
  \begin{align*}
\tilde \Phi_{*,2}(s)=&\tilde C_1\tilde \Phi_{*,1}(s)+\tilde C_2\tilde \Phi_{*,1}(s)\int_0^s{1\over |\tilde \Phi_{*,1}(\tilde s)|^2}e^{-\int_0^{\tilde s}{6\hat s\over \tilde s^2-1}d\hat s}d\tilde s\\
=&\tilde C_1\tilde \Phi_{*,1}(s)+\tilde C_2\tilde \Phi_{*,1}(s)\int_0^s{1\over |\tilde \Phi_{*,1}(\tilde s)|^2}{1\over (\tilde s^2-1)^3}d\tilde s.
  \end{align*}
  This, along with $\tilde\Phi_{*,1}(1)\neq0$, implies $|\tilde \Phi_{*,2}(1)|=\infty$.
 This implies that
 $\hat \Phi_*(s)=(1-s^2)\tilde\Phi_{*,1}(s)$ near $1$.
 By \eqref{hat Phi*}, we have
 \begin{align}\nonumber
\hat \Phi_*(1)
=C_1\Phi_{*,1}(1)+\Phi_{*,1}(1)\int_0^1\bigg(\Phi_{*,1}(\tilde s)\int_{\tilde s}^1{1\over |\Phi_{*,1}(\hat s)|^2}{1\over 1-\hat s^2} d\hat s\bigg){4\over  1-\tilde s^2}\hat \Phi_*(\tilde s) d\tilde s.
\end{align}
Noting that $\Phi_{*,1}(1)\neq 0$, we have $\left|\int_{\tilde s}^1{1\over |\Phi_{*,1}(\hat s)|^2}{1\over 1-\hat s^2} d\hat s\right|=\infty$ for $\tilde s\in(0,1)$. Moreover,
${4\over  1-\tilde s^2}\hat \Phi_*(\tilde s) ={4\over  1-\tilde s^2}(1-\tilde s^2)\tilde\Phi_{*,1}(\tilde s)=O\left(1\right)$ as $\tilde s\to1^-$ due to
$\tilde\Phi_{*,1}(1)\neq 0$. Thus, $|\hat \Phi_*(\pm1)|=\infty$ since $\hat \Phi_*$ is even.
This proves that the third eigenvalue of $\Delta_0-{2\omega+12\mu_{2,\omega}\over 15s^2-3+\mu_{2,\omega}}$ has no contributions to $\ker(\partial_\phi F_e(0,$ $\mu_{2,\omega}))$.

Therefore, from the $0$'th Fourier mode, there are no  contributions to $\ker(\partial_\phi F_e(0,\mu_{2,\omega}))$,
and thus,
 \begin{align*}
\ker(\partial_\phi F_e(0,\mu_{2,\omega}))=\text{span}\{\Phi_{\mu_{2,\omega},\omega,2}(s)\cos(2\varphi)\}.
 \end{align*}
The rest of the proof is similar to that of (i) and we sketch it as follows. Based on
 \begin{align*}
&\bigg(\Phi_{\mu_{2,\omega},\omega,2}(s)\cos(2\varphi), \partial_\mu\partial_\phi F_e(0,\mu_{2,\omega})(\Phi_{\mu_{2,\omega},\omega,2}(s)\cos(2\varphi))\bigg)_{L^2}=\pi\partial_\mu\tilde\lambda_1(\mu_{2,\omega},\omega),
 \end{align*}
we divide the discussion into the cases that $\partial_\mu\tilde\lambda_1(\mu_{2,\omega},\omega)\neq0$
 and $\partial_\mu\tilde \lambda_1(\mu_{2,\omega},\omega)=0$, respectively, where $\tilde\lambda_1(\mu_{2,\omega},\omega)$ is the principal eigenvalue of \eqref{Rayleigh-type equation lambda k=2} with $\mu=\mu_{2,\omega}$.
 In the first case, we apply the Crandall-Rabinowitz local bifurcation theorem to obtain the desired curve of travelling waves with stream functions satisfying
 \eqref{a curve of traveling wave solution 2 mode0}.
 In the second case that $\partial_\mu\tilde\lambda_1(\mu_{2,\omega},\omega)=0$, we can  prove that  there exists  odd
 $m_0>1$ such that $\partial_\mu^{j}\tilde\lambda_1(\mu_{2,\omega},\omega)=0$ for $1\leq j<m_0$ and $\partial_\mu^{m_0}\tilde \lambda_1(\mu_{2,\omega},\omega)\neq0$.
 This allows us to apply the Kielh\"{o}fer's degenerate local bifurcation theorem to obtain the curve of travelling waves with the stream functions satisfying  \eqref{a curve of traveling wave solution 2 mode0}.
  \\
 {\it Proof of (iii).} The proof of (iii) is similar and indeed simpler than that of (ii), since  $\partial_\mu\tilde \lambda_1(\mu_{3,\omega},\omega)=0$ never occurs by Corollary \ref{uniqueness4}.
 Thus, when applying  local bifurcation theorem, we only need to use the Crandall-Rabinowitz's theorem to obtain the curve of travelling waves with the stream functions satisfying \eqref{a curve of traveling wave solution 2 mode}.
 \end{proof}
Next, we construct  travelling waves near the $3$-jet  based on the Rossby-Haurwitz waves.
\begin{Lemma}\label{travelling waves c omega}
 Let $\omega\in\mathbb{R}$. Then
 the  travelling waves  $\Psi_0(s)+\varepsilon Y(\varphi+{1\over 6}\omega t,s)$  are sufficiently close in
analytic regularity to the $3$-jet, where $Y\in E_3$ is non-zonal and $\varepsilon$ is small enough. Moreover,

$(\rm{i})$ their streamlines have at least a cat's eyes structure for $\omega\in(-18,72)$,

$(\rm{ii})$ all their streamlines are unidirectional for $\omega\in(-\infty,-18)\cup(72,\infty)$.
\end{Lemma}
\begin{proof} The  travelling wave  $\Psi_0(s)+\varepsilon Y(\varphi+{1\over 6}\omega t,s)$ is a non-zonal Rossby-Haurwitz wave for any $\varepsilon\neq0$ and for any $\omega\in\mathbb{R}$.
For $\omega\in(-18,72)$, the travelling speed $c=-{1\over 6}\omega\in (-12,3)=\text{Ran} (-\Psi_0')^\circ$ and thus  the streamlines have at least a cat's eyes structure. For $\omega\in(-\infty,-18)\cup(72,\infty)$, the travelling speed $c=-{1\over 6}\omega\in (-\infty,-12)\cup(3,\infty)$ and thus $c\notin\text{Ran} (-\Psi_0')^\circ$. Then all the streamlines  are unidirectional.
\end{proof}

\begin{Remark}
In Lemma $\ref{travelling waves c omega}$, the travelling waves are due to the kernel of $J_\omega L$. By the method in \cite{Zelati-Elgindi-Widmayer2023,Nualart2023}, it is expected that there might exist  other cat's eyes travelling waves  arbitrarily close in analytic regularity to the $3$-jet. We do not pursue constructing more nearby travelling waves here since our concern is how the streamline patterns of nearby travelling waves change as the rotation speed increases, see Fig. \ref{Gradual-changes-in-streamline-patterns-of-travelling-waves}.
\end{Remark}

\subsection{Types of  imaginary eigenvalues  of the linearized operators}
To study the rigidity of travelling waves near the $3$-jet in the next subsection,
as a preparatory  work at the linear level,
we  study the types of purely imaginary eigenvalues  of the linearized operator $\mathcal{L}_{\omega,k}$ (the projection of $\mathcal{L}_{\omega}$ on the $k$'th Fourier mode) for $k\neq0$.

\begin{Lemma}\label{spectra of the linearized operatorLrigidity}
$({\rm{i}})$ Let $\omega\in(-3,{69\over2})$. Then  $\mathcal{L}_{\omega,k_0}|_{X^{k_0}}$ has  an   eigenvalue  ${ik_0\over 6}\omega$ embedded  in the  interior of $\sigma_e(\mathcal{L}_{\omega,k_0}|_{X^{k_0}})$ for $0<|k_0|\leq3$, and $\mathcal{L}_{\omega,k}|_{X^k}$ has  no purely imaginary isolated eigenvalues for $k\neq0$. Moreover, the endpoints $-3ki$ and $12ki$ of $\sigma_e(\mathcal{L}_{\omega,k}|_{X^{k}})$ are not eigenvalues of $\mathcal{L}_{\omega,k}|_{X^k}$ for $k\neq0$.

$({\rm{ii}})$ Let $\omega\in(-18,-3)\cup({69\over2},72)$. Then $\mathcal{L}_{\omega,k_0}|_{X^{k_0}}$ has  an embedded  eigenvalue  ${ik_0\over 6}\omega$ for $0<|k_0|\leq3$, and there exists $k_1\in\{1,2\}$ such that $\mathcal{L}_{\omega,k_1}|_{X^{k_1}}$ has  a purely imaginary isolated eigenvalue.

$({\rm{iii}})$
Let $\omega\in(-\infty,-18)\cup\omega\in(72,\infty)$. Then
$\mathcal{L}_{\omega,k}|_{X^k}$ has  no embedded  eigenvalues in the interior of  $\sigma_e(\mathcal{L}_{\omega,k})=ik Ran(\Psi_0')$ for $k\neq0$,
and $\mathcal{L}_{\omega,k_0}|_{X^{k_0}}$ has  an isolated eigenvalue ${ik_0\over 6}\omega$ for $0<|k_0|\leq3$.

$({\rm{iv}})$
Let $\omega\in(-\infty,-18)$ or $\omega\in(72,\infty)\setminus\{{15(j^2-m^2)+144\over2}\big|j\geq m, 0\leq m\leq3\}$. Then
$\mathcal{L}_{\omega,k}|_{X^k}$ has  no embedded  eigenvalues  for $k\neq0$.
\end{Lemma}
\begin{proof}
Since
$\mathcal{L}_\omega=
J_\omega L+{1\over 6}\omega\partial_\varphi$, we have $\mathcal{L}_{\omega,k}=
J_{\omega,k} L_k+{ik\over 6}\omega$ for $k\neq0$. For $0<|k_0|\leq3$, noting that $P_3^{k_0}\in\ker(L_{k_0}|_{X^{k_0}})$, we have $\mathcal{L}_{\omega,k_0}P_3^{k_0}=
J_{\omega,k_0} L_{k_0}P_3^{k_0}+{ik_0\over 6}\omega P_3^{k_0}={ik_0\over 6}\omega P_3^{k_0}$ for $\omega\in\mathbb{R}$. Thus, ${ik_0\over 6}\omega$ is an eigenvalue of $\mathcal{L}_{\omega,k_0}$. Note that $\sigma_e(\mathcal{L}_{\omega,k})=ik Ran(\Psi_0')=ik[-3,12]$ for $k\neq0$ and $\omega\in\mathbb{R}$. Then the eigenvalue ${ik_0\over 6}\omega$ of $\mathcal{L}_{\omega,k_0}$  is embedded in the interior of $\sigma_e(\mathcal{L}_{\omega,k_0}|_{X^{k_0}})$ for $\omega\in(-18,72)$ and is isolated for $\omega\in(-\infty,-18)\cup(72,\infty)$. Next, we prove the other parts of (i)-(iv), separately.

Note that  $-ik(c-c_\omega)=-ik(\omega+\mu-{5\over6}\omega)=-ik({1\over6}\omega+\mu)$ is an eigenvalue of  $J_{\omega,k}L_k$ if and only if $-ik\mu$ is an eigenvalue of $\mathcal{L}_{\omega,k}$.

Let $\omega\in(-3,{69\over2})$.  If $\omega\in(-3,12]$, by Lemmas \ref{tilde-psi-change-sign}-\ref{case-omega=12} $\mathcal{L}_{\omega,k}|_{X^k}$ has  no  isolated eigenvalues, and  the endpoints $-3ki$ and $12ki$ of $\sigma_e(\mathcal{L}_{\omega,k}|_{X^{k}})$ are not eigenvalues of $\mathcal{L}_{\omega,k}|_{X^k}$ for $k\neq0$.
If $\omega\in(12,{69\over2})$, by Lemma \ref{tilde Psi0 does not change sign} (2), $\mathcal{L}_{\omega,k}|_{X^k}$ has no  eigenvalues in $ik(-\infty,-3]$. Then we discuss $|k|=1$ and $|k|\geq2$ separately. For $|k|=1$, by \eqref{lambda1muomegalessthan-12}, $\mathcal{L}_{\omega,k}|_{X_o^k}$ has no eigenvalues in $ik[12,\infty)$.
Then we consider the restriction  in $X_e^k$. If $-ik\mu_0$ with $\mu_0\in(-\infty, -12]$ is an eigenvalue of $\mathcal{L}_{\omega,k}|_{X_e^k}$ with an  eigenfunction $\Delta_1\Phi_0$ satisfying $\|\Phi_0\|_{L^2}=1$, then
\begin{align*}
0\geq\int_{-1}^1-|\nabla_1\Phi_0|^2ds+12=\int_{-1}^1{2\omega+12\mu_0\over 15s-3+\mu_0}|\Phi_0|^2 ds\geq0,
\end{align*}
which implies $\Phi_0=0$ due to the fact that  ${2\omega+12\mu_0\over 15s-3+\mu_0}\geq0$ on $[-1,1]$. This is a contradiction.  Thus, $\mathcal{L}_{\omega,k}|_{X_e^k}$ and $\mathcal{L}_{\omega,k}|_{X^k}$ have no eigenvalues in $ik[12,\infty)$.
On the other hand, for $|k|\geq2$, suppose that there exists an eigenvalue
$-ik\mu_k$ with $\mu_k\in(-\infty, -12]$ of $\mathcal{L}_{\omega,k}$, and  a corresponding   eigenfunction $\Delta_k\Phi_k$ satisfies $\|\Phi_k\|_{L^2}=1$. By Remark \ref{rem tilde Psi0 does not change sign}, $\mu_k\in[-\omega,-12]$. Then
\begin{align}\nonumber
 &\Delta_k\Phi_k-{-12(15s^2-3)+2\omega\over15s^2-3+\mu_k}\Phi_k=0.
 \end{align}
 Let $
R(s)=\Psi_{0}'(s)+\mu_k=15s^2-3+\mu_k$ and $F(s)={\Phi_k(s)\over R(s)}$. Similar to \eqref{inte-t}, we have
\begin{align}\label{transformed-intRF}
\int_{-1}^{1}R^2\left|\left(\frac{F}{\sqrt{1-s^2}}\right)'\right|^2(1-s^2)^2ds+\int_{-1}^{1}\left(\frac{R^2F^2}{1-s^2}(k^2-1)+
2RF^2(\mu_k+\omega)\right)ds=0.
\end{align}
Since $-\omega\leq \mu_k\leq-12$,  $0\leq \mu_k+\omega\leq -12+\omega<{45\over2}$ and $|k|\geq2$, we have
\begin{align*}
\int_{-1}^{1}\left(\frac{R^2F^2}{1-s^2}(k^2-1)+
2RF^2(\mu_k+\omega)\right)ds=&\int_{-1}^{1}\left(\frac{k^2-1}{1-s^2}+
{2(\mu_k+\omega)\over R}\right)R^2F^2ds\\
\geq&\int_{-1}^{1}\left(k^2-1-
{2(\mu_k+\omega)\over 15}\right){\Phi_k^2\over 1-s^2}ds>0,
\end{align*}
which contradicts \eqref{transformed-intRF}. This proves that $\mathcal{L}_{\omega,k}$ has no  eigenvalues in $ik[12,\infty)$ for $|k|\geq2$.

Let $\omega\in(-18,-3)\cup({69\over2},72)$. If $\omega\in(-18,-3)\cup({99\over2},72)$, then $\mathcal{L}_{\omega,1}$ has an isolated eigenvalue in $i(-\infty,-3)\cup i(12,\infty)$ by
Corollaries \ref{uniqueness1} and \ref{uniqueness2}. If $\omega\in(-18,g^{-1}(-12))\cup({69\over2},72)$, then $\mathcal{L}_{\omega,2}$ has an isolated eigenvalue in $2i(-\infty,-3)\cup 2i(12,\infty)$ by
Corollaries \ref{uniqueness3} and \ref{uniqueness4}.

Let $\omega\in(-\infty,-18)\cup(72,\infty)$. By Theorem \ref{Psi prime change sign c},
  $\mathcal{L}_{\omega,k}$ has no eigenvalues embedded in $ik(-3,12)$ for $k\neq0$.

 If $\omega\in(-\infty,-18)$, then by Lemma \ref{tilde Psi0 does not change sign} (1),  $12ki$ is not an eigenvalue of  $\mathcal{L}_{\omega,k}$ for $k\neq0$. Then we show that $-3ki$ is not an eigenvalue of  $\mathcal{L}_{\omega,k}$ for $k\neq0$. Suppose that there exists $k_0\neq0$ such that $-3k_0i$ is an eigenvalue of  $\mathcal{L}_{\omega,k_0}$ with an eigenfunction $\Delta_{k_0}\hat \Phi_0$. Then
 \begin{align}\label{hat-Phi0mu3}
  ((1-s^2)\hat\Phi_0')'-{k_0^2\over 1-s^2}\hat\Phi_0-{2\omega+36\over 15s^2}\hat\Phi_0=-12\hat\Phi_0, \quad
 \Delta_{k_0}\hat\Phi_0\in L^2(-1,1).
 \end{align}
 Thus,
 \begin{align}\label{odeFrobeniusform}
  s^2\hat\Phi_0''-{2s^3\over 1-s^2}\hat\Phi_0'+\left(-{k_0^2s^2\over (1-s^2)^2}-{2\omega+36\over 15(1-s^2)}+{12s^2\over 1-s^2}\right)\hat\Phi_0=0.
 \end{align}
 Since
$
-{2s^2\over 1-s^2}$ and  $-{k_0^2s^2\over (1-s^2)^2}-{2\omega+36\over 15(1-s^2)}+{12s^2\over 1-s^2}
$
are analytic near $0$, by the Frobenius method
 the indicial equation of \eqref{odeFrobeniusform} is
$
 r(r-1)-{2\omega+36\over 15}=0,
$
whose two roots are $r_{\pm}={1\pm\sqrt{1+{8\omega+144\over 15}}\over2}$. For $\omega\in[-{159\over 8},-18)$, we have $r_+\in[{1\over2},1)$ and $r_-\in(0,{1\over2}]$.
Thus, two linearly independent solutions have the forms $s^{r_+}\sum_{j\geq0}a_js^j$ and $s^{r_-}\sum_{j\geq0}b_js^j$, where $a_j,b_j\in\mathbb{R}$ for $j\geq0$, $a_0\neq0$ and $b_0\neq0$. They are not $C^1$ near $0$, which is a contradiction.
For $\omega\in(-\infty, -{159\over 8})$, the two roots $r_{\pm}$ are non-real.
 For $s\geq0$, the real and imaginary parts of $s^{r_+}\sum_{j\geq0}d_js^{j}$ are two linearly independent solutions, where $d_j\in\mathbb{R}$ for $j\geq0$ and $d_0\neq0$ can be chosen as a real number. Note that
\begin{align*}
&s^{r_+}d_0=s^{{1+i\sqrt{-\left(1+{8\omega+144\over 15}\right)}\over2}}d_0=s^{{1}\over2}e^{{i\sqrt{-\left(1+{8\omega+144\over 15}\right)}\over2}\ln s} d_0\\
=&s^{1\over2}\left(\cos\left({\sqrt{-\left(1+{8\omega+144\over 15}\right)}\over2}\ln s\right)+i\sin\left({\sqrt{-\left(1+{8\omega+144\over 15}\right)}\over2}\ln s\right)\right)d_0.
\end{align*}
By \eqref{hat-Phi0mu3}, $\lim_{s\to0^+}{\hat\Phi_0(s)\over s}=0$.
On the other hand,
 there exist $a, b$, not both zero, such that
 \begin{align*}
 \lim_{s\to0^+}{\hat\Phi_0(s)\over s}= \lim_{s\to0^+}{ad_0\cos\left({\sqrt{-\left(1+{8\omega+144\over 15}\right)}\over2}\ln s\right)+bd_0\sin\left({\sqrt{-\left(1+{8\omega+144\over 15}\right)}\over2}\ln s\right)
 \over s^{1\over2}},
 \end{align*}
 from which one can choose a sequence $\{s_n\}$ such that $ \lim_{s_n\to0^+}\left|{\hat\Phi_0(s_n)\over s_n}\right|=\infty$, which is a contradiction. Thus,  $-3ki$ is not an eigenvalue of  $\mathcal{L}_{\omega,k}$ for $k\neq0$ and $\omega\in(-\infty,-18)$.

 If $\omega\in(72,\infty)$, then by Lemma \ref{tilde Psi0 does not change sign} (2),  $-3ki$ is not an eigenvalue of  $\mathcal{L}_{\omega,k}$ for $k\neq0$. Now, we prove that if $\omega\in(72,\infty)\setminus\{{15(j^2-m^2)+144\over2}\big|j\geq m, 0\leq m\leq3\}$, then $12ki$ is not an eigenvalue of  $\mathcal{L}_{\omega,k}$ for $k\neq0$.
 Suppose that there exists $k_1\neq0$ such that $12k_1i$ is an eigenvalue of  $\mathcal{L}_{\omega,k_1}$ with an eigenfunction $\Delta_{k_1}\hat \Phi_1$. Then
 \begin{align*}
  &((1-s^2)\hat\Phi_1')'-{k_1^2\over 1-s^2}\hat\Phi_1-{2\omega-144\over 15(s^2-1)}\hat\Phi_1
  =((1-s^2)\hat\Phi_1')'+{-k_1^2-{144-2\omega\over 15}\over 1-s^2}\hat\Phi_1=-12\hat\Phi_1,
 \end{align*}
where $
 \Delta_{k_1}\hat\Phi_1\in L^2(-1,1)$. Then there exists $0\leq m_1\leq 3$ such that $-k_1^2-{144-2\omega\over 15}=-m_1^2$. Thus, $|k_1|\geq|m_1|$ and $\omega={15(k_1^2-m_1^2)+144\over2}\in \{{15(j^2-m^2)+144\over2}\big|j\geq m, 0\leq m\leq3\}$, which is a contradiction.
\end{proof}

\subsection{Proof of rigidity of travelling waves  near the $3$-jet}
In this subsection, we prove Theorem $\ref{Rigidity of nearby traveling waves}$. 

\begin{proof}
First, we prove Theorem $\ref{Rigidity of nearby traveling waves}$ $(1)$. The proof is almost the same with Theorem 4.2 in \cite{csz}, and we mainly point out the difference.
Suppose that for any  $n>0$, there exists a travelling wave  $\Psi_n(\varphi-c_nt,s)$
satisfying
\begin{equation*}
\left\lVert\Delta\Psi_n-\Delta\Psi_0\right\rVert_{H_{p}^{2}(\mathbb{S}^2)}
+\left\lVert\partial_{s}\Psi_n-\Psi_{0}'\right\rVert_{C^0(\mathbb{S}^2)}\leq\frac{1}{n},
\end{equation*}
and $c_n\in (-\infty,-12-\delta]\cup[3+\delta,\infty)$, $c_n\neq-\omega$, but $\partial_{\varphi}\Psi_n\not\equiv0$. Moreover,
\begin{align*}
-\partial_{s}(c_ns+\Psi_n)\partial_{\varphi}(\Delta\Psi_n-2c_ns)+\partial_{\varphi}(\Psi_n+c_ns)\partial_{s}(\Delta\Psi_n-2c_ns+2(\omega+c_n)s)=0.
\end{align*}
Thus, $\Delta\Psi_n-2c_ns$ is a steady  solution of ${\rm(\mathcal{E}_{\omega+c_n})}$. Since $c_n+\omega\neq0$, by Theorem 4 in \cite{cg22} (iii) we have $\partial_\varphi\Psi_n\in X$, which is defined in \eqref{def-space-X}.
We normalize $\partial_{\varphi}\Psi_n$ by
$\xi_n=\frac{\partial_{\varphi}\Psi_n}{\left\lVert\partial_{\varphi}\Psi_n\right\rVert_{L^{2}(\mathbb{S}^2)}}$
so that $\left\lVert\xi_n\right\rVert_{L^{2}(\mathbb{S}^2)}=1$. Then $\xi_n\in X$. Similar to (4.17) in \cite{csz}, $\{\xi_n\}$ has a uniform $H_2^2(\mathbb{S}^2)$ bound. Thus, $\xi_n\rightharpoonup\xi_0$ in $H_2^2(\mathbb{S}^2)$ for some $\xi_0\in H_2^2(\mathbb{S}^2)\cap X$ and $\xi_{0,0}=0$ for the $0$'th Fourier mode.
Similar to Cases 1-2 in the proof of Theorem 4.2 of \cite{csz}, there exists $k_0\neq0$ such that $\mathcal{L}_{\omega,k_0}|_{X^{k_0}}$ has an imaginary eigenvalue located outside   $\sigma_e(\mathcal{L}_{\omega,k_0}|_{X^{k_0}})=ik_0 Ran(\Psi_0')=ik_0[-3,12]$.
However, by the spectral property of the linearized operator  $\mathcal{L}_{\omega,k}|_{X^k}$ in  Lemma $\ref{spectra of the linearized operatorLrigidity}$ $({\rm{i}})$, there are no imaginary eigenvalues of $\mathcal{L}_{\omega,k}|_{X^k}$ located outside   $\sigma_e(\mathcal{L}_{\omega,k}|_{X^{k}})=ik Ran(\Psi_0')=ik[-3,12]$   for $k\neq0$ and $\omega\in(-3,{69\over2})$. This is a contradiction.

Next, we prove Theorem $\ref{Rigidity of nearby traveling waves}$ $(2)$. Suppose that for any  $n>0$, there exists a travelling wave  $\Psi_n(\varphi-c_nt,s)$
satisfying
\begin{align}\label{Psi-Psi0distance-Hp2-C0}
\left\lVert\Delta\Psi_n-\Delta\Psi_0\right\rVert_{H_{p}^{2}(\mathbb{S}^2)}
+\sum_{j=1}^2\left\lVert\partial_{s}^j(\Psi_n-\Psi_{0})\right\rVert_{C^0(\mathbb{S}^2)}\leq{1\over n}
\end{align}
and $c_n\in [-12+\delta,3-\delta]$, but $\partial_{\varphi}\Psi_n\not\equiv0$. The travelling wave $\Psi_n(\varphi-c_nt,s)$ solves
\begin{equation}\label{travelling wave equation Psi-n}
-(c_n+\partial_{s}\Psi_n)\partial_{\varphi}\Delta\Psi_n+\partial_{\varphi}\Psi_n(\partial_{s}\Delta\Psi_n+2\omega)=0.
\end{equation}
For $n$ large enough so that ${1\over n}<{1\over 32}\delta$, by \eqref{Psi-Psi0distance-Hp2-C0} we have
\begin{align*}
\left|-\partial_s\Psi_n\left(\varphi,{\sqrt{\delta}\over 4}\right)-\left(-\Psi_0'\left({\sqrt{\delta}\over 4}\right)\right)\right|&=\left|-\partial_s\Psi_n\left(\varphi,{\sqrt{\delta}\over 4}\right)-\left(3-{15\over16}\delta\right)\right|<{1\over 32}\delta\end{align*}
and
\begin{align*}
\left|-\partial_s\Psi_n\left(\varphi,1-\delta^2\right)-\left(-\Psi_0'\left(1-\delta^2\right)\right)\right|&
=\left|-\partial_s\Psi_n\left(\varphi,1-\delta^2\right)-(-12+30\delta^2-15\delta^4)\right|<{1\over 32}\delta
\end{align*}
for $\varphi\in\mathbb{T}_{2\pi}$. Then
\begin{align*}
-\partial_s\Psi_n\left(\varphi,{\sqrt{\delta}\over 4}\right)>3-{15\over16}\delta-{1\over 32}\delta>3-\delta,
\end{align*}
and
\begin{align*}
-\partial_s\Psi_n\left(\varphi,1-\delta^2\right)<-12+30\delta^2-15\delta^4+{1\over 32}\delta<-12+\delta
\end{align*}
for $\varphi\in\mathbb{T}_{2\pi}$ and  $\delta>0$ small enough. Then
\begin{align*}
-\partial_s\Psi_n\left(\varphi,1-\delta^2\right)<-12+\delta\leq c_n\leq 3-\delta<-\partial_s\Psi_n\left(\varphi,{\sqrt{\delta}\over 4}\right)
\end{align*}
for $\varphi\in\mathbb{T}_{2\pi}$. By \eqref{Psi-Psi0distance-Hp2-C0}, we have $\|-\partial_s^2\Psi_n-(-\Psi_{0}'')\|_{C^0(\mathbb{S}^2)}=\|-\partial_s^2\Psi_n-(-30s)\|_{C^0(\mathbb{S}^2)}\leq{1\over n}<{1\over 32}\delta$.
Then  we have
\begin{align*}
-\partial_s^2\Psi_n(\varphi,s)\leq -30s+{\delta\over 32}\leq -{15\sqrt{\delta}\over 2}+{\delta\over 32}<0
\end{align*}
for $s\in\left[{\sqrt{\delta}\over 4},1-\delta^2\right]$, $\varphi\in\mathbb{T}_{2\pi}$ and $\delta>0$ small enough. Thus, for any $\varphi\in\mathbb{T}_{2\pi}$ and for  $n$ large enough, there exists a unique $s_{n,1}(\varphi)\in \left[{\sqrt{\delta}\over 4},1-\delta^2\right]$ such that $-\partial_s\Psi_n(\varphi,s_{n,1}(\varphi))=c_n$.
Similarly, for any $\varphi\in\mathbb{T}_{2\pi}$ and for $n$ large enough, there exists a unique $s_{n,2}(\varphi)\in \left[-1+\delta^2,-{\sqrt{\delta}\over 4}\right]$ such that $-\partial_s\Psi_n(\varphi,s_{n,2}(\varphi))=c_n$. Now, we divide $\mathbb{T}_{2\pi}\times [-1,1]$ into five parts:
\begin{align*}&D_1=\mathbb{T}_{2\pi}\times [-1,-1+\delta^2], D_2=\mathbb{T}_{2\pi}\times \left[-1+\delta^2,-{\sqrt{\delta}\over 4}\right],
D_3=\mathbb{T}_{2\pi}\times \left[-{\sqrt{\delta}\over 4},{\sqrt{\delta}\over 4}\right], \\
&D_4=\mathbb{T}_{2\pi}\times \left[{\sqrt{\delta}\over 4},1-\delta^2\right], D_5=\mathbb{T}_{2\pi}\times [1-\delta^2,1].
\end{align*}
We normalize $\partial_{\varphi}\Psi_n$ by
$\xi_n=\frac{\partial_{\varphi}\Psi_n}{\left\lVert\partial_{\varphi}\Psi_n\right\rVert_{L^{2}(\mathbb{S}^2)}}$
so that $\left\lVert\xi_n\right\rVert_{L^{2}(\mathbb{S}^2)}=1$. By \eqref{travelling wave equation Psi-n}, we have
\begin{equation}\label{travelling wave xin equation}
-(c_n+\partial_{s}\Psi_n)\Delta\xi_n+\xi_n(\partial_{s}\Delta\Psi_n+2\omega)=0.
\end{equation}
By Theorem 2.7 in \cite{Hebey2000}, $H_{3}^{2}(\mathbb{S}^2)$ is embedded in $C^{1}(\mathbb{S}^2)$. So
\begin{align}\label{DeltaPsinC1}
\left\lVert\Delta\Psi_n+2\omega s\right\rVert_{C^{1}(\mathbb{S}^2)}&\leq\left\lVert\Delta\Psi_n\right\rVert_{C^{1}(\mathbb{S}^2)}+C
\leq\left\lVert\Delta\Psi_n-\Delta\Psi_0\right\rVert_{C^{1}(\mathbb{S}^2)}+\left\lVert\Delta\Psi_0\right\rVert_{C^{1}(\mathbb{S}^2)}+C\\\nonumber
&\leq C\left\lVert\Delta\Psi_n-\Delta\Psi_0\right\rVert_{H_{3}^{2}(\mathbb{S}^2)}+C\leq\frac{C}{n}+C\leq C.
\end{align}

First, we estimate $\|\Delta\xi_n\|_{L^2(D_3)}$.  Since $-\Psi_0'(s)=-15s^2+3\in\left[3-{15\over 16}\delta,3\right]$ for $s\in \left[-{\sqrt{\delta}\over 4},{\sqrt{\delta}\over 4}\right]$, we infer from \eqref{Psi-Psi0distance-Hp2-C0} that
\begin{align*}
-\partial_{s}\Psi_n(\varphi,s)>-\Psi_0'(s)-{1\over 32}\delta\geq3-{15\over 16}\delta-{1\over 32}\delta=3-{31\over 32}\delta
\end{align*}
for $(\varphi,s)\in D_3$ and for $n$ large enough so that ${1\over n}<{1\over 32}\delta$. On the other hand, $c_n\leq 3-\delta$. Thus,
\begin{align*}
|\partial_s\Psi_n(\varphi,s)+c_n|=-\partial_s\Psi_n(\varphi,s)-c_n\geq3-{31\over 32}\delta-3+\delta={1\over 32}\delta
\end{align*}
for $(\varphi,s)\in D_3$ and for $n$ large enough so that ${1\over n}<{1\over 32}\delta$. Then
\begin{align}\label{DeltaxinD3}
\|\Delta\xi_n\|_{L^2(D_3)}=\left\|{\xi_n(\partial_s\Delta\Psi_n+2\omega)\over \partial_s\Psi_n+c_n} \right\|_{L^2(D_3)}\leq {32C\over \delta}\left\|{\xi_n\over \sqrt{1-s^2}}\right\|_{L^2(D_3)}\leq{C_\delta}\left\|\xi_n\right\|_{L^2(\mathbb{S}^2)}
\end{align}
for $n$ large enough so that ${1\over n}<{1\over 32}\delta$.

Next, we estimate $\|\Delta\xi_n\|_{L^2(D_1)}$ and $\|\Delta\xi_n\|_{L^2(D_5)}$. Since $-\Psi_0'(s)=-15s^2+3\in[-12, -15(-1+\delta^2)^2+3]$ for $s\in \left[-1,-1+\delta^2\right]$,  by \eqref{Psi-Psi0distance-Hp2-C0} we have
\begin{align*}
-\partial_{s}\Psi_n(\varphi,s)<-\Psi_0'(s)+{1\over 32}\delta\leq-12+30\delta^2-15\delta^4+{1\over 32}\delta< -12+{1\over 16}\delta
\end{align*}
for $(\varphi,s)\in D_1$,   $\delta>0$ small enough so that $30\delta^2-15\delta^4<{1\over 32}\delta$ and  $n$ large enough so that ${1\over n}<{1\over 32}\delta$.
On the other hand, $c_n\geq -12+\delta$. Thus,
\begin{align*}
|\partial_s\Psi_n(\varphi,s)+c_n|=c_n-(-\partial_s\Psi_n(\varphi,s))>-12+\delta-\left(-12+{1\over 16}\delta\right)={15\over 16}\delta
\end{align*}
for $(\varphi,s)\in D_1$,  $\delta>0$ small enough  and $n$ large enough as above. Then
\begin{align}\label{DeltaxinD1}
\|\Delta\xi_n\|_{L^2(D_1)}=&\left\|{\xi_n(\partial_s\Delta\Psi_n+2\omega)\over \partial_s\Psi_n+c_n} \right\|_{L^2(D_1)}\leq {16C\over 15\delta}\left\|{\xi_n\over \sqrt{1-s^2}}\right\|_{L^2(\mathbb{S}^2)}\\\nonumber
\leq&{C_\delta}\left\|{\partial_\varphi\xi_n\over \sqrt{1-s^2}}\right\|_{L^2(\mathbb{S}^2)}\leq C_\delta\|\nabla\xi_n\|_{L^2(\mathbb{S}^2)}\leq  C_\delta\|\xi_n\|_{H_1^2(\mathbb{S}^2)}
\end{align}
for $\delta>0$ small enough and $n$ large enough as above. Similarly, we have
\begin{align}\label{DeltaxinD5}
\|\Delta\xi_n\|_{L^2(D_5)}\leq  C_\delta\|\xi_n\|_{H_1^2(\mathbb{S}^2)}
\end{align}
for $\delta>0$ small enough and $n$ large enough.

Then we estimate $\|\Delta\xi_n\|_{L^2(D_2)}$ and $\|\Delta\xi_n\|_{L^2(D_4)}$. Note that $\partial_{s}\Delta\Psi_n(\varphi,s)+2\omega\neq0$ for $(\varphi,s)\in \mathbb{T}_{2\pi}\times [-1,1]$.
 For $\varphi\in\mathbb{T}_{2\pi}$, since $c_n+\partial_s\Psi_n(\varphi,s_{n,j}(\varphi))=0$,  by \eqref{travelling wave xin equation} we have
$\xi_n(\varphi,s_{n,j}(\varphi))=0$, where $j=1,2$.
By \eqref{Psi-Psi0distance-Hp2-C0} and ${1\over n}<{1\over 32}\delta$, we have $|-\partial_s^2\Psi_n(\varphi,s)-(-\Psi_0''(s))|<{1\over32}\delta$ for $(\varphi,s)\in D_2$. Thus,
\begin{align*}
-\partial_s^2\Psi_n(\varphi,s)>-\Psi_0''(s)-{1\over 32}\delta=-30s-{1\over 32}\delta>{15\over2}\sqrt{\delta}-{1\over 32}\delta>\sqrt{\delta}
\end{align*}
for $\delta>0$ small enough, $n$ large enough and $(\varphi,s)\in D_2$.
Then there exists $\tilde s(\varphi)\in\left[-1+\delta^2,-{\sqrt{\delta}\over 4}\right]$ such that
\begin{align*}
\left|{s-s_{n,2}(\varphi)\over -\partial_s\Psi_n(\varphi,s)-c_n}\right|=&\left|{s-s_{n,2}(\varphi)\over -\partial_s\Psi_n(\varphi,s)-(-\partial_s\Psi_n(\varphi,s_{n,2}(\varphi)))}\right|=\left|{s-s_{n,2}(\varphi)\over -\partial_s^2\Psi_n(\varphi,\tilde s(\varphi))(s-s_{n,2}(\varphi))}\right|\\
=&\left|{1\over -\partial_s^2\Psi_n(\varphi,\tilde s(\varphi))}\right|<{1\over \sqrt{\delta}}
\end{align*}
for $\delta>0$ small enough, $n$ large enough and  $(\varphi,s)\in D_2$.
Then by Hardy's inequality, we have
\begin{align}\label{DeltaxinD2}
\|\Delta\xi_n\|_{L^2(D_2)}=&\left\|{\xi_n(\partial_s\Delta\Psi_n+2\omega)\over \partial_s\Psi_n+c_n} \right\|_{L^2(D_2)}\leq C\left\|{\xi_n\over \sqrt{1-s^2}(\partial_s\Psi_n+c_n)}\right\|_{L^2(D_2)}\\\nonumber
\leq&{C_\delta}\left\|{\xi_n\over \partial_s\Psi_n+c_n}\right\|_{L^2(D_2)}\leq C_\delta\left\|{\xi_n\over s-s_{n,2}(\varphi)}{s-s_{n,2}(\varphi)\over -\partial_s\Psi_n-c_n}\right\|_{L^2(D_2)}\\\nonumber
\leq &C_\delta\left\|{\xi_n\over s-s_{n,2}(\varphi)}\right\|_{L^2(D_2)}=C_\delta\left(\int_0^{2\pi}\left(\int_{-1+\delta^2}^{-{\sqrt{\delta}\over 4}}{\xi_n^2\over (s-s_{n,2}(\varphi))^2}ds\right)d\varphi\right)^{1\over2}\\\nonumber
\leq& C_\delta\left(\int_0^{2\pi}\int_{-1+\delta^2}^{-{\sqrt{\delta}\over 4}}{|\partial_s\xi_n|^2}dsd\varphi\right)^{1\over2}
\leq C_\delta\|\nabla\xi_n\|_{L^2(\mathbb{S}^2)}\leq  C_\delta\|\xi_n\|_{H_1^2(\mathbb{S}^2)}
\end{align}
for $\delta>0$ small enough and $n$ large enough. Similarly,
\begin{align}\label{DeltaxinD4}
\|\Delta\xi_n\|_{L^2(D_4)}\leq  C_\delta\|\xi_n\|_{H_1^2(\mathbb{S}^2)}
\end{align}
for $\delta>0$ small enough and $n$ large enough. Combining \eqref{DeltaxinD3}-\eqref{DeltaxinD4}, we have
\begin{align*}
\|\xi_n\|_{H_2^2(\mathbb{S}^2)}^2\leq\|\Delta\xi_n\|_{L^2(\mathbb{S}^2)}^2\leq  C_\delta\|\xi_n\|_{H_1^2(\mathbb{S}^2)}^2\leq C_\delta\|\xi_n\|_{H_2^2(\mathbb{S}^2)}\|\xi_n\|_{L^2(\mathbb{S}^2)},
\end{align*}
and thus we obtain the uniform $H_2^2(\mathbb{S}^2)$ bound
\begin{align}\label{travelling wave xin uniform bound H2}
\|\xi_n\|_{H_2^2(\mathbb{S}^2)}\leq C_\delta
\end{align}
for $n$ large enough. Up to a subsequence,  there exists $\xi_0\in H_{2}^{2}(\mathbb{S}^2)$ such that $\xi_n\rightharpoonup\xi_0$ in $H_{2}^{2}(\mathbb{S}^2)$, $\xi_n\rightarrow\xi_0$ in $H_{1}^{2}(\mathbb{S}^2)$ and $\left\lVert\xi_0\right\rVert_{L^{2}(\mathbb{S}^2)}=1$.
By \eqref{travelling wave xin equation} we have
\begin{equation}\label{travelling wave-n}
\int_{\mathbb{S}^2}
-(c_n+\partial_{s}\Psi_n)\Delta\xi_n\Phi d\sigma_g+\int_{\mathbb{S}^2}\xi_n(\partial_{s}\Delta\Psi_n+2\omega)\Phi
 d\sigma_g=0
\end{equation}
for any $\Phi\in L^{2}(\mathbb{S}^2)$.
Up to a subsequence, $c_n\to c_0\in[-12+\delta,3-\delta]$. Note that
\begin{align*}
&\int_{\mathbb{S}^2}
\left(-(c_n+\partial_{s}\Psi_n)\Delta\xi_n\Phi+(c_0+\partial_{s}\Psi_0)\Delta\xi_0\Phi\right) d\sigma_g\\\nonumber
=&\int_{\mathbb{S}^2}(-(c_n+\partial_{s}\Psi_n)+(c_0+\partial_{s}\Psi_0))\Delta\xi_n\Phi d\sigma_g+\int_{\mathbb{S}^2}(-(c_0+\partial_{s}\Psi_0))(\Delta\xi_n-\Delta\xi_0)\Phi d\sigma_g.
\end{align*}
By \eqref{Psi-Psi0distance-Hp2-C0},  $c_n+\partial_{s}\Psi_n\to c_0+\partial_{s}\Psi_0$ in $C^0(\mathbb{S}^2)$. This, along with \eqref{travelling wave xin uniform bound H2}, implies that for any $\varepsilon>0$, we have
\begin{align*}
\left|\int_{\mathbb{S}^2}(-(c_n+\partial_{s}\Psi_n)+(c_0+\partial_{s}\Psi_0))\Delta\xi_n\Phi d\sigma_g\right|\leq & \varepsilon\|\Delta\xi_n\|_{L^2(\mathbb{S}^2)}\|\Phi\|_{L^2(\mathbb{S}^2)}\\
\leq&\varepsilon\|\xi_n\|_{H_2^2(\mathbb{S}^2)}\|\Phi\|_{L^2(\mathbb{S}^2)}\leq \varepsilon C_\delta\|\Phi\|_{L^2(\mathbb{S}^2)}
\end{align*}
when $n$ large enough. Moreover, $\Delta\xi_n\rightharpoonup\Delta\xi_0$ in $L^2(\mathbb{S}^2)$. Then
\begin{align}\label{travelling wave-n1}
&\int_{\mathbb{S}^2}
-(c_n+\partial_{s}\Psi_n)\Delta\xi_n\Phi d\sigma_g\to\int_{\mathbb{S}^2}-(c_0+\partial_{s}\Psi_0)\Delta\xi_0\Phi d\sigma_g.
\end{align}
On the other hand, noting that
for  $n\geq1$,
\[\xi_{n,0}=\frac{1}{2\pi}\int_{0}^{2\pi}\xi_nd\varphi=\frac{1}{2\pi}\int_{0}^{2\pi}\frac{\partial_{\varphi}\Psi_n}
{\left\lVert\partial_{\varphi}\Psi_n\right\rVert_{L^{2}(\mathbb{S}^2)}}d\varphi=0,\]
and
$\left\lVert\xi_{n,0}-\xi_{0,0}\right\rVert_{L^{2}(\mathbb{S}^2)}^2
\leq\left\lVert\xi_{n}-\xi_{0}\right\rVert_{L^{2}(\mathbb{S}^2)}^{2}\to0$
as $n\to+\infty$, we have $\xi_{0,0}=0$.
This, along with \eqref{DeltaPsinC1} and the fact that $H_{3}^{2}(\mathbb{S}^2)$ is embedded in $C^{1}(\mathbb{S}^2)$, implies
\begin{align}\label{travelling wave-n2}
&\left|\int_{\mathbb{S}^2}(\xi_n(\partial_{s}\Delta\Psi_n+2\omega)-\xi_0(\partial_{s}\Delta\Psi_0+2\omega))\Phi d\sigma_g\right|\\\nonumber
\leq&\left|\int_{\mathbb{S}^2}{\xi_n-\xi_0\over\sqrt{1-s^2}}\sqrt{1-s^2}(\partial_s\Delta\Psi_n+2\omega)\Phi d\sigma_g\right|\\\nonumber
&+\left|\int_{\mathbb{S}^2}{\xi_0\over\sqrt{1-s^2}}\Phi \sqrt{1-s^2}\partial_s(\Delta\Psi_n-\Delta\Psi_0) d\sigma_g\right|\\\nonumber
\leq&C\left\|{\xi_n-\xi_0\over\sqrt{1-s^2}}\right\|_{L^2(\mathbb{S}^2)}\left\|\Phi\right\|_{L^2(\mathbb{S}^2)}
+C\left\|{\xi_0\over\sqrt{1-s^2}}\right\|_{L^2(\mathbb{S}^2)}\|\Phi\|_{L^2(\mathbb{S}^2)}\|\Delta\Psi_n-\Delta\Psi_0\|_{C^1(\mathbb{S}^2)}\\\nonumber
\leq&C\left\|{\partial_\varphi(\xi_n-\xi_0)\over\sqrt{1-s^2}}\right\|_{L^2(\mathbb{S}^2)}\left\|\Phi\right\|_{L^2(\mathbb{S}^2)}
+C\left\|{\partial_\varphi\xi_0\over\sqrt{1-s^2}}\right\|_{L^2(\mathbb{S}^2)}\|\Phi\|_{L^2(\mathbb{S}^2)}\|\Delta\Psi_n-\Delta\Psi_0\|_{H_3^2(\mathbb{S}^2)}\\\nonumber
\leq&C\left\|\xi_n-\xi_0\right\|_{H_1^2(\mathbb{S}^2)}\left\|\Phi\right\|_{L^2(\mathbb{S}^2)}
+C\left\|{\xi_0}\right\|_{H_1^2(\mathbb{S}^2)}\|\Phi\|_{L^2(\mathbb{S}^2)}\|\Delta\Psi_n-\Delta\Psi_0\|_{H_3^2(\mathbb{S}^2)}\\\nonumber
\to&0
\end{align}
as $n\to \infty$. Combining \eqref{travelling wave-n}-\eqref{travelling wave-n2}, we have
\begin{equation*}
\int_{\mathbb{S}^2}
\left(-(c_0+\partial_{s}\Psi_0)\Delta\xi_0+\xi_0(\partial_{s}\Delta\Psi_0+2\omega)\right)\Phi d\sigma_g=0.
\end{equation*}
By the arbitrary choice of $\Phi\in L^2(\mathbb{S}^2)$, pointwisely  we have
\begin{align*}
-(c_0+\Psi_0')\Delta\xi_0+\xi_0(\Delta\Psi_0'+2\omega)=0.
\end{align*}
Since $\xi_0\neq0$, $\xi_0\in H_2^2(\mathbb{S}^2)$  and $\xi_{0,0}=0$, there exists $k_0\neq0$ such that $\xi_{0,k_0}\neq0$ satisfies $\Delta_{k_0}\xi_{0,k_0}\in L^2(-1,1)$ and
\begin{align}\label{travelling wave Rayleigh eq}
-(c_0+\Psi_0')\Delta_{k_0}\xi_{0,k_0}+\xi_{0,k_0}(\Delta\Psi_0'+2\omega)=0 \quad \text{on}\quad (-1,1).
\end{align}
Since $c_0\in Ran(-\Psi_0')^\circ=(-12,3)$, there exist two points $s_{1}\in(-1,0)$ and $s_{2}\in(0,1)$ solving $c_0+\Psi_0'(s_{i})=0$ for $i=1,2$.
Then
\begin{align}\label{travelling wave Rayleigh eq subinterval}
-\Delta_{k_0}\xi_{0,k_0}+{\Delta\Psi_0'+2\omega\over  \Psi_0'+c_0}
\xi_{0,k_0}=0\quad \text{on}\quad (-1,s_1)\cup(s_1,s_2)\cup(s_2,1).
\end{align}

We claim that there exists $1\leq i_0\leq 2$ such that $\xi_{0,k_0}(s_{i_0})\neq0$.
Suppose that $\xi_{0,k_0}(s_{1})=\xi_{0,k_0}(s_{2})=0$. For $\omega\in(-\infty,-18)$, since  $\xi_{0,k_0}(s_{1})=\xi_{0,k_0}(s_{2})=0$, by \eqref{travelling wave Rayleigh eq subinterval} on $(s_1,s_2)$  we have
\begin{align*}
\int_{s_1}^{s_2}\left(|\nabla_{k_0}\xi_{0,k_0}|^2+{\Delta\Psi_0'+2\omega\over  \Psi_0'+c_0}
|\xi_{0,k_0}|^2\right) ds=0.
\end{align*}
Since $\Delta\Psi_0'+2\omega<0$ and $\Psi_0'+c_0<0$, we have $\xi_{0,k_0}=0$ on $(s_1,s_2)$. For $\omega\in(72,\infty)$, since  $\xi_{0,k_0}(s_{1})=\xi_{0,k_0}(s_{2})=0$, by \eqref{travelling wave Rayleigh eq subinterval} on $(-1,s_1)\cup(s_2,1)$  we have
\begin{align*}
\left(\int_{-1}^{s_1}+\int_{s_2}^{1}\right)\left(|\nabla_{k_0}\xi_{0,k_0}|^2+{\Delta\Psi_0'+2\omega\over  \Psi_0'+c_0}
|\xi_{0,k_0}|^2\right) ds=0.
\end{align*}
Then $\Delta\Psi_0'+2\omega>0$ and $\Psi_0'+c_0>0$ imply $\xi_{0,k_0}=0$ on $(-1,s_1)\cup(s_2,1)$.
Note that $\xi_{0,k_0}\in C^1(-1,1)$ due to $\Delta_{k_0}\xi_{0,k_0}\in L^2(-1,1)$. Since $\xi_{0,k_0}=0$ on $(s_1,s_2)$ for $\omega\in(-\infty,-18)$ and $\xi_{0,k_0}=0$ on $(-1,s_1)\cup(s_2,1)$ for $\omega\in(72,\infty)$, we have $\xi_{0,k_0}'(s_{1})=\xi_{0,k_0}'(s_{2})=0$ for $\omega\in(-\infty,-18)\cup(72,\infty)$. By Lemma \ref{initial ode}, we have
$\xi_{0,k_0}=0$ on $(-1,1)$, which is a contradiction.

Since $\Delta\Psi_0'(s_{i_0})+2\omega\neq0$, $\xi_{0,k_0}(s_{i_0})\neq0$ and $\Psi_0'(s_{i_0})+c_0=0$,
by \eqref{travelling wave Rayleigh eq} we have $\Delta_{k_0}\xi_{0,k_0}\notin L^2(s_{i_0},s_{i_0}+\delta_0)$ for $\delta_0>0$ small enough, which is a contradiction.
\end{proof}

\begin{Remark} 
Based on the types of imaginary eigenvalues of the linearized operators,  the  rigidity in Theorem \ref{Rigidity of nearby traveling waves} might be improved. Precisely,
by Lemma $\ref{spectra of the linearized operatorLrigidity}$
$({\rm{i}})$, the imaginary eigenvalues of the linearized operator  $\mathcal{L}_{\omega,k}|_{X^k}$ have to be in the interior of $\sigma_e(\mathcal{L}_{\omega,k}|_{X^{k}})$  for $k\neq0$ and $\omega\in(-3,{69\over2})$. At the nonlinear level, this suggests that the  rigidity in Theorem $\ref{Rigidity of nearby traveling waves}$ $(1)$ might be improved  in the sense that
any nearby unidirectional travelling wave $($in the norm \eqref{travelling wave norm1}$)$
must be a zonal flow.
 By Lemma $\ref{spectra of the linearized operatorLrigidity}$
$({\rm{iv}})$, the imaginary eigenvalues of  $\mathcal{L}_{\omega,k}|_{X^k}$ are outside  $\sigma_e(\mathcal{L}_{\omega,k}|_{X^{k}})$  for $k\neq0$ and for almost all  $\omega\in(-\infty,-18)\cup(72,\infty)$. The  rigidity in Theorem $\ref{Rigidity of nearby traveling waves}$ $(2)$ might be improved in the sense that
any nearby cat's eyes travelling wave $($in the norm \eqref{travelling wave norm2}$)$
must be a zonal flow.
 These two possible improvements certainly require more sophisticated analysis to deal with the endpoints $-12$ and $3$ of $Ran(-\Psi_0')$.
In addition, the norms \eqref{travelling wave norm1} and \eqref{travelling wave norm2} might be improved to be  optimal.  

For $\omega\in(12,{69\over2})$, noting that $\pm\omega$ is an imaginary isolated eigenvalue of $\sigma(\mathcal{L}_{\omega,\pm1})$ which does not come from the space $X$ but from $E_1+E_3$ in Remark \ref{E1JL exact eigenvalues-rem}, one may construct unidirectional travelling waves with travelling speeds $c=-\omega$ near the $3$-jet.
If one takes this into consideration, the number ${69\over2}$ in Theorem $\ref{travelling wave solutions cat eye unidirectional thm}$ $(1)$-$(2)$ might be replaced by $12$.
\end{Remark}

\noindent
{\bf Acknowledgement} This research was supported by the Austrian Science Fund (FWF) [grant number Z 387-N] and by the National Natural Science Foundation of China [grant number 12494544 and 
grant number 12471229].

\end{CJK*}
\end{document}